\documentclass{amsart}
\usepackage{amssymb,mathrsfs}
\usepackage{bm}

\usepackage{xy}
\input{xy}
\xyoption{all}

\usepackage[hypertex]{hyperref}

\usepackage{threeparttable}

\usepackage{xspace}

\usepackage[geometry]{ifsym}

\makeatletter
 
 \@addtoreset{equation}{subsubsection}

 \@addtoreset{subsubsection}{section}
\makeatother

\newtheoremstyle{step}
 {}
 {}
 {}
 {}
 {\bf}
 {:}
 {0.5em}
 {}
\newcounter{steps}
\theoremstyle{step}
\newtheorem{step}[steps]{Step}

\newtheoremstyle{thm}
 {}
 {}
 {\itshape}
 {}
 {\bf}
 {. ---}
 {0.5em}
 {}

\newtheoremstyle{dfn}
 {}
 {}
 {}
 {}
 {\bf}
 {. {---}}
 {0.5em}
 {}

\swapnumbers
\theoremstyle{thm}
\newtheorem{thm}[subsubsection]{Theorem}
\newtheorem{lem}[subsubsection]{Lemma}
\newtheorem*{lem*}{Lemma}
\newtheorem{cor}[subsubsection]{Corollary}
\newtheorem*{cor*}{Corollary}
\newtheorem{prop}[subsubsection]{Proposition}
\newtheorem*{prop*}{Proposition}

\newtheorem*{conj*}{Conjecture}
\newtheorem*{thm*}{Theorem}

\theoremstyle{dfn}
\newtheorem{dfn}[subsubsection]{Definition}
\newtheorem*{dfn*}{Definition}

\newtheorem*{ex*}{Example}
\newtheorem{rem}[subsubsection]{Remark}
\newtheorem*{rem*}{Remark}
\newtheorem*{cl}{Claim}

\renewcommand{\qedsymbol}{$\blacksquare$}
\newtheorem{prob}[subsubsection]{Question}

\newcommand{\shom}{\mathop{\mc{H}om}\nolimits}

\usepackage{color}
\newenvironment{meta}{
\noindent \color{red}
\sffamily[}{\upshape]}

\newsavebox{\circlebox}
\savebox{\circlebox}{\fontencoding{OMS}\selectfont\char13}
\newlength{\circleboxwdht}
\newcommand{\ccirc}[1]{
  \setlength{\circleboxwdht}{\wd\circlebox}
  \addtolength{\circleboxwdht}{\dp\circlebox}
  \raisebox{0.2\dp\circlebox}{
    \parbox[][\circleboxwdht][c]{\wd\circlebox}
    {\centering\scriptsize #1}}
  \llap{\usebox{\circlebox}}
}

\newcommand{\fsch}[1]{{\ms{#1}}}

\newcommand{\st}[1]{\mf{#1}}
\renewcommand{\H}{\ms{H}}

\newcommand{\cH}{{}^\mr{c}\!\ms{H}}
\newcommand{\dcH}{{}^\mr{dc}\!\ms{H}}

\newcommand{\Chtpr}[2]{\overline{\mr{Cht}^{#1}_{#2}}^{\prime}}
\newcommand{\Chtprop}[2]{\overline{\mr{Cht}^{#1}_{#2}}}
\newcommand{\Ch}{\mr{Cht}}
\newcommand{\Gr}{\mr{Gr}}
\newcommand{\IH}{\mc{IH}}
\newcommand{\WI}{W^{\mr{isoc}}}
\newcommand{\cadm}{\mf{S}_{\mr{adm}}}
\newcommand{\base}{\blacktriangle}
\newcommand{\algsp}[1]{\mc{#1}}
\newcommand{\twolim}{\underleftrightarrow{\lim}}
\newcommand{\cc}[1]{(\!(#1)\!)}
\newcommand{\dd}[1]{[\![#1]\!]}
\newcommand{\eqtr}{\overset{\mr{Tr}}{=}}

\newcommand{\mr}[1]{\mathrm{#1}}
\newcommand{\ms}[1]{\mathscr{#1}}
\newcommand{\mc}[1]{\mathcal{#1}}
\newcommand{\mb}[1]{\mathbb{#1}}
\newcommand{\mf}[1]{\mathfrak{#1}}

\newcommand{\Ddag}[1]{\mathscr{D}^\dag_{#1}}
\newcommand{\DdagQ}[1]{\mathscr{D}^\dag_{{#1},\mathbb{Q}}}
\newcommand{\Dcomp}[2]{\widehat{\mathscr{D}}^{(#1)}_{#2}}

\newcommand{\Dmod}[2]{{\mathscr{D}}^{(#1)}_{#2}}

\newcommand{\invlim}{\mathop{\underleftarrow{\mathrm{lim}}}}
\newcommand{\indlim}{\mathop{\underrightarrow{\mathrm{lim}}}}

\newcommand{\stack}{good stack\xspace}
\newcommand{\stacks}{good stacks\xspace}
\newcommand{\presentation}{good presentation\xspace}

\newcounter{fundprop}
\newcounter{indcat}

\begin{document}
\title[Langlands correspondence for isocrystals]
{Langlands correspondence for isocrystals and the existence of
crystalline companions for curves}
\author{Tomoyuki Abe}
\address{Kavli Institute for the Physics and Mathematics of the Universe
(WPI), University of Tokyo,
5-1-5 Kashiwanoha, Kashiwa, Chiba, 277-8583, Japan}
\email{tomoyuki.abe@ipmu.jp}
\thanks{This work is supported by Grant-in-Aid for Research Activity
Start-up 23840006, Grant-in-Aid for Young Scientists (B) 25800004, and
Grant-in-Aid for Young Scientists (A) 16H05993.}
\begin{abstract}
 In this paper, we show the Langlands correspondence for isocrystals on
 curves, which asserts the existence of crystalline companions in the
 case of curves. For the proof, we generalize the theory of arithmetic
 $\mathscr{D}$-modules to algebraic stacks whose diagonal morphisms are
 finite. Finally, combining with methods of Deligne and Drinfeld, we show
 the existence of an ``$\ell$-adic companion'' for any isocrystal on a
 smooth scheme of any dimension under the assumption of a Bertini type
 conjecture.
\end{abstract}

\maketitle

\setcounter{tocdepth}{2}
\tableofcontents

\section*{Introduction}
The Weil conjectures were finally proven by P. Deligne in the 70's and
culminated in the theory of weights for $\ell$-adic cohomology in his
celebrated paper \cite{De}.
In that paper, Deligne made a following conjecture on the existence of
``compatible systems'':

\begin{conj*}[{\cite[1.2.10]{De}}]
 Soient $X$ normal connexe de type fini sur $\mb{F}_p$, et $\ms{F}$ un
 faisceau lisse irr\'{e}ducible dont le d\'{e}terminant est d\'{e}fini
 par un caract\`{e}re d'ordre fini du groupe fondamental.

 (ii) Il existe un corps de nombres $E\subset\overline{\mb{Q}}_{\ell}$
 tel que le polyn\^{o}me $\det(1- F_xt,\ms{F})$ pour $x\in |X|$, soit
 \`{a} coefficients dans $E$.

 (v) Pour $E$ convenable (peut-\^{e}tre plus grand qu'en (ii)), et
 chaque place non archim\'{e}dienne $\lambda$ premi\`{e}re \`{a} $p$, il
 existe un $E_{\lambda}$-faisceau compatible \`{a} $\ms{F}$ (m\^{e}mes
 valeurs propres des Frobenius).

 (vi) Pour $\lambda$ divisant $p$, on esp\`{e}re des petits camarades
 cristallins.
\end{conj*}
Part (vi) is written vaguely because a good theory of $p$-adic
cohomology was not available at the time Deligne conjectured it.
R. Crew made this conjecture more precise in \cite[4.13]{Cr} after
P. Berthelot's foundational works in $p$-adic cohomology theory.
This conjecture have been one of the driving forces to develop a
$p$-adic cohomology theory over fields of positive characteristic
parallel to the $\ell$-adic cohomology theory ({\it
e.g.}\ introduction of \cite{Ch}).

When $X$ is a curve, all parts of the the conjecture except for (vi) are
consequences of the Langlands correspondence, which was proven by
V. Drinfeld in the rank 2 case and by L. Lafforgue in the higher rank
case. Moreover, Deligne and Drinfeld proved all parts of the conjecture
except for (vi) for any smooth scheme $X$ as a consequence of Langlands
correspondence. In this paper, we prove part (vi) of the conjecture when
$X$ is a curve.
In fact, we prove a stronger result: a correspondence between
irreducible overconvergent $F$-isocrystals with finite determinant on an
open dense subscheme of $X$ and cuspidal automorphic representations of
the function field of $X$ with finite central character (see Theorem
\ref{main}). Finally, in Theorem \ref{convpcc}, we prove the converse of
Deligne's conjecture when $X$ is smooth using the techniques of
Deligne and Drinfeld in \cite{EK} and \cite{Dr} assuming a Bertini type
conjecture \ref{Bertiniconj}: for any overconvergent
$F$-isocrystal over a smooth scheme, there exists an $\ell$-adic
companion for any $\ell\neq p$.

Our strategy of proof is similar to the $\ell$-adic case.
First, we apply the ``product formula for epsilon factors'', which was
proven in the $p$-adic setting by the author together with A. Marmora
in \cite{AM}. By using Deligne's ``{\it principe de r\'{e}currence}'',
the product formula for epsilon factors reduces one to associating an
isocrystal to a cuspidal automorphic representation
(cf.\ \cite{ALang}).
Finally, we use the moduli spaces of ``shtukas'' to establish the
Langlands correspondence for isocrystals as was done by Drinfeld and
Lafforgue in the $\ell$-adic case.
In order to apply the methods of Drinfeld and Lafforgue, we construct a
six functor formalism for suitable $p$-adic cohomology theory for
certain algebraic stacks.
\medskip

Before explaining our construction, let us review the history of
attempts to construct a six functor formalism in the $p$-adic setting.
For more detailed overview of the history, we refer the reader to
\cite{I}, \cite{Kesurv}.
The first $p$-adic cohomology defined for an arbitrary separated scheme
of finite type over a perfect field of characteristic $p$ was proposed by
Berthelot around the 80's (called rigid cohomology).
He also defined a theory of coefficients for rigid cohomology, called
overconvergent $F$-isocrystals: these can be seen as a $p$-adic analogue
of vector bundles with an integrable connection.
One can see from this analogy that it is not reasonable to expect a six
functor formalism in the style of
A. Grothendieck in the framework of overconvergent $F$-isocrystals.
In order to remedy this, and in the analogy with the complex situation,
Berthelot introduced the theory of arithmetic $\ms{D}$-modules. We refer
to \cite{BerInt} for a beautiful survey by the founder himself.
In the $p$-adic case, the theory is very complicated since we need to
deal with differential operators
of infinite order unlike the complex situation.
As a result, many of the foundational properties had been left as
conjectures.
Among these conjectures, the most important one concerns the
preservation of finiteness properties of the arithmetic $\ms{D}$-modules
under various cohomological operations.
A big step toward this problem was the introduction of
overholonomic modules by D. Caro, which potentially bypasses Berthelot's
original strategy to construct the six functor formalism.
His work was successful in proving stability for most of the
standard cohomological operations, but the finiteness of
overconvergent $F$-isocrystals was still unresolved. A
breakthrough was achieved by K. S. Kedlaya in his resolution of Shiho's
conjecture, or the proof of the semistable reduction theorem \cite{KeSS}.
This extremely powerful theorem enabled us to answer many tough
questions in the theory of arithmetic $\ms{D}$-modules: a finiteness
result by Caro and N. Tsuzuki (cf.\ \cite{KT}), an analogue of Weil II
by Caro and the author (cf.\ \cite{AC}), and many more.
Even though we do not explicitly use Kedlaya's
theorem in the proof of the Langlands correspondence for isocrystals,
the main theorem of this paper can be seen as another application of
Kedlaya's result.
In previous works (\cite{KT, AC} {\it etc.}), Kedlaya's result was used
to develop a theory of arithmetic $\ms{D}$-modules for ``realizable
schemes''. We refer to \S\ref{revDmodtheory} for detailed overview.
In particular, quasi-projective schemes are included
in this framework. However, to construct the isocrystals corresponding
to cuspidal automorphic representations the category of quasi-projective
schemes is too restricted. A large part of this paper is devoted to
construct a theory of arithmetic $\ms{D}$-modules for ``admissible
stacks''. In particular, we end our search, since Monsky and Washnitzer,
for a $p$-adic six functor formalism for separated schemes
of finite type over a perfect field.

Our construction of a six functor formalism is more or less formal: making
full use of the existence of the formalism in local situation,
we glue. Even though we do not axiomatize, it can be carried
out for any cohomology theory over a field admitting a reasonable six
functor formalism locally. First, let us explain the construction in the
case of schemes. As we have already mentioned, for ``realizable schemes''
({\it e.g.}\ quasi-projective schemes)
we already have the formalism thanks to works of Caro and others.
For a realizable scheme $X$, we denote by
$D^{\mr{b}}_{\mr{hol}}(X)$ the associated triangulated category with
t-structure, and by $\mr{Hol}(X)$ its heart. The category $\mr{Hol}(X)$
is analogous to the category of perverse sheaves in the philosophy of the
Riemann-Hilbert correspondence.
When $X$ is a scheme of finite type over $k$, we are able to take
a finite open covering $\{U_i\}$ by realizable schemes, and define
$\mr{Hol}(X)$ by gluing $\bigl\{\mr{Hol}(U_i)\bigr\}$.
The first difficulty is to define the
derived category. A starting point of our construction is an analogy with
Beilinson's equivalence proven in \cite{AC2}
\begin{equation*}
 D^{\mr{b}}(\mr{Hol}(X))\xrightarrow{\sim} D^{\mr{b}}_{\mr{hol}}(X),
\end{equation*}
where $X$ is a realizable scheme.
This equivalence suggests to define
the derived category naively by $D^{\mr{b}}(\mr{Hol}(X))$ for general
scheme $X$. The next problem is to construct the cohomological
functors. To do this, we construct cohomological functors for finite
morphisms and projections separately and combine these for the
general case.
Let us explain the method in the easier case where $f\colon X\rightarrow
Y$ is a finite morphism between realizable schemes.
In that case the push-forward $f_+$ is an exact functor, so we can
define $f_+\colon\mr{Hol}(X)\rightarrow\mr{Hol}(Y)$  by gluing, and we
consider the associated derived functor to get the functor between
derived categories. The
definition of $f^!$ is more technical. When the morphism $f$ is between
realizable schemes, by using some general non-sense, we are able to show
that $f^!$ is the right derived functor of $\H^0f^!$ using the fact that
$(f_+,f^!)$ is an adjoint pair and $f_+$ is exact. Thus, for a general
finite morphism $f$, we define $\H^0f^!$ by gluing, and define the
functor between derived categories by taking the right derived
functor. Since the category $\mr{Hol}(X)$ does {\em not} possess enough
injectives, we need techniques of Ind-categories to overcome this
deficit. Even though the construction is more involved, we can define
the cohomological functors for projections
$X\times Y\rightarrow Y$ using similar ideas. Since any
morphism between separated schemes of finite type can be factored into
a closed immersion and a projection, we may define the cohomological
operations in complete generality by composition.

For an algebraic stack $\st{X}$, we use a simplicial technique to
construct the derived category $D^{\mr{b}}_{\mr{hol}}(\st{X})$:
take a presentation $X\rightarrow\st{X}$, and we consider the
simplicial algebraic space $X_\bullet:=\mr{cosk}_0(X\rightarrow
\st{X})$. The derived category of $\st{X}$ should coincide with that
of $X_\bullet$ with suitable conditions on the cohomology.
Since $\ms{D}$-modules behave like perverse sheaves, there are minor
differences with the analogous construction in the $\ell$-adic setting
(cf.\ \cite{OL}). However, the construction is mostly parallel.
Now, we would like to construct the cohomological operations for
algebraic stacks in a manner similarly to that for schemes described
above. However, in general
morphisms between algebraic stacks cannot be written as a composition of
finite morphisms and projections.
In this paper, we restrict our attention to admissible stacks, {\it
i.e.}\ algebraic stacks whose diagonal morphisms are finite.

This formalism is especially used to show the $\ell$-independence of
the traces of actions of correspondences on cohomology groups.
This is then used to calculate the traces of elements of certain Hecke
algebra acting of cohomology groups of the moduli spaces of shtukas.
We refer to \S\ref{langcorresstate} for a more detailed explanation of
the proof of the main theorem.
\bigskip

Let us overview the organization of this paper. We begin with collecting
known results concerning arithmetic $\ms{D}$-modules in
\S\ref{revDmodtheory}, and the subsection contains few new facts.
In \S\ref{Indcatsection}, we show some elementary properties of Ind
categories. In \S\ref{constt-struct}, we introduce a t-structure
corresponding to constructible sheaves in the spirit of the
Riemann-Hilbert correspondences.
This t-structure is useful when we construct various
types of trace maps. In the arithmetic $\ms{D}$-module theory, the
coefficient categories are $K$-additive where $K$ is a complete discrete
valuation field whose residue field is $k$. However, for the Langlands
correspondence it is convenient to work with
$\overline{\mb{Q}}_p$-coefficients. Passing from $K$-coefficients to
$\overline{K}$-coefficients is rather formal, and some generality
is developed in \S\ref{extofsclarsub}. In \S\ref{consttracsubsec}, we
conclude the first section by constructing the trace maps for flat
morphisms in the style of SGA 4.
This foundational property had been lacking in the theory of arithmetic
$\ms{D}$-modules, and it plays an important role in the proof of
$\ell$-independence type theorem, which is the main theme in
\S\ref{lindepsec}.

In \S\ref{arithDstacksec}, we develop a theory for algebraic stacks.
Most of the properties used in this section are formal in the six
functor formalism, and almost no knowledge of arithmetic
$\ms{D}$-modules is required. In \S\ref{Dmodforstack} we define the
triangulated category of holonomic complexes for algebraic stacks. Some
cohomological operations for algebraic stacks are introduced in
\S\ref{cohfunctorstack}. In \S\ref{sixfuncadmstsec}, we restrict our
attention to so called ``admissible stacks''. Any morphism between
admissible stacks can be factorized into morphisms which has already
been treated in \S\ref{cohfunctorstack}, and our
construction of six functor formalism for this type of stacks
completes. We show basic properties of the operations in this
subsection. The final subsection \S\ref{miscstsubsec} is complementary,
in which we collect some facts which is needed in the proof of the
Langlands correspondence.

In \S\ref{lindepsec}, we show an $\ell$-independence type theorem of the
trace of the action of a correspondence on cohomology groups. With the
trace formalism developed in \S\ref{consttracsubsec}, even though there
are some differences since we are dealing with algebraic stacks, our
task is to translate the proof of \cite{KS} in our language.

In the final section \S\ref{langcorrsec}, we show the Langlands
correspondence. In order to be friendly to readers who are only
interested in \S\ref{langcorrsec}, the section begins with some review
of $p$-adic theory as well as recalling some notations of this paper.
We state the main theorem and explain the idea of the
proof in the second subsection. The actual proof is written in the third
subsection, and we conclude the paper with some well-known applications.

\subsection*{Acknowledgment}
This project started from a comment of T. Saito suggesting the
possibility of proving the {\em petits camarades} conjecture when
the author explained his work with Marmora to Saito. This work would
have never appeared without his advice, and the author is greatly
indebted to him.

The author would like to thank D. Caro, O. Gabber, L. Lafforgue,
A. Shiho, N. Tsuzuki, S. Yasuda for various discussions.
He also thanks Y. Mieda for spending a lot of time for the author,
checking the proofs in the seminar, and Y. Toda for valuable discussions
in which the idea of using Ind categories came up.
He is also grateful to D. Patel for numerous linguistic advises
improving the paper, and A. Abbes and S. Saito for continuous supports
and encouragements.
Last but not least, he thanks three referees of Journal of AMS for
reading the manuscript very carefully, and giving him a lot of
invaluable comments.

\subsection*{Conventions and notations}
\subsubsection{}
In this paper, we usually use Roman fonts ({\it e.g.}\ $X$) for
schemes, script fonts ({\it e.g.}\ $\fsch{X}$) for formal schemes, and
Gothic fonts ({\it e.g.}\ $\st{X}$) for algebraic stacks. When we write
$(-)^{(\prime)}$ it means ``$(-)$ (resp.\ $(-)'$)''. Throughout this
paper, we fix a prime number $p$. When a discrete valuation field $K$ is
fixed and its residue field is finite, we often denote by
$\overline{\mb{Q}}_p$ an algebraic closure of $K$.
Throughout this paper, we fix a universe $\mb{U}$.

\subsubsection{}
\label{smpresentnotasur}
For the terminologies of algebraic stacks, we follow
\cite{LM}. Especially, any scheme, algebraic space, or algebraic
stack is assumed quasi-separated.
For an algebraic stack $\st{X}$, we denote by $\st{X}_{\mr{sm}}$ the
category of affine schemes over $\st{X}$ such that the structural
morphism $X\rightarrow\st{X}$ is smooth. Morphisms between
$X,Y\in\st{X}_{\mr{sm}}$ are smooth morphisms $X\rightarrow Y$ over
$\st{X}$.
Recall that a {\em presentation} of $\st{X}$ is a smooth surjective
morphism $\algsp{X}\rightarrow\st{X}$ from an algebraic space
$\algsp{X}$. Finite morphisms or universal homeomorphisms between
algebraic stacks are always assumed representable.

\subsubsection{}
\label{reldimfunc}
When $d\geq0$ is an integer, smooth morphisms between algebraic stacks
of relative dimension $d$ are understood to be equidimensional.
Let $P\colon\algsp{X}\rightarrow\st{X}$ be a smooth morphism from an
algebraic space to an algebraic stack. Then the continuous function
$\dim(P)\colon\algsp{X}\rightarrow\mb{N}$ is defined in
\cite[(11.14)]{LM}. This function is called the
{\em relative dimension of $P$}, and is sometimes denoted by
$d_{\algsp{X}/{\st{X}}}$.

\subsubsection{}
\label{tateshiftconv}
Let $X$ be a topological space, and let $\pi_0(X)$ be the set of
connected components of $X$. Let $d\colon\pi_0(X)\rightarrow\mb{Z}$ be a
map. For any connected component $Y$ of $X$, assume that a category
$\mc{C}_Y$ endowed with an auto-functor
$T_Y\colon\mc{C}_Y\xrightarrow{\sim}\mc{C}_Y$ is attached,
and $\mc{C}_X\cong\prod_{Y\in\pi_0(X)}\mc{C}_Y$ via
which $T_X$ is identified with $(T_Y)_{Y\in\pi_0(X)}$. For
$M\in\mc{C}_X$, we define $T^d(M)$ as follows: Let
$M=(M_Y)_{Y\in\pi_0(X)}\in\prod\mc{C}_Y$. Then
$T^d(M):=(T^{d(Y)}M_Y)_Y$. We may take the auto-functor $T$ to be the
shift functor or the Tate twist functor. When $T=[1]$ (resp.\ $T=(1)$),
the functor $T^d$ is denoted by $[d]$ (resp.\ $(d)$).

\section{Preliminaries}
\label{firstsection}

\subsection{Review of arithmetic $\ms{D}$-modules}
\label{revDmodtheory}
Let us recall the status of the theory of arithmetic $\ms{D}$-modules
briefly. Let $s$ be a positive integer, and put $q:=p^s$. Let $R$ be a
complete discrete valuation ring whose residue
field, which is assumed to be a perfect field of characteristic $p$, is
denoted by $k$. Put $K:=\mr{Frac}(R)$. We moreover assume that the
$s$-th absolute Frobenius homomorphism $\sigma\colon k\xrightarrow{\sim}
k$ sending $x$ to $x^{q}$ lifts to an automorphism
$R\xrightarrow{\sim}R$ also denoted by $\sigma$.

\begin{dfn}[{\cite[1.1.3]{AC}}]
 \label{defofrealcat}
 A scheme over $k$ is said to be {\em realizable} if it can be embedded
 into a proper smooth formal scheme over $\mr{Spf}(R)$. We denote by
 $\mr{Real}(k/R)$ the full subcategory of the category of $k$-schemes
 $\mr{Sch}(k)$ consisting of realizable schemes.
\end{dfn}

For a realizable scheme $X$, the triangulated category
of {\em holonomic complexes} $D^{\mr{b}}_{\mr{hol}}(X/K)$, endowed with
a t-structure, is defined.
Let us recall the construction. Let $\fsch{P}$ be a proper smooth formal
scheme over $\mr{Spf}(R)$. Then the category of overholonomic
$\DdagQ{\fsch{P}}$-modules (without Frobenius structure) is defined
by Caro in \cite{Caovhol}. We denote by $\mr{Hol}(\fsch{P})$ its thick
full subcategory generated by overholonomic $\DdagQ{\fsch{P}}$-modules
which can be endowed with $s'$-th Frobenius structure
for some positive integer $s'$ divisible by $s$ (but we do not consider
Frobenius structure). The objects of
$\mr{Hol}(\fsch{P})$ are called {\em holonomic modules}.
By definition, $D^{\mr{b}}_{\mr{hol}}(\DdagQ{\fsch{P}})$ is
the full subcategory of $D^{\mr{b}}(\DdagQ{\fsch{P}})$ whose cohomology
complexes are holonomic. Of course, this subcategory is triangulated by
\cite[13.2.7]{KSc}.

Now, let $X\hookrightarrow\fsch{P}$ be an embedding into a proper smooth
formal scheme, whose existence is assured since $X$ is a realizable
scheme. Then $D^{\mr{b}}_{\mr{hol}}(X/K)$ is the subcategory of
$D^{\mr{b}}_{\mr{hol}}(\fsch{P})$ which is supported on $X$. This
category does not depend on the choice of the embedding up to canonical
equivalence, and well-defined.
Moreover, the t-structure is compatible with this
equivalence (cf.\ \cite[1.2.8]{AC}).
The heart of the triangulated category is denoted by $\mr{Hol}(X/K)$.
For further details of this category, one can refer to
\cite[\S1.1, \S1.2]{AC} and \cite[\S1]{AC2}. In \cite{AC2},
$\mr{Hol}(X/K)$ and $D^{\mr{b}}_{\mr{hol}}(X/K)$ are denoted by
$\mr{Hol}_F(X/K)$ and $D^{\mr{b}}_{\mr{hol},F}(X/K)$ respectively.

\begin{rem*}
 In \cite{AC2}, the category $\mr{Hol}_F(X/K)$ is introduced using
 the category of ``overholonomic modules
 {\em after any base change}'', whereas, here, we simply used the
 category of overholonomic modules to define $\mr{Hol}(X/K)$.
 Since overholonomic modules with Frobenius structure are overholonomic
 modules after any base change by \cite[1.2]{AC2}, the categories
 $\mr{Hol}_F(X/K)$ in \cite{AC2} and $\mr{Hol}(X/K)$ defined above are
 the same. However, to prove that $\mr{Hol}(X)$ is a noetherian
 category (cf.\ [{\it ibid.}, 1.5]), it is convenient to work in the
 category of overholonomic modules after any base change.
\end{rem*}

\begin{rem}
 \label{Frobnotdep}
 Lifting $R\xrightarrow{\sim} R$ of the Frobenius automorphism
 of $k$ is not unique in general. Let $\sigma'\colon
 R\xrightarrow{\sim}R$ be another lifting. Let $\fsch{X}$ be a smooth
 formal scheme over $R$.
 We denote by $\fsch{X}^{\sigma^{(\prime)}}:=\fsch{X}\otimes_{R\nearrow
 \sigma^{(\prime)}}R$.
 Locally on $\fsch{X}$, we have the following commutative diagram over
 $\mr{Spf}(R)$:
 \begin{equation*}
  \xymatrix@C=50pt@R=3pt{
   &\fsch{X}^{\sigma}\ar[dd]_{\sim}^{G}\\
  \fsch{X}\ar[rd]_{F'}\ar[ru]^{F}&\\
  &\fsch{X}^{\sigma'},
   }
 \end{equation*}
 where $F$ and $F'$ denote liftings of the relative $s$-th Frobenius.
 For a $\DdagQ{\fsch{X}}$-module $\ms{M}$, we have
 \begin{equation*}
  F^*(\ms{M}^{\sigma})\cong F^*(G^*\ms{M}^{\sigma'})\cong
   F'^*(\ms{M}^{\sigma'}),
 \end{equation*}
 where $\ms{M}^{\sigma^{(\prime)}}$ denotes the
 $\DdagQ{\fsch{X}^{\sigma^{(\prime)}}}$-module defined by changing base
 using $\sigma^{(\prime)}$.
 This shows that endowing $\ms{M}$ with a Frobenius structure with
 respect to $\sigma$ is equivalent to endowing $\ms{M}$ with a Frobenius
 structure with respect to $\sigma'$.
 Thus, our category $\mr{Hol}(X/K)$ does not depend on the choice of
 $\sigma$. However, the category of modules with Frobenius structure
 {\em does} depend on the choice. For
 example, assume $k$ is algebraically closed and consider $\sigma$ and
 $\sigma'$. Put
 \begin{equation*}
  K^{(\prime)}_0:=\bigl\{x\in K\mid \sigma^{(\prime)}(x)=x\bigr\}.
 \end{equation*}
 We may take $\sigma$, $\sigma'$ so that $K_0$ and $K'_0$ are not the
 same. Consider the unit object $K$ in $\mr{Hol}(\mr{Spf}(R))$.
 Endow it with a Frobenius structure $\Phi^{(\prime)}$ with respect to
 $\sigma^{(\prime)}$. Then
 \begin{equation*}
  \mr{Hom}_{F^{(\prime)}\text{-}\mr{Hol}(X/K)}
   \bigl((K,\Phi^{(\prime)}),(K,\Phi^{(\prime)})\bigr)\cong
   K_0^{(\prime)}.
 \end{equation*}
 Thus, we do not have an equivalence of categories
 between $F\text{-}\mr{Hol}(X/K)$ and $F'\text{-}\mr{Hol}(X/K)$
 compatible with the forgetful functors to $\mr{Hol}(X/K)$.
\end{rem}

\subsubsection{}
\label{fundproprealsch}
The six functors have already been defined for realizable schemes.
For details one can refer to \cite{AC}, \cite{AC2}.
For the convenience of the reader, we collect known results. Let
$f\colon X\rightarrow Y$ be a morphism in $\mr{Real}(k/R)$. Then
we have the triangulated functors
\begin{alignat*}{2}
 f_!,f_+\colon D^{\mr{b}}_{\mr{hol}}(X/K)\rightarrow
 D^{\mr{b}}_{\mr{hol}}(Y/K),&\qquad
 f^!,f^+\colon D^{\mr{b}}_{\mr{hol}}(Y/K)\rightarrow
 D^{\mr{b}}_{\mr{hol}}(X/K).\\
\end{alignat*}
These functors satisfy the following fundamental properties of six
functor formalism:
\begin{enumerate}
 \item $D^{\mr{b}}_{\mr{hol}}(X/K)$ is a closed symmetric
       monoidal category, namely it is equipped with tensor product $\otimes$
       and the unit object $K_X$ with which it forms a symmetric
       monoidal category (or sometimes called commutative tensor
       category as in \cite[4.2.16]{KSc}),
       and $\otimes$ has the left adjoint
       functor $\shom$. The adjoint functor $\shom$ is denoted sometimes
       by $\shom_X$ for clarification, and it is called the {\em
       internal hom}. (cf.\ \cite[1.1.6, Appendix]{AC})

 \item $f^+$ is monoidal, namely it commutes with $\otimes$ and preserves
       the unit object.
       
 \item Given composable morphisms $f$ and $g$,
       there exists a canonical isomorphism
       $(f\circ g)^+\cong g^+\circ f^+$.
       We associate $D^{\mr{b}}_{\mr{hol}}(X/K)$ to
       $X\in\mr{Real}(k/R)$, and with this pull-back and the canonical
       isomorphisms for compositions,
       we have the fibered category over $\mr{Real}(k/R)$.
       Moreover, we have similar fibered category for $f^!$ as well.
       (cf.\ \cite[1.3.14]{AC}, checking of the category being fibered
       readily follows from the construction of the functor.)

 \item $(f^+,f_+)$ and $(f_!,f^!)$ are adjoint pairs. (cf.\ \cite[1.3.14
       (viii)]{AC})

 \item\label{propshpucoi}
      We have a morphism of functors $f_!\rightarrow
      f_+$ compatible with transitivity isomorphisms of composition.
      This morphism is an isomorphism when $f$ is proper. (cf.\
      \cite[1.3.7, 1.3.14 (vi)]{AC})

 \item\label{opencoinpull}
      When $j$ is an open immersion, there exists the isomorphism
      $j^+\xrightarrow{\sim}j^!$ compatible with transition isomorphism
      of the composition of two open immersions. (cf.\
      \cite[II.3.5]{Vi2})

 \item\label{globalhomdef}
      Let $\mr{Vec}_K$ be the abelian category of $K$-vector
      spaces, and we denote by $D^{\mr{b}}_{\mr{fin}}(\mr{Vec}_K)$ the
      derived category consisting of bounded complexes of $K$-vector
      spaces whose cohomologies are finite dimensional.
      There exists a canonical equivalence of monoidal categories
      $\mb{R}\Gamma\colon
      D^{\mr{b}}_{\mr{hol}}(\mr{Spec}(k)/K)\xrightarrow{\sim}
      D^{\mr{b}}_{\mr{fin}}(\mr{Vec}_K)$. For
      $X\rightarrow\mr{Spec}(k)$ in $\mr{Real}(k/R)$, we put
      $\mb{R}\mr{Hom}(-,-):=\mb{R}\Gamma\circ
      f_+\circ\shom(-,-)$.
      Note that we have an isomorphism
      $\mb{R}^i\mr{Hom}(\ms{F},\ms{G})\cong\mr{Hom}(\ms{F},\ms{G}[i])$
      for $\ms{F},\ms{G}\in D^{\mr{b}}_{\mr{hol}}(X)$.

 \item\label{basechangeprop}
      Consider the following cartesian diagram of schemes:
       \begin{equation}
	\label{cartediagforbc}
	\xymatrix{
	 X'\ar[r]^-{g'}\ar[d]_-{f'}\ar@{}[rd]|\square&
	 X\ar[d]^{f}\\
	Y'\ar[r]_-{g}&Y.}
       \end{equation}
      Assume that the schemes are realizable.
      Then we have a canonical isomorphism $g^+f_!\cong f'_!g'^+$
      compatible with compositions. When $f$ is proper (resp.\ open
      immersion), this isomorphism is the base change homomorphism
      defined by the adjointness of $(f^+,f_+)$ (resp.\ $(f_!,f^!)$)
      via the isomorphism of \eqref{propshpucoi} (resp.\
      \eqref{opencoinpull}). (cf.\ \cite[1.3.14 (vii)]{AC})

 \item\label{projformulas}
      We have a canonical isomorphism $f_!\ms{F}\otimes\ms{G}\cong
      f_!(\ms{F}\otimes f^+\ms{G})$. (cf.\ \cite[Appendix]{AC})

 \item Let $i$ be a closed immersion in $\mr{Real}(k/K)$, and $j$ be the
       open immersion defined by the complement.
       Then we have a canonical distinguished triangle of functors
       \begin{equation*}
	j_!j^!\rightarrow\mr{id}\rightarrow i_+i^+\rightarrow
       \end{equation*}
       where the first and second morphisms are adjunction morphisms.
 \setcounter{fundprop}{\theenumi}
\end{enumerate}

Before recalling several more properties, let us show the following
lemma:
\begin{lem*}
 Let $\iota\colon X\rightarrow X'$ be a universal homeomorphism in
 $\mr{Real}(k/R)$. Then the adjoint pair $(\iota_+,\iota^!)$ induces an
 equivalence between $D^{\mr{b}}_{\mr{hol}}(X/K)$ and
 $D^{\mr{b}}_{\mr{hol}}(X'/K)$, and we have a canonical isomorphism
 $\iota^+\cong\iota^!$. Moreover assume given the following commutative
 diagram where $\iota$ and $\iota'$ are universal homeomorphisms:
 \begin{equation*}
  \xymatrix{
   X\ar[r]^-{f}\ar[d]_{\iota}&Y\ar[d]^{\iota'}\\
  X'\ar[r]_-{f'}&Y'.
   }
 \end{equation*}
 Then $f^{(\prime)}_+$, $f^{(\prime)}_!$, $f^{(\prime)+}$, and
 $f^{(\prime)!}$ commute canonically with
 $\iota^{(\prime)+}\cong\iota^{(\prime)!}$.
\end{lem*}
\begin{proof}
 The first equivalence is nothing but \cite[1.3.12]{AC}.
 Since $\iota_!\xrightarrow{\sim}\iota_+$ by \eqref{propshpucoi} above,
 we have
 $\iota^+\cong(\iota_!)^{-1}\cong(\iota_+)^{-1}\cong\iota^!$.
 Commutation results follows by transitivity of push-forwards and
 pull-backs.
\end{proof}
This result can be applied in particular when $f$ is the relative
Frobenius morphism (see the remark below). We need a few more
properties, which may not be regarded as standard properties of six
functor formalism:
\begin{enumerate}
 \setcounter{enumi}{\thefundprop}
 \item\label{frobdesrecall}
      For $X$ in $\mr{Real}(k/R)$, let $X^\sigma:=X\otimes_{k\nearrow
      \sigma}k$. Then we have a pull-back $\sigma^*\colon
      D^{\mr{b}}_{\mr{hol}}(X/K)\xrightarrow{\sim}
      D^{\mr{b}}_{\mr{hol}}(X^{\sigma}/K)$, which is exact, and all the
      cohomological functors commute canonically with this
      pull-back. (This follows easily from the definition of
      the cohomological functors. See also \cite[4.5]{BeDmod2}.)
 \setcounter{fundprop}{\theenumi}
\end{enumerate}
Now, for a separated scheme of finite type over $k$, we denote by
$\mr{Isoc}^\dag(X/K)$ the thick full subcategory of the category of
overconvergent isocrystals on $X$ generated by those which can be
endowed with $s'$-th Frobenius structure for some $s|s'$. Caution that
the notation is slightly different from the standard one as in
\cite[2.3.6]{Berrigcoh}. In fact, by the notation, Berthelot simply
means the category of overconvergent isocrystals and Frobenius structure
does not play any role in his definition.

\begin{enumerate}
 \setcounter{enumi}{\thefundprop}
 \item\label{carodagdagcat}
      Let $X$ be a realizable scheme such that $X_{\mr{red}}$ is a
      smooth realizable scheme of dimension
      $d\colon\pi_0(X)\rightarrow\mb{N}$ (cf.\ \ref{reldimfunc}). Then
      there exists a fully faithful functor
      $\mr{sp}_+\colon\mr{Isoc}^\dag(X/K)\rightarrow
      \mr{Hol}(X/K)$ called the {\em specialization
      functor}. We denote the essential image\footnote{
      In \cite{Cafai}, the essential image is denoted by
      $\mr{Isoc}^{\dag\dag}(X/K)$.} of $\mr{sp}_+$ shifted by $-d$ by
      $\mr{Sm}(X/K)\subset\mr{Hol}(X/K)[-d]\subset
      D^{\mr{b}}_{\mr{hol}}(X/K)$.
      (cf.\ \cite[4.2.2]{Cafai})
 \end{enumerate}

\begin{rem*}
 (i) Let $F\colon X\rightarrow X$ be the $s$-th absolute Frobenius
 endomorphism. Combining \eqref{frobdesrecall} and the lemma above
 applied to the $s$-th relative Frobenius morphism $F_{X/k}\colon
 X\rightarrow X^{\sigma}$, we get an equivalence of categories
 \begin{equation*}
  F^*:=F_{X/k}^+\circ\sigma^*\colon D^{\mr{b}}_{\mr{hol}}(X/K)
   \rightarrow D^{\mr{b}}_{\mr{hol}}(X/K).
 \end{equation*}
 This pull-back is nothing but the one used in \cite[Definition
 4.5]{BeDmod2}. For a cohomological functor
 $C\colon D^{\mr{b}}_{\mr{hol}}(X/K)\rightarrow
 D^{\mr{b}}_{\mr{hol}}(Y/K)$, we say that $C$ {\em commutes with
 Frobenius pull-back} if there exists a ``canonical'' isomorphism
 $C\circ F^*\cong F^*\circ C$.

 (ii) Let $F\colon X\rightarrow X'$ be the relative Frobenius
 morphism. For $\ms{M}\in D^{\mr{b}}_{\mr{hol}}(X')$, the lemma above
 yields a canonical isomorphism $\alpha\colon
 F_+F^!\ms{M}\xrightarrow{\sim}\ms{M}$. On the
 other hand, when $X$ and $X'$ can be lifted to smooth formal schemes
 $\fsch{X}$, $\fsch{X}'$, \cite[4.2.4]{BeDmod2} gives us another
 isomorphism $\beta\colon F_+F^!\ms{M}\cong\ms{M}$, since
 $F^{\flat}\Ddag{\fsch{X'}}\cong
 F^*\Ddag{\fsch{X}'}\otimes\omega_{\fsch{X}/\fsch{X}'}$ by
 \cite[2.4.4]{BeDmod2}.
 These two isomorphisms coincide. Indeed, to see this, it suffices to
 check the coincidence for $\ms{M}=\mc{O}_{\fsch{X}'}$. By left-to-right
 conversion, it suffices to check for $\ms{M}=\omega_{\fsch{X}'}$
 and $F_+$, $F^!$ the corresponding functors for right modules.
 The homomorphism $\alpha$ is nothing but the adjunction map, and this
 is by definition Virrion's trace map
 $\mr{Tr}^{\mr{Vir}}\colon F_*\bigl(\omega_{\fsch{X}}
 \otimes_{\Ddag{\fsch{X}}}F^*\Ddag{\fsch{X}'}\bigr)
 \rightarrow\omega_{\fsch{X}'}$ defined in \cite[III.7.1]{Vir}. By
 \cite[III.5.4]{Vir}, this map can be characterized as the unique map
 $\gamma$ such that the following diagram commutes:
 \begin{equation*}
  \xymatrix@C=60pt@R=10pt{
   F_*\omega_{\fsch{X}}\ar@{->>}[r]\ar[d]_{\mr{Tr}_F}&
   F_*\bigl(\omega_{\fsch{X}}
   \otimes_{\Ddag{\fsch{X}}}F^*\Ddag{\fsch{X}'}\bigr),
   \ar[ld]^-{\gamma}\\
  \omega_{\fsch{X}'}&
   }
 \end{equation*}
 where $\mr{Tr}_F$ is the homomorphism induced by the trace map of
 Hartshorne \cite[2.4.2]{BeDmod2}.
 Thus, it suffices to check the diagram is commutative when
 $\gamma=\beta$. By construction of the equivalence of
 \cite[4.2.4]{BeDmod2}, $\beta$ is defined by taking limit to
 \cite[(2.5.6.2)]{BeDmod2}. It is not hard to check the commutativity
 using \cite[1.5]{A} and the description of Garnier \cite[(2.2.1)]{A} of
 the isomorphism \cite[2.5.2]{BeDmod2}. The details are left to the
 reader.

 (iii) The identification of (ii) shows that the commutation isomorphism
 $F^!\circ f_+\cong f_+\circ F^!$ defined by the lemma and by
 \cite[4.3.9]{BerInt} coincide.
\end{rem*}

\subsubsection{}
\label{bidualrealcase}
We also have the duality formalism.
Let $p\colon X\rightarrow\mr{Spec}(k)$ be the structural
morphism of a realizable scheme.
We put $K_X^{\omega}:=p^!(K)$ and call it the {\em
dualizing complex}. We put $\mb{D}_X:=\shom(-,K_X^\omega)$, and call it
the {\em dual functor}. For $\ms{F}\in D^{\mr{b}}_{\mr{hol}}(X/K)$, we
have
\begin{align}
 \label{bidualmapdefreal}
 \tag{$\star$}
 \mr{Hom}(\mb{D}_X\ms{F},\mb{D}_X\ms{F})
 &\cong
 \mr{Hom}(\mb{D}_X\ms{F}\otimes\ms{F},K^{\omega}_X)\\
 \notag
 &\cong
 \mr{Hom}(\ms{F}\otimes\mb{D}_X\ms{F},K^{\omega}_X)
 \cong
 \mr{Hom}(\ms{F},\mb{D}_X\mb{D}_X\ms{F}).
\end{align}
The identity in the abelian group on the left hand side induces a
homomorphism $\ms{F}\rightarrow\mb{D}_X\mb{D}_X\ms{F}$.

\begin{lem*}
 The induced homomorphism of functors
 $\mr{id}\rightarrow\mb{D}_X\circ\mb{D}_X$ is an isomorphism.
\end{lem*}
\begin{proof}
 We already know that $\mr{id}\cong\mb{D}_X\circ\mb{D}_X$ by
 \cite[II.3.5]{Vi2}, even though the isomorphism may not be equal to the
 one in the claim.
 Let us show that the given homomorphism in the lemma is actually an
 isomorphism using this Virrion's isomorphism.
 By {\it d\'{e}vissage}, it suffices to check the equivalence for
 holonomic modules. We recall that objects of $\mr{Hol}(X/K)$ have
 finite length by \cite[1.5]{AC2}. Thus, we only need to show the lemma
 for irreducible modules $\ms{F}$. By Virrion's result,
 $\mb{D}_X\mb{D}_X(\ms{F})$ is irreducible as well, and it remains to
 show that the homomorphism is not $0$. If this were $0$, the
 corresponding element of the left side of (\ref{bidualmapdefreal})
 should also be $0$, which is a contradiction.
\end{proof}
This isomorphism induces a canonical isomorphism
$\shom(\ms{F},\ms{G})\cong\mb{D}_X\bigl(\ms{F}\otimes\mb{D}_X
(\ms{G})\bigr)$ for $\ms{F},\ms{G}$ in
$D^{\mr{b}}_{\mr{hol}}(X/K)$. This can be seen by a similar argument to
\cite[V.2.6]{RD}.

\subsubsection{}
\label{dualfunccommuform}
Let $f\colon X\rightarrow Y$ be a morphism between realizable schemes,
and $\ms{F},\ms{F}'\in D^{\mr{b}}_{\mr{hol}}(X)$. Since $f^+$ is
monoidal, we have the canonical isomorphism
$f^+\bigl((-)\otimes(-)\bigr)\cong
f^+(-)\otimes f^+(-)$. By taking the adjoint, we have the homomorphism
$f_+(-)\otimes f_+(-)\rightarrow f_+\bigl((-)\otimes(-)\bigr)$.
This induces a homomorphism
\begin{equation*}
 f_+\shom_X(\ms{F},\ms{F}')\otimes f_+(\ms{F})\rightarrow
  f_+\bigl(\shom_X(\ms{F},\ms{F}')\otimes\ms{F}\bigr)
  \rightarrow
  f_+(\ms{F}'),
\end{equation*}
where the second homomorphism is the evaluation map.
Taking the adjoint, we have the homomorphism
\begin{equation*}
 f_+\shom_X(\ms{F},\ms{F}')\rightarrow
  \shom_Y\bigl(f_+(\ms{F}), f_+(\ms{F}')\bigr).
\end{equation*}
Now, let $\ms{G}\in D^{\mr{b}}_{\mr{hol}}(Y)$.
When $f$ is {\em proper}, \ref{fundproprealsch}~\eqref{propshpucoi} and
the adjointness of the pair $(f_!,f^!)$ induce a homomorphism
\begin{equation}
 \label{adjointproperreal}
  \tag{$\star$}
  f_+\shom_X\bigl(\ms{F},f^!\ms{G}\bigr)\rightarrow
  \shom_Y\bigl(f_+\ms{F},\ms{G}\bigr).
\end{equation}

\begin{prop*}
 This homomorphism is an isomorphism.
\end{prop*}
\begin{proof}
 For an open subscheme $U$ of $X$, we denote by $j_U\colon
 U\hookrightarrow X$ the open immersion. Let
 $p_Z$ be the structural morphism of a scheme $Z$,
 and $\ms{M}\in D^{\mr{b}}_{\mr{hol}}(X)$. Assume
 that $p_{U+}(j_U^+\ms{M})=0$
 for any $U$. Then $\ms{M}=0$. Indeed, for any closed point $i_x\colon
 x\hookrightarrow X$, put $U_x:=X\setminus\{x\}$,
 and we have the localization triangle
 \begin{equation*}
  i_{x}^!(\ms{M})\rightarrow p_{X+}(\ms{M})\rightarrow
   p_{U_x+}j^+_{U_x}(\ms{M})\xrightarrow{+1}.
 \end{equation*}
 By assumption, we have $i_x^!(\ms{M})=0$. By \cite[1.3.11]{AC},
 $\ms{M}=0$.

 By the construction, the homomorphism in question is compatible
 with restriction to open subschemes. Thus, the observation above
 reduces to checking that the induced homomorphism
 $p_{Y+}f_+\shom_X\bigl(\ms{F},f^!\ms{G}\bigr)\rightarrow
 p_{Y+}\shom_Y\bigl(f_+\ms{F},\ms{G}\bigr)$ is isomorphic. 
 Via the isomorphism \ref{fundproprealsch} \eqref{globalhomdef}, this
 homomorphism is nothing but the canonical homomorphism
 $\mr{Hom}_X\bigl(\ms{F},f^!\ms{G}\bigr)\rightarrow
 \mr{Hom}_Y\bigl(f_+\ms{F},\ms{G}\bigr)$ of the adjoint pair
 $(f_+,f^!)$.
\end{proof}

Let $j\colon U\rightarrow X$ be an open immersion. Then there exists a
unique homomorphism
\begin{equation}
 \label{adjointopenreal}\tag{$\star\star$}
 \shom\bigl(j_!\ms{F},\ms{G}\bigr)\rightarrow
  j_+\shom\bigl(\ms{F},j^!\ms{G}\bigr)
\end{equation}
such that its restriction to $U$ is the identity.
This is an isomorphism since we know that these two objects are
isomorphic abstractly by \cite[A.7]{AC}. Let $f\colon
X\rightarrow Y$ be a morphism between realizable schemes. 
This morphism factorizes in $\mr{Real}(k/R)$ as
$X\xrightarrow{j}\overline{X}\xrightarrow{\overline{f}}Y$ where $j$ is
an open immersion and $\overline{f}$ is proper. We define
\begin{equation*}
 f_+\shom_X\bigl(\ms{F},f^!\ms{G}\bigr)
  \xrightarrow[\mbox{(\ref{adjointproperreal})}]{\sim}
  j_+\shom\bigl(\overline{f}_+\ms{F},j^!\ms{G}\bigr)
  \xleftarrow[\mbox{(\ref{adjointopenreal})}]{\sim}
  \shom_Y\bigl(f_!\ms{F},\ms{G}\bigr). 
\end{equation*}
It is standard to check that this does not depend on the choice of the
factorization. Using this isomorphism, we may prove the following two
more isomorphisms, which we record here for future use:

\begin{equation*}
 \shom(\ms{F},f_+\ms{G})\cong
  f_+\shom(f^+\ms{F},\ms{G}),\qquad
  f^!\shom(\ms{F},\ms{G})\cong
  \shom(f^+\ms{F},f^!\ms{G}).
\end{equation*}

\subsubsection{}
\label{compatleftadj}
Consider the cartesian diagram (\ref{cartediagforbc}).
We assume that the schemes are realizable.
Then we define the base change homomorphism $g'^+\circ f^!\rightarrow
f'^!\circ g^+$ to be the adjunction of the following composition:
\begin{equation*}
 f'_!\circ g'^+\circ f^!\xleftarrow{\sim}g^+\circ f_!\circ f^!
  \xrightarrow{\mr{adj}_{f'}}g^+.
\end{equation*}
By definition, the following diagram is commutative, which we will use
later:
\begin{equation*}
 \xymatrix{
  g^+f_!f^!\ar[d]_-{\mr{adj}_f}\ar[r]^{\sim}&
  f'_!g'^+f^!\ar[r]&
  f'_!f'^!g^+\ar[d]^-{\mr{adj}_{f'}}\\
 g^+\ar@{=}[rr]&&g^+.
  } 
\end{equation*}

\subsubsection{}
\label{Kunnethsch}
For realizable schemes $X_1$, $X_2$ and $\ms{M}_1\in
D^{\mr{b}}_{\mr{hol}}(X_1)$, $\ms{M}_2\in D^{\mr{b}}_{\mr{hol}}(X_2)$,
we put $\ms{M}_1\boxtimes\ms{M}_2:=p_1^+(\ms{M}_1)\otimes
p^+_2(\ms{M}_2)$ where $p_i\colon X_1\times X_2\rightarrow X_i$ denotes
the $i$-th projection. This functor is called the {\em exterior tensor
product}.

Now, let $f^{(\prime)}\colon X^{(\prime)}\rightarrow Y^{(\prime)}$ be
a morphism of realizable schemes, and take an object $\ms{M}^{(\prime)}$
in $D^{\mr{b}}_{\mr{hol}}(X^{(\prime)}/K)$. We have the canonical
isomorphism $(f\times f')^+\bigl((-)\boxtimes(-)\bigr)\cong
f^+(-)\boxtimes f'^+(-)$ since $f^+$ and $f'^+$ are monoidal. By taking
the adjoint, we have a homomorphism
\begin{equation*}
 f_+(\ms{M})\boxtimes f'_+(\ms{M}')\rightarrow
  (f\times f')_+\bigl(\ms{M}\boxtimes\ms{M}'\bigr).
\end{equation*}

\begin{prop*}
 This homomorphism is an isomorphism.
\end{prop*}
\begin{proof}
 When $f$ and $f'$ are immersions, the proposition is essentially
 contained in the proof of \cite[1.3.3 (i)]{AC}. Thus, we may assume $f$
 and $f'$ to be smooth proper and $X^{(\prime)}$, $Y^{(\prime)}$ can be
 lifted to proper smooth formal schemes $\fsch{X}^{(\prime)}$,
 $\fsch{Y}^{(\prime)}$.
 In this situation, we have the canonical isomorphism
 $f^!(-)\boxtimes f'^!(-)\cong(f\times
 f')^!\bigl((-)\boxtimes(-)\bigr)$, and by taking the adjunction of
 \cite{Vir}, we have the homomorphism $\rho\colon(f\times
 f')_+\bigl(\ms{M}\boxtimes\ms{M}'\bigr) \rightarrow
 f_+(\ms{M})\boxtimes f'_+(\ms{M}')$. The homomorphism in the
 statement is the dual of this homomorphism. Thus it suffices to show
 that $\rho$ is an isomorphism. Let
 $f_n^{(\prime)}\colon X_n^{(\prime)}\rightarrow Y_n^{(\prime)}$ be a
 smooth morphism of relative dimension $d^{(\prime)}$ between proper
 smooth schemes over $R/\pi^{n+1}$ where $\pi$ is a uniformizer of $R$.
 For perfect $\Dmod{m}{}$-complexes on
 $X^{(\prime)}_n$ and $Y^{(\prime)}_n$, have the homomorphism
 $\rho_n\colon (f_n\times f'_n)_+\bigl((-)\boxtimes(-)\bigr)\rightarrow
 f_{n+}(-)\boxtimes f'_{n+}(-)$ by similar construction to $\rho$, and
 it suffices to show that $\rho_n$ is an isomorphism
 since  $D^{\mr{b}}_{\mr{coh}}(\Dcomp{m}{\fsch{X}})=
 D_{\mr{perf}}(\Dcomp{m}{\fsch{X}})$ by \cite[4.4.8]{BeDmod2}.
 By \cite[VII.4.1]{RD}, the following
 diagram commutes:
 \begin{equation*}
  \xymatrix{
   R^df_*(\omega_{X_n/Y_n})\boxtimes R^{d'}f'_*(\omega_{X'_n/Y'_n})
   \ar[r]\ar[d]&
   \mc{O}_{Y_n}\boxtimes\mc{O}_{Y'_n}\ar[d]^{\sim}\\
  R^{d+d'}(f_n\times f'_n)_*(\omega_{X_n\times X'_n/Y_n\times
   Y'_n})\ar[r]&\mc{O}_{Y_n\times Y'_n}
   }
 \end{equation*}
 where the vertical homomorphisms are trace maps, and $\omega$ denotes
 the canonical bundle sheaf. The commutativity shows that $\rho_n$ is
 nothing but the homomorphism induced by the isomorphism
 \begin{equation*}
  \Dmod{m}{Y_n\times Y'_n\leftarrow X_n\times X'_n}\cong 
   \Dmod{m}{Y_n\leftarrow X_n}\boxtimes\Dmod{m}{Y'_n\leftarrow X'_n}
 \end{equation*}
 (cf.\ \cite[Lemma 4.5 (ii)]{A}), and we get the proposition by using
 the K\"{u}nneth formula for quasi-coherent sheaves.
\end{proof}

\subsubsection{}
\label{Beilequivcat}
Finally, we recall the following result:
\begin{thm*}[\cite{AC2}]
 Let $X$ be a realizable scheme. Then the canonical functor
 $D^{\mr{b}}(\mr{Hol}(X/K))\rightarrow D^{\mr{b}}_{\mr{hol}}(X/K)$
 induces an equivalence of triangulated categories.
\end{thm*}

In the current formalism, cycle class map is missing. We shall construct
trace maps and a cycle class formalism in the coming subsections, which
are important tools to show the $\ell$-independence type
result.

\subsection{Ind-categories}
\label{Indcatsection}
\begin{lem}
 \label{deradjlemm}
 Let $\mc{A}$, $\mc{B}$ be abelian categories, and assume $\mc{A}$ has
 enough injective objects. Let $F\colon\mc{A}\rightarrow\mc{B}$ be a left
 exact functor, and assume that we have an adjoint pair $(G,F)$ such
 that $G$ is {\em exact}. Then for $M\in D^+(\mc{A})$ and $N\in
 D(\mc{B})$, we have
 \begin{equation*}
  \mr{Hom}_{D(\mc{A})}(G(N),M)\cong
   \mr{Hom}_{D(\mc{B})}(N,\mb{R}F(M)).
 \end{equation*}
\end{lem}
\begin{proof}
 By the exactness of $G$, $F$ sends injective objects to injective
 objects. Thus, for a complex of injective objects $I^{\bullet}\in
 C^+(\mc{A})$ and a complex $N^{\bullet}\in C(\mc{B})$, it suffices to
 show that
 \begin{equation*}
  \mr{Hom}_{K(\mc{A})}(G(N^{\bullet}),I^{\bullet})\cong
   \mr{Hom}_{K(\mc{B})}(N^{\bullet},F(I^{\bullet})).
 \end{equation*}
 By the adjointness, we have
 $\mr{Hom}^{\bullet}(G(N^{\bullet}),I^{\bullet})\cong\mr{Hom}^{\bullet}
 (N^{\bullet},F(I^{\bullet}))$ in $C(\mr{Ab})$ where
 $\mr{Hom}^{\bullet}$ is the functor defined in \cite[I.6]{RD}. Since
 $\mr{Hom}_{K(\mc{A})}=\H^0\mr{Hom}^{\bullet}$, we get the isomorphism.
\end{proof}

\begin{rem*}
 The proof also shows that if moreover $\mc{B}$ has enough injectives,
 we have
 \begin{equation*}
  \mb{R}\mr{Hom}_{D(\mc{A})}\bigl(G(N),M\bigr)\cong
   \mb{R}\mr{Hom}_{D(\mc{B})}\bigl(N,\mb{R}F(M)\bigr).
 \end{equation*}
\end{rem*}

\subsubsection{}
\label{indcatrecall}
Let us collect some facts on Ind-categories. Let $\mc{A}$ be a category.
Let $\mc{A}^\wedge$ be the category of presheaves on $\mc{A}$, and
$\mr{h}_{\mc{A}}\colon\mc{A}\rightarrow\mc{A}^\wedge$ be the canonical
embedding. Then $\mr{Ind}(\mc{A})$ is the full subcategory of
$\mc{A}^\wedge$ consisting of objects which can be written as a filtrant
small inductive limit of the image of $\mr{h}_{\mc{A}}$.
By definition, $\mr{h}_{\mc{A}}$ induces a functor
$\iota_{\mc{A}}\colon\mc{A}\rightarrow\mr{Ind}(\mc{A})$. We sometimes
abbreviate this as $\iota$. Since
$\mr{h}_{\mc{A}}$ is fully faithful by the Yoneda lemma
\cite[1.4.4]{KSc}, $\iota_{\mc{A}}$ is fully faithful as well. For the
detail see \cite[\S6]{KSc}.

Now, we assume that $\mc{A}$ is an abelian category.
We have the following properties:

\begin{enumerate}
 \item The category $\mr{Ind}(\mc{A})$ is abelian, and the functor
       $\iota_{\mc{A}}$ is exact. Moreover, $\mr{Ind}(\mc{A})$ admits
       small inductive limits, and small filtrant inductive
       limits are exact. (cf.\ \cite[8.6.5]{KSc})

 \item\label{esssmallgrothcate}
      Assume $\mc{A}$ to be essentially small. Then $\mr{Ind}(\mc{A})$
      is a Grothendieck category, and in particular, it possesses
      enough injectives, and admits small projective limits. (cf.\
      \cite[8.6.5, 9.6.2, 8.3.27]{KSc})

 \item\label{directfactorstillinA}
      The category $\mc{A}$ is a thick subcategory of
       $\mr{Ind}(\mc{A})$ by \cite[8.6.11]{KSc}. This in particular
       shows that any direct factor of objects of $\mc{A}$ is in
       $\mc{A}$, since direct factor is the kernel of a projector.

 \item\label{relofindlim}
      Let $X_\bullet\colon I\rightarrow\mc{A}$ be an inductive
      system. Since $\iota_{\mc{A}}$ is fully faithful, if
      $\indlim\iota_{\mc{A}}(X_i)$ is in the essential image of
      $\iota_{\mc{A}}$, then $\indlim X_i$ exists in $\mc{A}$ and
      $\indlim\iota_{\mc{A}}(X_i)\xrightarrow{\sim}
      \iota_{\mc{A}}(\indlim X_i)$.
 \setcounter{indcat}{\theenumi}
\end{enumerate}

Now, let $F\colon\mc{A}\rightarrow\mc{B}$ be an additive functor between
abelian categories. Then it extends uniquely to an additive
functor $IF\colon\mr{Ind}(\mc{A})\rightarrow\mr{Ind}(\mc{B})$ such that
$IF$ commutes with arbitrary small filtrant inductive limits by
\cite[6.1.9]{KSc}. Since small direct sum can be written as a filtrant
inductive limit of finite sums, $IF$ commutes with small direct sums as
well. We have the following more properties:
\begin{enumerate}
 \setcounter{enumi}{\theindcat}
 \item If $F$ is left (resp.\ right) exact, so is $IF$. (cf.\
       \cite[8.6.8]{KSc})

 \item Let $G\colon\mc{B}\rightarrow\mc{C}$ be another additive functor
       between abelian categories. Then $I(G\circ F)\cong IG\circ IF$.
       (cf.\ \cite[6.1.11]{KSc})
\end{enumerate}
If there is nothing to be confused, by abuse of notation, we denote $IF$
simply by $F$.

\begin{rem*}
 In general $\iota_{\mc{A}}$ does not commute with inductive limits
 (cf.\ \cite[6.1.20]{KSc}), and in \cite{KSc}, inductive limits in
 $\mr{Ind}(\mc{A})$ are distinguished by using $``\indlim"$.
 In this paper, we simply denote this limit by $\indlim$ if no
 confusion can arise, and when we use inductive limits, it is understood
 to be taken in $\mr{Ind}(\mc{A})$, not in $\mc{A}$, unless
 otherwise stated.
\end{rem*}

\begin{lem}
 \label{extenuniqcomfil}
 Let $\mc{A}$, $\mc{B}$ be abelian categories, and assume that $\mc{B}$
 admits small filtrant inductive limits.
 Then the restriction functor yields an equivalence
 $\mr{Fct}^{\mr{il,add}}(\mr{Ind}(\mc{A}),\mc{B})\xrightarrow{\sim}
 \mr{Fct}^{\mr{add}}(\mc{A},\mc{B})$
 where the target (resp.\ source) is the category of additive functors
 (resp.\ of additive functors which commute with small filtrant
 inductive limits).
\end{lem}
\begin{proof}
 This is a reorganization of \cite[6.3.2]{KSc}, or see [SGA 4, I,
 8.7.3]. The quasi-inverse is
 the functor sending $F$ to $\sigma_{\mc{B}}\circ IF$ where
 $\sigma_{\mc{B}}\colon\mr{Ind}(\mc{B})\rightarrow\mc{B}$ is the functor
 taking the inductive limit (cf.\ \cite[6.3.1]{KSc}).
\end{proof}

\subsubsection{}
\label{noethcatindcat}
Let $\mc{A}$, $\mc{B}$ be abelian categories, and assume moreover that
$\mc{A}$ is a {\em noetherian category} ({\it i.e.}\ an essentially
small category whose objects are noetherian cf.\
\cite[II.4]{Ga}).
Let $f_*\colon\mc{A}\rightarrow\mc{B}$ be a left exact functor. Recall
that $\mr{Ind}(\mc{A})$ has enough injectives (cf.\
\ref{indcatrecall} \eqref{esssmallgrothcate}). Thus, $If_*$ can be
derived to get $\mb{R}f_*\colon D^{+}(\mr{Ind}(\mc{A}))\rightarrow
D^+(\mr{Ind}(\mc{B}))$, by abuse of notation. We also recall that the
canonical functor $\iota_{\mc{A}}\colon D^{\mr{b}}(\mc{A})\rightarrow
D^{\mr{b}}_{\mc{A}}(\mr{Ind}(\mc{A}))$ gives an equivalence by
\cite[15.3.1]{KSc}, and the same for $\mc{B}$.

\begin{lem*}
 \label{indlimcommderfun}
 Let $f_*\colon\mc{A}\rightarrow\mc{B}$ be a left exact functor as
 above. Then for any integer $i\geq0$,
 $\mb{R}^if_*\colon\mr{Ind}(\mc{A})\rightarrow\mr{Ind}(\mc{B})$ commutes
 with arbitrary small filtrant inductive limit.
\end{lem*}
\begin{proof}
 Since we are assuming $\mc{A}$ to be noetherian, $\mr{Ind}(\mc{A})$ is
 locally noetherian category (cf.\ \cite[II.4]{Ga}), and by \cite[II.4,
 Cor 1 of Thm 1]{Ga}, small filtrant inductive limits of injective
 objects in $\mr{Ind}(\mc{A})$ remain to be injective. Thus, we may
 apply \cite[15.3.3]{KSc} to conclude the proof.
\end{proof}

\subsubsection{}
\label{rightderfunc}
Recall that a $\delta$-functor $\{f^i\}$ between abelian categories is
called the {\em right satellite of $f^0$} if $f^{i}=0$ for $i<0$, and it
is universal among them (cf.\ \cite[2.2]{Tohoku}).

\begin{lem*}
 The composition functor $\{\mb{R}^if_*\circ\iota\}\colon\mc{A}
 \rightarrow \mr{Ind}(\mc{A})\rightarrow\mr{Ind}(\mc{B})$
 is the right satellite functor of
 $If_*\circ\iota_{\mc{A}}\cong\iota_{\mc{B}}\circ f_*$.
\end{lem*}
\begin{proof}
 Since $\{\mb{R}^if_*\circ\iota\}$ is a right $\delta$-functor, it
 remains to show that it is universal. Let
 $\{G^i\}\colon\mc{A}\rightarrow\mr{Ind}(\mc{B})$ be a right
 $\delta$-functor with a morphism of functors $\iota_{\mc{B}}\circ f_*
 \rightarrow G^0$. By Lemma \ref{extenuniqcomfil}, $\{G^i\}$ extends
 uniquely to a collection of functors
 $\{\widetilde{G}^i\}\colon\mr{Ind}(\mc{A})\rightarrow\mr{Ind}(\mc{B})$.
 Then
 $\{\widetilde{G}^i\}$ is a right $\delta$-functor as well by
 \cite[8.6.6]{KSc}, with a morphism $If_*\rightarrow\widetilde{G}^0$.
 By the universal property of $\{\mb{R}^if_*\}$,
 we get a morphism $\{\mb{R}^if_*\}\rightarrow\{\widetilde{G}^i\}$,
 which induces the morphism
 $\varphi\colon\{\mb{R}^if_*\circ\iota\}\rightarrow\{G^i\}$. Now, any
 morphism $\mb{R}^if_*\circ\iota\rightarrow G^i$ extends uniquely to
 $\mb{R}^if_*\rightarrow\widetilde{G}^i$ by Lemma \ref{indlimcommderfun}
 and \ref{extenuniqcomfil}. Using \cite[8.6.6]{KSc} again,
 $\{\mb{R}^if_*\}\rightarrow\{\widetilde{G}^i\}$ is a morphism of
 $\delta$-functors if $\varphi$ is.
 Thus the uniqueness of $\varphi$ follows, and we conclude that the
 $\delta$-functor in question is universal.
\end{proof}

\begin{lem}
 \label{rightadjisderv}
 Let $f^*\colon\mc{B}\rightarrow\mc{A}$ be an exact functor such
 that $(f^*,f_*)$ is an adjoint pair. Assume given a functor
 $f_+\colon D^{\mr{b}}(\mc{A})\rightarrow D^{\mr{b}}(\mc{B})$ such that
 $(f^*,f_+)$ is an adjoint pair.
 Then $f_+\cong \mb{R}f_*\circ\iota$ on $D^{\mr{b}}(\mc{A})$.
\end{lem}
\begin{proof}
 First, let us show that for $X\in\mc{A}$, $\H^if_+(X)=0$ for $i<0$. If
 $f_+(X)\neq0$, by boundedness condition, there exists an
 integer $d$ such that $\H^df_+(X)\neq0$ and $\H^if_+(X)=0$ for
 $i<d$. Assume $d<0$. Then for any $Y\in\mc{B}$, we have
 \begin{equation*}
  \mr{Hom}_{\mc{B}}\bigl(Y,\H^df_+(X)\bigr)\cong
   \mr{Hom}_{D(\mc{B})}\bigl(Y,f_+(X)[d]\bigr)\cong
   \mr{Hom}_{D(\mc{A})}\bigl(f^*(Y),X[d]\bigr)=0
 \end{equation*}
 where the last equality holds since $\H^i(X[d])=0$ for $i\leq0$.
 This contradicts with the assumption, and thus, $\H^if_+(X)=0$ for
 $i<0$. In the same way, we get that $(f^*,\H^0f_+)$ is
 an adjoint pair, and in particular, $\H^0f_+\cong f_*$.

 This shows that the collection of functors $\{\H^if_+\}$ is a (right)
 $\delta$-functor. Since  $\{\H^if_+\}$ is a $\delta$-functor from
 $\mc{A}$ to $\mc{B}$ with the isomorphism $f_*\cong\H^0f_+$, Lemma
 \ref{rightderfunc} yields a homomorphism
 $\{\mb{R}^if_*\circ\iota\}\rightarrow\{\H^if_+\}$ of
 $\delta$-functors.
 For $X\in\mc{A}$, we have
 \begin{align*}
  \mr{Hom}_{D(\mr{Ind}(\mc{B}))}\bigl(f_+(X),&\mb{R}f_*(X)\bigr)\cong
  \mr{Hom}_{D(\mr{Ind}(\mc{A}))}\bigl(f^*f_+(X),X\bigr)\\
   &\cong\mr{Hom}_{D(\mc{A})}\bigl(f^*f_+(X),X\bigr)\cong
  \mr{Hom}_{D(\mc{B})}\bigl(f_+(X),f_+(X)\bigr)
 \end{align*}
 where we used the canonical equivalence $D^{\mr{b}}(\mc{A})
 \xrightarrow{\sim} D^{\mr{b}}_{\mc{A}}(\mr{Ind}(\mc{A}))$ recalled in
 \ref{noethcatindcat}.
 Thus, the identity of $f_+(X)$ defines a homomorphism
 $\rho\colon f_+(X)\rightarrow \mb{R}f_*(X)$, which induces the
 isomorphism on $\H^0$. By the universal property of satellite functor,
 the composition $\{\mb{R}^if_*(X)\}\rightarrow
 \{\H^if_+(X)\}\xrightarrow{\rho}\{\mb{R}^if_*(X)\}$ is the identity,
 which shows that $\mb{R}^if_*(X)$ is a direct factor of $\H^if_+(X)$.
 This shows that $\mb{R}^if_*(X)$ is in $\mc{B}$ by
 \ref{indcatrecall} \eqref{directfactorstillinA} and $\mb{R}^if_*(X)=0$
 for $i\gg0$, which means that $\mb{R}f_*(X)$ is in
 $D^{\mr{b}}_{\mc{B}}(\mr{Ind}(\mc{B}))
 \xleftarrow{\sim}D^{\mr{b}}(\mc{B})$.
 Thus $\mb{R}f_*$ induces a functor from $D^{\mr{b}}(\mc{A})$ to
 $D^{\mr{b}}(\mc{B})$. For any $Y\in D^{\mr{b}}(\mc{B})$, we have
 \begin{equation*}
  \mr{Hom}_{D(\mc{B})}(Y,\mb{R}f_*(X))\cong
   \mr{Hom}_{D(\mc{A})}(f^*(Y),X)\cong
   \mr{Hom}_{D(\mc{B})}(Y,f_+(X)).
 \end{equation*}
 Thus $\mb{R}f_*(X)\xrightarrow{\sim}f_+(X)$ by the Yoneda lemma as
 required.
\end{proof}

\subsubsection{}
\label{smallnoeth}
Now, let us apply the preceding general results to the theory of
arithmetic $\ms{D}$-modules. We use the notation of
\ref{revDmodtheory}. First, we need:
\begin{lem*}
 For a realizable scheme $X$, the category $\mr{Hol}(X/K)$ is
 a noetherian and artinian.
\end{lem*}
\begin{proof}
 Let $X$ be a realizable variety. Then the category $\mr{Hol}(X/K)$ is
 essentially small. Indeed, to check this, it suffices to show that for
 a smooth formal scheme $\fsch{X}$, the category of coherent
 $\DdagQ{\fsch{X}}$-modules is essentially small. The verification is
 standard. Now, any object of $\mr{Hol}(X/K)$ is noetherian and artinian
 by \cite[1.5]{AC2}.
\end{proof}

\begin{dfn*}
 For a realizable scheme $X$, we put
 $M(X/K):=\mr{Ind}(\mr{Hol}(X/K))$. This is a Grothendieck category by
 \ref{indcatrecall} \eqref{esssmallgrothcate} and the lemma above.
\end{dfn*}

\subsubsection{}
\label{smoothpushpulsch}
Let $\phi\colon X\rightarrow Y$ be a smooth morphism equidimensional of
relative dimension $d$ between realizable schemes.
We have the following functors via the equivalence of Theorem
\ref{Beilequivcat} compatible with Frobenius pull-backs:
\begin{equation*}
 \phi_+[-d]
  \colon D^{\mr{b}}(\mr{Hol}(X/K))\rightleftarrows
  D^{\mr{b}}(\mr{Hol}(Y/K))\colon
  \phi^+[d].
\end{equation*}

\begin{lem*}
 We have an adjoint pair $(\phi^+[d],\phi_+[-d])$, and $\phi^+[d]$ is
 exact. The adjunction map is compatible with Frobenius
 pull-backs.
\end{lem*}
\begin{proof}
 Since $(\phi^+,\phi_+)$ is an adjoint pair, the adjointness
 follows. The exactness is by \cite[1.3.2 (i)]{AC}.
\end{proof}
We put $\phi_*:=\H^0(\phi_+[-d])$, $\phi^*:=\H^0(\phi^+[d])$.
We have the right derived functor
$\mb{R}\phi_*\colon D^{+}(M(X/K))\rightarrow D^+(M(Y/K))$.
By Lemma \ref{rightadjisderv}
together with the lemma above, $\phi_+[-d]$ is the right derived
functor of $\phi_*$, namely, $\phi_+[-d]\cong\mb{R}\phi_*$ on
$D^{\mr{b}}(\mr{Hol}(X/K))$, which is a full subcategory of
$D^+(M(X/K))$.

Now, let $\phi\colon X\rightarrow Y$ be a smooth morphism which may
not be equidimensional. Then there exists a decomposition $X=\coprod
X_i$ where $X_i$ is an open subscheme of $X$ such that the induced
morphism $\phi_i\colon X_i\rightarrow Y$ is equidimensional. We put
$\phi_*:=\sum \phi_{i*}$ and $\phi^*:=\sum \phi_i^*$. Note that when
$\phi$ is an open immersion, then we have $\mb{R}\phi_*\cong\phi_+$.

\subsubsection{}
\label{finiteadjpropsch}
Let $f\colon X\rightarrow Y$ be a finite morphism between realizable
schemes. Consider the functors
\begin{equation*}
 f_!\xrightarrow{\sim}f_+\colon
  D^{\mr{b}}(\mr{Hol}(X/K))\rightleftarrows
  D^{\mr{b}}(\mr{Hol}(Y/K))\colon f^!.
\end{equation*}
We have the following:

\begin{lem*}
 Then $f_!\xrightarrow{\sim}f_+$ are exact and $(f_+,f^!)$ is an adjoint
 pair compatible with Frobenius pull-backs.
\end{lem*}
\begin{proof}
 The exactness is by \cite[1.3.13]{AC}, and the other claims follows by
 \ref{fundproprealsch}.
\end{proof}

Now, we have the associated right derived functor
$\mb{R}(\H^0f^!)\colon D^{+}(M(X/K))\rightarrow D^+(M(Y/K))$. By Lemma
\ref{rightadjisderv} together with the lemma above, we have
$f^!\cong\mb{R}(\H^0f^!)$ on $D^{\mr{b}}(\mr{Hol}(Y/K))$.

\subsubsection{}
\label{projectcasepullpush}
Let $X$, $Y$ be realizable schemes, and consider the projections
$p\colon X\times Y\rightarrow Y$, $q\colon X\times Y\rightarrow X$. Let
$\ms{A}$ be an object in $\mr{Hol}(X/K)$. We have functors:
\begin{align*}
 p_{\ms{A}+}(-):=p_+\shom(q^+\ms{A},-)
 \colon
 D^{\mr{b}}&(\mr{Hol}(X\times Y/K))\\
 &\rightleftarrows
 D^{\mr{b}}(\mr{Hol}(Y/K))\colon
 \ms{A}\boxtimes(-)=:p^+_{\ms{A}}.
\end{align*}
Now, assume that $\ms{A}$ is endowed with Frobenius structure
$\ms{A}\xrightarrow{\sim}F^*\ms{A}$.
Then we have an isomorphism of functors $F^*\circ p_{\ms{A}+}\cong
p_{\ms{A}+}\circ F^*$, and $F^*\circ p^+_{\ms{A}}\cong
p^+_{\ms{A}}\circ F^*$. Thus, $p_{\ms{A}+}$ and $p_{\ms{A}}^+$ are
compatible with Frobenius pull-backs. We have:
\begin{lem*}
 The functor $p^+_{\ms{A}}$ is exact, and $(p^+_{\ms{A}},p_{\ms{A}+})$
 is an adjoint pair. Moreover, if $\ms{A}$ is endowed with Frobenius
 structure, the pair is compatible with Frobenius pull-backs.
\end{lem*}
\begin{proof}
 The exactness of $p_{\ms{A}}^+$ follows from \cite[1.3.3 (ii)]{AC}.
 By definition \cite[1.1.8 (i)]{AC}, we have $q^+\ms{A}\otimes
 p^+(-)\cong\ms{A}\boxtimes(-)$. Thus, we get
 \begin{align*}
  \mr{Hom}_{X\times Y}\bigl(p^+_{\ms{A}}(-),-\bigr)&\cong
  \mr{Hom}_{X\times Y}\bigl(q^+\ms{A}\otimes p^+(-),-\bigr)\\
  &\cong
  \mr{Hom}_{X\times Y}\bigl(p^+(-),\shom(q^+\ms{A},-)\bigr)\\
  &\cong
  \mr{Hom}_{Y}\bigl(-,p_+\shom(q^+\ms{A},-)\bigr),
 \end{align*}
 where the second and the last isomorphism holds by the adjunction
 properties (cf.\ \ref{fundproprealsch}).
\end{proof}
We put $p_{\ms{A}*}:=\H^0p_{\ms{A}+}$,
$p_{\ms{A}}^*:=\H^0p_{\ms{A}}^+$.
Once again, we get $p_{\ms{A}+}\cong\mb{R}p_{\ms{A}*}$ on
$D^{\mr{b}}(\mr{Hol}(X\times Y/K))$.

\begin{lem}
 \label{locextrsch}
 Let $X$ be a realizable scheme, $j\colon U\hookrightarrow X$ be an open
 immersion, and $i\colon Z\hookrightarrow X$ be its complement. For an
 injective object $\ms{I}$ in $M(X/K)$, we have an exact sequence
 \begin{equation*}
  0\rightarrow\H^0i_+i^!(\ms{I})\rightarrow\ms{I}\rightarrow
   \H^0j_+j^+(\ms{I})\rightarrow0.
 \end{equation*}
\end{lem}
\begin{proof}
 Since $\H^0i_+i^!$ is a left exact functor, we may take its right
 derived functor, and this is denoted by $\mb{R}(\H^0i_+i^!)$.
 Let us put
 \begin{equation*}
  F:=\mr{Coker}\bigl(\mr{id}\rightarrow\H^0(j_+j^+)\bigr)\colon
   \mr{Hol}(X/K)\rightarrow\mr{Hol}(X/K).
 \end{equation*}
 We show that $\mb{R}^1(\H^0i_+i^!)(\ms{M})\cong IF(\ms{M})$ for
 $\ms{M}\in M(X/K)$. Since
 $\H^0i^!$ is left exact and $i_+$ is exact, we have
 $\mb{R}^1(\H^0i_+i^!)\cong i_+\mb{R}^1(\H^0i^!)$. When
 $\ms{M}\in\mr{Hol}(X/K)$, we have the isomorphism by
 Lemma \ref{finiteadjpropsch} and the localization triangle.
 Lemma \ref{smallnoeth} and Lemma \ref{indlimcommderfun} show that the
 functor $\mb{R}^1(\H^0i_+i^!)$ commutes
 with small filtrant inductive limits. Thus, by Lemma
 \ref{extenuniqcomfil}, this isomorphism uniquely extends to the
 isomorphism we want. This shows that for an injective object $\ms{I}$,
 $IF(\ms{I})=0$, and we get the short exact sequence in the statement of
 the lemma.
\end{proof}

\subsection{Constructible t-structures}
\label{constt-struct}
We need to introduce a t-structure on the triangulated category
$D^{\mr{b}}_{\mr{hol}}(X/K)$ whose
heart corresponds to the category of ``constructible sheaves'' in the
philosophy of the Riemann-Hilbert correspondence. We keep the notation
from \ref{revDmodtheory}.

\subsubsection{}
\label{dfnofsupportandoth}
Let $X$ be a realizable scheme. 
For $\ms{M}\in\mr{Hol}(X/K)$ we define the {\em support}, denoted by
$\mr{Supp}(\ms{M})$, to be the smallest closed subset $Z\subset X$ such
that $\ms{M}$ is $0$ if we pull-back to $X\setminus Z$. 
When $X_{\mr{red}}$ is smooth of dimension $d$, we say that a complex
$\ms{M}\in D^{\mr{b}}_{\mr{hol}}(X/K)$ is {\em smooth} if
$\H^i(\ms{M})[-d]$ is in $\mr{Sm}(X/K)$ for any $i$ (cf.\
\ref{fundproprealsch} \eqref{carodagdagcat}).

Now, we define the following two full subcategories of
$D^{\mr{b}}_{\mr{hol}}(X/K)$:
\begin{itemize}
 \item ${}^{\mr{c}}D^{\geq0}$ consists of complexes $\ms{M}$ such that
       $\mr{dim}\bigl(\mr{Supp}(\H^n(\ms{M}))\bigr)\leq n$ for any
       $n\geq 0$, and $\H^n(\ms{M})=0$ for any $n<0$.

 \item ${}^{\mr{c}}D^{\leq0}$ consists of complexes $\ms{M}$ such that
       $\H^ki^+_W(\ms{M})=0$ for any closed subscheme $i_W\colon
       W\hookrightarrow X$ and $k>\mr{dim}(W)$.
\end{itemize}
We note that the extension property holds, namely, for a triangle
$\ms{M}'\rightarrow\ms{M}\rightarrow\ms{M}''\xrightarrow{+1}$, if
$\ms{M}'$ and $\ms{M}''$ are in ${}^{\mr{c}}D^{\star}(X)$
($\star\in\{\geq0,\leq0\}$) then so is $\ms{M}$.

\begin{ex*}
 Let $X$ be a smooth curve. Then ${}^{\mr{c}}D^{\geq0}$ consists of
 complexes $\ms{M}$ such that $\H^i(\ms{M})=0$ for $i<0$, and
 $\H^0(\ms{M})$ is supported on a finite union of points. The
 category ${}^{\mr{c}}D^{\leq0}$ consists of complexes $\ms{N}$ such
 that $\H^i(\ms{N})=0$ for $i>1$, and $\H^0i_{x}^+\H^1(\ms{M})=0$ for
 any closed point $x$. For example, $i_{x+}(K)$ and
 $K_X$ ($\cong\mr{sp}_+(\mc{O}_{X,\mb{Q}})[-1]$ where
 $\mc{O}_{X,\mb{Q}}$ denotes the constant overconvergent isocrystal) are
 in both ${}^{\mr{c}}D^{\geq0}$ and ${}^{\mr{c}}D^{\leq0}$. For a smooth
 realizable scheme $X$, any object of $\mr{Sm}(X/K)$ (cf.\
 \ref{fundproprealsch} \eqref{carodagdagcat}) is in both $D^{\geq0}$ and
 $D^{\leq0}$. This can be checked by the right exactness of $i^+$ (cf.\
 \cite[1.3.2 (ii)]{AC}).
\end{ex*}

\begin{lem}
 \label{cexactclosedpull}
 Let $i\colon Z\hookrightarrow X$ be a closed immersion, and $j\colon
 U\hookrightarrow X$ be its complement. Then $i^+$, $j_!$, $i_+$, $j^+$
 all preserve both ${}^{\mr{c}}D^{\geq0}$ and ${}^{\mr{c}}D^{\leq0}$.
\end{lem}
\begin{proof}
 Since $i_!\cong i_+$ and $j^+$ are exact by \cite[1.3.2]{AC}, the
 verification is easy. Let us show the preservation for $i^+$.
 Since the verification is Zariski local with respect to $X$, we may
 assume that $X$ is affine.
 Then the verification is reduced to the case where $Z$ is
 defined by a function $f\in\mc{O}_X$. In this case, we know that for
 any $\ms{M}\in\mr{Hol}(X)$, $\H^ki^+\ms{M}=0$ for $i\neq0,-1$.

 Since $i^+_W$ is right exact by \cite[1.3.2 (ii)]{AC}, the
 preservation for ${}^{\mr{c}}D^{\leq0}$ is easy.
 Let us show the preservation for ${}^{\mr{c}}D^{\geq0}$.
 By the extension property, it suffices to check for $\ms{M}$ of the
 form $\ms{M}=\ms{N}[-n]$ such that $\ms{N}\in\mr{Hol}(X)$ and
 $\dim\bigl(\mr{Supp}(\ms{N})\bigr)\leq n$. By using the extension
 property again, we are reduced even to the case where $\ms{N}$ is
 irreducible. In particular, we may assume that the support of $\ms{N}$
 is irreducible. In this case, we have two possibilities:
 $\mr{Supp}(\ms{N})\subset Z$ or $\mr{Supp}(\ms{N})\not\subset Z$. When
 $\mr{Supp}(\ms{N})\subset Z$, we get $\H^{-1}i^+(\ms{M})=0$, and the
 other case follows since $\dim\bigl(\mr{Supp}(\ms{N})\cap Z\bigr)<
 \dim\bigl(\mr{Supp}(\ms{N})\bigr)$.

 Let us show the lemma for $j_!$ by using the induction on the dimension
 of $X$. When $j$ is affine, the claim follows easily since $j_!$ is
 exact by \cite[1.3.13]{AC}. In general, take
 $\ms{M}\in{}^{\mr{c}}D^{\star}(U)$. Let $j'\colon V\hookrightarrow U$
 be an affine open dense subscheme, and $i'$ be the closed immersion
 into $U$ defined by the complement. Consider the triangle
 $j'_!j'^+\ms{M}\rightarrow\ms{M}\rightarrow
 i'_+i'^+\ms{M}\xrightarrow{+1}$. Since $j\circ j'$ is affine,
 $j_!j'_!j'^+\ms{M}$ is in ${}^{\mr{c}}D^{\star}(X)$, and
 $j_!i'_+i'^+\ms{M}$ is in ${}^{\mr{c}}D^{\star}(X)$ as well by the
 induction hypothesis together with the lemma for $i'^+$ we have already
 treated. Using the extension property, we conclude.
\end{proof}

\begin{prop}
 The categories ${}^{\mr{c}}D^{\geq0}$ and ${}^{\mr{c}}D^{\leq0}$
 define a t-structure on $D^{\mr{b}}_{\mr{hol}}(X/K)$.
\end{prop}
\begin{proof}
 We put $D(X):=D^{\mr{b}}_{\mr{hol}}(X/K)$.
 Let $U$ be an open subset of $X$ and $Z$ be its complement. Put
 $i\colon Z\hookrightarrow X$ and $j\colon U\hookrightarrow X$. For
 $\star\in\{\geq0,\leq0\}$, $\ms{M}$ is in
 ${}^{\mr{c}}D^{\star}(X)$ if and only if $i^+(\ms{M})$ and
 $j^+(\ms{M})$ are in ${}^{\mr{c}}D^{\star}(Z)$ and
 ${}^{\mr{c}}D^{\star}(U)$ respectively. This follows by the extension
 property, and Lemma \ref{cexactclosedpull}.

 Now, we proceed as \cite[p.143]{KW}. We use the induction on the
 dimension of $X$. We may assume $X$ to be reduced by Lemma
 \ref{fundproprealsch}. It suffices to check, for a smooth open affine
 subscheme $j\colon U\hookrightarrow X$
 equidimensional of dimension $\dim(X)$, the restriction of
 ${}^{\mr{c}}D^{\geq0}$ and ${}^{\mr{c}}D^{\leq0}$ to the subcategory
 \begin{equation*}
  T(X,U):=\bigl\{\ms{E}\in D(X)\mid \text{$\H^i(j^+\ms{E})$ is
   smooth on $U$ for any $i$}\bigr\}
 \end{equation*}
 defines a t-structure, since $\bigcup_{U}T(X,U)=D(X)$.
 Let $i\colon Z\hookrightarrow X$ be the
 complement of $U$. By the observation above, $\ms{M}\in T(X,U)$ is in
 ${}^\mr{c}D^{\star}(X)$ if and only if $j^+(\ms{M})$ and $i^+(\ms{M})$
 are in ${}^\mr{c}D^{\star}(U)$ and ${}^\mr{c}D^{\star}(Z)$
 respectively. We note that ${}^\mr{c}D^{\star}(Z)$ defines a
 t-structure on $D(Z)$ by induction hypothesis.

 Let us check the axioms of t-structure \cite[1.3.1]{BBD}.
 Axiom (ii) is obvious, and axiom (iii) can be shown by a similar
 argument to \cite[p.140, 141]{KW} using the t-structures of $D(Z)$.
 Let us check (i). By {\it d\'{e}vissage}
 using the localization triangle $j_!j^+\rightarrow\mr{id}\rightarrow
 i_+i^+\xrightarrow{+1}$
 twice and by induction hypothesis, we are reduced to showing that
 $\mr{Hom}(i_+\ms{B},j_!\ms{C})=0$ when $\ms{B}\in {}^{\mr{c}}D^{<0}(Z)$
 and $\ms{C}\in{}^{\mr{c}}D^{\geq0}(U)$ such that $\H^i(\ms{C})$ is
 smooth for any $i$. Then, $\H^{k-1}(i_+\ms{B})=0$ for $k>\dim(Z)$. On
 the other hand, $j_!$ is exact by \cite[1.3.13]{AC} since $j$ is
 affine, and thus $\H^k(j_!\ms{C})=0$ for $k<\dim(U)$ by the smoothness
 of $\H^i(\ms{C})$, so the claim follows.
\end{proof}

\begin{dfn*}
 \label{dfnofctstruc}
 The t-structure on $D^{\mr{b}}_{\mr{hol}}(X/K)$ defined in the
 proposition is called the {\em constructible t-structure}, and shortly,
 {\em c-t-structure}.
 The heart of the t-structure is denoted by $\mr{Con}(X/K)$, and
 called the category of {\em constructible modules}. The cohomology
 functor for this t-structure is denoted by $\cH^*$.
\end{dfn*}

\begin{rem*}
 Our constructible t-structure can be regarded as a generalization of
 ``perverse t-structure'' introduced in \cite{Letsru}, and also as a
 $p$-adic analogue of the t-structure defined in \cite{Katsr}.
\end{rem*}

\begin{lem}
 \label{constproplem}
 Let $f\colon X\rightarrow Y$ be a morphism between realizable
 schemes.

 (i) The functor $f^+$ is c-t-exact, and $f_+$ is left
 c-t-exact. Moreover, the pair $(\cH^0f^+,\cH^0f_+)$ is an adjoint
 pair.

 (ii) When $f=:i$ is a closed immersion, $i_+$ is c-t-exact and
 $\cH^0i^!$ is left c-t-exact. Moreover, $(i_+,\cH^0i^!)$ is an adjoint
 pair.

 (iii) When $f=:j$ is an open immersion, $j_!$ is c-t-exact, and
 $(j_!,j^+)$ is an adjoint pair.
\end{lem}
\begin{proof}
 Claims (ii) and (iii) are nothing but Lemma \ref{cexactclosedpull},
 and we reproduced these for record.

 Let us show (i). We only need to show the exactness of $f^+$.
 The verification is Zariski local, so we may
 assume $X$ and $Y$ to be realizable. Thus we can take the following
 commutative diagram:
 \begin{equation*}
  \xymatrix{
   X\ar@{^{(}->}[r]^{i'}\ar[d]_f&P\ar[d]^{\tilde{f}}\\
  Y\ar@{^{(}->}[r]_i&Q,}
 \end{equation*}
 where horizontal morphisms are closed immersions, $Q$ is smooth, and
 $\tilde{f}$ is smooth. By (ii), which we have already verified,
 it suffices to show the claim for $(i\circ f)^+$.
 By Lemma \ref{cexactclosedpull}, we already know that $i'^+$ is
 c-t-exact. Thus, it remains to show that $\tilde{f}^+$ is c-t-exact,
 which we can check easily using \cite[1.3.2 (i)]{AC}.
\end{proof}

\begin{lem}
 \label{genrankconstmod}
 Let $X$ be an irreducible realizable scheme. Let $\ms{M}$ be a
 constructible module on $X$ such that $\mr{Supp}(\ms{M})=X$.
 Then there exists an open dense subscheme
 $j\colon U\hookrightarrow X$ such that $j^+\ms{M}$ is in
 $\mr{Sm}(U/K)$.
 The rank of $j^+\ms{M}$ is called the {\em generic rank of
 $\ms{M}$}.
\end{lem}
\begin{proof}
 For any complex in $D^{\mr{b}}_{\mr{hol}}(X)$, there exists an open
 dense subscheme $j\colon U\hookrightarrow X$ such that the cohomology
 modules of $j^+\ms{M}$ is smooth.
\end{proof}

\begin{lem}
 \label{constnoetheriancat}
 Let $X$ be a realizable scheme. Then the category $\mr{Con}(X/K)$ is
 noetherian.
\end{lem}
\begin{proof}
 Since $\mr{Hol}(X/K)$ is essentially small by Lemma \ref{smallnoeth},
 so is $D^{\mr{b}}(\mr{Hol}(X))$. Since $\mr{Con}(X)$ is a full
 subcategory, it is also essentially small.

 Let us show that the category is noetherian. It suffices to show the
 claim for each irreducible component of $X$, so we may assume $X$ to be
 irreducible. Assume $X$ is smooth, and let $\ms{M}$ be a smooth
 constructible module on $X$.
 We claim that any submodule $\ms{N}$ of $\ms{M}$, there
 exists an open dense subscheme $U$ such that $\ms{N}$ is a non-zero
 smooth constructible module on $U$. Assume contrary. Then there exists
 a nowhere dense closed subset $i\colon Z\hookrightarrow X$ such that we
 have non-zero homomorphism $i_+i^+(\ms{N})\rightarrow\ms{M}$. Shrinking
 $X$ if necessary, we may assume that both $X$ and $Z$ are smooth and
 $i^+\ms{N}$ is smooth on $Z$.
 Taking the adjoint, we get a non-zero homomorphism
 $i^+\ms{N}\rightarrow i^!\ms{M}$. By \cite[5.6]{A}, we have
 $i^!\ms{M}\cong i^+\ms{M}(-d)[-2d]$ where $d$ is the codimension of $Z$
 in $X$, which is impossible.

 We use the noetherian induction on the support of $\ms{M}$. We may
 assume that $X$ is reduced. Moreover, we may assume
 $\mr{Supp}(\ms{M})=X$, otherwise, we can conclude by the induction
 hypothesis. Let $\ms{M}$ be a constructible
 module, and let $\{\ms{M}_i\}_{i\in\mb{N}}$ be an ascending chain of
 submodules of $\ms{M}$. There exists $N$ such that the generic rank
 (cf.\ Lemma \ref{genrankconstmod}) of $\ms{M}_i$ is the same for any
 $i\geq N$. Since it suffices to show that the ascending chain
 $\{\ms{M}_i/\ms{M}_N\}_{i\geq N}$ is stationary in $\ms{M}/\ms{M}_N$,
 we may assume that $\mr{Supp}(\ms{M}_i)\subset\mr{Supp}(\ms{M})$ is
 nowhere dense. Let $U$ be an open dense smooth subscheme of $X$ such
 that $\ms{M}$ is smooth on $U$. By what we have shown, $\ms{M}_i$ is
 $0$ on $U$. Let $i\colon Z\hookrightarrow X$ be the complement. Then by
 the c-t-exactness of $i_+$ and $i^+$, and induction hypothesis,
 $\ms{M}_i\cong i_+i^+\ms{M}_i$ is stationary in $i_+i^+\ms{M}$ as
 required.
\end{proof}

\begin{rem*}
 Contrary to $\mr{Hol}(X/K)$, $\mr{Con}(X/K)$ is not artinian.
 Indeed, let $X$ be a smooth realizable scheme and take a descending
 sequence of open subschemes $X\supset U_1\supsetneq
 U_2\supsetneq\dots$. Denote by $j_i\colon U_i\hookrightarrow X$ the
 inclusion. For any $\ms{M}\in\mr{Sm}(X/K)$, we have $\ms{M}\supset
 j_{1!}j_1^+(\ms{M})\supsetneq j_{2!}j_2^+(\ms{M})\supsetneq\dots$, and
 the claim follows.
\end{rem*}

\begin{lem}
 \label{consttstrpro}
 Let $X$ be a realizable scheme. For a closed point $x\in X$, denote by
 $i_x\colon \{x\}\rightarrow X$ the closed immersion.

 (i) For $\ms{F}\in\mr{Con}(X)$, $\ms{F}=0$ if and only if
 $i_x^+(\ms{F})=0$ for any closed point $x$. In particular, a
 homomorphism $\phi$ in $\mr{Con}(X)$ is $0$ if and only if
 $i^+_x(\phi)=0$ for any closed point $x$.

 (ii) Let $f\colon s'\rightarrow s$ be a morphism of points
 ({\it i.e.}\ connected schemes of dimension $0$ of finite type over
 $k$). Then $f^+$ is faithful and conservative.
\end{lem}
\begin{proof}
 For (i), use \cite[1.3.11]{AC}, and (ii) is left to the reader.
\end{proof}

\begin{lem}
 \label{cohdimlem}
 Let $f\colon X\rightarrow Y$ be a morphism of realizable schemes such
 that for any $y\in Y$, the dimension of the fiber $f^{-1}(y)$ is $\leq
 d$. Then for any $\ms{M}\in\mr{Con}(X)$, $\cH^if_!(\ms{M})=0$
 for $i>2d$ and $i<0$.
\end{lem}
\begin{proof}
 By Lemma \ref{consttstrpro}, it suffices to show that for any closed
 point $y\in Y$, $i_y^+\,\cH^if_!(\ms{M})=0$ for $i\not\in[0,2d]$. By
 the c-t-exactness of $i_y^+$ and base change, it is reduced to showing
 that $\cH^if_{y!}(\ms{M})=0$ for $i\not\in[0,2d]$ where $f_y\colon
 X\times_{Y}\{y\}\rightarrow\{y\}$. Since over a point, c-t-structure
 and the usual t-structure coincide, it remains to show that if $X$ is a
 realizable scheme of dimension $d$, and $\ms{M}\in\mr{Con}(X)$, then
 $\H^if_!\ms{M}=0$ for $i\not\in[0,2d]$. We use the induction on the
 dimension of $X$. When the dimension of $X$ is $0$, then the
 verification is easy. This in particular implies that $f_!$ is
 c-t-exact when $f$ is quasi-finite.
 Let us assume that the lemma holds for $d<N$.
 By Lemma \ref{fundproprealsch}, we may assume that
 $X$ is reduced. By the c-t-exactness \ref{constproplem} and the
 induction hypothesis, we may shrink $X$ by its open dense subscheme.
 Thus, we may assume that there exists a finite morphism
 $g\colon X\rightarrow\mb{A}^N$. Since $g_!\ms{M}$ is constructible by
 the c-t-exactness of $g_!$ we have already verified, it suffices to
 check the claim for $X=\mb{A}^N$. Shrinking $X$ further,
 we may assume that there exists a divisor $Z$ of $P:=\mb{P}^N$
 such that $X=P\setminus Z$ and $\ms{M}=\ms{N}[-N]$
 where $\ms{N}\in\mr{Hol}(X)$ by Lemma \ref{genrankconstmod}.
 Then the lemma follows by the definition of $f_!$ as well as
 \cite[5.4.1]{Huy}.
\end{proof}

\begin{lem}
 \label{gluinglemtsru}
 Let $\ms{M}\in\mr{Con}(X)$, and let $\bigl\{u_i\colon
 U_i\hookrightarrow X\bigr\}$ be a finite open covering of $X$. Let
 $u_{ij}\colon U_i\cap U_j\hookrightarrow X$ be the immersion. Then the
 following sequence is exact in $\mr{Con}(X)$:
 \begin{equation*}
  \bigoplus_{i,j}u_{ij!}u^+_{ij}\ms{M}\rightarrow
   \bigoplus_{i}u_{i!}u^+_{i}\ms{M}\rightarrow\ms{M}\rightarrow0.
 \end{equation*}
\end{lem}
\begin{proof}
 To check the exactness, it suffices to check it after taking $i^+_x$
 for each closed point $x\in X$ by Lemma \ref{consttstrpro}. By the
 commutativity of $i_x^+$ and $u_{\star!}$, the verification is
 just a combinatorial problem.
\end{proof}

\subsubsection{}
\label{propindcatcon}
We use the category $\mr{Ind}(\mr{Con}(X))$ later. Let us prepare some
properties of this category. Let $f\colon X\rightarrow Y$ be a morphism
between realizable schemes. Since $f^+$ is c-t-exact, we have a functor
$f^+\colon\mr{Ind}(\mr{Con}(Y))\rightarrow\mr{Ind}(\mr{Con}(X))$.

\begin{lem*}
 We use the notation of Lemma \ref{consttstrpro}.

 (i) Let $\ms{F}\in\mr{Ind}(\mr{Con}(X))$. Assume that $i_x^+\ms{F}=0$
 for any closed point $x$ of $X$. Then $\ms{F}=0$.

 (ii) Let $g\colon s'\rightarrow s$ be a morphism of points. Then for
 $\ms{F}\in\mr{Ind}(\mr{Con}(s))$, $\ms{F}=0$ if and only if
 $g^+(\ms{F})=0$. In particular, a homomorphism $\phi$ in
 $\mr{Ind}(\mr{Con}(s))$ is an isomorphism if and only if so is
 $g^+(\phi)$.
\end{lem*}
\begin{proof}
 Let us show (i). Write $\ms{F}=\indlim_{i\in I}\ms{F}_i$ where $I$
 is a small filtrant category and $\ms{F}_i\in\mr{Con}(X)$.
 Fix $i\in I$, and let
 $\ms{E}_j:=\mr{Ker}(\ms{F}_i\rightarrow\ms{F}_j)$.
 Since $\mr{Con}(X)$ is noetherian by Lemma \ref{constnoetheriancat},
 there exists $j_0\in I$ such that $\ms{E}_{j_0}=\ms{E}_j$ for any
 $j\geq j_0$. Now, we have
 $i_x^+(\ms{F}_i/\ms{E}_{j_0})=i_x^+\indlim_j(\ms{F}_i/\ms{E}_j)
 \hookrightarrow i_x^+\ms{F}=0$, thus $\ms{E}_{j_0}=\ms{F}_i$ by Lemma
 \ref{consttstrpro}. This shows that the homomorphism
 $\ms{F}_i\rightarrow\ms{F}_{j_0}$ is $0$, and the claim follows.

 Let us show (ii). Let $\ms{F}=\indlim\ms{F}_i$ where $I$ is small
 filtrant and $\ms{F}_i\in\mr{Con}(s)$.
 For each $i\in I$, there exists $j\in I$
 such that $g^+\ms{F}_i\rightarrow g^+\ms{F}_j$ is $0$. Thus, by Lemma
 \ref{consttstrpro}, we get that $\ms{F}_i\rightarrow\ms{F}_j$ is $0$ as
 well.
\end{proof}

\subsection{Extension of scalars and Frobenius structures}
\label{extofsclarsub}
So far, the coefficient categories we have treated ({\it e.g.}\
$\mr{Hol}(X/K)$ or $D^{\mr{b}}_{\mr{hol}}(X/K)$) are $K$-additive. For
the Langlands correspondence, we need to consider $L$-coefficients with
Frobenius structure where $L$ is an algebraic field extension of $K$.
We introduce such categories in this subsection when the extension is
finite. Scalar extended categories of isocrystals have already been
introduced in \cite[7.3]{AM}, and the idea of our construction is
essentially the same, but we hope that the usability is improved.

\subsubsection*{Extension of scalars}
\subsubsection{}
\label{extenscasetup}
Let $K$ be an arbitrary field and $\mc{A}$ be a $K$-additive category.
We take a finite field extension $L$ of $K$. We define the category
$\mc{A}_L$ as follows. The objects consist of pairs $(X,\rho)$ such that
$X\in\mr{Ob}(\mc{A})$, and a $K$-algebra homomorphism $\rho\colon
L\rightarrow\mr{End}(X)$,
called the {\em $L$-structure} (cf.\ \cite[after Remark 3.10]{Mi}).
The morphisms are morphisms in $\mc{A}$ compatible with
$L$-structures, or more precisely,
\begin{align*}
 \mr{Hom}_{\mc{A}_L}&\bigl((X,\rho),(X',\rho')\bigr)\\
 &=
  \bigl\{f\in\mr{Hom}_{\mc{A}}(X,X')\mid
  \mbox{$f\circ\rho(x)=\rho'(x)\circ f$ for any $x\in L$}
  \bigr\}.
\end{align*}

We have the forgetful functor
$\mr{for}_L\colon\mc{A}_L\rightarrow\mc{A}$. Let $X\in\mc{A}$. Then we
define $X\otimes_KL\in\mc{A}_L$ as follows: Take a basis $x_1,\dots,x_d$
of $L$ over $K$. Then $X\otimes_KL:=(\bigoplus_{i=1}^d X\otimes
x_i,\rho')$ such that for $x\in L$, write $x\cdot x_k=\sum a_ix_i$ with
$a_i\in K$, and $\rho'(x)|_{X\otimes x_k}:=\sum \rho(a_i)\otimes x_i$
where $\rho(a_i)$ denotes the structural action of $a_i\in K$ on $X$.
We can check easily that this does not depend on the choice of the basis
of $L$ up to canonical isomorphism. We denote by
$\iota_L:=(-)\otimes_KL\colon\mc{A}\rightarrow\mc{A}_L$.
If $\mc{A}$ is moreover abelian, then $\mc{A}_L$ is abelian as well, and
$\iota_L$ is exact.

For $X\in\mc{A}$ and $Y\in\mc{A}_L$, we have
\begin{equation*}
 \mr{Hom}_{\mc{A}}(X,\mr{for}_L(Y))\cong
  \mr{Hom}_{\mc{A}_L}(\iota_L(X),Y),
\end{equation*}
in other words, we have an adjoint pair $(\iota_L,\mr{for}_L)$. Thus, if
$\mc{A}$ is an abelian category, the functor $\mr{for}_L$ sends
injective objects in $\mc{A}_L$ to injective objects in $\mc{A}$.

Let $f\colon\mc{A}\rightarrow\mc{B}$ be a $K$-additive functor
between $K$-additive categories. Then there exists a unique functor
$f_L\colon\mc{A}_L\rightarrow\mc{B}_L$ which is compatible with both
$\iota_L$ and $\mr{for}_L$. If $\mc{A}$ and $\mc{B}$ are abelian,
$f_L$ is left (resp.\ right) exact, if $f$ is so.

\subsubsection{}
\label{calchominextcat}
Let $\mc{A}$ be a $K$-additive category.
Let $X,Y\in\mc{A}_L$. On the abelian group
$\mr{Hom}(\mr{for}_L(X),\mr{for}_L(Y))$, we
endow with $L\otimes_KL$-module structure as follows: we define the left
$L$-structure by $(a\cdot\phi)(x):=a(\phi(x))$, and the right
$L$-structure by $(\phi\cdot a)(x)=\phi(ax)$. For $a\in K$, both
$L$-structures are compatible, and we get the $L\otimes_KL$-module
structure. This $L\otimes_KL$-module is denoted by $\mr{Hom}_K(X,Y)$. By
definition, we have
\begin{equation*}
 \mr{Hom}(X,Y)\cong\bigl\{\phi\in\mr{Hom}_K(X,Y)\mid
  \mbox{$(a\otimes1)\phi=(1\otimes a)\phi=0$ for any $a\in L$}
  \bigr\}.
\end{equation*}
Note that if $L/K$ is a separable extension, then $L\otimes_KL$ is a
product of fields, and any $L\otimes_KL$-module is flat.

\begin{lem*}
 Let $L/K$ be a separable extension, and $M$ be an
 $L\otimes_KL$-module. Put $I:=\mr{Ker}(L\otimes_KL\rightarrow L)$. Then
 we have a canonical isomorphism
 \begin{equation*}
  M_0:=\bigl\{m\in M\mid a\cdot m=0
   \mbox{ for any $a\in I$}\bigr\}\xrightarrow{\sim}
   M/IM.
 \end{equation*}
\end{lem*}
\begin{proof}
 This follows from the following more general fact: Let $i\colon
 Z\hookrightarrow X$ be a closed immersion of schemes and $\mc{M}$ be a
 quasi-coherent $\mc{O}_X$-module. Then the composition
 $\underline{\Gamma}_Z(\mc{M})\rightarrow\mc{M}\rightarrow
 i_*i^*(\mc{M})$ is an isomorphism if $X=Z\sqcup(X\setminus Z)$ as
 schemes.
\end{proof}

\begin{cor*}
 Let $L/K$ be a separable extension. Then, for $X,Y\in\mc{A}_L$, we have
 a canonical isomorphism
 \begin{equation*}
  \mr{Hom}(X,Y)\xrightarrow{\sim}
   L\otimes_{L\otimes_KL}\mr{Hom}_K(X,Y).
 \end{equation*}
\end{cor*}

\subsubsection{}
\label{fullfaithderlem}
Now, let $\mc{A}$ be a $K$-abelian category.
Assume $L/K$ is a separable extension, and we consider the derived
category $D(\mc{A}_L)$. We have the following functor
\begin{equation*}
 \mr{Hom}^{\bullet}_K\colon C(\mc{A}_L)^\circ\times C(\mc{A}_L)
  \rightarrow C(L\otimes_KL);\quad
  (X^\bullet,Y^\bullet)\mapsto
  \prod_{i\in\mb{Z}}\mr{Hom}_K\bigl(X^i,Y^{i+\bullet}\bigr),
\end{equation*}
and the differential is defined as in \cite[I.6]{RD}. Since $\mr{for}_L$
sends injective objects to injective objects, we can take the associated
derive functor and get $\mb{R}\mr{Hom}^{\bullet}_K(-,-)$ as in
\cite[II.3]{RD}. We have the following:

\begin{lem*}
 Let $X\in D(\mc{A}_L)$, $Y\in D^+(\mc{A}_L)$. Then,
 we have an isomorphism
 \begin{equation*}
  \mb{R}\mr{Hom}(X,Y)\xrightarrow{\sim}
  L\otimes_{L\otimes_KL}\mb{R}\mr{Hom}_K^\bullet(X,Y).
 \end{equation*}
\end{lem*}
\begin{proof}
 Use Corollary \ref{calchominextcat}.
\end{proof}

\begin{cor*}
 The functor
 $D^{\mr{b}}(\mc{A}_L)\rightarrow D^{\mr{b}}(\mc{A})_L$ is fully
 faithful.
\end{cor*}
\begin{proof}
 For $X\in D(\mc{A}_L)$, let us denote by $X'$ the image in
 $D(\mc{A})_L$. We have
 \begin{equation*}
  \mb{R}^i\mr{Hom}^\bullet_K(X,Y)\cong
   \mr{Hom}_K(X',Y'[i])
 \end{equation*}
 as $L\otimes_KL$-modules.
 This shows that the functor is fully faithful by the lemma above and
 Corollary \ref{calchominextcat}.
\end{proof}

\begin{rem*}
 This corollary shows that if $X,Y\in D^{\mr{b}}(\mc{A}_L)$ are
 isomorphic in $D^{\mr{b}}(\mc{A})_L$, then they are isomorphic in
 $D^{\mr{b}}(\mc{A}_L)$. For example, assume given two $K$-additive
 functors $F,G\colon D^{\mr{b}}(\mc{A})\rightarrow D^{\mr{b}}(\mc{B})$
 and a $K$-additive morphism $\alpha\colon F\rightarrow G$ of
 functors. Furthermore, assume that these functors have $L$-additive
 liftings $\widetilde{F},\widetilde{G}\colon
 D^{\mr{b}}(\mc{A}_L)\rightarrow D^{\mr{b}}(\mc{B}_L)$.
 Then the fully faithfulness implies that $\alpha$ can be lifted to
 $\widetilde{\alpha}\colon\widetilde{F}\rightarrow\widetilde{G}$, and
 $\widetilde{\alpha}$ is an isomorphism if $\alpha$ is.
\end{rem*}

\begin{lem}
 Let $\mc{A}$ be a $K$-additive noetherian category.
 Then we have a canonical equivalence
 $\mr{Ind}(\mc{A}_L)\xrightarrow{\sim}\mr{Ind}(\mc{A})_L$.
\end{lem}
\begin{proof}
 It is easy to check that it is fully faithful. Let $(X,\rho)$ be an
 object in $\mr{Ind}(\mc{A})_L$. We may write $X=\indlim_{i\in I}X_i$
 where $X_i\in\mc{A}$ and $X_i\subset X$ by \cite[4.2.1 (ii)]{Depone}.
 Let $x\in L$ such that $K[x]=L$, and $[L:K]=:d$. Put
 $X'_i:=\sum_{j=0}^{d-1}\rho(x^j)(X_i)$ where the sum is taken in
 $X$. Then $X'_i$ is stable under the action of $L$, and defines an
 object in $\mc{A}_L$. The limit $\indlim_{i\in I}(X'_i,\rho)$ is sent
 to $(X,\rho)$.
\end{proof}

\subsubsection{}
\label{commuiotaderiv}
Let $F\colon\mc{A}\rightarrow\mc{B}$ be a $K$-additive functor between
$K$-abelian categories. Assume that $F$ is left exact and $\mc{A}_L$ has
enough injectives. Note that $\mc{A}$ also has enough injectives since
$\mr{id}\hookrightarrow\mr{for}_L\circ\iota_L$ and $\mr{for}_L$
preserves injective objects. Since, again, $\mr{for}_L$ preserves
injective objects and commutes with $F$, the functors $\mb{R}F$ and
$\mr{for}_L$ commute.
Moreover, $\mb{R}F$ and $\iota_L$ commute. Indeed, since $F$ and
$\iota_L$ commute, it suffices to show that for an injective object $I$
in $\mc{A}$, $\mb{R}^iF(\iota_L(I))=0$ for $i>0$. For this, it suffices
to show that $\mr{for}_L\circ(\mb{R}^iF)\circ\iota_L(I)=0$. We have
\begin{equation*}
 \mr{for}_L\circ(\mb{R}^iF)\circ\iota_L(I)\cong
  (\mb{R}^iF)\circ\mr{for}_L\circ\iota_L(I)=0,
\end{equation*}
where the second equality holds by the fact that
$\mr{for}_L\circ\iota_L(I)$ is a finite direct sum of copies of $I$ and
thus injective.

\subsubsection*{Frobenius structure}
\subsubsection{}
\label{dfnoffrobstru}
Now, let us consider the ``Frobenius structure''.
We fix an automorphism $\sigma\colon K\rightarrow K$, and put
$K_0:=K^{\sigma=1}$.
Let $\mc{A}$ be a $K$-additive category, and
$F^*\colon\mc{A}\rightarrow\mc{A}$ be a $\sigma$-semilinear
functor, namely for $X,Y\in\mc{A}$ the homomorphism
$\mr{Hom}(X,Y)\rightarrow\mr{Hom}(F^*X,F^*Y)$ is $\sigma$-semilinear.
We define the category $F\text{-}\mc{A}$ to be the category of pairs
$(X',\Phi)$ such that $X'\in\mr{Ob}(\mc{A})$, and an isomorphism
$\Phi\colon F^*X'\xrightarrow{\sim}X'$ called the {\em Frobenius
structure}\footnote{
In \cite[4.5.1]{BeDmod2}, Frobenius structure is defined to be an
isomorphism with the opposite direction $\Psi\colon
X'\xrightarrow{\sim}F^*X'$.
To be compatible with that of rigid cohomology, we choose the other
convention. See footnote (1) in 2.7 of \cite{A}.
}.
Morphisms in $F\text{-}\mc{A}$ are morphisms in $\mc{A}$ respecting
$\Phi$. Then the category $F\text{-}\mc{A}$ is $K_0$-additive.

There exists the forgetful functor
\begin{equation*}
 \mr{for}_F\colon F\text{-}\mc{A}\rightarrow\mc{A};\quad
  (X',\Phi)\mapsto X'.
\end{equation*}
This functor is faithful. For $X$, $Y$ in $F\text{-}\mc{A}$, we have a
$K_0$-linear endomorphism
\begin{align}
 \label{Homfrobaction}
 F\colon\mr{Hom}\bigl(\mr{for}_F(X),\mr{for}_F(Y)\bigr)\rightarrow\,
 &\mr{Hom}\bigl(\mr{for}_F(F^*X),\mr{for}_F(F^*Y)\bigr)\\
 \notag
 \cong\,&\mr{Hom}\bigl(\mr{for}_F(X),\mr{for}_F(Y)\bigr),
\end{align}
where the last isomorphism is induced by the Frobenius structures of $X$
and $Y$.

Now, assume that $\mc{A}$ is abelian. Then $F\text{-}\mc{A}$ is abelian
as well. Indeed, assume given a morphism $f\colon X\rightarrow Y$ in
$F\text{-}\mc{A}$. Then the Frobenius structure on $X$ induces a
Frobenius structure on $\mr{Ker}(\mr{for}_F(f))$, which is the kernel of
$f$. This construction shows that
$\mr{for}_F(\mr{Ker}(f))\cong\mr{for}_F(\mr{Ker}(f))$. Replacing
$\mr{Ker}$ by $\mr{Coker}$, we get the same result. Thus, we get the
claim.

The construction shows that $\mr{for}_F$ is an exact functor, and the
following diagram is commutative:
\begin{equation*}
 \xymatrix@C=50pt{
  D(F\text{-}\mc{A})\ar[r]^-{\H^i}\ar[d]_{\mr{for}_F}&
  F\text{-}\mc{A}\ar[d]^-{\mr{for}_F}\\
 D(\mc{A})\ar[r]_{\H^i}&\mc{A}.
  }
\end{equation*}
Moreover, if $\mr{for}_F(X)=0$ then $X=0$. This implies that a sequence
$C$ in $F\text{-}\mc{A}$ is exact if and only if the sequence
$\mr{for}_F(C)$ is exact.

Finally, let $(\mc{A},F)$ and $(\mc{B},G)$ be $K$-additive categories
with semilinear endofunctor. Assume given a functor
$f\colon\mc{A}\rightarrow\mc{B}$ and an equivalence $f\circ F\cong
G\circ f$. Then we have a canonical functor $\widetilde{f}\colon
F\text{-}\mc{A}\rightarrow G\text{-}\mc{B}$ such that
$f\circ\mr{for}_F\cong\mr{for}_G\circ\widetilde{f}$.

\subsubsection{}
\label{grocatindfrob}
Now, assume further that $F^*$ is an equivalence of categories, and
$\mc{A}$ is a Grothendieck category. We define
\begin{equation*}
 (-)_F\colon\mc{A}\rightarrow
  F\text{-}\mc{A};\,X\mapsto X_F:=\bigoplus_{n\in\mb{Z}}(F^*)^nX.
\end{equation*}
Then it can be checked easily that $((-)_F,\mr{for}_F)$ is an adjoint
pair. Furthermore, $(-)_F$ is exact, since the functor
$\mr{for}_F\circ(-)_F$ is exact.
Thus $\mr{for}_F$ sends injective objects to
injective objects. Filtrant inductive limits are representable in
$F\text{-}\mc{A}$, and commute with $\mr{for}_F$. Let $G$ be a
generator of $\mc{A}$. Then $G_F$ is a generator of
$F\text{-}\mc{A}$. Indeed, assume given two morphisms $f,g\colon
X\rightarrow Y$ in $F\text{-}\mc{A}$. Then there exists $\phi\colon
G\rightarrow\mr{for}_F(X)$ such that
$\phi\circ\mr{for}_F(f)\neq\phi\circ\mr{for}_F(g)$. By taking adjoint,
we have $\phi_F\colon G_F\rightarrow X$. Then $\phi_F\circ
f\neq\phi_F\circ g$ as required. This shows that $F\text{-}\mc{A}$ is a
Grothendieck category as well.

In the following, we often assume:
\begin{quote}
 (*) $F^*$ is an equivalence, and $\mc{A}$ is a noetherian category.
\end{quote}
This assumption implies that $\mr{Ind}(\mc{A})$ is a Grothendieck
category endowed with semilinear auto-functor $F^*$. Thus by the result
we get above, the category $F\text{-}\mr{Ind}(\mc{A})$ is a Grothendieck
category.

\subsubsection{}
\label{hochildserrespecseq}
We retain the assumption (*) in \ref{grocatindfrob}. Take
$X,Y$ in $F\text{-}\mr{Ind}(\mc{A})$.
Then the homomorphism $F$ in (\ref{Homfrobaction}) is an isomorphism,
and $\mr{Hom}\bigl(\mr{for}_F(X),\mr{for}_F(Y)\bigr)$ is a
$K_0[F^{\pm1}]$-module. Here the $K_0[F^{\pm1}]$-module structure is
defined so that $F\cdot\varphi:=F\circ\varphi\circ F^{-1}$ for
$\varphi\colon\mr{for}_F(X)\rightarrow\mr{for}_F(Y)$.
This module is denoted by $\mr{Hom}_\rho(X,Y)$.
On the other hand, for a $K_0[F^{\pm}]$-module $M$ and $X=(X',\Phi)\in
F\text{-}\mr{Ind}(\mc{A})$, we define $M\otimes_{K_0}X$ in
$F\text{-}\mr{Ind}(\mc{A})$ as follows:
Write $M=\indlim M_i$ as $K_0$-vector spaces such that $M_i$ is finite
dimensional. As an object in $\mr{Ind}(\mc{A})$, it is
$\indlim(M_i\otimes_{K_0}X')$. The Frobenius structure is defined as
\begin{equation*}
 M\otimes_{K_0}X'\xrightarrow[\sim]{(F\cdot)\otimes\Phi}
  M\otimes_{K_0}F^*X'\cong F^*(M\otimes_{K_0}X')
\end{equation*}
where the last isomorphism follows since $F^*$ is an exact functor thus
commutes with the functor $M_i\otimes_{K_0}$.
The functor $M\otimes_{K_0}$ is exact. Now, for any
$K_0[F^{\pm1}]$-module $M$, we have
\begin{equation*}
  \mr{Hom}_{F\text{-}\mr{Ind}(\mc{A})}(M\otimes_{K_0}X,Y)
   \xrightarrow{\sim}
 \mr{Hom}_{K_0[F^{\pm1}]}\bigl(M,\mr{Hom}_\rho(X,Y)\bigr).
\end{equation*}
This shows that if $Y$ is an injective object, $\mr{Hom}_\rho(X,Y)$ is
an injective $K_0[F^{\pm1}]$-module.

As in \cite[I.6]{RD}, for $X,Y\in C(F\text{-}\mr{Ind}(\mc{A}))$,
we define a complex of $K_0[F^{\pm1}]$-modules
$\mr{Hom}^\bullet_\rho(X,Y)$. We can take the derived functor, and get
\begin{equation*}
 \mb{R}\mr{Hom}^{\bullet}_\rho\colon
  D(F\text{-}\mr{Ind}(\mc{A}))^\circ
  \times
  D^+(F\text{-}\mr{Ind}(\mc{A}))\rightarrow D(K_0[F^{\pm1}]).
\end{equation*}
Abusing notation, we write
$\mr{Hom}_\rho:=\mb{R}^0\mr{Hom}^{\bullet}_\rho$. Let
$\varphi\colon\mr{Spec}(K_0[F^{\pm1}])\rightarrow\mr{Spec}(K_0)$ be the
canonical morphism. We have the canonical isomorphism
$\varphi_*\mr{Hom}_\rho(X,Y)\cong\mr{Hom}(\mr{for}_F(X),\mr{for}_F(Y))$
as $K_0$-vector spaces.

\begin{lem*}
 We regard $K_0$ as a $K_0[F^{\pm}]$-module such that $F$ acts
 trivially. For $X,Y\in D(F\text{-}\mr{Ind}(\mc{A}))$ such that $X\in
 D^-$, $Y\in D^+$, we have
 \begin{equation*}
  \mb{R}\mr{Hom}_{K_0[F^{\pm1}]}\bigl(K_0,\mb{R}\mr{Hom}^{\bullet}
   _\rho(X,Y)\bigr)
   \cong\mb{R}\mr{Hom}(X,Y).
 \end{equation*}
\end{lem*}
\begin{proof}
 We have a canonical isomorphism
 \begin{equation*}
  \mr{Hom}_{K_0[F^{\pm1}]}\bigl(K_0,\mr{Hom}_\rho(X,Y)\bigr)
   \cong\mr{Hom}(X,Y).
 \end{equation*}
 Since the functor $\mr{Hom}_\rho(X,-)$ preserves injective objects, we
 get the lemma.
\end{proof}

For a $K_0[F^{\pm}]$-module $M$, we put
\begin{equation*}
 M^F:=\mr{Hom}_{K_0[F^{\pm}]}(K_0,M),\qquad
  M_F:=\mr{Ext}^1_{K_0[F^{\pm}]}(K_0,M).
\end{equation*}

\begin{cor*}
 Let $X,Y\in D(F\text{-}\mr{Ind}(\mc{A}))$ such that $X\in D^-$ and
 $Y\in D^+$. Then there exists the
 short exact sequence:
 \begin{equation*}
  0\rightarrow\mr{Hom}_\rho(X,Y[-1])_{F}
   \rightarrow
   \mr{Hom}(X,Y)
   \rightarrow
   \mr{Hom}_\rho(X,Y)^F
   \rightarrow0.
 \end{equation*}
\end{cor*}

\subsubsection{}
\label{Frobsetupp}
Let $\mc{A}$ be a $K$-additive category, and
$F^*\colon\mc{A}\rightarrow\mc{A}$ be a $\sigma$-semilinear
functor. We fix a finite field extension $L$ and an
isomorphism $\sigma_L\colon L\rightarrow L$ compatible with
$\sigma$. Put $L_0:=L^{\sigma=1}$.
We define $F_L^*\colon\mc{A}_L\rightarrow\mc{A}_L$ as follows: Let
$\rho\colon L\rightarrow\mr{End}(X)$ be an object of $\mc{A}_L$. We have
a $\sigma$-semilinear homomorphism $F^*(\rho)\colon
L\xrightarrow{\rho}\mr{End}(X)\rightarrow\mr{End}(F^*X)$. We put
\begin{equation*}
 F^*_L(\rho):=F^*(\rho)\circ\sigma^{-1}_L
  \colon L\rightarrow\mr{End}(F^*X).
\end{equation*}
This is a homomorphism of $K$-algebras. We define
$F^*_L\colon\mc{A}_L\rightarrow\mc{A}_L$ by sending $(X,\rho)$ to
$(F^*X,F_L^*(\rho))$. The $L_0$-additive category $F_L\text{-}\mc{A}_L$
is sometimes denoted by $F\text{-}\mc{A}_L$.
The $K$-additive functors $\iota_L\colon\mc{A}\rightarrow\mc{A}_L$ and
$\mr{for}_L\colon\mc{A}_L\rightarrow\mc{A}$ induce the functors
$F\text{-}\mc{A}\rightarrow F_L\text{-}\mc{A}_L$ and
$F_L\text{-}\mc{A}_L\rightarrow F\text{-}\mc{A}$ denoted abusively by
$\iota_L$ and $\mr{for}_L$ respectively. We can check that
$(\iota_L,\mr{for}_L)$ is an adjoint pair, and $\iota_L$ is exact.
In particular, $\mr{for}_L$ preserves injective objects.

Let $\mc{B}$ be another $K$-additive category endowed with an
$\sigma$-semilinear endofunctor $G^*$.
Let $f\colon\mc{A}\rightarrow\mc{B}$ be a $K$-additive functor between
$K$-additive categories compatible with $F^*$ and $G^*$. Then there
exists a unique functor $f_L\colon F\text{-}\mc{A}_L\rightarrow
G\text{-}\mc{B}_L$ compatible with $\mr{for}_L$ and $\iota_L$.

\begin{rem*}
 Assume that $\sigma$ and $\sigma_L$ are identity. Then
 $F\text{-}\mc{A}$ is a $K$-additive category, and it makes sense to
 consider the category $(F\mbox{-}\mc{A})_L$. We leave the reader to
 check that there exists a canonical equivalence
 $(F\mbox{-}\mc{A})_L\cong F\text{-}\mc{A}_L$.
\end{rem*}

\subsubsection*{Application of the theory}
\subsubsection{}
\label{5tuplesover}
\label{defofcatwithcoeff}
To work with $p$-adic cohomology theory, we often need to fix a base
as in \ref{revDmodtheory}. Let $R$ is a complete discrete valuation ring
with residue field $k$ which is assumed perfect, $K=\mr{Frac}(R)$. We
assume that there exists a positive integer $s$ such that $\sigma\colon
K\xrightarrow{\sim}K$ is the extension of a lifting
$R\xrightarrow{\sim}R$ of the $s$-th absolute Frobenius automorphism of
$k$. Now, we consider the following two types of data:
\begin{description}
 \item[Geometric case]
	    We fix a finite field extension $L$ of $K$, and put
	    $\mf{T}_{\emptyset}:=(k,R,K,L)$. This is called a {\em geometric base
	    tuple}. We put $L_0:=L$ in this case.

 \item[Arithmetic case] 
	    We fix a finite field extension $L$ and an automorphism $\sigma\colon
	    L\rightarrow L$ such that $\sigma(K)=K$ and $\sigma|_K$ is a
	    lifting of the $s$-th Frobenius automorphism of $k$. We put
	    $\mf{T}_F:=(k,R,K,L,s,\sigma)$, and call this an {\em arithmetic base
	    tuple}. We put $L_0:=L^{\sigma=1}$. We call the geometric base tuple
	    $(k,R,K,L)$ the {\em associated geometric base tuple}.
\end{description}
By {\em base tuple}, we mean either geometric or arithmetic
base tuple. For a geometric base tuple, we sometimes put an index
$\cdot_{\emptyset}$, and for arithmetic base tuple, $\cdot_F$.

\begin{dfn*}
 Let $X$ be a realizable scheme over $k$. The category $\mr{Hol}(X/K)$
 (cf.\ \ref{defofrealcat}) is endowed with ($s$-th) Frobenius pull-back
 $F^*$ (cf.\ Remark \ref{fundproprealsch}). Moreover, $F^*$ induces an
 auto-equivalence of $\mr{Hol}(X/K)$.
 Thus, we can apply the general results and constructions developed in
 the preceding paragraphs.
 \begin{description}
  \item[Geometric case]
	     Let $\mf{T}_{\emptyset}$ be a geometric base tuple as above.
	     In this case, we put
   \begin{alignat*}{2}
    \mr{Hol}(X/\mf{T}_\emptyset)&:=\mr{Hol}(X/K)_L,&\qquad
    \mr{Isoc}^\dag(X/\mf{T}_\emptyset)&:=\mr{Isoc}^\dag(X/K)_L,\\
    M(X/\mf{T}_\emptyset)&:=\mr{Ind}(\mr{Hol}(X/\mf{T}_\emptyset)),&
    D(X/\mf{T}_\emptyset)&:=D(M(X/\mf{T}_\emptyset)).
   \end{alignat*}

  \item[Arithmetic case] 
	     Let $\mf{T}_F$ be an arithmetic base tuple, and let
	     $\mf{T}_\emptyset$ be the associated geometric tuple. In this
	     case, we put
   \begin{alignat*}{2}
    \mr{Hol}(X/\mf{T}_F)&:=F\text{-}\mr{Hol}(X/K)_L,&\qquad
    \mr{Isoc}^\dag(X/\mf{T}_F)&
    :=F\text{-}\mr{Isoc}^\dag(X/K)_L,\\
    M(X/\mf{T}_F)&:=F_L\text{-}
    \mr{Ind}(\mr{Hol}(X/\mf{T}_\emptyset)),&\qquad
    D(X/\mf{T}_F)&:=D(M(X/\mf{T}_F)).
   \end{alignat*}
 \end{description}
 If there is nothing to be confused, we sometimes omit the base tuple
 $/\mf{T}_\emptyset$ or $/\mf{T}_F$.
 We also denote $\mr{Hol}(X/\mf{T}_\emptyset)$ (resp.\
 $\mr{Hol}(X/\mf{T}_F)$) {\it etc.}\ by $\mr{Hol}(X/L_\emptyset)$
 (resp.\ $\mr{Hol}(X/L_F)$) {\it etc.}.
\end{dfn*}

\begin{rem*}
 (i) For a realizable scheme over $k$, recall that $\mr{Isoc}^\dag(X/K)$
 is slightly smaller than the category of overconvergent isocrystals
 (cf.\ \ref{fundproprealsch}). However, our category
 $\mr{Isoc}^\dag(X/K_F)$
 coincides with the category of overconvergent $F$-isocrystals
 $F\text{-}\mr{Isoc}^\dag(X/K)$
 defined in \cite[2.3.7]{Berrigcoh}.

 (ii) For a scheme $X$ over a field, let us denote by
 $D^{\mr{b}}_{\mr{c}}(X)$ the category of constructible
 $\mb{Q}_\ell$-complexes.
 Let $k$ be a finite field, $\overline{k}$ be its algebraic closure, and
 $X$ be a scheme of finite type over $k$. Put
 $\overline{X}:=X\otimes_k\overline{k}$.
 Under the philosophy of the Riemann-Hilbert correspondence,
 $D^{\mr{b}}_{\mr{hol}}(X/L_{\emptyset})$ (resp.\
 $D^{\mr{b}}_{\mr{hol}}(X/L_F)$) plays the role of
 $D^{\mr{b}}_{\mr{c}}(\overline{X})$ (resp.\
 $D^{\mr{b}}_{\mr{c}}(X)$) in $p$-adic cohomology theory.

 Now, note that $D^{\mr{b}}_{\mr{c}}(X)$ does not depend on the base
 field $k$. On the other hand, {\it a priori},
 $D^{\mr{b}}_{\mr{hol}}(X/L_F)$ depends on $\mf{T}_F$. However,
 we show in Corollary \ref{behaviorofbasech} that the category, in fact,
 does {\em not} depend on the choice of the base tuple under some
 conditions, which reinforces the justification of the analogy.
\end{rem*}

\begin{lem}
 \label{behaviorofbasech}
 (i) Let $\mf{T}'_\emptyset=(k',R',K',K')$ be a geometric base tuple
 over a tuple $\mf{T}_\emptyset:=(k,R,K,K')$, namely $K'/K$ is a finite
 extension. Then, there exists a canonical equivalence
 $\mr{Hol}(X\otimes_kk'/\mf{T}'_\emptyset)
 \xrightarrow{\sim}
 \mr{Hol}(X/\mf{T}_\emptyset)$.

 (ii) Let $\mf{T}'_F=(k',R',K',K',s,\sigma)$ be an arithmetic base tuple
 over a tuple $\mf{T}_F:=(k,R,K,K',s,\sigma)$. Then there exists a
 canonical equivalence
 $\mr{Hol}(X\otimes_kk'/\mf{T}'_F)\xrightarrow{\sim}
 \mr{Hol}(X/\mf{T}_F)$.
\end{lem}
\begin{proof}
 We may reduce to the case where $X$ can be lifted to a smooth formal
 scheme $\fsch{X}$ over $R$. Let $\fsch{X}':=\fsch{X}\otimes_RR'$. There
 exists the functor
 \begin{equation*}
  M(\DdagQ{\fsch{X}'/R'})\rightarrow M(\DdagQ{\fsch{X}/R})_{K'},
 \end{equation*}
 where $M(\ms{A})$ denotes the category of $\ms{A}$-modules.
 It is straightforward to show that this functor induces an equivalence
 of categories. By Remark \ref{Frobnotdep}, we get (i).
 Now, the following diagram is commutative:
 \begin{equation*}
  \xymatrix{
   M(\DdagQ{\fsch{X}'/R'})
   \ar[r]^-{\sim}\ar[d]_{F^*}&
  M(\DdagQ{\fsch{X}/R})_{K'}
   \ar[d]^{F^*}\\
  M(\DdagQ{\fsch{X}'/R'})
   \ar[r]_-{\sim}&
   M(\DdagQ{\fsch{X}/R})_{K'}.
   }
 \end{equation*}
 This diagram implies (ii).
\end{proof}

\begin{cor*}
 Assume $k$ is a finite field with $q=p^s$-elements.
 Let $K'$ be a finite extension of $K$, and put
 $\mf{T}_F:=(k,R,K,K',s,\mr{id})$,
 $\mf{T}_{K',F}:=(k',R',K',K',s':=[k':k]\cdot s,\mr{id})$.
 Let $X$ be a scheme over $k'$. Then, we have an equivalence of
 categories
 \begin{equation*}
  \mr{Hol}(X/\mf{T}_F)\xrightarrow{\sim}
   \mr{Hol}(X/\mf{T}_{K',F}).
 \end{equation*}
\end{cor*}
\begin{proof}
 When the extension $K'/K$ is totally ramified, the claim follows from
 (ii) of the lemma.
 Thus, we may assume that the extension is unramified. In this
 case, the verification is essentially the same as \cite[1.1.10]{De},
 so we only sketch. Since $K'/K$ is assumed unramified, we have
 $K'\cong K\otimes_{W(k)}W(k')$.
 As a scheme over $k$, we have a canonical isomorphism
 $X\otimes_kk'\cong\coprod_{\sigma\in\mr{Gal}(k'/k)} X^{\sigma}$ where
 each $X^\sigma$ is canonically isomorphic to $X$, and the Galois action
 on $k'$ is compatible in an obvious sense. Put
 $\mf{T}'_F:=(k',R',K',K',s,\mr{id})$. Then by the lemma, we get
 $\mr{Hol}(X\otimes_kk'/\mf{T}'_F)\xrightarrow{\sim}
 \mr{Hol}(X/\mf{T}_F)$. There exists $\varphi\in\mr{Gal}(k'/k)$
 such that, by $F$, each $X^{\sigma}$ is sent to
 $X^{\sigma\cdot\varphi}$. Assume given $(\ms{M},\Phi)\in
 \mr{Hol}(X/\mf{T}_{K',F})$. For $0\leq i<[k':k]$, we put
 $(F^*)^i(\ms{M})$ on $X^{\varphi^i}$, which defines $\ms{N}$ in
 $\mr{Hol}(X'\otimes_kk'/K')$. The $s'$-th Frobenius structure $\Phi$
 defines an $s$-th Frobenius structure on $\ms{N}$, and defines an
 object of $\mr{Hol}(X/\mf{T}_F)$. It is easy to check that this
 correspondence yields the equivalence of categories.
\end{proof}

\begin{rem}
 (i) Let
 $\mf{T}:=(k,R,K,L)$ and
 $\mf{T}_0:=(k,W(k),\mr{Frac}(W(k)),L)$. Then
 Lemma \ref{behaviorofbasech} implies that there exists an equivalence
 $D(X/\mf{T}_0)\cong D(X/\mf{T})$. This implies that the datum $K$
 (and $R$) is unnecessarily to define the category $D(X/L_\emptyset)$.
 Similarly, we do not need $K$ in the definition of $D(X/L_F)$.

 (ii) In the proof of the Langlands correspondence, it is convenient to
 work with $\overline{\mb{Q}}_p$-coefficient. For this, we use the
 2-inductive limit method as in \cite[1.1.3]{De} to construct the
 theory. The detail will be explained in \ref{algcloscoefftheory}.
\end{rem}

\begin{dfn}
 \label{tatetwistdfn}
 Assume we are in the situation of \ref{Frobsetupp}, and
 let $X:=(X',\Phi)$ be an object of $F\text{-}\mc{A}_L$.
 Assume we are given an arithmetic tuple as in \ref{5tuplesover}. For an
 integer $n$, we define $X(n):=\bigl(X',p^{-sn}\cdot\Phi\bigr)$, and
 call it the {\em $n$-th Tate twist} of $X$.
\end{dfn}

\subsubsection{}
\label{globalsectionfunct}
Let $\base\in\{\emptyset,F\}$
Let $L:=\iota_L(K)$ in $\mr{Hol}(\mr{Spec}(k)/L_{\base})$.
We have the left exact functor
\begin{equation*}
 \Gamma\colon
  M(\mr{Spec}(k)/L_{\base})\rightarrow\mr{Vec}_{L_0};\quad
  \ms{M}\mapsto\mr{Hom}(L,\ms{M}).
\end{equation*}
We can take the associated derived functor
$\mb{R}\Gamma\colon D^+(\mr{Spec}(k)/L_{\base})\rightarrow
D^+(\mr{Vec}_{L_0})$.

\subsection{Trace map}
\label{consttracsubsec}
In order to establish a cycle class formalism, we need the trace
map in the style of SGA 4.
Remark that, in \cite[5.5]{A}, we constructed an isomorphism
$f^!\cong f^+(d)[2d]$ for a smooth morphism $f$ of relative dimension
$d$. However, the construction of this homomorphism is {\it ad hoc.},
and does not seem to be easy to check the properties that the trace map
should satisfy, for example, transitivity. Furthermore, we need the
trace maps for flat morphisms to define the cycle class map.

\subsubsection{}
\label{traceexitstate}
We fix $\base\in\{\emptyset,F\}$, and fix a base tuple
$\mf{T}:=\mf{T}_{\base}$ using the notation of \ref{5tuplesover} in this
subsection. Let $X$ be a realizable
scheme over $k$. We only treat $L=K$ case in this subsection. This will
be generalized in Theorem \ref{poincaredualrela}.
When $\base=\emptyset$, we
denote $D^{\mr{b}}(X/K_\emptyset)$ simply by $D^{\mr{b}}(X)$, in which
case, the Tate twist $(n)$ is the identity functor. When $\base=F$, we
denote $D^{\mr{b}}(X/K_F)$ by $D^{\mr{b}}(X)$.
The main result of this subsection is
the following theorem on the existence of the trace map:

\begin{thm*}
 Let $f\colon X\rightarrow Y$ be a morphism between realizable
 schemes over $k$. Let $\mf{M}_d$ be the following set of
 morphisms of realizable schemes, and
 $\mf{M}:=\bigcup_{d\geq0}\mf{M}_d$:
 \begin{quote}
  there exists an open subscheme $U\subset Y$ such that
  $X\times_YU\rightarrow U$ is flat of relative dimension $d$, and for
  each $x\in Y\setminus U$, the dimension of $f^{-1}(x)$ is $<d$.
 \end{quote}
 Then there exists a unique homomorphism $\mr{Tr}_f\colon
 f_!f^+\ms{F}(d)[2d]\rightarrow\ms{F}$ for any $\ms{F}$ in
 $D^{\mr{b}}_{\mr{hol}}(Y)$, called the {\em trace map}, satisfying the
 following conditions.

 (Var 1) $\mr{Tr}_f$ is functorial with respect to $\ms{F}$.

 (Var 2) Consider the cartesian diagram (\ref{cartediagforbc}) of
 realizable schemes. Assume $f\in\mf{M}_d$. Then the following
 diagram is commutative:
 \begin{equation*}
  \xymatrix{
   g^+f_!f^+(d)[2d]\ar[d]_-{g^+\mr{Tr}_f}\ar[r]^{\sim}&
   f'_!g'^+f^+(d)[2d]\ar@{=}[r]&
   f'_!f'^+g^+(d)[2d]\ar[d]^-{\mr{Tr}_{f'}}\\
  g^+\ar@{=}[rr]&&g^+.
   }
 \end{equation*}

 (Var 3) Let $X\xrightarrow{g}Y\xrightarrow{f}Z$ be morphisms of
 realizable schemes such that $f\in\mf{M}_d$ and $g\in\mf{M}_e$. Then
 the following diagram is commutative:
 \begin{equation*}
  \xymatrix@C=50pt{
   f_!g_!g^+f^+(d+e)[2(d+e)]
   \ar[r]^-{\mr{Tr}_g}\ar@{<->}[d]_{\sim}&
   f_!f^+(d)[2d]\ar[d]^{\mr{Tr}_f}\\
  (f\circ g)_!(f\circ g)^+(d+e)[2(d+e)]
   \ar[r]_-{\mr{Tr}_{f\circ g}}&
   \mr{id}.
   }
 \end{equation*}

 (Var 4-I) Let $f\in\mf{M}_0$ be a finite locally free morphism of rank
 $n$. Then the composition
 \begin{equation*}
  \ms{F}\rightarrow f_+f^+\ms{F}\xleftarrow{\sim}f_!f^+\ms{F}
   \xrightarrow{\mr{Tr}_f}\ms{F}
 \end{equation*}
 is the multiplication by $n$.

 (Var 4-II) When $X$ and $Y$ can be lifted to a proper smooth formal
 scheme, and $f$ can be lifted to a smooth morphism of relative
 dimension $1$ between them, then the trace map is the one defined in
 \ref{deftrforsmoothformprocu} below.

 (Var 5) The following diagram is commutative
 \begin{equation*}
  \xymatrix{
   f_!f^+(\ms{F}\otimes\ms{G})(d)[2d]
   \ar[d]_{\mr{Tr}_{f}}\ar[r]^{\sim}&
   (f_!f^+\ms{F}(d)[2d])\otimes\ms{G}
   \ar[d]^{\mr{Tr}_f\otimes\mr{id}}\\
  \ms{F}\otimes\ms{G}\ar@{=}[r]&
   \ms{F}\otimes\ms{G}
   }
 \end{equation*}
 where the upper horizontal homomorphism is the projection
 formula \ref{fundproprealsch} \eqref{projformulas}.
\end{thm*}

\subsubsection{}
\label{firstproptraceexis}
Even though there are many technical differences, the idea of the
construction of trace map is essentially the same as that in [SGA 4,
XVIII]. Let us start to construct the trace map. First, we list up
direct consequences from the requested properties.

\begin{enumerate}
 \item\label{constenoughconstr}
      By (Var 5), it suffices to construct the trace map for
      $\ms{F}=K_Y$.

 \item\label{constisconstrdim}
      By assumption, for $\ms{F}\in\mr{Con}(X)$, we have
      $\cH^if_!f^+\ms{F}=0$ for $i>2d$ by Lemma
      \ref{cohdimlem}. Thus, we have an isomorphism
      $\mr{Hom}(f_!f^+\ms{F}(d)[2d],\ms{F})\cong
      \mr{Hom}(\cH^{2d}f_!f^+\ms{F}(d),\ms{F})$.

 \item\label{pointredunique}
      Assume that we have already constructed the trace map when
      $Y$ is a point.
      Then by \eqref{constenoughconstr}, \eqref{constisconstrdim}
      above, Lemma \ref{consttstrpro} (i), and the base change property
      shows that extension of this trace map to the general situation is
      unique, if they exist.

 \item\label{reducetopoinrt}
      When $f=f\amalg f'\colon X'\coprod X''\rightarrow Y$, and assume
      $\mr{Tr}_f$ and $\mr{Tr}_{f'}$ have already been
      constructed. Then, by the same argument as [SGA 4, XVII, 6.2.3.1],
      $\mr{Tr}_f=\mr{Tr}_{f'}+\mr{Tr}_{f''}$.

 \item\label{univhomeodet}
      If $f$ is a universal homeomorphism, then the canonical
      homomorphism
      \begin{equation*}
       \alpha\colon\ms{F}\rightarrow f_+f^+\ms{F}
	\xleftarrow{\sim}f_!f^+\ms{F}
      \end{equation*}
      is an isomorphism by Lemma \ref{fundproprealsch}.
      By (Var 4-I), we have $\mr{Tr}_f:=\deg(f)\cdot\alpha^{-1}$.

 \item Consider the cartesian diagram (\ref{cartediagforbc}) of
       realizable schemes.
       Then the compatibility (Var 2) is equivalent to the commutativity
       of one of the following diagrams:
       \begin{equation}
	\label{commuequivBC}
	 \xymatrix@C=40pt{
	 g'^+f^+\ms{F}(d)[2d]\ar[r]^-{g'^+\mr{Tr}^{\mr{ad}}_f}
	 \ar[d]_{\sim}&g'^+f^!\ms{F}\ar[d]\\
	f'^+g^+\ms{F}(d)[2d]\ar[r]_-{\mr{Tr}^{\mr{ad}}_{f'}g^+}&
	 f'^!g^+\ms{F},
	 }\qquad
	 \xymatrix@C=40pt{
	 f'^+g^!\ms{F}(d)[2d]\ar[r]^-{\mr{Tr}^{\mr{ad}}_{f'}g^!}\ar[d]&
	 f'^!g^!\ms{F}\ar[d]^{\sim}\\
	g'^!f^+\ms{F}(d)[2d]\ar[r]_-{g'^!\mr{Tr}^{\mr{ad}}_f}&
	 g'^!f^!\ms{F},
	 }
       \end{equation}
       where $\mr{Tr}^{\mr{ad}}$ denotes the adjoint of the trace map,
       and the vertical arrows are the base change homomorphisms. The
       verification is standard using the diagram of \ref{compatleftadj}.
\end{enumerate}

\begin{lem}
 \label{glueindimd}
 Let $f\colon X\rightarrow Y$ be a morphism of realizable schemes of
 relative dimension $\leq d$. Let $\{U_i\}$ be a finite open covering of
 $X$, and $U_{ij}:=U_i\cap U_j$. For $\star\in\{i,ij\}$, we put
 $u_{\star}\colon U_\star\rightarrow X$, and
 $f_{\star}\colon U_\star\hookrightarrow X\xrightarrow{f}Y$. Then for
 $\ms{F}\in\mr{Con}(X)$, the following sequence is exact:
 \begin{equation*}
  \bigoplus_{i,j}\cH^{2d}f_{ij!}u_{ij}^+\ms{F}\rightrightarrows
   \bigoplus_i\cH^{2d}f_{i!}u_i^+\ms{F}\rightarrow
   \cH^{2d}f_!\ms{F}\rightarrow0.
 \end{equation*}
\end{lem}
\begin{proof}
 By Lemma \ref{cohdimlem}, $\cH^{2d}f_!$ is right exact, and the
 claim of the lemma follows by applying this functor to the exact
 sequence in Lemma \ref{gluinglemtsru}.
\end{proof}

\subsubsection{}
\label{stapoduet}
First, let $\mf{M}_{\mr{et}}$ be the set of \'{e}tale morphisms
between realizable schemes. We show the theorem for $\mf{M}_{\mr{et}}$
instead of $\mf{M}$. By \ref{firstproptraceexis} \eqref{pointredunique},
combining with \ref{firstproptraceexis} \eqref{reducetopoinrt},
\eqref{univhomeodet} and Lemma \ref{consttstrpro}
(ii), if the trace maps exist for morphisms in $\mf{M}_{\mr{et}}$, then
they are unique. We show the following lemma:

\begin{lem*}
 For $f\in\mf{M}_{\mr{et}}$, there exists a unique trace map
 $f_!f^+\ms{M}\rightarrow\ms{M}$ for $\ms{M}\in
 D^{\mr{b}}_{\mr{hol}}(Y)$ satisfying the properties (Var
 1,2,3,4-I,5) if we replace $\mf{M}$ by $\mf{M}_{\mr{et}}$. Moreover,
 the homomorphism $f^+(\ms{M})\rightarrow f^!(\ms{M})$ defined by taking
 the adjoint is an isomorphism.
\end{lem*}
The proof is divided into several parts, and it is given in
\ref{proofoslemmetdual}.

\begin{lem}[Smooth base change for open immersion]
 \label{smbcforopenimm}
 Consider the following cartesian diagram
 \begin{equation*}
  \xymatrix{
   U'\ar[d]_-{g'}\ar[r]^{j'}\ar@{}[rd]|\square&
   X'\ar[d]^{g}\ar@{}[rd]|\square&
   Z'\ar[l]_-{i'}\ar[d]^{g''}\\
  U\ar[r]_-{j}&X&Z\ar[l]^-{i}.
   }
 \end{equation*}
 Assume $g$ is smooth, $j$ is an open immersion, and $i$ is the closed
 immersion defined by the complement.
 Then the base change homomorphisms
 $g^+j_+(\ms{M})\rightarrow j'_+g'^+(\ms{M})$ and
 $g''^+i^!(\ms{M})\rightarrow i'^!g^+(\ms{M})$ are isomorphisms for any
 $\ms{M}$ in $D^{\mr{b}}_{\mr{hol}}(U)$.
\end{lem}
\begin{proof}
 By the localization triangle $i_+i^!\rightarrow\mr{id}\rightarrow
 j_+j^+\xrightarrow{+1}$, it suffices to show only the
 first isomorphism.
 Obviously, if we restrict the base change homomorphism to $U'$, the
 homomorphism is an isomorphism. Thus, by the localization triangle, it
 suffices to show that $i'^!g^+j_+\ms{M}=0$.
 Since the verification is Zariski local with respect to $X'$, we may
 assume that $g$ is factored into an \'{e}tale morphism followed by the
 projection $\mb{A}^n_X\rightarrow X$. We can treat \'{e}tale and
 projection cases separately.
 Thus, using [EGA IV, 18.4.6], we may assume that there is a
 smooth morphism $\fsch{P}'\rightarrow\fsch{P}$ of smooth formal schemes
 and a closed embedding $X\hookrightarrow\fsch{P}$ such that $X'\cong
 X\times_{\fsch{P}}\fsch{P}'$. Then we may use \cite[4.3.12]{BeDmod1}
 and \cite[Theorem 5.5]{A} to conclude.
\end{proof}

\begin{cor*}[Smooth base change]
 Consider the diagram of realizable schemes (\ref{cartediagforbc}).
 Assume that $g$ is smooth. Then the base change homomorphism $g^+\circ
 f_+\rightarrow f'_+\circ g'^+$ is an isomorphism.
\end{cor*}
\begin{proof}
 We may factor $f$ as $X\xrightarrow{j}\overline{X}\xrightarrow{p} Y$
 where $j$ is an open immersion and $p$ is proper. The base change for
 $p$ is the proper base change theorem
 (cf. \ref{fundproprealsch} \eqref{basechangeprop}), and that for $j$ is
 the lemma above.
\end{proof}

\subsubsection{}
\label{Lemmaofadjtraet}
First, suppose that $Y$ is smooth liftable purely of dimension $d$, and
$f$ is affine. In this case let us construct an isomorphism
$f^+K_Y\xrightarrow{\sim}f^!K_Y$. By taking the dual, it is equivalent
to constructing $f^+K^\omega_Y\xrightarrow{\sim}f^!K^\omega_Y$ (cf.\
\ref{bidualrealcase} for $K^{\omega}_{\star}$). Let $\fsch{Y}$ be a
smooth lifting of $Y$.
Since the \'{e}tale site of $Y$ and $\fsch{Y}$ are equivalent, there
exists the following cartesian diagram where $\fsch{X}$ and $\fsch{Y}$
are smooth formal schemes and $X$ and $Y$ are special fibers:
\begin{equation*}
 \xymatrix{
  X\ar[r]\ar[d]_{f}\ar@{}[rd]|\square&
  \fsch{X}\ar[d]^{\widetilde{f}}\\
 Y\ar[r]&\fsch{Y}.
  }
\end{equation*}
By \cite[4.1.8, 4.1.9, 4.3.5]{Caro-SMF} and \cite[3.12]{A}, we have
canonical isomorphisms
\begin{equation*}
 \mb{D}_{\fsch{X}}\circ\widetilde{f}^!\circ\mb{D}_{\fsch{Y}}
  (\mr{sp}_+(\mc{O}_{\fsch{Y}_K}))
  \cong \mr{sp}_+\bigl((\widetilde{f}^*
  (\mc{O}_{\fsch{Y}_K}^\vee(-d)))^\vee(-d)\bigr)
  \cong
  \mr{sp}_+(\mc{O}_{\fsch{X}_K}).
\end{equation*}
This gives us a canonical isomorphism
\begin{equation*}
 \rho_f\colon
 \widetilde{f}^+\mc{O}_{\fsch{Y},\mb{Q}}\xrightarrow{\sim}
  \mc{O}_{\fsch{X},\mb{Q}}\cong
  \widetilde{f}^!\mc{O}_{\fsch{Y},\mb{Q}}.
\end{equation*}
By Kedlaya's fully faithfulness \cite{Ke}, $\rho_f$ extends to the
desired isomorphism. Taking the adjoint, we get a homomorphism
$t_{\fsch{Y}}\colon K^\omega_Y\rightarrow f_+f^!K^\omega_Y$.
We need the lemma below to show that $t_{\fsch{Y}}$, in fact, does not
depend on the choice of $\fsch{Y}$.

\begin{rem*}
 We remark that for an isocrystal $M$, the following diagram is
 commutative:
 \begin{equation*}
  \xymatrix@C=50pt@R=5pt{
   &\mr{sp}_+((M^\vee)^\vee)\ar@{-}[dd]^{\sim}\\
   \mr{sp}_+(M)\ar[ur]\ar[dr]&\\
  &\mb{D}\circ\mb{D}(\mr{sp}_+(M)).}
 \end{equation*}
 To check this, by definition (cf.\ \cite[2.2.12]{carodual}), it
 suffices to check the commutativity for
 $M\cong\mc{O}_{\fsch{Y},\mb{Q}}$. Then it is reduced to the
 commutativity of the following diagram of complexes, whose verification
 is easy:
 \begin{equation*}
  \xymatrix@C=30pt@R=5pt{
   &\DdagQ{\fsch{Y}}\otimes\bigwedge\ms{T}_{\fsch{Y}}\otimes
   \mr{Hom}_{\mc{O}}\bigl(\mr{Hom}_{\mc{O}}
   (\mc{O}_{\fsch{Y}},\mc{O}_{\fsch{Y}}),\mc{O}_\fsch{Y}\bigr)
   \ar[dd]^{\sim}\\
  \DdagQ{\fsch{Y}}\otimes\bigwedge\ms{T}_{\fsch{Y}}
   \ar[ur]+L\ar[dr]+L&\\
  &\mr{Hom}_{\ms{D}}\bigl(\mr{Hom}_{\ms{D}}(\DdagQ{\fsch{Y}}
   \otimes\bigwedge\ms{T}_{\fsch{Y}},
   \DdagQ{\fsch{Y}}\otimes\omega^{-1}_{\fsch{Y}}),
   \DdagQ{\fsch{Y}}\otimes\omega^{-1}_{\fsch{Y}}\bigr).
   }
 \end{equation*}
\end{rem*}

Let $\ms{M}$ be a holonomic module on $Y$. By definition, there
exists a smooth proper formal scheme $\fsch{P}$ such that $\ms{M}$ can
be realized as a $\DdagQ{\fsch{P}}$-module $\ms{M}_{\fsch{P}}$.
There exists an immersion (not necessarily closed)
$i\colon\fsch{Y}\hookrightarrow\fsch{P}$. Then the
$\DdagQ{\fsch{Y}}$-module $i^!(\ms{M}_{\fsch{P}})$
(which is overholonomic and canonically isomorphic to
$i^+(\ms{M}_{\fsch{P}})$) is denoted by $\ms{M}\Vert_{\fsch{Y}}$ for a
moment. This module does not depend on the auxiliary choices up to
canonical equivalence.

\begin{lem*}
 The following diagram is commutative
 \begin{equation*}
  \xymatrix@C=40pt{
   K^\omega_Y\Vert_{\fsch{Y}}\ar[r]^-{t_{\fsch{Y}}}\ar@{=}[d]&
   f_+f^!K^\omega_Y\Vert_{\fsch{Y}}\ar@{^{(}->}[d]\\
  \mc{O}_{\fsch{Y},\mb{Q}}\ar[r]_-{\mr{adj}_f}&
   \widetilde{f}_*\widetilde{f}^*\mc{O}_{\fsch{Y},\mb{Q}}.
   }
 \end{equation*}
\end{lem*}
\begin{proof}
 First, let us show the lemma when $f$ is finite \'{e}tale of rank
 $n$. Since $\widetilde{f}$ is finite \'{e}tale, we can identify
 $\widetilde{f}_+$ and $\widetilde{f}^!$ by $\widetilde{f}_*$ and
 $\widetilde{f}^*$ if we consider the underlying
 $\mc{O}_{\fsch{Y},\mb{Q}}$-module structure. In the following, for
 simplicity, we do not make any difference between $\widetilde{f}$ and
 $f$. In this case, the right vertical homomorphism is, in fact,
 isomorphic.
 Let $\ms{F}$ be an
 $\mc{O}_{\fsch{Y},\mb{Q}}$-module, and $\iota\colon
 f_*(f^*\ms{F})^\vee\rightarrow(f_*f^*\ms{F})^\vee$ be the homomorphism
 sending $\varphi$ to $\mr{Tr}_f\circ\varphi$, where $\mr{Tr}_f\colon
 f_*f^*\mc{O}_{\fsch{Y}}\rightarrow\mc{O}_{\fsch{Y}}$ is the classical
 trace map. If $\ms{F}$ is a locally
 free $\mc{O}_{\fsch{Y},\mb{Q}}$-module, $\iota$ is an isomorphism. We
 have the following diagram, where we omit $\mr{sp}_+$ and the
 subscripts $\mb{Q}$:
 \begin{equation*}
  \xymatrix@C=40pt{
   \mc{O}_{\fsch{Y}}\ar[r]\ar[d]&
   f_!f^+\mc{O}_{\fsch{Y}}\ar[r]_-{\sim}\ar[d]&
   f_+f^+\mc{O}_{\fsch{Y}}\ar[r]_-{\sim}^-{\rho_f}\ar[d]&
   f_+f^!\mc{O}_{\fsch{Y}}\ar[d]\\
  (\mc{O}_{\fsch{Y}}^\vee)^\vee
   \ar[r]_-{(\mr{Tr}_f)^\vee}&
   (f_*f^*\mc{O}_{\fsch{Y}}^\vee)^\vee
   \ar[r]_{\iota^{-1}}^-{\sim}&
   f_*(f^*\mc{O}_{\fsch{Y}}^\vee)^\vee
   \ar[r]^-{\sim}&
   f_*f^*\mc{O}_{\fsch{Y}}.\\
   }
 \end{equation*}
 Here the vertical morphisms are isomorphic.
 This diagram is commutative. The commutativity of the left and right
 square immediately follows by definition.
 To check the commutativity for the middle one, we need to go back to
 the definition, which is \cite[IV.1.3]{Vir}. We note that since $f$ is
 finite \'{e}tale, the trace map
 $f_+\mc{O}_{\fsch{X},\mb{Q}}\rightarrow\mc{O}_{\fsch{Y},\mb{Q}}$
 defined in \cite[III.5.1]{Vir} is equal to $\mr{Tr}_f$ via the
 identification $f_+\mc{O}_{\fsch{X},\mb{Q}}\cong
 f_*\mc{O}_{\fsch{X},\mb{Q}}$. Since the commutativity is standard
 routine work, we leave the detail to the reader. Now, the
 verification of the lemma in the finite \'{e}tale case is reduced to
 showing the composition of the lower row is the adjunction
 homomorphism. This is easy.

 In general, there exists an open dense formal subscheme
 $j\colon\fsch{U}\subset\fsch{Y}$ such that
 $f'\colon\fsch{X}':=\fsch{X}\times_{\fsch{Y}}\fsch{U}
 \rightarrow\fsch{U}$ is finite \'{e}tale. Put
 $j'\colon\fsch{X}'\hookrightarrow\fsch{X}$. Then by
 \cite[4.3.10]{BeDmod1}, we have an injection
 $K^{\omega}_{X}\hookrightarrow
 j'_+K^{\omega}_{X'}$ where $X'$ is the special fiber of $\fsch{X}'$.
 Since $f$ is affine, $f_+$ is exact, and the homomorphism
 $f_+f^!K_{Y}^{\omega}\rightarrow j_+f'_+f'^!K_U^{\omega}$ is
 injective. Consider the following diagram where we omit
 $\Vert_{\fsch{Y}}$:
 \begin{equation*}
  \xymatrix@C=10pt@R=15pt{
   &j_+K^\omega_U\ar[rr]\ar[dd]|\hole&&
   j_+f'_+f'^!K^\omega_U\ar[dd]^{\sim}\\
  K_Y^\omega\ar[ur]\ar[rr]\ar[dd]_{\sim}&&
   f_+f^!K^\omega_Y\ar@{^{(}->}[ur]\ar[dd]&\\
  &j_*\mc{O}_{\fsch{U},\mb{Q}}\ar[rr]|-\hole&&
   j_*f'_*f'^*\mc{O}_{\fsch{U},\mb{Q}}\\
  \mc{O}_{\fsch{Y},\mb{Q}}\ar[rr]\ar[ru]&&
   f_*f^*\mc{O}_{\fsch{Y},\mb{Q}}.\ar@{^{(}->}[ur]_-{\star}&
   }
 \end{equation*}
 The diagram is known to be commutative except for the forehead square
 diagram. By the injectivity of $\star$, the desired commutativity
 follows by the commutativity of other faces.
\end{proof}

The trace map satisfies the base change property: namely
considering the cartesian diagram of (\ref{cartediagforbc}) such
that $Y$ and $Y'$ possess liftings $\fsch{Y}$, $\fsch{Y}'$, then the
diagram of (Var 2) is commutative if we replace $\mr{Tr}_{f^{(\prime)}}$
by the dual of $t_{\fsch{Y}^{(\prime)}}$.
To check this, it suffices to check the dual of the base
change property for $t_{\fsch{Y}}\colon K^\omega_Y\rightarrow
f_+f^!K^\omega_Y$. By using \cite[4.3.10]{BeDmod1}, it suffices to show
the base change property after taking $\Vert_{\fsch{Y}}$. The lemma
above reduces the verification to the base change property for the
adjunction homomorphism $\mc{O}_{\fsch{Y},\mb{Q}}\rightarrow
f_*f^*\mc{O}_{\fsch{Y},\mb{Q}}$ which follows by the base change
property of coherent $\mc{O}_{\fsch{Y},\mb{Q}}$-modules. 
The transitivity can also be checked by a similar argument.

This implies that $t_{\fsch{Y}}$ does not depend on the choice of
$\fsch{Y}$. Indeed, when $Y$ is a point, all smooth liftings of $Y$ are
canonically isomorphic, and $t_{\fsch{Y}}$ does not
depend on the choice. In general, by the base change property and the
uniqueness mentioned at the beginning of \ref{stapoduet} shows that
$t_{\fsch{Y}}$ depends only on $Y$. This justifies to denote the dual of
$t_{\fsch{Y}}$ by $\mr{Tr}_f\colon f_!f^+K_Y\rightarrow K_Y$.

\subsubsection{Proof of Lemma {\ref{stapoduet}}}
\label{proofoslemmetdual}
\mbox{}\\
Let us construct the trace map for general \'{e}tale morphism.
Assume that we have the following cartesian diagram $D$
\begin{equation*}
 \xymatrix{
 X\ar@{^{(}->}[r]^-{i'}\ar[d]_{f}\ar@{}[rd]|\square&
  \widetilde{X}\ar[d]^{g}\\
 Y\ar@{^{(}->}[r]_-{i}&\widetilde{Y},
  }
\end{equation*}
where $\widetilde{Y}$ is smooth liftable, $g$ is affine \'{e}tale, and
the horizontal morphisms are closed immersions. We have the following
uncompleted diagram of solid arrows:
\begin{equation*}
 \xymatrix@C=40pt{
  i'^+g^+K_{\widetilde{Y}}
  \ar[r]_-{\sim}^{\mr{Tr}_g}\ar[d]_{\sim}&
  i'^+g^!K_{\widetilde{Y}}\ar[d]\\
 f^+i^+K_{\widetilde{Y}}\ar@{.>}[r]&
  f^!i^+K_{\widetilde{Y}}.
  }
\end{equation*}
The left vertical homomorphism is an isomorphism by the transitivity,
and the dotted homomorphism is defined so that the diagram commutes.
By taking the adjoint, the dotted arrow gives us a homomorphism
$\mr{Tr}_D\colon f_!f^+K_Y\rightarrow K_Y$.
Let us check that $\mr{Tr}_D$ does not depend on the choice of $D$. If
$Y$ is liftable, $\mr{Tr}_D$ coincides with the trace map $\mr{Tr}_f$
(for liftable schemes) by the base change property of the trace map for
liftable schemes we have already checked.
In particular, for a closed point $i_s\colon
s\hookrightarrow Y$, $i_s^+\mr{Tr}_D$ coincides with the trace map for
the liftable schemes $X\times_Ys\rightarrow s$.
By the uniqueness at the beginning of \ref{stapoduet}, $\mr{Tr}_D$ does
not depend on the choice of diagram $D$, and we are allowed to denote
$\mr{Tr}_D$ by $\mr{Tr}_f$. This also shows the base change
and transitivity property of $\mr{Tr}_f$ when the homomorphisms in the
diagrams of (Var 2) and (Var 3) are defined.

In general, we can take a diagram $D$ locally on
$X$ and $Y$. By using Lemma \ref{glueindimd}, we can glue, and get the
desired trace map $f_!f^+K_Y\rightarrow K_Y$ similarly to
[SGA 4, XVIII, 2.9 c)], and check that this is the desired trace map. The
details are left to the reader.

Finally, let us show that $f^+(\ms{F})\cong f^!(\ms{F})$ for $\ms{F}\in
D^{\mr{b}}_{\mr{hol}}(Y)$. By (Var-5), the trace map
for $\ms{F}$ should be
\begin{equation*}
 f_!f^+\ms{F}\cong f_!f^+K_Y\otimes\ms{F}
  \xrightarrow{\mr{Tr}_f}\ms{F}.
\end{equation*}
This map satisfies (Var 1) to (Var 4) since they hold for
$\ms{F}=K_Y$. Taking the adjunction, we have a homomorphism
$f^+(\ms{F})\rightarrow f^!(\ms{F})$.
When $Y$ is a point, this is easy. Let $s\in
Y$ be a closed point, and consider the following diagram:
\begin{equation*}
 \xymatrix{
  X_s\ar[r]^-{i'_s}\ar[d]_{f_s}\ar@{}[rd]|\square&
  X\ar[d]^{f}\\
 s\ar[r]_-{i_s}&Y.
  }
\end{equation*}
By the compatibility of trace map by base change, the following diagram
is commutative
\begin{equation*}
 \xymatrix{
  i'^+_sf^+(\ms{F})\ar[r]\ar[d]_{\sim}&
  i'^+_sf^!\ms{F}\ar[d]\\
 f_s^+i_s^+(\ms{F})\ar[r]&
  f_s^!i_s^+(\ms{F}).
  }
\end{equation*}
By (the dual of) smooth base change \ref{smbcforopenimm}, the right
vertical homomorphism is an isomorphism, and by the point case, the
bottom horizontal homomorphism is an isomorphism as well. Thus by
\cite[1.3.11]{AC}, the claim follows.
\qed

\subsubsection{}
Let us construct the trace map for quasi-finite flat morphism. We follow
the construction of [SGA 4, XVII, 6.2]. Let $f\colon X\rightarrow Y$ be
a quasi-finite flat morphism. For an \'{e}tale morphism $U\rightarrow
Y$, we consider the category $\Psi_f(U)$ defined as follows: objects
consist of collections $(V_i)_{i\in I}$, where $I$ is a pointed finite
set with the marked point $0\in I$ and a decomposition
$X\times_YU=\coprod_{i\in I}V_i$ such that 
$V_i\rightarrow U$ is finite for $i\neq0$.
We denote by $I^*:=I\setminus\{0\}$. A morphism from
$\varphi=(V_i)_{i\in I}$ to $\varphi'=(V'_i)_{i\in I'}$ is a map
$\sigma\colon I\rightarrow I'$ such that $\sigma(0)=0$ and
$V_i=\bigcup_{j\in\sigma^{-1}(i)}V'_j$. For
a morphism $U\rightarrow V$ in $Y_{\mr{et}}$, there exists the obvious
functor $\Psi_f(V)\rightarrow\Psi_f(U)$, and $\Psi_f(U)$ is a fibered
category over $Y_{\mr{et}}$. This category is denoted by $\Psi_f$, and
an object of the fiber over $U\in Y_{\mr{et}}$ is denoted by
$\bigl\{U;(V_i)_{i\in I}\bigr\}$. We refer to {\em ibid.}\ for some
details.

\begin{lem*}
 Let $\ms{F}\in\mr{Con}(Y)$. Then there exists a canonical isomorphism
 \begin{equation*}
  \tau_f\colon \indlim_{\{U;\varphi\}\in\Psi_f}
   j_{U!}j_U^+\ms{F}^{I^*}\xrightarrow{\sim}
   f_!f^+\ms{F},
 \end{equation*}
 where $j_U\colon U\rightarrow Y$ is the \'{e}tale morphism.
\end{lem*}

\begin{rem*}
 Before proving the lemma, we remark that the inductive system is 
{\em not} filtrant.
\end{rem*}
\begin{proof}
 The verification is essentially the same as {\it ibid.}.
 Let us construct the homomorphism.
 Take $\varphi=\bigl\{U;(V_i)_{i\in I}\bigr\}$. Let $j_i\colon
 V_i\rightarrow X$ be the \'{e}tale morphism.
 Since $f_i\colon V_i\rightarrow U$ is assumed finite for $i\in
 I^*$, we have the following homomorphism for $\ms{G}\in\mr{Con}(U)$:
 \begin{equation*}
  \ms{G}\rightarrow f_{i+}f_i^+(\ms{G})
   \xleftarrow{\sim}
   f_{i!}f_i^+(\ms{G}).
 \end{equation*}
 By using the trace map in Lemma \ref{stapoduet}, we get
 the homomorphism
 \begin{equation*}
  j_{U!}j_U^+(\ms{F})\rightarrow
   j_{U!}f_{i!}f_i^+j_U^+(\ms{F})\cong
   f_!j_{i!}j_i^+f^+(\ms{F})\xrightarrow{\mr{Tr}_{j_i}}
   f_!f^+(\ms{F}),
 \end{equation*}
 which induces the homomorphism in the statement.

 Now, by \ref{indcatrecall} \eqref{relofindlim}, it suffices to show
 that the homomorphism is an isomorphism in $\mr{Ind}(\mr{Con}(X))$.
 When $f$ is a universal homeomorphism, the canonical homomorphism
 $\ms{F}\rightarrow f_!f^+\ms{F}$ is an isomorphism by Lemma
 \ref{fundproprealsch}.
 Assume $Y=:s$ is a point. There is a separable
 extension $s'\rightarrow s$ such that $X\times s'\rightarrow s'$ is
 disjoint union of universal homeomorphisms.
 Thus, the lemma follows by Lemma \ref{propindcatcon} (ii).

 Let $s$ be a closed point of $Y$.
 Put $i_s\colon s\rightarrow Y$ to be closed immersion. Since $i_s^+$ is
 an exact functor and commutes with direct sum, it commutes with
 arbitrary inductive limits. Thus, we have
 \begin{equation*}
  i_s^+\bigl(\indlim_{\Psi_f}j_{U!}\ms{F}^{I^*}\bigr)\cong
   \indlim_{\Psi_f}i_s^+j_{U!}\ms{F}^{I^*}.
 \end{equation*}
 Let $f_s\colon X\times_Ys\rightarrow s$. There exists a functor
 $\Psi_f\rightarrow\Psi_{f_s}$. This functor is cofinal by
 [EGA IV, 18.12.1]. Then by Lemma \ref{stapoduet},
 $i_s^+\tau_f\cong\tau_{f_s}$, and by the proven case where $Y$ is a
 point, $i_s^+\tau_f$ is an isomorphism. By Lemma \ref{propindcatcon}
 (i), this implies that $\tau_f$ is an isomorphism as required.
\end{proof}

\subsubsection{}
Let $f\colon X\rightarrow Y$ be a quasi-finite flat morphism between
realizable schemes. Let us construct the unique trace map
$f_!f^+K_Y\rightarrow K_Y$ satisfying (Var 1,2,3,4-I). When $f$ is
\'{e}tale, we remark that this trace map coincides with that of Lemma
\ref{stapoduet} by the uniqueness. The construction is the same as
[SGA 4, XVII, 6.2], so we only sketch.

Let $\Psi'_f$ be the full subcategory of $\Psi_f$ consisting of
$\bigl\{U;(V_i)_{i\in I}\bigr\}$ such that $V_i$ is locally free of
constant rank over $U$ for any $i\neq0$.
This category is cofinal in $\Psi_f$. For each $\bigl\{U;(V_i)_{i\in
I}\bigl\}\in\Psi'_f$, we have a homomorphism
\begin{equation*}
 \sum_{i\in I^*}\deg(V_i/U)\cdot\mr{Tr}_{j_U}\colon
 j_{U!}j^+_U(K_Y^{I^*})\rightarrow K_Y.
\end{equation*}
Since the compatibility follows by that of Lemma \ref{stapoduet}, this
homomorphism induces
\begin{equation*}
 f_!f^+K_Y\xleftarrow{\sim}
 \indlim_{\varphi\in\Psi'_f}j_{U!}j_U^+K_Y^{I^*}
 \rightarrow K_Y.
\end{equation*}
It is easy to check that this is what we are looking for.

\begin{lem}
 \label{quafintracalc}
 Let $f\colon X\rightarrow Y$ be the special fiber of a finite \'{e}tale
 morphism between smooth formal curves
 $\widetilde{f}\colon\fsch{X}\rightarrow\fsch{Y}$.
 By taking the dual of the trace map, we get $K^{\omega}_Y\rightarrow
 f_+K^{\omega}_X$. When we restrict this homomorphism to $\fsch{Y}$
 ({\it i.e.}\ taking $\Vert_{\fsch{Y}}$ of \ref{Lemmaofadjtraet}),
 this dual of trace map is nothing but the homomorphism induced by the
 adjunction homomorphism $\phi_{\widetilde{f}}\colon
 \mc{O}_{\fsch{Y},\mb{Q}}\rightarrow
 \widetilde{f}_*\mc{O}_{\fsch{X},\mb{Q}}$ with the identification
 $\widetilde{f}_*\mc{O}_{\fsch{X},\mb{Q}}\cong
 \widetilde{f}_+\mc{O}_{\fsch{X},\mb{Q}}$.
\end{lem}

\begin{proof}
 We may assume $\fsch{X}$ and $\fsch{Y}$ to be connected. Let
 $L_{\fsch{X}}$, $L_{\fsch{X}}$ be the largest finite extension of $K$
 in $\fsch{X}$, $\fsch{Y}$. If $L_{\fsch{X}}\neq L_{\fsch{Y}}$, then
 $\widetilde{f}$ factors as
 \begin{equation*}
  \fsch{X}\xrightarrow{\alpha}\fsch{Y}
   \otimes_{R_{\fsch{Y}}}R_{\fsch{X}}
   \xrightarrow{\beta}\fsch{Y},
 \end{equation*}
 where $R_\star$ denotes the ring of integers of $L_\star$, and it
 suffices to show the lemma for $\alpha$ and $\beta$ separately. The
 verification for $\beta$ is easy. Let us show that for $\alpha$. In
 this case, we let $L_{\fsch{X}}=L_{\fsch{Y}}=:L$.

 Since the formal schemes are curves, by Kedlaya's fully faithfulness
 theorem \cite{Ke}, the homomorphism
 $\mc{O}_{\fsch{Y},\mb{Q}}\rightarrow
 \widetilde{f}_+\mc{O}_{\fsch{X},\mb{Q}}$ induced by
 $\phi_{\widetilde{f}}$ extends uniquely to the homomorphism
 $\phi\colon K^{\omega}_Y\rightarrow f_+K^{\omega}_X$. We have
 \begin{equation*}
  \mr{Hom}(K^{\omega}_Y,f_+f^!K^{\omega}_Y)\cong
   \mr{Hom}(f^+K^{\omega}_Y,f^!K^{\omega}_Y)\sim
   \mr{Hom}(K^{\omega}_X,K^{\omega}_X)\cong L
 \end{equation*}
 where $\sim$ is the isomorphism induced by \cite[Theorem 5.5]{A} since
 $\fsch{X}$ and $\fsch{Y}$ are smooth. Thus, there exists
 $c\in L$ such that $c\cdot\phi=\mb{D}(\mr{Tr}_f)$. It remains to show
 that $c=1$. By definition, the composition $K_Y\rightarrow
 f_!f^+K_Y\xrightarrow{\mr{Tr}_f}K_Y$ is the multiplication by
 $n:=\deg(f)$. Take the dual of this homomorphism, and we get
 \begin{equation*}
  n\colon
  K^{\omega}_Y\xrightarrow{\mb{D}(\mr{Tr}_f)} f_+f^!K^{\omega}_Y
   \xrightarrow{\mr{Tr}^{\mr{Vir}}_f}K^{\omega}_Y,
 \end{equation*}
 where the second homomorphism is the trace map of \cite{Vir}.
 On the other hand, by property of $\mr{Tr}^{\mr{Vir}}_f$
 (cf.\ \cite[III.5.4]{Vir}), we get
 $\mr{Tr}^{\mr{Vir}}_f\circ\phi=n$. Thus $c=1$ since
 $\mr{Hom}(K^{\omega}_Y,K^{\omega}_Y)\cong L$.
\end{proof}

\subsubsection{}
\label{deftrforsmoothformprocu}
Let $f\colon\fsch{X}\rightarrow\fsch{Y}$ be a proper smooth morphism of
relative dimension $1$ between smooth proper formal schemes.
The homomorphism of rings $\mc{O}_{\fsch{Y},\mb{Q}}\rightarrow
f_*\mc{O}_{\fsch{X},\mb{Q}}$ induces the homomorphism
\begin{equation}
 \label{defofderhamhom}
 \mc{O}_{\fsch{Y},\mb{Q}}\rightarrow
  \mb{R}f_*\bigl[
  0\rightarrow\mc{O}_{\fsch{X},\mb{Q}}\rightarrow
  \Omega^1_{\fsch{X}/\fsch{Y},\mb{Q}}\rightarrow0
  \bigr]
\end{equation}
in $D(\DdagQ{\fsch{Y}})$.
Since the target of the homomorphism
is canonically isomorphic to $f_+\mc{O}_{\fsch{X},\mb{Q}}[-1]$,
we have a homomorphism $\alpha_f\colon\mc{O}_{\fsch{Y},\mb{Q}}(1)[2]
\rightarrow f_+f^!\mc{O}_{\fsch{Y},\mb{Q}}$. By \cite[3.14]{A}
and Remark \ref{fundproprealsch} (iii), this
homomorphism is compatible with Frobenius structure when $\base=F$.
This trace map only depends on the special fibers because the unit
element is sent to the unit element by ring homomorphism. Thus, we have
a homomorphism $K^{\omega}_Y(1)[2]\rightarrow f_+f^!K^{\omega}_Y$. By
construction, this homomorphism is compatible with base change, namely,
given a morphism of proper smooth formal schemes
$g\colon\fsch{Y}'\rightarrow\fsch{Y}$ such that
$d:=\dim(\fsch{Y}')-\dim(\fsch{Y})$, let
$f'\colon\fsch{X}':=\fsch{X}\times_{\fsch{Y}}\fsch{Y}'$. Then the
following diagram is commutative:
\begin{equation*}
 \xymatrix{
  \mc{O}_{\fsch{Y}',\mb{Q}}(1)[2]
  \ar[r]^-{\sim}\ar[d]_{\alpha_{f'}}&
  g^!\mc{O}_{\fsch{Y},\mb{Q}}(1)[2-d]
  \ar[d]^{g^!\alpha_f}\\
 f'_+f'^!\mc{O}_{\fsch{Y}',\mb{Q}}
  \ar[r]_-{\sim}&
 g^!f_+f^!\mc{O}_{\fsch{Y},\mb{Q}}[-d]
  }
\end{equation*}
where the horizontal homomorphisms are canonical homomorphisms.
By taking the dual, we get the {\em trace map} $\mr{Tr}_f\colon
f_!f^+K_Y(1)[2]\rightarrow K_Y$. This trace map is compatible with
pull-back $g^+$ where $g$ is a morphism between liftable proper smooth
schemes.

Let us consider the case where $\fsch{Y}$ is a point. Consider a
commutative diagram:
\begin{equation*}
 \xymatrix@R=15pt@C=20pt{
  \fsch{U}\ar@{^{(}->}[rr]\ar[rd]_-{g}&&\fsch{X}\ar[ld]^-{f}\\
  &\mr{Spf}(R)&
  }
\end{equation*}
where $\fsch{U}$ is dense open in $\fsch{X}$. Put $Z:=X\setminus U$
where $X$, $U$ are the special fibers of $\fsch{X}$, $\fsch{U}$, and
assume $Z$ is a divisor of $X$. Then we have an
injection
\begin{equation}
 \label{injrembound}
 \H^{-2}g_+g^!K\cong\H^{-1}f_+\mc{O}_{\fsch{X},\mb{Q}}(^\dag Z)
 \hookrightarrow\H^{-1}g_+\mc{O}_{\fsch{U},\mb{Q}}.
\end{equation}

\subsubsection{Proof of Theorem \ref{traceexitstate}}
\mbox{}\\
Now, we construct the trace map. We need several steps for the
construction.
\medskip

\noindent
{\bf (i) Absolute curve case:}
Let $f\colon X\rightarrow\mr{Spec}(k)$ be a realizable variety. We put
$H_{\mr{c}}^i(X):=\mr{Hom}\bigl(K, f_!(K_X)[i]\bigr)$.
Now, assume $X$ to be a curve, and let us construct the trace map.
First, let us construct the trace map when $X_{\mr{red}}$ is smooth and
irreducible. Let $\iota\colon X_{\mr{red}}\hookrightarrow X$ be the
closed immersion, and
$f'\colon\overline{X}_{\mr{red}}\rightarrow\mr{Spec}(k)$
be the smooth compactification of $X_{\mr{red}}$.
We have already defined $\mr{Tr}_{f'}$ in
\ref{deftrforsmoothformprocu}.
We define
\begin{equation*}
 \mr{Tr}_f\colon
  H_{\mr{c}}^2(X)(1)
  \xrightarrow{\mr{lg}(\mc{O}_{X,\eta})\cdot\iota^*}
  H_{\mr{c}}^2(X_{\mr{red}})(1)
  \xrightarrow{\sim}
  H_{\mr{c}}^2(\overline{X}_{\mr{red}})(1)
  \xrightarrow{\mr{Tr}_{f'}}
  K.
\end{equation*}
In general, we may take an open dense subscheme $U\subset X$ such that
$U_{\mr{red}}$ is smooth. Then we have the canonical isomorphism
$H_{\mr{c}}^2(U)\xrightarrow{\sim}H_{\mr{c}}^2(X)$ by Lemma
\ref{cohdimlem}.
Let $U=\coprod_{i\in I}U_i$ be the decomposition into connected
components. Then we define
\begin{equation*}
 \mr{Tr}_f\colon H_{\mr{c}}^2(X)(1)
  \xrightarrow{\sim}
  H_{\mr{c}}^2(U)(1)
  \cong
  \bigoplus_{i\in I}H_{\mr{c}}^2(U_i)(1)
  \xrightarrow{\sum\mr{Tr}_{f|_{U_i}}}
  K.
\end{equation*}

\begin{lem*}
 Let $X\xrightarrow{f}Y\xrightarrow{g}\mr{Spec}(k)$ be a morphism of
 realizable schemes such that $f$ is quasi-finite flat morphism and
 $g$ is of relative dimension $1$. Then we have
 \begin{equation*}
  \mr{Tr}_{g\circ f}=\mr{Tr}_g\circ g_!(\mr{Tr}_f)\colon
   (g\circ f)_!K_X(1)[2]\rightarrow K.
 \end{equation*}
\end{lem*}
\begin{proof}
 Arguing as [SGA 4, XVIII, 1.1.5], we may assume
 that $X$ and $Y$ are connected smooth affine, and $f$ factors as
 $X\xrightarrow{F}X'\xrightarrow{f'}Y$ where $F$ is
 an iterated relative Frobenius and $f'$ is a finite \'{e}tale
 morphism. It suffices to check the equality for $F$ and $f'$
 individually. For $F$, the claim follows since $\mr{Tr}_{g\circ f}$ is
 compatible with Frobenius structure. It remains to check the lemma when
 $f$ is finite and \'{e}tale. Since $X$ and $Y$ are assumed to be smooth
 and affine, there exist smooth liftings
 $\fsch{X}\xrightarrow{\widetilde{f}}\fsch{Y}\rightarrow\mr{Spf}(R)$
 such that $\widetilde{f}$ is finite flat. In this case, it suffices to
 check the transitivity after removing the boundary by the injectivity
 of (\ref{injrembound}), and the lemma follows by Lemma
 \ref{quafintracalc} and the definition of (\ref{defofderhamhom}).
\end{proof}
\medskip

\noindent
{\bf (ii) Relative affine space case:}
Let $Y$ be a realizable scheme, and consider the projection
$f\colon X:=\mb{P}^1_Y\rightarrow Y$. There exists a proper smooth
formal scheme $\fsch{P}$ such that $Y\hookrightarrow\fsch{P}$. Then $f$
can be lifted to the following cartesian diagram:
\begin{equation*}
 \xymatrix{
  \mb{P}^1_Y\ar@{^{(}->}[r]\ar[d]_{f}\ar@{}[rd]|\square&
  \widehat{\mb{P}}^1_{\fsch{P}}\ar[d]^{\widetilde{f}}\\
 Y\ar@{^{(}->}[r]_i&\fsch{P}.
  }
\end{equation*}
We define the trace map
\begin{equation*}
 \mr{Tr}_f\colon f_!f^+K_Y(1)[2]\cong i^+\widetilde{f}_!
  \widetilde{f}^+K_{\fsch{P}}(1)[2]\xrightarrow{i^+\mr{Tr}_{\widetilde{f}}}
  i^+K_{\fsch{P}}\cong K_Y,
\end{equation*}
where $\mr{Tr}_{\widetilde{f}}$ is the one defined
in \ref{deftrforsmoothformprocu}. This map does not depend on the choice
of $\fsch{P}$ by the base change property of $\mr{Tr}_{\widetilde{f}}$.

When $f\colon \mb{A}^1_X\rightarrow X$ is the projection, then we have
the factorization $\mb{A}^1_X\xrightarrow{j}\mb{P}^1_X\xrightarrow{p}X$,
and the trace map $\mr{Tr}_f$ is defined to be the composition
$\mr{Tr}_p\circ p_!(\mr{Tr}_j)$. The base change property can be checked
by the base change property for $j$ and $p$.
Now, let $f\colon X:=\mb{A}^d_Y\rightarrow Y$. In this case,
we define by iteration as in [SGA 4, XVIII, 2.8].
\medskip

\noindent
{\bf (iii) Factorization case:}
Let $f\colon X\rightarrow Y$ be a morphism which possesses a
factorization $X\xrightarrow{u}\mb{A}^d_Y\xrightarrow{a^d}Y$ such that
$u$ is a quasi-finite flat morphism. Then we define
$t(f,u):=\mr{Tr}_{a^d}\circ a^d_!(\mr{Tr}_u)$. We need to check that
$t(f,u)$ does not depend on the choice of the factorization. By using
the lemma in (i), the verification is the same as [SGA 4, XVIII, 2.9 b)].
\medskip

\noindent
{\bf (iv) General case:}
The construction is the same as [{\it ibid.}, 2.9 c), d), e)]. We sketch
the construction. When $f$ is a Cohen-Macaulay morphism, then there
exists a finite covering of $\{U_i\}$ of $X$ such that the compositions
$U_i\rightarrow X\rightarrow Y$ possess factorizations considered in
case (iii). By gluing lemma \ref{glueindimd}, we have the trace map in
this case. In general, we shrink $X$ suitably, so that $f$ is
Cohen-Macaulay.
Thus the trace map is constructed, and we conclude the proof of Theorem
\ref{traceexitstate}.\qed

\begin{thm}[Poincar\'{e} duality]
 \label{Poindual}
 Let $X\rightarrow Y$ be a smooth morphism of relative dimension $d$
 between realizable schemes. Then for $\ms{F}\in
 D^{\mr{b}}_{\mr{hol}}(Y)$, the adjoint of the trace map
 $\phi_{\ms{F}}\colon f^+\ms{F}(d)[2d]\rightarrow f^!\ms{F}$ is an
 isomorphism.
\end{thm}
\begin{proof}
 Since the verification is local on $X$, it suffices to treat the case
 where $f$ is \'{e}tale and is the projection $\mb{A}^1_Y\rightarrow Y$
 separately. The \'{e}tale case has already been treated in Lemma
 \ref{stapoduet}.

 Let us treat the projection case. We may shrink $Y$. Then we can embed
 $Y$ into a proper smooth formal scheme. By using \cite[Theorem 5.5]{A},
 we have an isomorphism $f^+\ms{F}(d)[2d]\sim f^!\ms{F}$, where $\sim$
 is the isomorphism induced by {\it ibid.}\ and may
 not be the same as the one defined by the trace map.
 It suffices to show the theorem for $\ms{F}\in\mr{Hol}(Y)$. For
 this, we may assume $\ms{F}$ to be irreducible. Let $k'$ be a finite
 extension of $k$. It suffices to show that $\phi_{\ms{F}}$ is an
 isomorphism after pulling-back to $X\otimes_kk'$. Thus, we may assume
 moreover that $\ms{F}$ is irreducible also on $X\otimes_kk'$ for any
 extension $k'$ of $k$. For a closed point $a\in\mb{A}^1$, we denote
 $i_a\colon Y\otimes_kk(a)\rightarrow\mb{A}^1_Y$ the closed immersion
 defined by $a$. We claim that $f^+(\ms{F})$ is irreducible. Indeed,
 first, let us assume $Y$ is smooth and $\ms{F}$ is smooth. Assume
 $f^+(\ms{F})$ were not irreducible. Then there would exist a smooth
 object $\ms{N}\subset f^+(\ms{F})$ and a closed point $a$ of
 $\mb{A}^1(\overline{k})$ such that $i_a^+\ms{N}$ and
 $i_a^+\bigl(f^+(\ms{F})/\ms{N}\bigr)$ are non-zero. This is a
 contradiction. In general, $\ms{F}$ can be written as $j_{!+}$ of a
 smooth irreducible object by \cite[1.4.9]{AC} where $j$ is an open
 immersion of $Y$. Since $f$ is smooth, $f^+$ and $j_{!+}$ commute, and
 $f^+(\ms{F})$ is irreducible.

 Now, we know that
 \begin{equation*}
  \mr{Hom}(f^+\ms{F}(d)[2d],f^!\ms{F})\sim
   \mr{Hom}(f^+\ms{F}(d)[2d],f^+\ms{F}(d)[2d]).
 \end{equation*}
 Since $f^+\ms{F}$ is irreducible, the Hom group is a division algebra,
 and it remains to show that $\phi_{\ms{F}}$ is not $0$. For this, it
 suffices to check that the trace map
 $f_!f^+\ms{F}(d)[2d]\rightarrow\ms{F}$ is non-zero. By base change
 property of trace map, we may assume $Y$ to be a point, in which case,
 the trace map is non-zero by construction.
\end{proof}

\begin{cor*}
 Let $f\colon X\rightarrow Y$ be a flat morphism of relative dimension
 $d$ between smooth realizable schemes. Then the adjoint of trace map
 $f^+K_Y(d)[2d]\rightarrow f^!K_Y$ is an isomorphism.
\end{cor*}
\begin{proof}
 This follows by the transitivity of trace map and the Poincar\'{e}
 duality for both $X$ and $Y$.
\end{proof}

\subsubsection{}
\label{purityforrealsch}
Let $i\colon Z\hookrightarrow X$ be a closed immersion of codimension
$c$ between smooth realizable schemes. By using the Poincar\'{e}
duality, we have a canonical isomorphism
$i^+K^{\omega}_X(-c)[-2c]\xrightarrow{\sim}i^!K^{\omega}_X$.
Let us denote by $(-)\widetilde{\otimes}(-):=\mb{D}
\bigl(\mb{D}(-)\otimes\mb{D}(-)\bigr)$. The projection formula yields
the homomorphism
$i^+\bigl(\ms{N}\widetilde{\otimes}\ms{M}\bigr)\rightarrow
i^+(\ms{N})\widetilde{\otimes}i^!(\ms{M})$ for $\ms{M}$, $\ms{N}$ in
$D^{\mr{b}}_{\mr{hol}}(X/K)$. Using this homomorphism, we get a
homomorphism
\begin{align}
 \notag
 i^+(\ms{M})(-c)[-2c]&\cong
  i^+\bigl(K_X^{\omega}\widetilde{\otimes}\ms{M}\bigr)(-c)[-2c]
  \rightarrow i^+(K_X^{\omega})(-c)[-2c]\widetilde{\otimes}i^!(\ms{M})\\
 \label{purityhomdef}
  &\xrightarrow{\sim} i^!K_X^{\omega}\widetilde{\otimes} i^!(\ms{M})
  \cong i^!(\ms{M}).
\end{align}

\begin{thm*}
 If $\ms{M}$ is an smooth, then the canonical homomorphism
 (\ref{purityhomdef}) is an isomorphism.
\end{thm*}
\begin{proof}
 It suffices to show that when $\ms{M}$ is a smooth holonomic module,
 the canonical homomorphism
 $i^+\bigl(\ms{N}\widetilde{\otimes}\ms{M}\bigr)\rightarrow
 i^+(\ms{N})\widetilde{\otimes}i^!(\ms{M})$ is an isomorphism for any
 $\ms{N}\in D^{\mr{b}}_{\mr{hol}}(X/K)$. Since
 $i_+$ is conservative, it suffices to show that the homomorphism
 \begin{equation*}
  \rho\colon
   i_+i^+\bigl(\ms{N}\widetilde{\otimes}\ms{M}\bigr)\rightarrow
   i_+\bigl(i^+(\ms{N})\widetilde{\otimes}i^!(\ms{M})\bigr)\cong
   i_+i^+(\ms{N})\widetilde{\otimes}\ms{M}
 \end{equation*}
 is an isomorphism. By definition, this is the unique homomorphism which
 makes the following diagram commutative:
 \begin{equation*}
  \xymatrix@R=3pt@C=50pt{
   &i_+i^+\bigl(\ms{N}\widetilde{\otimes}\ms{M}\bigr)\ar[dd]^{\rho}\\
  \ms{N}\widetilde{\otimes}\ms{M}
   \ar[rd]_-{\beta}
   \ar[ur]^-{\alpha}&\\
  &i_+i^+(\ms{N})\widetilde{\otimes}\ms{M}
   }
 \end{equation*}
 where $\alpha:=\mr{adj}_i$ and $\beta:=\mr{adj}_{i}\otimes\mr{id}$.
 Now, since the verification is local, we may assume that $Z$ and $X$
 can be lifted to smooth formal schemes $\fsch{Z}$ and $\fsch{X}$.
 It suffices to show the claim after removing the boundaries by
 \cite[4.3.10]{BeDmod1}. In this situation, recall that
 $(-)\widetilde{\otimes}(-)\cong(-)\otimes^\dag_{\mc{O}_{\fsch{X},\mb{Q}}}
 (-)[-\dim(\fsch{X})]$ (cf.\ \cite[1.1.6]{AC}). Since $\ms{M}$ is a
 coherent $\mc{O}_{\fsch{X},\mb{Q}}$-module, we have a canonical
 isomorphism
 \begin{equation*}
  \mb{R}\shom_{\ms{D}_{\fsch{X}}}
   \bigl(\ms{N}\otimes_{\mc{O}_{\fsch{X}}}
   \ms{M},\ms{D}_{\fsch{X}}\otimes\omega_{\fsch{X}}^{-1}\bigr)
   \cong
  \mb{R}\shom_{\ms{D}_{\fsch{X}}}
  (\ms{N},\ms{D}_{\fsch{X}}\otimes\omega_{\fsch{X}}^{-1})
  \otimes_{\mc{O}_{\fsch{X}}}\ms{M}^\vee,
 \end{equation*}
 where $\ms{D}_{\fsch{X}}$ denotes $\DdagQ{\fsch{X}}$.
 This yields an isomorphism
 $\gamma\colon\mb{D}(\ms{N}\otimes_{\mc{O}_{\fsch{X}}}
 \ms{M})\cong\mb{D}(\ms{N})\otimes_{\mc{O}_{\fsch{X}}}\ms{M}^\vee$.
 Consider the following diagram:
 \begin{equation*}
  \xymatrix@R=15pt@C=30pt{
   \mb{D}\bigl(i_+i^+(\ms{N}\otimes_{\mc{O}_{\fsch{X}}}\ms{M})\bigr)
   \ar@{-}[r]^-{\sim}\ar@/^15pt/[rr]^-{\mb{D}\alpha}&
   i_+i^!\bigl(\mb{D}(\ms{N}\otimes_{\mc{O}_{\fsch{X}}}\ms{M})\bigr)
   \ar[r]\ar@{-}[d]^{\sim}_{\gamma}\ar[r]_-{\mr{adj}_i}&
   \mb{D}(\ms{N}\otimes_{\mc{O}_{\fsch{X}}}\ms{M})
   \ar@{-}[d]^{\sim}_{\gamma}\\
  &i_+i^!\bigl(\mb{D}(\ms{N})\otimes_{\mc{O}_{\fsch{X}}}
   \ms{M}^\vee\bigr)\ar@{-}[d]_{\sim}\ar[r]\ar[r]^-{\mr{adj}_i}
   \ar@{}[rd]|(.4){\heartsuit}&
   \mb{D}(\ms{N})\otimes_{\mc{O}_{\fsch{X}}}\ms{M}^\vee\\
  \mb{D}\bigl(i_+i^+(\ms{N})\otimes_{\mc{O}_{\fsch{X}}}\ms{M}\bigr)
   \ar@{-}[r]_-{\sim}^-{\gamma}
   \ar@{.>}[uu]^{\star}\ar[d]_-{\mb{D}\beta}&
   i_+i^!\mb{D}(\ms{N})\otimes_{\mc{O}_{\fsch{X}}}\ms{M}^\vee
   \ar@/_10pt/[ur]_(.5){\mr{adj}_i}\ar[d]^{\mr{adj}_i}&\\
  \mb{D}(\ms{N}\otimes_{\mc{O}_{\fsch{X}}}\ms{M})
   \ar@{-}[r]_-{\sim}^{\gamma}&
   \mb{D}(\ms{N})\otimes_{\mc{O}_{\fsch{X}}}\ms{M}^\vee
   \ar@{=}@/_25pt/[uur]&
   }
 \end{equation*}
 where $\mr{adj}_i$ is the homomorphism induced by the adjunction
 homomorphisms of $i$. The diagram $\heartsuit$ is commutative by
 \cite[2.2.7]{Casur}, and the other diagrams formed by solid arrows
 are commutative as well.
 We define the isomorphism $\star$ so that the diagram is commutative.
 By the characterization of $\rho$, the isomorphism $\star$ is the dual
 of $\rho$, which implies that $\rho$ is an isomorphism.
\end{proof}

\section{Arithmetic $\ms{D}$-modules for algebraic stacks}
\label{arithDstacksec}
This section is devoted to construct a $p$-adic cohomology theory for
algebraic stacks. Even though we do not try to axiomatize, the
ideas of this section work also for any rational cohomology
theories with standard six functors formalism without essential
changes.
({\it e.g.}\ algebraic $\ms{D}$-module theory, \'{e}tale cohomology
theory with perverse t-structures over a separably closed field with
rational coefficients, {\it etc.})

\subsubsection{}
\label{baseicgeneralsetup}
In this section, $\base\in\{\emptyset,F\}$ is fixed.
Throughout this section, we also fix a base tuple either
$\mf{T}_{\emptyset}:=(k,R,K,L)$ or
$\mf{T}_F:=(k,R,K,L,s,\sigma)$ (cf.\ \ref{defofcatwithcoeff})
depending on which $\base$ we take.
We often denote $D(X/\mf{T}_{\base})$ by $D(X/L_\base)$ or even $D(X/L)$
or $D(X)$ if no confusion may arise.
When $\base=\emptyset$, Tate twists $(n)$ are understood to be
the identity functors as usual. All algebraic stacks are understood to
be over $k$ unless otherwise specified.

\subsubsection{}
\label{fixterminolspace}
To construct a theory for general algebraic stacks, we will first need
to construct a theory for algebraic spaces.
Then using the theory for algebraic spaces, we shall construct a theory
for all algebraic stacks. Since the process for generalizing the
construction from the case of realizable schemes to algebraic
spaces is the same as generalizing from algebraic spaces to algebraic
stacks, we shall present the processes at the same time.
In order to obtain the theory for general algebraic spaces,
the reader should read sections \S\ref{Dmodforstack},
\S\ref{cohfunctorstack} as follows.
First, read \S\ref{Dmodforstack}, \S\ref{cohfunctorstack}
by replacing the word ``space'' (resp.\ ``good stack'', resp.\ ``good
presentation'') by ``quasi-projective scheme'' (resp.\ \dots),
the corresponding terminology indicated in the first
block of the table below.
The reader should then reread \S\ref{Dmodforstack},
\S\ref{cohfunctorstack}. This time using the second block of the table
below instead of the first block.
\begin{center}
 {\renewcommand{\arraystretch}{1.4}
 \begin{tabular}[t]{lll}
  \fbox{1st read}& space:&quasi-projective scheme\\
  & \stack:
      &\parbox[t]{20em}{algebraic stack of finite type whose
	  diagonal morphism is quasi-projective}\\
  & \presentation:
      &\parbox[t]{20em}{smooth surjective morphism from a
	  quasi-projective scheme}\\
  \fbox{2nd read}&space:
      &\parbox[t]{20em}{separated algebraic space of finite type}\\
  &\stack:
      &\parbox[t]{20em}{algebraic stack of finite type}\\
  & \presentation:
      &\parbox[t]{20em}{smooth surjective morphism from a separated
	  quasi-compact algebraic space}
 \end{tabular}
 }
\end{center}
\medskip
Note that since \stacks are of finite type, they are in particular
quasi-compact. See paragraph
\ref{secondreadsub}--\ref{remarkendofsecondpar} for additional
explanation.
Finally, we remark that admissible stacks we define in
\S\ref{sixfuncadmstsec} are \stack in the sense of first read, and
second read is not really necessary if the reader is only interested
in the six functor formalism for schemes or Deligne-Mumford stacks.

\subsection{Definition of the derived category of $\ms{D}^\dag$-modules
  for stacks}
\label{Dmodforstack}

\subsubsection{}
\label{prepcohoprealsch}
We first need basic cohomological operations for spaces. Let $X$ be a
space over $k$. In the ``first read case'', we have already
defined $M(X/L)$ and $D(X/L)$ in \ref{defofcatwithcoeff}, and in the
``second read case'', see \ref{secondreadsub}.
\medskip

\noindent
{\bf Smooth morphism:}
Let $f\colon X\rightarrow Y$ be a smooth
morphism between spaces over $k$. The exact functor $f^*\colon
\mr{Hol}(Y/K)\rightarrow\mr{Hol}(X/K)$ (cf.\ \ref{smoothpushpulsch} or
\ref{smoothpullback}) can
be extended canonically to $M(Y/L)\rightarrow M(X/L)$ by
\ref{indcatrecall} and \ref{Frobsetupp}, which remains to be exact.
The derived functor is also denoted by $f^*$. Similarly, we have the
left exact functor $f_*\colon M(X/L)\rightarrow M(Y/L)$, and we can take
the derived functor $\mb{R}f_*\colon D^+(X/L)\rightarrow D^+(Y/L)$. By
\ref{deradjlemm}, $(f^*,\mb{R}f_*)$ is an adjoint pair.
Since $\mr{for}_L$ and $\mr{for}_F$ commute with $f^*$ and $\mb{R}f_*$
by \ref{extenscasetup} and \ref{dfnoffrobstru}, $f^*$ and $\mb{R}f_*$
preserve holonomicity, and induce functors between
$D^{+}_{\mr{hol}}(X/L)$ and $D^{+}_{\mr{hol}}(Y/L)$.
\medskip

\noindent
{\bf Finite morphism:}
Let $f\colon X\rightarrow Y$ be a finite morphism
between spaces over $k$. Just as in the smooth morphism case, we
can define functors
\begin{equation*}
 f_+\colon D^{\star}(X/L)\rightleftarrows D^{\star}(Y/L)\colon f^!
\end{equation*}
where $\star=\emptyset$ for $f_+$ and $+$ for $f^!$.
The pair $(f_+,f^!)$ is an adjoint pair. These functors commute with
$\mr{for}_L$ and $\mr{for}_F$, and preserve boundedness
and holonomicity.
\medskip

\noindent
{\bf External tensor product:}
Let $X$, $Y$ be spaces over $k$.
Extending the scalar of the external tensor product functor, we have the
bifunctor $\boxtimes\colon M(X/L)\times
M(Y/L)\rightarrow M(X\times Y/L)$, which is exact.
Thus, we can take the derived functor. This derived functor preserves
boundedness and holonomicity as well.
\medskip

\noindent
{\bf Dual functor:}
Let $X$ be a space over $k$.
The dual functor extends canonically to
$\mb{D}_X\colon\mr{Hol}(X/L)^{\circ}\rightarrow\mr{Hol}(X/L)$. This
induces the functor
\begin{equation*}
 \mb{D}_X\colon
  D^{\mr{b}}(\mr{Hol}(X/L))^{\circ}\rightarrow
  D^{\mr{b}}(\mr{Hol}(X/L)).
\end{equation*}

\begin{lem*}
 (i) Consider the cartesian diagram (\ref{cartediagforbc}) of spaces
 such that $f$ is smooth and $g$ is finite. Then the base change
 homomorphisms $f'^*\circ \H^0(g^!)\rightarrow \H^0(g'^!)\circ f^*\colon
 M(Y/L)\rightarrow M(X'/L)$ and $f^*\circ g_+\rightarrow g'_+\circ
 f'^*\colon M(Y'/L)\rightarrow M(X/L)$ are isomorphisms.
 Moreover, we have
 $\H^0(g^!)\circ f_*\cong f'_*\circ\H^0(g')^!\colon M(X/L)\rightarrow
 M(Y'/L)$.

 (ii) For a smooth morphism $f\colon X\rightarrow Y$ of spaces of
 relative dimension $d$, we have the canonical isomorphism
 \begin{equation*}
  \bigl(f^*\circ\mb{D}_Y\bigr)(d)\cong\mb{D}_X\circ f^*\colon
   \mr{Hol}(Y)^\circ\rightarrow\mr{Hol}(X).
\end{equation*}
\end{lem*}

\begin{rem*}
 These equalities also hold on the level of
 $D^{\mr{b}}_{\mr{hol}}(-/L)$, which we show later.
\end{rem*}

\begin{proof}
 In the second read case, we can easily reduce to the first read case,
 so we assume we are in the first read case.
 When $\base=\emptyset$ and $L=K$, we only need to show the equality for
 $\mr{Hol}(-/K)$, and these are consequences of Corollary
 \ref{smbcforopenimm}, Theorem \ref{Poindual},
 \ref{fundproprealsch} \eqref{basechangeprop}.
 Since these equalities are on the level of modules, the results can
 be extended automatically to $\base=F$ and general $L$.
\end{proof}

\subsubsection{}
\label{simplinotfix}
Now, let us fix some terminologies on simplicial spaces.
\begin{dfn*}
 For an integer $i\geq0$, let $[i]=\{0,\dots,n\}$ be the ordered set,
 and put $[-1]:=\emptyset$. Let
 $\Delta^+$ be the category of those objects, and morphisms are
 increasing {\em injective} maps.

 (i) An {\em admissible simplicial space} is a contravariant functor
 $(\Delta^+)^\circ\rightarrow\mr{Sp}^{\mr{sm}}(k)$ where
 $\mr{Sp}^{\mr{sm}}(k)$ denotes the category of spaces
 over $k$ whose morphisms are
 smooth. Let $X_\bullet$ be an admissible simplicial space. For
 $i\geq0$, we often denote $X_\bullet([i])$ by $X_i$. This can be
 described as follows:
 \begin{equation*}
  X_\bullet:
   \left[
    \xymatrix{X_0&X_1\ar@<0.5ex>[l]\ar@<-0.5ex>[l]&
    X_2\ar@<0.7ex>[l]\ar[l]\ar@<-0.7ex>[l]&
    \dots\ar@<1.2ex>[l]\ar@<0.4ex>[l]
    \ar@<-0.4ex>[l]\ar@<-1.2ex>[l]
    }
   \right].
 \end{equation*}
 For a space $S$,
 let $S_\bullet$ be the constant admissible simplicial space. A morphism
 of an admissible simplicial space to a space $X_\bullet\rightarrow S$
 is a morphism $X_\bullet\rightarrow S_\bullet$. The morphism is said to
 be {\em smooth} if $X_0\rightarrow S$ is.

 (ii) A morphism of simplicial spaces $f\colon X_\bullet\rightarrow
 Y_\bullet$ is said to be {\em cartesian} if for any
 $\phi\colon[i]\rightarrow[j]$, the following diagram is cartesian:
 \begin{equation*}
  \xymatrix{
   X_j\ar[r]^{f_j}\ar[d]_{X(\phi)}\ar@{}[rd]|\square&
   Y_j\ar[d]^{Y(\phi)}\\
  X_i\ar[r]_{f_i}&Y_i.
   }
 \end{equation*}

 (iii) An {\em admissible double simplicial space} $X_{\bullet\bullet}$
 is a functor
 $((\Delta^+)^2)^\circ\rightarrow\mr{Sp}^{\mr{sm}}(k)$. An
 admissible simplicial space $S_{\bullet}$ yields the constant
 admissible double simplicial space $S_{\bullet\bullet}$ by setting
 $S_{\bullet i}:=S_\bullet$. A morphism from an admissible double
 simplicial space to an admissible
 simplicial space $X_{\bullet\bullet}\rightarrow S_{\bullet}$ is defined
 to be $X_{\bullet\bullet}\rightarrow S_{\bullet\bullet}$. This is a
 collection of morphisms $X_{n\bullet}\rightarrow S_n$ satisfying
 the compatibility conditions.
\end{dfn*}

\begin{rem*}
 Note that we do not consider ``degeneracy maps'', and only ``face
 maps''. This type of objects are sometimes called strictly simplicial
 schemes ({\it e.g.}\ \cite{OL}).
\end{rem*}

\begin{dfn}
 Let $\st{X}$ be an algebraic stack over $k$, and
 $\algsp{X}\rightarrow\st{X}$ be a presentation (cf.\
 \ref{smpresentnotasur}).
 Put $\algsp{X}_\bullet:=\mr{cosk}_0(\algsp{X}\rightarrow\st{X})$
 ({\it i.e.}\ the simplicial space such that
 $\algsp{X}_n:=\underbrace{\algsp{X}\times_{\st{X}}\times\dots
 \times_{\st{X}}\algsp{X}}_{n}$ and the face morphisms are
 projections).  We say that $\algsp{X}_{\bullet}$ is a
 {\em simplicial algebraic space presentation of $\st{X}$}.
 Let $\mathbf{P}$ be a set of algebraic spaces. A {\em simplicial
 $\mathbf{P}$ presentation} is a simplicial algebraic space presentation
 $X_\bullet$ consisting of algebraic spaces belonging to $\mathbf{P}$
 ({\it e.g.}\ simplicial realizable schemes presentation,...).

 Now, let $\st{X}$ be a \stack, and $X\rightarrow\st{X}$ be a
 \presentation. Since $\st{X}$ is a \stack,
 $\mr{cosk}_0(X\rightarrow\st{X})$ is an admissible simplicial space. In
 particular, for any \stack, we may take a simplicial space
 presentation.
\end{dfn}

\begin{dfn}
 \label{dfncofibcatrel}
 Let $X_\bullet$ be an admissible simplicial space.
 For a morphism $\phi\colon[i]\rightarrow[j]$, since $X(\phi)$ is
 smooth, the pull-back
 \begin{equation*}
  X(\phi)^*\colon M(X_i/L)\rightarrow M(X_j/L),
 \end{equation*}
 which is exact, is defined (cf.\ \ref{prepcohoprealsch}). This defines
 a cofibered category $M(X_\bullet/L)_\bullet$ over $\Delta^+$.

 (i) We put $M(X_\bullet/L):=\mr{sec}_+(M(X_\bullet/L)_\bullet)$ (see
 \S\ref{BDgluingcat} for the notation). We
 often denote $M(X_\bullet/L)$ by $M(X_\bullet)$.
 For $\ms{M}_\bullet$ or $\ms{M}$ in $M(X_\bullet/L)$, the fiber over
 $[i]$ is denoted by $\ms{M}_i$. For $\phi\colon[i]\rightarrow[j]$, the
 homomorphism $X(\phi)^*\ms{M}_i\rightarrow\ms{M}_j$ is called the {\em
 gluing homomorphism}.

 (ii) We denote by $\mr{Hol}(X_\bullet/L)$ or simply by
 $\mr{Hol}(X_\bullet)$ the full subcategory of
 $M(X_\bullet/L)_{\bullet,\mr{tot}}$ consisting of $\ms{M}_{\bullet}$
 such that $\ms{M}_i\in\mr{Hol}(X_i/L)$ for any $i\geq0$.

 (iii) We denote by $D^{\star}_{\mr{hol}}(X_{\bullet}/L)$ or
 $D^{\star}_{\mr{hol}}(X_\bullet)$ the full subcategory of
 $D^{\star}(M(X_\bullet/L))$ whose cohomology objects are in
 $\mr{Hol}(X_\bullet/L)$. We denote
 $D^\star_{\mr{tot}}(M(X_\bullet/L)_{\bullet})$ by
 $D^\star_{\mr{tot}}(X_\bullet/L)$ or $D^\star_{\mr{tot}}(X_\bullet)$.
\end{dfn}

\subsubsection{}
\label{sepacseqBDandmis}
Let $\ms{M}_\bullet$ and $\ms{N}_\bullet$ be in
$D^{\mr{b}}_{\mr{tot}}(X_\bullet/L)$.
Then by (\ref{BDssingen}), we have the following spectral sequence:
\begin{equation}
 \label{BDspecseq}
 E^{p,q}_1:=\mr{Ext}^{q}_{D(X_{p})}(\ms{M}_p,\ms{N}_p)\Rightarrow
  \mr{Ext}^{p+q}_{D(X_\bullet)}(\ms{M}_\bullet,\ms{N}_\bullet).
\end{equation}
We have kernels, cokernels, and inductive limits in $M(X_\bullet)$, and
they can be calculated termwise, namely these functors commute with the
functor sending $\ms{M}_{\bullet}$ to $\ms{M}_i$ for any $i\geq0$.
This is because, for any $\phi\colon[i]\rightarrow[j]$, the functors
$\mr{Ker}$, $\mr{Coker}$,
$\indlim$ commute with $X(\phi)^*$. In particular, $M(X_\bullet)$ is an
abelian category. Moreover, projective limits are representable in
$M(X_\bullet)$ and they can be calculated termwise as well, by using the
canonical homomorphism
$X(\phi)^*\circ\invlim\rightarrow\invlim\circ X(\phi)^*$.

\subsubsection{}
\label{functorrhodef}
Let us define three basic functors.
Take $i\geq0$. We define a functor $\rho^*_{i}\colon
M(X_\bullet)\rightarrow M(X_i)$ by sending $\ms{M}_\bullet\in
M(X_\bullet)$ to $\ms{M}_i$. Obviously, this is an exact functor.
Now, take $\ms{N}\in M(X_i)$. We define
\begin{equation*}
 \rho_{i*}(\ms{N}):=\left\{\prod_{\phi\colon[k]\rightarrow[i]}
  X(\phi)_*(\ms{N})\right\}_{k},\qquad
  \rho_{i!}(\ms{N}):=\left\{\bigoplus_{\phi\colon[i]\rightarrow[k]}
		      X(\phi)^*(\ms{N})\right\}_{k},
\end{equation*}
and the gluing homomorphisms are defined as follows: for
$\psi\colon[k]\rightarrow[k']$, the map $X(\psi)^*\rho_{i*}(\ms{N})_k
\rightarrow\rho_{i*}(\ms{N})_{k'}$ (resp.\ $X(\psi)^*\rho_{i!}(\ms{N})_k
\rightarrow\rho_{i!}(\ms{N})_{k'}$) is the product (resp.\ direct sum)
of the adjunction (resp.\ canonical) homomorphisms
\begin{equation*}
 X(\psi)^*X(\phi)_*(\ms{N})\rightarrow X(\phi')_*(\ms{N}),\qquad
  \bigl(\mbox{resp.\ }
  X(\psi)^*X(\phi)^*(\ms{N})\rightarrow X(\phi')^*(\ms{N})
  \bigr)
\end{equation*}
where $\phi\colon[k]\xrightarrow{\psi}[k']\xrightarrow{\phi'}[i]$
and $0$ if $\phi$ cannot be factored through $\psi$ (resp.\
$\phi'\colon[i]\xrightarrow{\phi}[k]\xrightarrow{\psi}[k']$).
These data define functors
$\rho_{i*},\rho_{i!}\colon M(X_i)\rightarrow M(X_\bullet)$.

\begin{lem*}
 (i) We have adjoint pairs $(\rho^*_i,\rho_{i*})$ and
 $(\rho_{i!},\rho^*_i)$.

 (ii) The functors $\rho^*_i$ and $\rho_{i!}$ are exact. In particular
 $\rho^*_i$ and $\rho_{i*}$ preserve injective objects.

 (iii) The category $M(X_\bullet)$ is a Grothendieck category.
\end{lem*}
\begin{proof}
 Since the verification of (i) is standard, we leave the details to the
 reader.
 The first claim of (ii) follows from the exactness of $X(\phi)^*$. Let
 us check (iii). We have arbitrary inductive limit as observed in
 \ref{sepacseqBDandmis}, and filtrant inductive limits are exact.
 We need to show that it has a generator. Let $\ms{G}_i$
 be a generator of $M(X_i)$. Then
 $\bigl\{\rho_{i!}(\ms{G}_i)\bigr\}_{i\geq0}$ is a set of
 generators. Indeed, let $\ms{M}\in M(X_\bullet)$, and assume
 $\mr{Hom}(\rho_{i!}(\ms{G}_i),\ms{M})=0$ for any $i\geq0$. Then by (i),
 we have $\mr{Hom}(\ms{G}_i,\rho^*_i(\ms{M}))=0$, and thus
 $\ms{M}_i=0$. Thus by definition, $\ms{M}=0$.
\end{proof}

\subsubsection{}
Let $X_\bullet$ be an admissible simplicial space. Let $k$ be a
non-negative integer. Given $\ms{M}_\bullet\in M(\mr{sk}_k(X_\bullet))$,
we get an object in $M(\mr{sk}_{k-1}(X_\bullet))$ denoted by
$\sigma_{k}^*(\ms{M}_\bullet)\in M(X_\bullet)$ by putting
$\bigl(\sigma^*_k(\ms{M}_\bullet)\bigr)_i\cong\ms{M}_i$ for $i<k$. Let
us construct the left adjoint functor of $\sigma_{k}^*$.

Let $D_k$ be the category of homomorphisms $\phi\colon[i]\rightarrow[k]$
such that the morphism from $\phi$ to $\psi\colon[j]\rightarrow[k]$ is
a morphism $\alpha\colon[i]\rightarrow[j]$ such that
$\psi\circ\alpha=\phi$. Now, given $\ms{M}_\bullet$ in
$M(\mr{sk}_{k-1}(X_\bullet))$, we can construct an object in
$M(\mr{sk}_{k}(X_\bullet))$ as follows: Put
\begin{equation*}
 \ms{M}_{k}:=\indlim_{\phi\in D_k}X(\phi)^*(\ms{M}_i).
\end{equation*}
Then $\bigl\{\ms{M}_i\bigr\}_{i\leq k}$ with obvious gluing
homomorphism defines the desired object.
The functor is denoted by $\sigma_{k!}$. We can check easily
that $(\sigma_{k!},\sigma^*_k)$ is an adjoint pair.

\begin{rem}
 The functors defined in the last two paragraphs have natural
 interpretation in terms of the language of topos.
 Let $X_\bullet$ be a (strictly) simplicial topos (cf.\ [SGA 4,
 $\text{V}^{\text{bis}}$]).
 Since sheaves $\ms{F}$ on $X_\bullet$ can be described as data
 $\{\ms{F}_i,\phi_{ij}\}$ where $\ms{F}_i$ is a sheaf on $X_i$ and
 $\phi_{ij}$ is a gluing homomorphism.
 We have a functor sending a sheaf $\ms{F}$ on $X_\bullet$ to the
 sheaf $\ms{F}_i$ on $X_i$. This functor defines a morphism of topos
 $e_i\colon X_i\rightarrow X_\bullet$.
 The functors $\rho_{i!}$, $\rho_i^*$, $\rho_{i*}$ are nothing but
 analogues of the functors $e_{i!}$, $e_i^*$, $e_{i*}$  (cf.\ [{\it
 ibid.}, 1.2.8--1.2.12]). The interpretation of $\sigma_{i!}$,
 $\sigma_i^*$ is similar.
\end{rem}

\begin{lem}
 \label{catequforhol}
 Let us denote by $\mr{Hol}(X_\bullet)_\bullet$ the cofibered category
 over $\Delta^+$ such that the fiber over $[i]$ is $\mr{Hol}(X_i)$.
 Let $D^{\mr{b}}_{\mr{tot}}(\mr{Hol}(X_\bullet)_\bullet)$ be the derived
 category defined in \S\ref{BDgluingcat}. The canonical functor
 $D^{\mr{b}}_{\mr{tot}}(\mr{Hol}(X_\bullet)_\bullet)\rightarrow
 D^{\mr{b}}_{\mr{hol}}(X_\bullet)$ is an equivalence of categories.
\end{lem}
\begin{proof}
 We can argue as \cite[15.3.1]{KSc}. It suffices to show that the
 functor
 \begin{equation*}
  D^{\mr{b}}(\mr{Hol}(X_\bullet)_\bullet)\rightarrow
   D^{\mr{b}}(X_\bullet)
 \end{equation*}
 is fully faithful. By \cite[13.2.8]{KSc}, it suffices to
 show the following: given a surjection $\ms{A}\rightarrow\ms{M}$ such
 that $\ms{M}\in\mr{sec}_+\mr{Hol}(X_\bullet)_{\bullet}$
 and $\ms{A}\in M(X_\bullet)$,
 there exists a homomorphism $\ms{N}\rightarrow\ms{A}$ such that
 $\ms{N}\in\mr{sec}_+\mr{Hol}(X_\bullet)_{\bullet}$ and the composition
 $\ms{N}\rightarrow\ms{M}$ is surjective. To check this, it suffices to
 construct, for each $k\geq0$, the following $\ms{N}_{(k)}$ in
 $M(\mr{sk}_k(X_\bullet))$: 1.\
 $\sigma_{k}^*(\ms{N}_{(k)})\cong\ms{N}_{(k-1)}$;
 2.\ we have a homomorphism $\ms{N}_{(k)}\rightarrow\sigma_k^*(\ms{A})$
 such that the composition $\ms{N}_{(k)}\rightarrow\sigma^*_k
 (\ms{A})\rightarrow\sigma^*_k(\ms{M})$ is surjective.

 We use the induction on $k$. For $\ms{N}_{(0)}$, take one as in [{\it
 ibid.}, 15.3.1]. Assume we have constructed $\ms{N}_{(k-1)}$. Take
 $\ms{N}'\rightarrow\ms{A}_k$ such that $\ms{N}'\in\mr{Hol}(X_k)$ and
 the composition with $\ms{A}_k\rightarrow\ms{M}_k$ is surjective, as in
 [{\it ibid.}, 15.3.1]. Given $\phi\colon[i]\rightarrow[k]$, we have the
 following diagram:
 \begin{equation*}
  \xymatrix{
   X(\phi)^*\ms{N}_{(k-1),i}\ar[r]\ar@{.>}[d]_{?}&
   X(\phi)^*\ms{A}_i\ar[r]\ar[d]&
   X(\phi)^*\ms{M}_i\ar[d]\\
  \ms{N}'\ar[r]&\ms{A}_k\ar[r]&\ms{M}_k.
  }
 \end{equation*}
 We need to construct the dotted homomorphism making the diagram
 commutative, for which we modify $\ms{N}'$. We put
 \begin{equation*}
  (\ms{N}_{(k)})_i:=
   \begin{cases}
    (\ms{N}_{(k-1)})_i&\mbox{for $i<k$}\\
    \ms{N}'\oplus(\sigma_{k!}(\ms{N}_{(k-1)}))_k&\mbox{for $i=k$}.
   \end{cases}
 \end{equation*}
 With the obvious gluing homomorphisms, these data define an object
 $\ms{N}_{(k)}$ in $M(\mr{sk}_k(X_\bullet))$, which is what we are
 looking for.
\end{proof}

\subsubsection{}
\label{defofpushfinitesimpl}
Let $f_\bullet\colon X_\bullet\rightarrow Y_\bullet$ be a {\em
cartesian} morphism of admissible simplicial spaces such that $f_i$ is
{\em finite} for any $i\geq0$. We call such a morphism {\em cartesian
finite morphism} for short.
For a morphism $\phi\colon[j]\rightarrow[k]$, we have the following
cartesian diagram
\begin{equation*}
 \xymatrix{
  X_k\ar[r]^{f_k}\ar[d]_{X(\phi)}\ar@{}[rd]|\square&
  Y_k\ar[d]^{Y(\phi)}\\
 X_j\ar[r]_{f_j}&Y_j
  }
\end{equation*}
where $f_j$ and $f_k$ are finite. Let $\ms{M}_{\bullet}$ be an object in
$M(Y_\bullet)$. For a finite morphism $g$, we denote $\H^0g^!$ by
$g^{\circ}$, which is left exact by \ref{finiteadjpropsch}. We have a
canonical homomorphism
\begin{equation*}
 X(\phi)^*f_j^{\circ}(\ms{M}_j)\cong
  f_k^{\circ}Y(\phi)^*(\ms{M}_j)
  \rightarrow f_k^{\circ}(\ms{M}_k)
\end{equation*}
by Lemma \ref{prepcohoprealsch}.
Using this homomorphism,
$\bigl\{f^{\circ}_k(\ms{M}_k)\bigr\}$ defines an object in
$M(X_\bullet)$, and defines a functor $f^{\circ}\colon
M(Y_\bullet)\rightarrow M(X_\bullet)$. Since $f^{\circ}_k$ is left
exact, $f^{\circ}$ is left exact as well.
We can take the associated derive functor to get
\begin{equation*}
 f^!:=\mb{R}f^{\circ}\colon D^+(M(Y_\bullet))
  \rightarrow
  D^+(M(X_\bullet)).
\end{equation*}

On the other hand, the functor $f_{k+}$ is exact. For $\ms{N}\in
M(X_\bullet)$, using Lemma \ref{prepcohoprealsch}, we have a
homomorphism
\begin{equation*}
 Y(\phi)^*f_{j+}(\ms{N}_j)\cong
  f_{k+}X(\phi)^*(\ms{N}_j)
  \rightarrow f_{k+}(\ms{N}_k),
\end{equation*}
which defines an object $\bigl\{f_{k+}(\ms{N}_k)\bigr\}$ in
$M(Y_\bullet)$. The functor is denoted by $f_+$. Since this functor is
exact, we can take the derive functor
\begin{equation*}
 f_+\colon D(M(X_\bullet))\rightarrow D(M(Y_\bullet)).
\end{equation*}

\begin{lem*}
 Let $X_\bullet\xrightarrow{f}Y_\bullet\xrightarrow{g}Z_\bullet$ be
 cartesian finite morphisms of admissible simplicial spaces.

 (i) We have a canonical isomorphism $\rho^*_k\circ f^!\cong
 f_k^!\circ\rho^*_k$. In particular, $f^!$ sends
 $D^{\star}_{\mr{hol}}(Y_\bullet)$ into
 $D^{\star}_{\mr{hol}}(X_\bullet)$ for
 $\star\in\bigl\{+,\mr{b}\bigr\}$.

 (ii) We have an adjoint pair $(f_+,f^!)$, and $f_+$ is exact. In
 particular, $f^{\circ}$ preserves injective objects.

 (iii) We have a canonical isomorphism $f^!\circ g^!\cong(g\circ
 f)^!$.
\end{lem*}
\begin{proof}
 Let us show (i). Since $\rho^*_k$ is exact and preserves injective
 objects by Lemma \ref{functorrhodef}, the first isomorphism follows by
 definition. Let us check the second one. The
 functor $f^!$ preserves total complexes by Lemma \ref{prepcohoprealsch}
 (i). It remains to show that it preserves holonomicity and boundedness.
 These are immediate consequences of the isomorphism $\rho_k^*\circ
 f^!\cong f_k^!\circ\rho_k^*$.

 The verification of (ii) is easy. To show (iii), by (ii), it suffices
 to show that $f^{\circ}\circ g^{\circ}\cong(g^{\circ}\circ
 f^{\circ})$. This follows by definition, and the corresponding
 statement for spaces.
\end{proof}

\subsubsection{}
\label{constofpairfunct}
Now, we use the notations of \ref{prepcohoprealsch}.
Let $f\colon X_\bullet\rightarrow S$ be a smooth morphism
(cf.\ \ref{simplinotfix})
from an admissible simplicial space to a space. In this
situation, let us define an adjoint pair of functors
$(f^*,\mb{R}f_*)$. Let $f_i\colon X_i\rightarrow S$ be the induced
morphism. The pull-back is easy to define: Let $\ms{N}\in M(S)$. We put
$\ms{N}_i:=f_i^*(\ms{N})$ which is defined in $M(X_i)$. Let
$\phi\colon[i]\rightarrow[j]$ be a map. Then we define a homomorphism,
which is in fact an isomorphism, $X(\phi)^*\ms{N}_i\rightarrow\ms{N}_j$
to be the gluing homomorphism. The object we constructed in
$M(X_\bullet)$ is denoted by $f^*(\ms{N})$. Thus, we have a functor
\begin{equation*}
 f^*\colon M(S)\rightarrow M(X_\bullet).
\end{equation*}
The functor $f^*$ is exact since each $f_i^*$ is.

Let us define its right adjoint. Take $\ms{M}_\bullet$ in
$M(X_\bullet)$. For $\phi\colon[i]\rightarrow[j]$, we have the
homomorphism
\begin{equation*}
 \alpha_\phi\colon f_{i*}(\ms{M}_i)\rightarrow f_{j*}(\ms{M}_j).
\end{equation*}
We put
\begin{equation*}
 f_*(\ms{M}_\bullet):=\mr{Ker}\bigl(f_{0*}(\ms{M}_0)\rightrightarrows
  f_{1*}(\ms{M}_1)\bigr).
\end{equation*}
Since $f_{0*}$ and $f_{1*}$ are left exact, the functor $f_*$ is left
exact as well. Thus we may take the associated derived functor to get
\begin{equation*}
 \mb{R}f_*\colon D^{+}(X_\bullet)\rightarrow D^+(S).
\end{equation*}

These constructions can be generalized to a smooth morphism
$f\colon X_{\bullet\bullet}\rightarrow S_\bullet$ from a double
simplicial space to a simplicial space: Pull-back is defined in an
obvious manner. Let $\ms{M}_{\bullet\bullet}$ in
$M(X_{\bullet\bullet})$. Recall that $f$ is a collection of
$f_{n\bullet}\colon X_{n\bullet}\rightarrow S_n$ satisfying
the compatibility conditions. We get $f_{n\bullet*}(\ms{M}_{n\bullet})$
in $M(S_n)$. The transition homomorphism is defined by adjunction, and
we have $f_{\bullet\bullet*}\colon M(X_{\bullet\bullet})\rightarrow
M(S_\bullet)$. This is a left exact functor, and we get the derived
functor $\mb{R}f_*\colon D^+(X_{\bullet\bullet})\rightarrow
D^+(S_\bullet)$.

\begin{lem*}
 We have an adjoint pair $(f^*,f_*)$. Thus, so is the pair
 $(f^*,\mb{R}f_*)$, and the obvious analogue holds for the double
 simplicial case.
\end{lem*}
\begin{proof}
 The adjointness of $(f^*,\mb{R}f_*)$ follows by the first one using
 Lemma \ref{deradjlemm}.
 Let $S_\bullet$ be the constant simplicial space. We have the morphism
 $f_\bullet\colon X_\bullet\rightarrow S_\bullet$, and the adjoint pair
 $(f_i^*,f_{i*})$ defines a pair of functors
 $(f_{\bullet}^*,f_{\bullet*})$ between $M(X_\bullet)$ and
 $M(S_\bullet)$. It is straightforward to check that this is an adjoint
 pair. Thus, it suffices to check the lemma for the morphism
 $S_\bullet\rightarrow S$. This follows from the following general fact:
 Let $\mc{A}$ be an abelian category, and $\Delta^+\mc{A}$ be the
 category of cosimplicial objects, namely the abelian category of
 functors $\Delta^+\rightarrow\mc{A}$. Let
 $\rho^*\colon\mc{A}\rightarrow\Delta^+\mc{A}$ be functor assigning the
 constant object, and $\rho_*\colon\Delta^+\mc{A}\rightarrow\mc{A}$ to
 be the functor associating $\mr{Ker}(M_0\rightrightarrows M_1)$ to
 $\{M_i\}$. Then $(\rho^*,\rho_*)$ is an adjoint pair. The verification
 is straightforward. The double simplicial case is similar.
\end{proof}

\subsubsection{}
The following spectral sequence is one of the keys to show the
cohomological descent.

\begin{lem*}
 \label{fundssforsimpldes}
 Let $\ms{M}\in D^+(X_\bullet)$. Then we have the following spectral
 sequence:
 \begin{equation*}
  E_1^{p,q}:=\mb{R}^qf_{p*}(\ms{M}_p)\Rightarrow
   \mb{R}^{p+q}f_*(\ms{M}).
 \end{equation*}
\end{lem*}
\begin{proof}
 The construction is essentially the same as \cite[Corollary
 2.7]{OlSA}. Since we use the similar argument again later, we sketch
 the proof. Let $\ms{N}\in M(X_i)$. We have a homomorphism
 $f_{0*}\bigl(\rho_{i*}(\ms{N})_0\bigr)\cong\prod_{[0]\rightarrow[i]}
 f_{i*}(\ms{N})\rightarrow f_{i*}(\ms{N})$ where the second homomorphism
 is the projection to the component of the map
 $\alpha\colon[0]\rightarrow[i]$ such that the image is $0\in[i]$. This
 induces a homomorphism from the \v{C}ech type complex:
 \begin{equation*}
  C^{\bullet}_i(\ms{N}):=\left[
  0\rightarrow f_{0*}\bigl(
  \rho_0^*\rho_{i*}(\ms{N})\bigr)\rightarrow
   f_{1*}\bigl(\rho_1^*\rho_{i*}(\ms{N})\bigr)\rightarrow
   f_{2*}\bigl(\rho_2^*\rho_{i*}(\ms{N})\bigr)\rightarrow\dots
   \right]
 \end{equation*}
 to $f_{i*}(\ms{N})$. This homomorphism is in fact a homotopy equivalence.
 Indeed, the cohomologies of the complex $\widetilde{T}:=\left[
 0\rightarrow\prod_{[0]\rightarrow[i]}L
 \rightarrow\prod_{[1]\rightarrow[i]}L\rightarrow\dots\right]$,
 which is isomorphic to that of the $i$-simplex, vanish except
 for degree $0$, and since $K(\mr{Vec}_L)\cong D(\mr{Vec}_L)$, the
 homomorphism $\widetilde{T}\rightarrow L$ of the projection to
 $\alpha$-component is a homotopy equivalence. Since
 $C_i^{\bullet}(\ms{N})\cong\widetilde{T}\otimes_Lf_{i*}(\ms{N})$, we
 get the claim.

 Now, for any $\ms{N}\in M(X_\bullet)$, there exists an embedding
 $\ms{N}\hookrightarrow\ms{I}$ into an injective object in
 $M(X_\bullet)$ such that the complex
 \begin{equation}
  \label{Cechcomplforsipu}
   \tag{$\star$}
  0\rightarrow f_{0*}(\rho_0^*(\ms{I}))
  \rightarrow f_{1*}(\rho_1^*(\ms{I}))
   \rightarrow f_{2*}(\rho_2^*(\ms{I}))
   \rightarrow\cdots
 \end{equation}
 is exact away from degree $0$ part. For this, take an embedding
 $\rho_i^*\ms{N}\hookrightarrow\ms{I}_{(i)}$ into an injective object in
 $M(X_i)$, and put $\ms{I}:=\prod_i\rho_{i*}(\ms{I}_{(i)})$. Note that
 products can be calculated termwise by
 \ref{sepacseqBDandmis}. Moreover, small products and $f_{i*}$ commute
 by \cite[2.1.10]{KSc} since $f_{i*}$ admits a left adjoint $f^*_i$.
 This implies that (\ref{Cechcomplforsipu}) is isomorphic to
 $\prod_iC_i^{\bullet}(\ms{I}_{(i)})$. Now, each
 $C^{\bullet}_i(\ms{I}_{(i)})$ is homotopic to $f_{i*}(\ms{I}_{(i)})$ by
 the observation above. Since homotopy equivalence is preserved even
 after taking product, we get that (\ref{Cechcomplforsipu}) is homotopic
 to $\prod_{i} f_{i*}\ms{I}_{(i)}$, in particular, the complex is exact
 away from $0$. Since $\rho_{i*}$ preserves injective objects by Lemma
 \ref{functorrhodef} and the product of injective objects remains to be
 injective, $\ms{I}$ is a desired object.

 Finally, let $\ms{M}\rightarrow\ms{I}^{\bullet}$ be an resolution of
 $\ms{M}$ consisting of complex as above. We consider the double complex
 $\bigl\{f_{p*}\rho_p^*(\ms{I}^q)\bigr\}_{p,q\geq0}$. The acyclicity of
 (\ref{Cechcomplforsipu}) except for degree $0$ shows that the total
 complex is $\mb{R}f_*(\ms{M})$, and thus the associated spectral
 sequence is the one we want.
\end{proof}

Recall that by Lemma \ref{smoothpushpulsch}, or by \ref{smoothpullback}
in the second read case, $\mb{R}f_{i*}$ preserves holonomicity. Thus, we
have the following corollary:

\begin{cor*}
 The functor $\mb{R}f_*$ preserves holonomicity, and induces a functor
 $D^+_{\mr{hol}}(X_\bullet)\rightarrow D^+_{\mr{hol}}(S)$.
\end{cor*}

\begin{prop}
 \label{smoothdesre}
 Let $X\rightarrow S$ be a smooth surjective morphism between spaces,
 and put $f\colon X_\bullet:=\mr{cosk}_0(X/S)\rightarrow S$.
 The adjoint pair of functors $(f^*,\mb{R}f_*)$ induces an equivalence
 between $D^{\star}_{\mr{tot}}(X_\bullet)$ and $D^{\star}(S)$
 for $\star\in\{+,\mr{b}\}$. Moreover, it induces an equivalence between
 $D^{\star}_{\mr{hol}}(X_\bullet)$ and $D^{\star}_{\mr{hol}}(S)$.
\end{prop}
\begin{proof}
 The second claim follows by the first one since $f^*$ and $\mb{R}f_*$
 preserve the holonomicity by Corollary \ref{fundssforsimpldes}. Thus,
 it suffices to show that the canonical homomorphisms
 $\mr{id}\rightarrow\mb{R}f_*f^*$ and $f^*\mb{R}f_*\rightarrow\mr{id}$
 are isomorphisms. For the first one, it suffices to show the equalities
 after taking $i^!_s$ where $s$ is a closed point of $S$ and
 $i_s\colon\{s\}\hookrightarrow S$. By taking the fiber product, it
 induces a morphism of simplicial spaces $i_{X,s}\colon X_{s
 \bullet}\hookrightarrow X_\bullet$. Using Lemma \ref{prepcohoprealsch},
 we have $\mb{R}f_{s*}\circ i^!_{X,s}\cong i^!_s\circ\mb{R}f_*$ where
 $f_s\colon X_{s\bullet}\rightarrow s$.
 Thus, we can reduce to the situation where we have a section
 $s\colon S\rightarrow X_0$ of $f_0$.
 In this case, the verification for the first homomorphism is the same
 as \cite[Thm 7.2]{Con}. Let us recall the argument briefly.
 By using Lemma \ref{fundssforsimpldes}, we have the spectral sequence
 $E_1^{p,q}=\mb{R}^qf_{p*}(\ms{M}_p)\Rightarrow\mb{R}^{p+q}f_*(\ms{M})$.
 Put $E_1^{-1,q}:=\mb{R}^q\mr{id}_*(\ms{M})$, which is $\ms{M}$ if $q=0$
 and $0$ otherwise.
 We claim that the complex defined by adjunction $0\rightarrow
 E_1^{-1,q}\rightarrow E_1^{\bullet,q}\rightarrow0$ where $E_1^{-1,q}$
 is placed at degree $-1$ is acyclic. To check this, we construct a
 concrete homotopy $E_1^{p,q}\rightarrow E_1^{p-1,q}$ using the section
 $s$. For example $E_1^{0,q}\rightarrow E_1^{-1,q}$ is constructed as
 follows: We have an isomorphism
 $s^*f^*_{0}\ms{M}\xrightarrow{\sim}\ms{M}$, which induces
 $f^*_{0}(\ms{M})\rightarrow s_*(\ms{M})$ by adjunction. By taking
 $f_{0*}$ and using the isomorphism $f_{0*}s_*\cong\mr{id}$, we obtain
 the desired homotopy. See {\it ibid.}\ for the details.

 Let us show the second one. It suffices to show that the homomorphism
 of functors, before taking the derived functors,
 $f^*f_*\rightarrow\mr{id}$ is an isomorphism. Indeed, if this is shown,
 we get that for $\ms{M}\in M(X_\bullet)$, we have
 \begin{equation*}
  \mb{R}f_*(\ms{M})\xleftarrow{\sim}\mb{R}f_*
   \bigl(f^*f_*(\ms{M})\bigr)\xleftarrow{\sim}f_*(\ms{M})
 \end{equation*}
 where the second quasi-isomorphism follows by the the first part of the
 proof.
 Thus we have $\mb{R}^if_*(\ms{M})=0$ for $i\neq0$. Finally, let
 us show $f^*f_*(\ms{M})\rightarrow\ms{M}$ is an isomorphism. As the
 proof of the first isomorphism,
 it suffices to show the claim when $S$ is a point,
 and in particular there is a section $S\rightarrow X$.
 In this case, it suffices to show that there exists
 $\ms{N}\in M(S)$ such that $f^*(\ms{N})\cong\ms{M}$, namely,
 $\ms{M}$ is ``effective descent''.
 Because of the existence of the section, this is automatic
 (for example, see \cite[right after 6.15]{Gi}).
\end{proof}

\subsubsection{}
\label{smoothbasechforsimplspa}
Let $f\colon X_\bullet\rightarrow S$ be a smooth morphism from an
admissible simplicial space to a space, and $g\colon S'\rightarrow S$ be
a smooth morphism between spaces. Consider the following cartesian
diagram:
\begin{equation*}
 \xymatrix{
  X'_\bullet\ar[r]^-{g'}\ar[d]_{f'}\ar@{}[rd]|\square&
  X_\bullet\ar[d]^{f}\\
 S'\ar[r]_-{g}&
  S
  }
\end{equation*}
For $\ms{M}\in M(X_\bullet)$, the system
$\bigl\{g'^*_i(\ms{M}_i)\bigr\}_{i}$ defines an object of
$M(X'_\bullet)$, denoted by $g'^*(\ms{M})$. This functor $g'^*$ is
exact, and preserves holonomicity.
We can take the derived functor $g'^*\colon
D^{\mr{b}}_{\mr{hol}}(X_\bullet)\rightarrow
D^{\mr{b}}_{\mr{hol}}(X'_\bullet)$. Similarly, for $\ms{N}\in
M(X'_\bullet)$, the system $\bigl\{g'_{i*}(\ms{N}_i)\bigr\}$ defines an
object of $M(X_\bullet)$ by Corollary \ref{smbcforopenimm}, or by the
following lemma in the second read case, and the assumption that $g'$ is
cartesian. This functor is left exact, and we have the right derived
functor $\mb{R}g'_*$. The couple $(g'^*,\mb{R}g'_*)$ is an adjoint
pair.

\begin{lem*}
 The canonical homomorphism $g^*\circ\mb{R}f_*\rightarrow
 \mb{R}f'_*\circ g'^*$ is an isomorphism.
\end{lem*}
\begin{proof}
 By using the spectral sequence of Lemma \ref{fundssforsimpldes}, the
 verification is reduced to Corollary \ref{smbcforopenimm}, or to the
 present lemma of the first read case in the second read.
\end{proof}

\begin{prop}
 \label{equivcatstackdfn}
 Let $\st{X}$ be a \stack, and $X_\bullet\rightarrow\st{X}$ be a
 simplicial space presentation. Then for $\star\in\{\mr{b},+\}$, the
 categories $D^{\star}_{\mr{tot}}(X_\bullet)$ and
 $D^{\star}_{\mr{hol}}(X_\bullet)$ do not depend on the choice of the
 presentation up to canonical equivalence, and the t-structure as well.
\end{prop}
\begin{proof}
 Let $X_\bullet\rightarrow\st{X}$ and $X'_\bullet\rightarrow\st{X}$ be
 two presentations. Let $Z_{n,n'}:=X_n\times_{\st{X}}X'_{n'}$. This
 defines a double simplicial space $Z_{\bullet\bullet}$ with projections
 $Z_{\bullet\bullet}\rightarrow X_\bullet$, and
 $Z_{\bullet\bullet}\rightarrow X'_\bullet$.
 Thus, it suffices to show the following: given a smooth morphism
 $f\colon Z_{\bullet\bullet}\rightarrow X_\bullet$ such that
 $(Z_{i\bullet}\rightarrow X_i)=\mr{cosk}_0(Z_{i0}\rightarrow X_i)$, the
 functors $\mb{R}f_*$ and $f^*$ induce an equivalence of
 categories. First, the functors $\mb{R}f_*$ and $f^*$ preserve
 holonomicity. Indeed, the preservation of holonomicity for $f^*$ is
 easy, and for $\mb{R}f_*$, use (double simplicial analogue of) Lemma
 \ref{functorrhodef} (ii), Lemma \ref{smoothbasechforsimplspa} and
 Corollary \ref{fundssforsimpldes}. If $f^*$, $\mb{R}f_*$ yield an
 equivalence between $D^{\star}(X_\bullet)$ and
 $D^{\star}(Z_{\bullet\bullet})$, since they preserve holonomicity,
 they induce the equivalence of $D^{\star}_{\mr{hol}}(X_\bullet)$ and
 $D^{\star}_{\mr{hol}}(Z_{\bullet\bullet})$ and the proposition
 follows.

 Now, it remains to show that for $\ms{N}\in M(X_\bullet)$
 and $\ms{M}\in M(Z_{\bullet\bullet})$, the homomorphisms
 \begin{equation*}
  f^*\mb{R}f_*(\ms{M})\rightarrow\ms{M},\qquad
  \ms{N}\rightarrow\mb{R}f_*f^*(\ms{N})
 \end{equation*}
 are isomorphisms. Since it suffices to show the isomorphism for each
 $X_i$, this follows from Proposition \ref{smoothdesre}.
\end{proof}

\begin{dfn}
 \label{dfnofcat}
 Let $\st{X}$ be a \stack. Take a simplicial space presentation
 $X_\bullet\rightarrow\st{X}$. By the proposition above, for
 $\star\in\{\mr{b},+\}$, the categories
 $D^{\star}_{\mr{tot}}(X_\bullet/L_{\base})$ and
 $D^{\star}_{\mr{hol}}(X_\bullet/L_{\base})$
 do not depend on the choice of the presentation. We denote these
 categories by $D^{\star}(\st{X}/L_{\base})$ and
 $D^{\star}_{\mr{hol}}(\st{X}/L_{\base})$, or more precisely
 $D^{\star}(\st{X}/\mf{T}_{\base})$ and
 $D^{\star}_{\mr{hol}}(\st{X}/\mf{T}_{\base})$.
 We often omit $(\cdot)_\base$ as usual.
 These categories are endowed with t-structure, and their hearts are
 denoted by $M(\st{X}/L)$ and
 $\mr{Hol}(\st{X}/L)$ respectively. Objects of $\mr{Hol}(\st{X}/L)$ are
 called {\em holonomic modules on $\st{X}$}. As usual, we often even
 omit ``$/L$'' from the notation of categories.
\end{dfn}

\begin{rem*}
 (i) When $\st{X}$ is a realizable scheme
 $D^{\mr{b}}_{\mr{hol}}(\st{X})$ is equivalent to the one defined in
 \ref{defofrealcat} by Proposition \ref{smoothdesre}.

 (ii) Let $X_{\bullet}\rightarrow\st{X}$ be a simplicial space
 presentation. Then $\rho_0^*$ is conservative, namely,
 $\rho_0^*(\ms{M})=0$
 for $\ms{M}\in D^{+}_{\mr{hol}}(\st{X})$ if and only if
 $\ms{M}=0$. Indeed, $\rho_0^*(\ms{M})=0$ implies $\rho_i^*(\ms{M})=0$,
 so the only if part holds.
\end{rem*}

\begin{dfn}
 \label{isocsuppdef}
 (i) Let $\st{X}$ be a \stack. Assume further that the associated
 reduced algebraic stack $\st{X}_{\mr{red}}$ is smooth. Let
 $X_\bullet\rightarrow\st{X}$ be a simplicial space presentation.
 The category of {\em smooth objects} denoted by $\mr{Sm}(\st{X}/L)$ is
 the full subcategory of $D^{\mr{b}}_{\mr{hol}}(\st{X}/L)$ consisting of
 $\ms{M}$ such that, for any $i$, $\mr{for}_L(\rho^*_i\ms{M})\in
 D^{\mr{b}}_{\mr{hol}}(X_i/K)$ is in $\mr{Sm}(X_i/K)[d_i]$ where $d_i$
 denotes the relative dimension function (cf.\ \ref{reldimfunc}) of
 $X\rightarrow\st{X}$, and see \ref{fundproprealsch}
 \eqref{carodagdagcat} for the notation of $\mr{Sm}(X/K)$. It is
 straightforward to check that the category does not depend on the
 choice of the presentation.

 (ii) Let $\st{X}$ be a \stack, and $\ms{M}\in\mr{Hol}(\st{X})$. The
 {\em support of $\ms{M}$} is the minimum closed subset $Z\subset\st{X}$
 such that the restriction of $\ms{M}$ to $\st{X}\setminus Z$ is $0$.
 The support is denoted by $\mr{Supp}(\ms{M})$. For $\ms{M}\in
 D^{\mr{b}}_{\mr{hol}}(\st{X}/L)$, we put
 $\mr{Supp}(\ms{M}):=\bigcup_i\mr{Supp}(\H^i\ms{M})$.
\end{dfn}

\subsubsection{}
\label{homdefindcat}
Let $X_\bullet$ be an admissible simplicial space. Since $M(X_\bullet)$
is enough injectives, we have the bifunctor
$\mb{R}\mr{Hom}_{D(X_{\bullet})}(-,-)\colon D(X_\bullet)^{\circ}\times
D^+(X_\bullet)\rightarrow D(\mr{Vec}_L)$ (cf.\ \cite[I, \S6]{RD}). This
induces the bifunctor
\begin{equation*}
 \mb{R}\mr{Hom}_{D(\st{X})}(-,-)\colon
  D^{+}(\st{X})^{\circ}\times
  D^{+}(\st{X})\rightarrow
  D^+(\mr{Vec}_L).
\end{equation*}
Indeed, $\mb{R}\mr{Hom}_{D(X_\bullet)}$ does not depend on the choice of
simplicial space presentation $X_\bullet\rightarrow\st{X}$. To check
this, let $f\colon Z_{\bullet\bullet}\rightarrow X_\bullet$ be as in the
proof of Proposition \ref{equivcatstackdfn}. Then we have a canonical
homomorphism $\mb{R}\mr{Hom}_{D(X_{\bullet})}(\ms{M},\ms{N})\rightarrow
\mb{R}\mr{Hom}_{D(Z_{\bullet\bullet})}\bigl(f^*(\ms{M}),
f^*(\ms{N})\bigr)$ for $\ms{M}\in D(X_\bullet)$, $\ms{N}\in
D^+(X_{\bullet})$. It suffices to show that this homomorphism is a
quasi-isomorphism when $\ms{M},\ms{N}\in D^{+}_{\mr{tot}}(X_\bullet)$.
This follows since the pair $(f^*,\mb{R}f_*)$ is an equivalence of
categories.

\begin{quote}
 In the following, for simplicity, we use particularly
 $D^{\star}_{\mr{hol}}(\st{X}/L)$ even when we
 can generalize statements or constructions to $D^{\star}(\st{X}/L)$
 easily.
\end{quote}

\subsubsection{}
\label{doublesimplicspalemm}
We conclude this subsection with the following lemma we use later, whose
proof is similarly to the proof of Proposition \ref{equivcatstackdfn}.

\begin{lem*}
 \label{doublecompleeq}
 Let $\st{X}$ and $\st{Y}$ be \stacks, and $X_\bullet$ and
 $Y_\bullet$ be simplicial space presentations. Let $(X\times
 Y)_{n,n'}:=X_n\times Y_{n'}$, which forms a double simplicial spaces
 denoted by $(X\times Y)_{\bullet\bullet}$. Then we have a canonical
 equivalence
 \begin{equation*}
  D^+_{\mr{hol}}(\st{X}\times\st{Y})\xrightarrow{\sim}
   D^+_{\mr{hol}}((X\times Y)_{\bullet\bullet}).
 \end{equation*}
\end{lem*}

\subsection{Cohomological functors}
\label{cohfunctorstack}
In this subsection, we define some cohomological functors for
algebraic stacks. Even though the six functor formalism is expected for
algebraic stacks, unfortunately, at this moment, we can obtain full
formalism only for ``admissible stacks'' (cf.\ Definition
 \ref{admissdfn}), which is enough for our purpose. In
this subsection, we define functors that we can define for general
algebraic stacks.

\subsubsection*{Finite morphism case}
\subsubsection{}
\label{finitemorphfibercat}
First, we will define the adjoint pair $(f_+,f^!)$ when $f$ is a finite
morphism between \stacks. To do this, we only need to translate the
functor constructed in \ref{defofpushfinitesimpl} in the language of
algebraic stacks.

Let $f\colon\st{X}\rightarrow\st{Y}$ be a finite morphism
between \stacks. Let us define $f^!$ and $f_+$. Take a simplicial
space presentation
$Y_\bullet\rightarrow\st{Y}$. By pulling-back, we get a simplicial
space presentation
$X_\bullet\rightarrow\st{X}$. Let $f_{\bullet}\colon
X_\bullet\rightarrow Y_\bullet$ be the finite cartesian morphism. Let
$\star\in\{\mr{b},+\}$. We define
\begin{equation*}
 f_+\colon D^{\star}_{\mr{hol}}(\st{X})\cong
  D^{\star}_{\mr{hol}}(X_\bullet)\rightleftarrows
  D^{\star}_{\mr{hol}}(Y_\bullet)\cong D^{\star}_{\mr{hol}}(\st{Y})
  \colon f^!.
\end{equation*}
We need to check the well-definedness, namely independence of the
presentation. By the adjointness property, it suffices to show the
independence for $f_+$. As in the proof of Proposition
\ref{equivcatstackdfn}, it suffices to show the following: Consider the
cartesian diagram
\begin{equation*}
 \xymatrix{
  Z_{\bullet\bullet}\ar[r]^{g_{\bullet}}\ar[d]_{p}
  \ar@{}[rd]|\square&
  W_{\bullet\bullet}\ar[d]^q\\
 X_\bullet\ar[r]_{f_\bullet}&Y_\bullet.
  }
\end{equation*}
Then $q^*\circ f_{\bullet+}\cong g_{\bullet+}\circ p^*\colon
M(X_\bullet)\rightarrow M(W_{\bullet\bullet})$. The verification is
straightforward and left to the reader. We have the pair $(f_+,f^!)$ of
adjoint functors between $D^{\star}_{\mr{hol}}(\st{X})$ and
$D^{\star}_{\mr{hol}}(\st{Y})$.

Now, let $\st{X}\xrightarrow{f}\st{Y}\xrightarrow{g}\st{Z}$ be finite
morphisms of \stacks. We have canonical isomorphisms
\begin{equation*}
 c_{(g,f)}\colon f^!\circ g^!\xrightarrow{\sim}(g\circ f)^!,\qquad
  c^{(g,f)}\colon (g\circ f)_+\xrightarrow{\sim} g_+\circ f_+.
\end{equation*}
These isomorphisms are subject to the following two conditions: 1.\ we
have $c_{(f,\mr{id})}=c_{(\mr{id},f)}=\mr{id}$,
$c^{(f,\mr{id})}=c^{(\mr{id},f)}=\mr{id}$; 2.\ given homomorphisms
$\st{X}\xrightarrow{f}\st{Y}\xrightarrow{g}\st{Z}\xrightarrow{h}\st{W}$,
we have
\begin{equation*}
 c_{(h,g\circ f)}\circ c_{(g,f)}(h^!)=
  c_{(h\circ g,f)}\circ f^!(c_{(h,g)}),\qquad
 h_+c^{(g,f)}\circ c^{(h,g\circ f)}=
  c^{(h,g)}(f_+)\circ c^{(h\circ g,f)}.
\end{equation*}

These results can be rephrased by using the language of (co)fibered
categories (cf.\ [SGA 1, Exp.\ VII, end of 7]) as follows.
Let $\mr{St}^{\mr{fin}}(k)$ the be the category of \stacks (we
do not consider the 2-morphisms) over $k$ such that the
morphisms are {\em finite morphisms} between
\stacks. To a \stack $\st{X}$, we associate the triangulated category
$D^{\star}_{\mr{hol}}(\st{X})$. For a finite morphism
$\st{X}\rightarrow\st{Y}$, we consider the functor $f^!$, and
$c_{(f,g)}$. Then these data form a fibered category
$\ms{F}^!\rightarrow\mr{St}^{\mr{fin}}(k)$. Considering $f_+$
and $c^{(f,g)}$, we get a cofibered category
$\ms{F}_{\oplus}\rightarrow\mr{St}^{\mr{fin}}(k)$.

\subsubsection*{Dual functors}
\subsubsection{}
Let $f\colon X\rightarrow Y$ be a smooth morphism of relative dimension
$d$ between spaces. Then we have a canonical isomorphism
$\bigl(f^*\circ\mb{D}_Y\bigr)(d)\xrightarrow{\sim}\mb{D}_X\circ f^*$ by
Lemma \ref{prepcohoprealsch}.
Now, let $X_\bullet$ be an admissible simplicial space, and assume given
a smooth morphism $X_\bullet\rightarrow\st{X}$ to an algebraic stack.
We have the dual auto-functor
$\mb{D}_{X_i}\colon \mr{Hol}(X_i)^{\circ}\xrightarrow{\sim}
\mr{Hol}(X_i)$. We modify this functor by putting
$\widetilde{\mb{D}}_i:=(d_{X_i/\st{X}})\circ\mb{D}_{X_i}$ where
$d_{X_i/\st{X}}$ denotes the relative dimension function (cf.\
\ref{reldimfunc}).

Now, we use the notation of Lemma \ref{catequforhol} and
\S\ref{BDgluingcat}.
Let $\ms{M}_\bullet\in\mr{sec}_+(\mr{Hol}(X_\bullet)_\bullet)$.
For a morphism $\phi\colon[i]\rightarrow[j]$, let $\alpha_\phi\colon
X(\phi)^*\ms{M}_i\rightarrow\ms{M}_j$ be the gluing homomorphism. Let
\begin{equation*}
 \beta_\phi\colon\widetilde{\mb{D}}_{j}(\ms{M}_j)\rightarrow
  \widetilde{\mb{D}}_{j}\bigl(X(\phi)^*(\ms{M}_i)\bigr)
  \xleftarrow{\sim} X(\phi)^*
  \widetilde{\mb{D}}_{i}(\ms{M}_i).
\end{equation*}
The data $\bigl\{\widetilde{\mb{D}}_i(\ms{M}_i),\beta_\phi\bigr\}$
defines an object in $\mr{sec}_-(\mr{Hol}(X_\bullet)_\bullet)$ and
defines a functor
\begin{equation*}
 \mb{D}_{X_\bullet/\st{X}-}\colon
  \mr{sec}_+(\mr{Hol}(X_\bullet)_\bullet)^{\circ}\rightarrow
  \mr{sec}_-(\mr{Hol}(X_\bullet)_\bullet).
\end{equation*}
Similarly, we can define the functor $\mb{D}_{X_\bullet/\st{X}+}
\colon\mr{sec}_-(\mr{Hol}(X_\bullet)_\bullet)^{\circ}\rightarrow
\mr{sec}_+(\mr{Hol}(X_\bullet)_\bullet)$, and we have canonical
isomorphisms $c_\mp\circ\mb{D}_{X_\bullet/\st{X}\pm}\cong
\mb{D}_{X_\bullet/\st{X}\mp}\circ c_\pm$ by definition of $c_{\pm}$
(cf.\ \cite[7.4.2]{BD}).
These functors are exact since $\widetilde{\mb{D}}_{i}$ are. Then we
have
\begin{equation*}
 \xymatrix{
  D_{\mr{tot}}(\mr{sec}_\pm(\mr{Hol}(X_\bullet)_\bullet))^{\circ}
  \ar[d]_{\mb{D}_{X\bullet/\st{X}}\mp}\ar@{-}[r]^-{\sim}&
  D_{\mr{tot}}(\mr{Hol}(X_\bullet)_\bullet)^{\circ}
  \ar[r]^-{\sim}&
  D_{\mr{hol}}(X_\bullet)^{\circ}
  \ar@{.>}[d]^{\mb{D}'_{X_\bullet}}\\
 D_{\mr{tot}}(\mr{sec}_\mp(\mr{Hol}(X_\bullet)_\bullet))
  \ar@{-}[r]_-{\sim}&
  D_{\mr{tot}}(\mr{Hol}(X_\bullet)_\bullet)
  \ar[r]_-{\sim}&
  D_{\mr{hol}}(X_\bullet)
  }
\end{equation*}
where the left horizontal isomorphism follows by \ref{BDgluingcat}, and
the right horizontal isomorphisms follow by Lemma
\ref{catequforhol}. We define the dotted functor so that the square is
commutative. The dotted functor is called the {\em dual functor}
on $D(X_\bullet)$. By construction, the functor is exact. Moreover, we
have a canonical isomorphism
$\mb{D}'_{X_\bullet}\circ\mb{D}'_{X_\bullet}\cong\mr{id}$.

Let $\st{X}$ be a \stack. Take a simplicial space presentation
$X_\bullet\rightarrow\st{X}$. We can check that
$\mb{D}_{X_\bullet/\st{X}}$ does not depend on the choice of
presentation. Thus, we get a functor
\begin{equation*}
 \mb{D}'_{\st{X}}\colon D^{\mr{b}}_{\mr{hol}}(\st{X})^{\circ}
  \rightarrow
  D^{\mr{b}}_{\mr{hol}}(\st{X}).
\end{equation*}
We have a canonical isomorphism of functors
$\mb{D}'_{\st{X}}\circ\mb{D}'_{\st{X}}\cong\mr{id}$.

\begin{rem*}
 Later, in \ref{bidualityofstack}, we define another dual functor
 $\mb{D}$. This is because $\mb{D}'$ is not suited to show some
 fundamental properties of dual functors. The reason why we introduced
 $\mb{D}'$ at this point is to show the existence of the left adjoint of
 $f_+$ for finite morphism $f$.
\end{rem*}

\begin{lem}
 \label{commudualprimpushfin}
 Let $f\colon\st{X}\rightarrow\st{Y}$ be a finite morphism between
 \stacks. Then there exists a canonical isomorphism
 \begin{equation*}
  \mb{D}'_{\st{Y}}\circ f_+\xrightarrow{\sim}f_+\circ\mb{D}'_{\st{X}}
   \colon
   D^{\mr{b}}_{\mr{hol}}(\st{X})^{\circ}\rightarrow
   D^{\mr{b}}_{\mr{hol}}(\st{Y}).
 \end{equation*}
\end{lem}
\begin{proof}
 Take a simplicial space presentation $Y_\bullet\rightarrow\st{Y}$, and
 let $X_\bullet\rightarrow\st{X}$ be the pull-back. We denote by
 $f_k\colon X_k\rightarrow Y_k$ the
 finite morphism induced by $f$. Since $f_{k+}(\ms{M})$ is in
 $\mr{Hol}(Y_k)$ when $\ms{M}\in\mr{Hol}(X_k)$ and $f_{k+}$ and
 $X(\phi)^*$ commute canonically, we can define the push-forward
 functors $f_{\pm*}\colon\mr{sec}_{\pm}
 (\mr{Hol}(X_\bullet)_\bullet)\rightarrow
 \mr{sec}_{\pm}(\mr{Hol}(Y_\bullet)_\bullet)$
 by sending $\{\ms{M}_i\}$ to $\{f_{i+}\ms{M}_i\}$ with obvious gluing
 homomorphism. By the definition of the functors $c_{\pm}$, the
 following diagrams are commutative:
 \begin{equation*}
  \xymatrix@C=25pt{
   \mr{sec}_{+}(\mr{Hol}(X_\bullet)_\bullet)\ar[r]^-{f_{+*}}\ar[d]&
   \mr{sec}_{+}(\mr{Hol}(Y_\bullet)_\bullet)\ar[d]\\
  M(X_\bullet)\ar[r]_-{f_+}&
   M(Y_\bullet),
   }\qquad
   \xymatrix@C=25pt{
   \mr{sec}_{-}(\mr{Hol}(X_\bullet)_\bullet)
   \ar[r]^-{f_{-*}}\ar[d]_{c_+}&
   \mr{sec}_{-}(\mr{Hol}(Y_\bullet)_\bullet)
   \ar[d]^{c_+}\\
  C(M(X_\bullet))\ar[r]_-{f_+}&
   C(M(Y_\bullet)).
   }
 \end{equation*}
 Thus, it is reduced to constructing an isomorphism
 $\mb{D}_{Y_{\bullet}/\st{Y}}\circ f_{+*}\cong
 f_{-*}\circ\mb{D}_{X_{\bullet}/\st{X}}$. Since all the functors
 we used are exact, the verification is easy.
\end{proof}

\begin{dfn*}
 The lemma shows that $f_+$ has a left adjoint functor
 \begin{equation*}
  \mb{D}'_{\st{X}}\circ f^!\circ\mb{D}'_{\st{Y}}\colon
   D^{\mr{b}}_{\mr{hol}}(\st{Y})\rightarrow
   D^{\mr{b}}_{\mr{hol}}(\st{X}).
 \end{equation*}
 This right adjoint functor is denoted by $f^+$. Since $f^!$ is left
 exact, $f^+$ is right exact. Summing up, when $f$ is finite, we have
 two pairs of adjoint functors $(f^+,f_+)$ and $(f_+,f^!)$. By taking
 the dual to \ref{finitemorphfibercat}, $f^+$ yields a fibered category
 $\ms{F}^{\oplus}\rightarrow\mr{St}^{\mr{fin}}(k)$.
\end{dfn*}

\begin{lem}
 \label{directfactorlemmaeasy}
 Let $f\colon\st{C}\rightarrow\st{C}'$ be a finite morphism in
 $\mr{St}^{\mr{fin}}(k)$.

 (i) Assume $f$ is surjective radicial morphism. Then $f_+$ and
 $f^{+}$ define an equivalence of categories between
 $D^{\mr{b}}_{\mr{hol}}(\st{C}')$ and $D^{\mr{b}}_{\mr{hol}}(\st{C})$,
 and $\mr{Hol}(\st{C}')$ and $\mr{Hol}(\st{C})$.

 (ii) Assume $f$ is an \'{e}tale morphism. Then for any
 $\ms{M}\in\mr{Hol}(\st{C})$, it is a direct factor of
 $f^{+}f_+(\ms{M})$.

 (iii) If $f$ is a flat morphism, then for any
 $\ms{M}\in D^{\mr{b}}_{\mr{hol}}(\st{C}')$, $\ms{M}$ is a direct factor
 of $f_+f^+(\ms{M})$.
\end{lem}
\begin{proof}
 For (i), it suffices to show the claim when $\st{C}$ and $\st{C}'$ are
 schemes. This is nothing but Lemma \ref{fundproprealsch} in the first
 read case, and the second read case can be reduced to the first one
 immediately. For (iii), we can define homomorphisms
 $f_+f^+(\ms{M})\rightarrow\ms{M}$ using the trace map of realizable
 schemes, and the claim follows easily.

 For (ii), let $C'_\bullet\rightarrow\st{C}'$ be a simplicial
 space presentation. Put $\ms{M}_i:=\rho_i^*(\ms{M})$, and
 $f_i\colon\st{C}\times_{\st{C}'}C'_i\rightarrow C'_i$. We have the
 morphisms $\ms{M}_i\rightarrow
 f^+_if_{i+}(\ms{M}_i)\rightarrow\ms{M}_i$.
 Here the first morphism is defined by the trace map in the first
 read case and the homomorphism defined in the first read case in the
 second read case. The composition is an isomorphism.
 Indeed, by \cite[1.3.11]{AC}, it is reduced to checking the claim when
 $f_i$ is $\coprod_{j\in J}\mr{Spec}(k')\rightarrow\mr{Spec}(k')$ where
 $k'$ is a finite extension of $k$ and $J$ is a finite set. In this
 case, the verification is easy.
 Moreover, these homomorphisms are compatible with gluing
 homomorphisms, so they define homomorphisms $\ms{M}\rightarrow
 f^{+}f_+\ms{M}\rightarrow\ms{M}$ whose composition is an isomorphism,
 thus the claim follows.
\end{proof}

\subsubsection*{Exterior tensor product}
\subsubsection{}
Let us define the exterior tensor product. Let $X_\bullet$ and
$Y_\bullet$ be admissible simplicial spaces. Given $\ms{M}_\bullet$ and
$\ms{N}_\bullet$ in $M(X_\bullet)$ and $M(Y_\bullet)$ respectively, the
collection $\bigl\{\ms{M}_i\boxtimes\ms{N}_i\bigr\}_i$ defines an object
in $M(X_\bullet\times Y_\bullet)$. This is denoted by
$\ms{M}\boxtimes\ms{N}$. The functor is exact, and we can take the
derived functor to get
\begin{equation*}
 (-)\boxtimes(-)\colon D(M(X_\bullet))\times D(M(Y_\bullet))\rightarrow
  D(M(X_\bullet\times Y_\bullet)).
\end{equation*}
We can check easily that this preserves holonomicity and boundedness.

Let $\st{X}$ and $\st{Y}$ be \stacks, and take simplicial space
presentations $X_\bullet\rightarrow\st{X}$ and
$Y_\bullet\rightarrow\st{Y}$. Then $X_\bullet\times Y_\bullet$ is a
simplicial space presentation of the \stack $\st{X}\times\st{Y}$. Since
$\boxtimes$ does not depend on the choice of presentation, we get the
exterior tensor product for \stacks.

\begin{lem}
 \label{compexttensoth}
 Let $f\colon\st{X}\rightarrow\st{Y}$,
 $g\colon\st{X}'\rightarrow\st{Y}'$ be finite morphisms between
 \stacks. Then we have canonical isomorphisms $f_+(-)\boxtimes
 g_+(-)\cong(f\times g)_+\bigl((-)\boxtimes(-)\bigr)$,
 $f^{\star}(-)\boxtimes g^{\star}(-)\cong(f\times
 g)^{\star}\bigl((-)\boxtimes(-)\bigr)$ where
 $\star\in\bigl\{+,!\bigr\}$, and
 $\mb{D}'\bigl((-)\boxtimes(-)\bigr)\cong
 \mb{D}'(-)\boxtimes\mb{D}'(-)$.
\end{lem}
\begin{proof}
 For the commutation of external tensor product and dual functors, see
 \cite[1.3.3]{AC}, and for the push-forward, use Proposition
 \ref{Kunnethsch}. Since the proofs are straightforward, we leave the
 details.
\end{proof}

\subsubsection*{Smooth morphism case}
\subsubsection{}
\label{smoothpullback}
Let $f\colon\st{X}\rightarrow\st{Y}$ be a smooth morphism between
\stacks. Take simplicial space presentations
$Y_\bullet\rightarrow\st{Y}$ and $X_\bullet\rightarrow\st{X}$. Let
$X_{n,m}:=X_n\times_{\st{Y}}Y_m$, which defines a double simplicial
space $X_{\bullet\bullet}$ with obvious face morphisms, and morphisms
$\tilde{f}\colon X_{\bullet\bullet}\rightarrow Y_{\bullet}$ and
$g\colon X_{\bullet\bullet}\rightarrow X_{\bullet}$. By (double
simplicial analogue of) Proposition \ref{smoothdesre}, $g^*$
and $\mb{R}g_*$ induce an equivalence between
$D^{+}(X_{\bullet\bullet})$ and $D^{+}(X_{\bullet})\cong
D^{+}(\st{X})$. Thus, we have functors
\begin{equation*}
 \mb{R}f_*\colon D^+(\st{X})\cong
  D^{+}(X_{\bullet\bullet})\rightleftarrows
  D^+(Y_{\bullet})\cong D^+(\st{Y})\colon f^*,
\end{equation*}
where the middle functors are induced by $\mb{R}\tilde{f}_*$ and
$\tilde{f}^*$.
These functors preserve holonomicity by Corollary
\ref{fundssforsimpldes}.
We need to check that these functors do not depend on the choice of the
presentations. By adjointness property, it suffices to check it for
$f^*$. The verification is easy and left to the reader. It is also
straightforward to check that $f^*$ is an exact functor, and satisfies
the transitivity, namely given smooth morphisms
$\st{X}\xrightarrow{f}\st{Y}\xrightarrow{g}\st{Z}$ between \stacks, we
have a canonical isomorphism $(g\circ f)^*\cong f^*\circ g^*$.

\begin{lem*}
 Assume $f$ is of relative dimension $d$. We have a canonical
 isomorphism $\mb{D}'_{\st{X}}\circ f^*\cong (d)\circ
 f^*\circ\mb{D}'_{\st{Y}}$.
\end{lem*}
\begin{proof}
 The proof is similar to Lemma \ref{commudualprimpushfin}, using Lemma
 \ref{prepcohoprealsch} (ii).
\end{proof}

\subsubsection{}
\label{openimmdefshri}
Now, assume that $f$ is an {\em open immersion} of \stacks. Then
$X_\bullet:=Y_\bullet\times_{\st{Y}}\st{X}$ is a simplicial space
presentation of $\st{X}$. The canonical morphism
$X_{\bullet}\rightarrow Y_{\bullet}$ is denoted by $h$.
Then we have the commutative diagram of double simplicial spaces
\begin{equation*}
 \xymatrix@C=30pt@R=5pt{
  X_{\bullet\bullet}\ar[rr]^-{\tilde{f}}\ar[dr]_{g}&&Y_{\bullet}\\
 &X_{\bullet}\ar[ur]_{h}.&
  }
\end{equation*}
This implies that the composition $D^+(\st{X})\cong D^+(X_{\bullet})
\xrightarrow{\mb{R}h_*}D^+(Y_\bullet)\cong D^+(\st{Y})$ is canonically
isomorphic to $\mb{R}f_*$, and similarly for $f^*$. Now we have:
\begin{lem*}
 The functor $\mb{R}f_*$ sends $D^{\mr{b}}_{\mr{hol}}(\st{X})$ to
 $D^{\mr{b}}_{\mr{hol}}(\st{Y})$.
\end{lem*}
\begin{proof}
 It suffices to show that $\mb{R}h_*$ preserves the boundedness. This
 can be seen similarly to Lemma \ref{defofpushfinitesimpl} (i).
\end{proof}

Changing the notation, we put $j:=f$ and $j^+:=j^*\colon
D^{\mr{b}}_{\mr{hol}}(\st{Y})\rightarrow
D^{\mr{b}}_{\mr{hol}}(\st{X})$, $j_+:=\mb{R}j_*\colon
D^{\mr{b}}_{\mr{hol}}(\st{X})\rightarrow
D^{\mr{b}}_{\mr{hol}}(\st{Y})$. We define the functor $j_!$ so that
$(j_!,j^+)$ is an adjoint pair. Such a functor exists since by the
above lemma and Lemma \ref{smoothpullback},
$\mb{D}'_{\st{Y}}\circ j_+\circ\mb{D}'_{\st{X}}$ is left adjoint to
$j^+$. Thus, we have pairs of adjoint functors $(j^+,j_+)$ and $(j_!,j^+)$.
Now, since $j^+j_+\cong\mr{id}$, we have a canonical homomorphism of
functors $j_!\rightarrow j_+$. In particular, this induces a functor
\begin{equation*}
 j_{!+}:=\mr{Im}\bigl(\H^0j_!\rightarrow\H^0j_+\bigr)\colon
  \mr{Hol}(\st{X})\rightarrow\mr{Hol}(\st{Y}),
\end{equation*}
called the {\em intermediate extension functor}.

\begin{lem}
 \label{localtriangs}
 Let $j\colon\st{U}\hookrightarrow\st{X}$ be an open immersion between
 \stacks, and let $i\colon\st{Z}\hookrightarrow\st{X}$ be its
 complement. Then, we have the following exact triangles:
 \begin{equation*}
  i_+i^!\rightarrow\mr{id}\rightarrow j_+j^+\xrightarrow{+1},
   \qquad
   j_!j^+\rightarrow\mr{id}\rightarrow i_+i^+\xrightarrow{+1},
 \end{equation*}
 where the homomorphisms are defined by adjunctions.
\end{lem}
\begin{proof}
 Let us check the left one. Let $X_\bullet\rightarrow\st{X}$ be a
 simplicial realizable scheme presentation,
 and let $\ms{I}^{\bullet}$ be a complex of injective objects in
 $C(X_\bullet)$. By abuse of notation, we denote the open and
 closed immersions $\st{U}\times_{\st{X}}X_k\hookrightarrow X_k$ and
 $\st{Z}\times_{\st{X}}X_k\hookrightarrow X_k$ by $j$ and $i$
 respectively. By construction and Lemma \ref{defofpushfinitesimpl},
 $i_+$ and $i^!$ commute with
 $\rho_k^*$. By definition, $j^+$ commutes also with
 $\rho_k^*$. Moreover, $j^+$ commutes with $\rho_{k!}$ as well,
 which implies that their right adjoint functors $j_+$ and $\rho^*_k$
 commute. Thus, it suffices to show that the sequence
 \begin{equation*}
  0\rightarrow
   i_+i^!(\rho_k^*(\ms{I}^l))\rightarrow
   \rho^*_k(\ms{I}^j)\rightarrow
   j_+j^+(\rho_k^*(\ms{I}^l))
   \rightarrow0
 \end{equation*}
 is exact. This follows by Lemma \ref{locextrsch}. The right triangle is
 exact by duality.
\end{proof}

\subsubsection*{Projection case}
\subsubsection{}
We define the push-forward functor for a projection
$\st{X}\times\st{Y}\rightarrow\st{Y}$. The method here is close to the
definition of $Rf^!$ in [SGA 4, XVIII, 3.1].
We start with the following lemma:

\begin{lem*}
 Let $X_\bullet$ and $Y_\bullet$ be admissible simplicial spaces, and
 $\ms{A}$ be an object in $M(X_\bullet)$. Let
 \begin{equation*}
  p^*_{\ms{A}}:=\ms{A}\boxtimes(-)\colon
   M(Y_\bullet)\rightarrow M((X\times Y)_{\bullet\bullet}),
 \end{equation*}
 where we use the notation of Lemma \ref{doublecompleeq}.
 Then there exists a right adjoint denoted by $p_{\ms{A}*}$.
\end{lem*}
\begin{proof}
 Since the functor $p^*_{\ms{A}}$ is exact and commutes with direct sums
 by definition, it commutes with arbitrary inductive limits by
 \cite[2.2.9]{KSc}. Since $M(Y_\bullet)$ is a Grothendieck category and
 $p^*_{\ms{A}}$ commutes with inductive limits, the existence follows
 from the adjoint functor theorem (cf.\ \cite[8.3.27 (iii)]{KSc}).
\end{proof}

Given a homomorphism $\ms{A}\rightarrow\ms{B}$
in $M(X_\bullet)$, we have a natural homomorphisms of functors
\begin{equation*}
 p^*_{\ms{A}}p_{\ms{B}*}\rightarrow
   p^*_{\ms{B}}p_{\ms{B}*}\rightarrow
   \mr{id}
\end{equation*}
where the last homomorphism is the adjunction. Taking the adjoint, we get
a homomorphism
$p_{\ms{B}*}\rightarrow p_{\ms{A}*}$. Obviously, if the homomorphism
$\ms{A}\rightarrow\ms{B}$ is $0$, the induced morphism of functors is
$0$ as well. Thus, for a complex $\ms{A}^{\bullet}\in C(M(X_\bullet))$,
we have a complex of functors
\begin{equation*}
 p_{\ms{A}^{\bullet}*}:\left[
  \cdots\rightarrow p_{\ms{A}^{i+1}*}\rightarrow
  p_{\ms{A}^i*}\rightarrow\cdots\right],
\end{equation*}
where $p_{\ms{A}^i*}$ is placed at the degree $-i$ term.

\begin{lem}
 Let $\ms{I}$ be an {\em injective object} in $M((X\times
 Y)_{\bullet\bullet})$. Then the contravariant functor
 \begin{equation*}
  p_{-*}(\ms{I})\colon M(X_\bullet)^{\circ}\rightarrow
   M(Y_\bullet)
 \end{equation*}
 sending $\ms{A}$ to $p_{\ms{A}*}(\ms{I})$ is exact.
\end{lem}
\begin{proof}
 Let $0\rightarrow\ms{A}'\rightarrow\ms{A}\rightarrow\ms{A}''
 \rightarrow0$ be a short exact sequence.
 It suffices to show that, for any $\ms{N}$ in $M(Y_\bullet)$, the
 complex
 \begin{equation*}
  0\rightarrow\mr{Hom}(\ms{N},p_{\ms{A}''*}(\ms{I}))\rightarrow
   \mr{Hom}(\ms{N},p_{\ms{A}*}(\ms{I}))\rightarrow
   \mr{Hom}(\ms{N},p_{\ms{A}'*}(\ms{I}))\rightarrow0
 \end{equation*}
 is exact, which implies in fact that the sequence $0\rightarrow
 p_{\ms{A}''*}(\ms{I})\rightarrow p_{\ms{A}*}(\ms{I})\rightarrow
 p_{\ms{A}'*}(\ms{I})\rightarrow 0$ is split exact. This follows by
 definition.
\end{proof}

\begin{cor}
 \label{sheafdevsseq}
 Let $\ms{C}$ be in $C(M(X_\bullet))$, and $\ms{M}\in C((X\times
 Y)_{\bullet\bullet})$. We have the spectral
 sequence
 \begin{equation*}
  E_2^{p,q}:=\mb{R}^{p}p_{\H^{-q}(\ms{C})*}(\ms{M})
   \Rightarrow
   \mb{R}^{p+q}p_{\ms{C}*}(\ms{M}).
 \end{equation*}
\end{cor}
\begin{proof}
 The lemma shows that if $\ms{I}$ is an injective object in $M((X\times
 Y)_{\bullet\bullet})$, we have
 \begin{equation*}
  \H^{-i}\bigl([\cdots\rightarrow p_{\ms{A}^{i+1}*}(\ms{I})
   \rightarrow
   p_{\ms{A}^{i}*}(\ms{I})\rightarrow p_{\ms{A}^{i-1}*}(\ms{I})
   \rightarrow\cdots]\bigr)\cong
  p_{\H^i(\ms{A}^{\bullet})*}(\ms{I}).
 \end{equation*}
 Let $\ms{I}^{\bullet}$ be an injective resolution of $\ms{M}$. Then the
 spectral sequence associated to the double complex
 $p_{\ms{A}^{\bullet}*}(\ms{I}^{\bullet})$ is the desired one.
\end{proof}

\begin{cor}
 \label{uptotwpushquasi}
 If a homomorphism of complexes
 $\ms{A}^\bullet\rightarrow\ms{B}^\bullet$ in $C(M(X_\bullet))$ is a
 quasi-isomorphism, the induced homomorphism of derived functors
 $\mb{R}p_{\ms{B}^\bullet*}\rightarrow\mb{R}p_{\ms{A}^\bullet*}$ is a
 quasi-isomorphism as well.
\end{cor}
\begin{proof}
 The homomorphism of functors
 $\mb{R}p_{\ms{B}^\bullet*}\rightarrow\mb{R}p_{\ms{A}^\bullet*}$ induces
 the homomorphism of spectral sequences
 \begin{equation*}
  \xymatrix@R=15pt{
   {}^{\mr{I}}E_2^{p,q}=\mb{R}^pp_{\H^{-q}(\ms{B})*}\ar@{=>}[r]\ar[d]&
   \mb{R}^{p+q}p_{\ms{B}*}\ar[d]\\
  {^{\mr{II}}}E_2^{p,q}=\mb{R}^pp_{\H^{-q}(\ms{A})*}\ar@{=>}[r]&
   \mb{R}^{p+q}p_{\ms{A}*}.
   }
 \end{equation*}
 Since the left vertical homomorphism is an isomorphism, so is the right.
\end{proof}

\begin{dfn}
 \label{dfnofplusprodpr}
 Let $\star\in\{\emptyset,\mr{b},+,-\}$, and $\ms{C}\in
 C^{\star}(M(X_\bullet))$. We can take the derived functor of
 $p_{\ms{C}*}$ to get
 \begin{equation*}
  p_{\ms{C}+}:=\mb{R}p_{\ms{C}*}\colon
   D^{+}((X\times Y)_{\bullet\bullet})
   \rightarrow D(Y_\bullet),\qquad
   p_{\ms{C}}^+:=p_{\ms{C}}^*\colon
   D^{\star}(Y_\bullet)\rightarrow D^{\star}((X\times
   Y)_{\bullet\bullet}).
 \end{equation*}
 By Corollary \ref{uptotwpushquasi}, we may even take $\ms{C}$ in
 $D(X_\bullet)$. By definition, the pair
 $(p^+_{\ms{C}},p_{\ms{C}+})$ is an adjoint pair.
\end{dfn}

\begin{rem*}
 Assume $\ms{C}\in M(X_\bullet)$. Then by definition, for $\ms{M}\in
 M((X\times Y)_{\bullet\bullet})$, $\H^ip_{\ms{C}+}(\ms{M})=0$ for
 $i<0$, and in particular, it sends $D^+$ to $D^+$. Now, let
 $\ms{C}\in D(X_\bullet)$ such that there exists an integer $a$ with
 $\H^i\ms{C}=0$ for $a<i$. Then by Corollary \ref{uptotwpushquasi}, for
 $\ms{M}\in M((X\times Y)_{\bullet\bullet})$, we have
 $\H^ip_{\ms{C}+}(\ms{M})=0$ for $i<-a$. In particular, the functor
 sends $D^+$ to $D^+$ as well.
\end{rem*}

\subsubsection{}
\label{compuprojcon}
Let us compute $p_{\ms{A}*}$ more concretely when
$\ms{A}_{\bullet}\in\mr{Hol}(X_\bullet)$.
Let $X_\bullet$ be an admissible simplicial space, and $Y$ be a
space. Let $p\colon X_\bullet\times Y\rightarrow Y$ be the
projection. Let $\ms{A}_\bullet\in\mr{Hol}(X_\bullet)$.
Take $\ms{M}_\bullet\in M(X_\bullet\times Y)$. For
$\phi\colon[i]\rightarrow[j]$, we have the commutative diagram
\begin{equation*}
 \xymatrix@C=50pt@R=3pt{
  X_j\times Y\ar[rd]^-{p_j}\ar[dd]_{(X\times Y)(\phi)}&\\
 &Y\\
 X_i\times Y\ar[ru]_-{p_i}&
  }
\end{equation*}
Recall \ref{projectcasepullpush}. Since
$\bigl((X\times Y)(\phi)^*\circ p_{i\ms{A}_i}^*, p_{i\ms{A}_i*}\circ
(X\times Y)(\phi)_*\bigr)$ is an adjoint pair and we have the canonical
isomorphism $(X\times Y)(\phi)^*\circ p_{i\ms{A}_i}^*\cong
p_{j\ms{A}_j}^*$, we have
\begin{equation*}
 p_{i\ms{A}_i*}\circ(X\times Y)(\phi)_*\cong p_{j\ms{A}_j*}.
\end{equation*}
This isomorphism together with the gluing homomorphism of
$\ms{M}_\bullet$ for $\phi$ induces a
homomorphism $\alpha_\phi\colon p_{i\ms{A}_i*}(\ms{M}_i)
\rightarrow p_{j\ms{A}_j*}(\ms{M}_j)$. With this homomorphism, we define
\begin{equation*}
 p_{\ms{A}\bullet\times}(\ms{M}_\bullet):=\mr{Ker}
  \bigl(p_{0\ms{A}_0*}(\ms{M}_0)\rightrightarrows
  p_{1\ms{A}_1*}(\ms{{M}_1})\bigr).
\end{equation*}

Now, recall the notation of \ref{doublesimplicspalemm}, and let
$Y_\bullet$ be an admissible simplicial space, and
$\ms{M}_{\bullet\bullet}$ be in $M((X\times Y)_{\bullet\bullet})$. Given
a morphism $\psi\colon[k]\rightarrow[l]$, consider the following
diagram:
\begin{equation}
 \label{doulbsimpconndia}
 \xymatrix{
  X_\bullet\times Y_l
  \ar[r]^-{p^l}\ar[d]_{(X\times Y)(\psi)}\ar@{}[rd]|\square&
  Y_l\ar[d]^{Y(\psi)}\\
 X_\bullet\times Y_k\ar[r]_-{p^k}&Y_k.
  }
\end{equation}
Then we have canonical homomorphisms in $M(Y_l)$
\begin{equation*}
 Y(\psi)^*p^k_{\ms{A}_\bullet\times}(\ms{M}_{\bullet k})
  \xrightarrow{\sim}
  p^l_{\ms{A}_\bullet\times}\bigl((X\times
  Y)(\psi)^*(\ms{M}_{\bullet k})\bigr)\rightarrow
  p^l_{\ms{A}_\bullet\times}(\ms{M}_{\bullet l}).
\end{equation*}
With this gluing homomorphism, we get a left exact functor
\begin{equation*}
 p_{\ms{A}_\bullet\times}:=
 p^{\bullet}_{\ms{A}_\bullet\times}\colon
  M((X\times Y)_{\bullet\bullet})\rightarrow
  M(Y_\bullet).
\end{equation*}

\begin{lem*}
 \label{compladjpushpull}
 The pair $(p^*_{\ms{A}_\bullet},p_{\ms{A}_\bullet\times})$ is an
 adjoint pair. Thus, we have an isomorphism
 $p_{\ms{A}_\bullet\times}\cong p_{\ms{A}_\bullet*}$
\end{lem*}
\begin{proof}
 The proof is essentially the same as that of Lemma
 \ref{constofpairfunct}, and we do not repeat here.
\end{proof}

\begin{lem}
 \label{specseqprojcase}
 Assume $\ms{A}_\bullet\in\mr{Hol}(X_\bullet)$. Then, there exists the
 following spectral sequence
 \begin{equation*}
  E^{i,j}_1=\mb{R}^jp_{i\ms{A}_i*}(\ms{M}_i)\Rightarrow
   \H^{i+j}p_{\ms{A}+}(\ms{M})
 \end{equation*}
 where $p_i\colon X_i\times Y_\bullet\rightarrow Y_\bullet$ is the
 projection.
\end{lem}
\begin{proof}
 By Lemma \ref{compladjpushpull}, it suffices to construct the spectral
 sequence for $p_{\ms{A}_\bullet\times}$.
 The construction is the same as that of Lemma
 \ref{fundssforsimpldes} which is analogous to
 \cite[Corollary 2.7]{OlSA}.
\end{proof}

\subsubsection{}
Now, we use these computations to show finiteness results of the
functor defined in Definition \ref{dfnofplusprodpr}:
\begin{prop*}
 Let $\ms{A}$ be an object of $D^{\mr{b}}_{\mr{hol}}(X_\bullet)$.

 (i) The functors $p_{\ms{A}+}$ and $p_{\ms{A}}^+$ induce functors
 between
 $D^{+}_{\mr{hol}}((X\times Y)_{\bullet\bullet})$ and
 $D^+_{\mr{hol}}(Y_\bullet)$.

 (ii) We have an adjoint pair $(p^+_{\ms{A}},p_{\ms{A}+})$ between
 $D^+_{\mr{hol}}(Y_\bullet)$ and $D^+_{\mr{hol}}((X\times Y)
 _{\bullet\bullet})$.
\end{prop*}
\begin{proof}
 Let us show (i). First, let us check that for $\ms{M}\in
 D^{\mr{b}}_{\mr{hol}}((X\times Y)_{\bullet\bullet})$ and integers
 $i\geq0$ and $n$,
 $\bigl(\H^np_{\ms{A}+}(\ms{M})\bigr)_i$ is in
 $\mr{Hol}(Y_i)$. By combining the spectral sequences of Corollary
 \ref{sheafdevsseq} and Lemma \ref{specseqprojcase}, it suffices
 to check the holonomicity of $\mb{R}^qp_{i\H^p(\ms{A})*}(\ms{M}_i)$ for
 any integers $p$, $q$, $i$. This follows from
 \ref{prepcohoprealsch} or \ref{smoothpullback}.
 It remains to show that $\H^np_{\ms{A}+}(\ms{M})$ is a total complex,
 namely for $\psi\colon[k]\rightarrow[l]$, the homomorphism
 \begin{equation*}
  Y(\psi)^*p^k_{\ms{A}+}(\ms{M}_{\bullet k})\rightarrow
   p^l_{\ms{A}+}(\ms{M}_{\bullet l})
 \end{equation*}
 is a quasi-isomorphism. We may assume
 $\ms{A}\in\mr{Hol}(X_\bullet)$. In this case, this follows by smooth
 base change (cf.\ Corollary \ref{smbcforopenimm}) and Lemma
 \ref{specseqprojcase}. Using (i), (ii) follows immediately by
 construction.
\end{proof}

\begin{rem*}
 Unfortunately, the boundedness is not preserved as we can see from the
 standard example \cite[18.3.3]{LM}.
 Thus, to get a six functor formalism for algebraic stacks, dealing with
 unbounded derived category is essential as in \cite{OL}. However, we
 only construct a complete formalism for ``admissible stacks'' (cf.\
 Definition \ref{admissdfn}), and for morphism between admissible
 stacks, the boundedness is preserved, so we do not use unbounded
 category.
\end{rem*}

\subsubsection{}
As one can expect, these functors define functors for \stacks. Let
$\st{X}$ and $\st{Y}$ be \stacks, and $\ms{A}\in
D^{\mr{b}}_{\mr{hol}}(\st{X})$. Take simplicial space presentations
$X_\bullet\rightarrow\st{X}$ and $Y_\bullet\rightarrow\st{Y}$. Then we
have a functor
\begin{equation*}
 D^+_{\mr{hol}}(\st{X}\times\st{Y})\cong D^+_{\mr{hol}}((X\times
  Y)_{\bullet\bullet})
  \xrightarrow{p_{\ms{A}+}} D^+_{\mr{hol}}(Y_\bullet)\cong
  D^+_{\mr{hol}}(\st{Y}),
\end{equation*}
and the same for $p^+_{\ms{A}}\colon
D^{\star}_{\mr{hol}}(\st{Y})\rightarrow
D^{\star}_{\mr{hol}}(\st{X}\times\st{Y})$. The pair
$(p_{\ms{A}}^+,p_{\ms{A}+})$ is an adjoint pair. Now, we have:

\begin{lem*}
 The functors do not depend on the choice of presentation.
\end{lem*}
\begin{proof}
 By the adjointness property, it suffices to show the lemma for
 $p^+_{\ms{A}}$, in which case the verification is easy.
\end{proof}

\begin{rem*}
 When $\st{X}$ and $\st{Y}$ are realizable schemes, then $p_{\ms{A}+}$
 coincides with the functor defined in \ref{projectcasepullpush}, which
 justifies the notation. This follows since both functors are right
 adjoint to $p_{\ms{A}}^+$.
\end{rem*}

\begin{prop}
 \label{pushforcohrel}
 Let $p\colon\st{X}\rightarrow\mr{Spec}(k)$ be the structural
 morphism of a \stack. Let $\ms{A}$ be an object in
 $D^{\mr{b}}_{\mr{hol}}(\st{X})$. For any $\ms{M}$ in
 $D^+_{\mr{hol}}(\st{X})$, we have a canonical isomorphism (recall
 \ref{globalsectionfunct} and \ref{homdefindcat} for the notation)
 \begin{equation*}
  \mb{R}\Gamma\circ p_{\ms{A}+}(\ms{M})\cong
   \mb{R}\mr{Hom}_{D(\st{X})}(\ms{A},\ms{M}).
 \end{equation*}
\end{prop}
\begin{proof}
 Take a simplicial space presentation $X_\bullet\rightarrow\st{X}$.
 For $\ms{A}\in C^{\mr{b}}(X_\bullet)$ and $\ms{M}\in M(X_\bullet)$, we
 have
 \begin{equation*}
  \Gamma\circ p_{\ms{A}*}(\ms{M})
   \cong
   \mr{Hom}_{M(\mr{Spec}(k))}\bigl(L,p_{\ms{A}*}(\ms{M})\bigr)
   \cong
   \mr{Hom}_{M(X_\bullet)}(\ms{A},\ms{M}).
 \end{equation*}
 Now, $p_{\ms{A}^i*}$ preserves injective objects since the left adjoint
 functor $p_{\ms{A}^i}^*$ is exact. This shows that
 $\mb{R}\bigl(\Gamma\circ p_{\ms{A}*}\bigr)\cong
 \mb{R}\Gamma\circ\mb{R}p_{\ms{A}*}$.
 Thus by the definition of $p_{\ms{A}+}$, the proposition follows.
\end{proof}

\subsubsection{}
\label{BBDgluingthm}
We have defined a pair of adjoint functors
$(p^+_{\ms{A}},p_{\ms{A}+})$, which depends on the choice of the complex
$\ms{A}$. For the construction of normal push-forward and pull-back, we
need a ``canonical choice'' of $\ms{A}$, which is nothing but the
unit object $L_{\st{X}}$ when $\st{X}$ is a smooth realizable
scheme. To construct this complex for \stacks, we need the following
theorem of \cite{BBD}, as in the construction of \cite{OL}.

\begin{thm*}[{\cite[3.2.4]{BBD}}]
 Let $X_\bullet$ be an admissible simplicial space.
 Assume given data
 $\bigl\{\ms{C}_i,\alpha_{\phi}\bigr\}$ where
 $\ms{C}_i\in
 D^\mr{b}(X_\bullet)$ and for $\phi\colon[i]\rightarrow[j]$,
 $\alpha_\phi\colon X(\phi)^*\ms{C}_i\xrightarrow{\sim}\ms{C}_j$
 satisfying the cocycle condition. Assume moreover that
 \begin{equation*}
  \mb{R}^i\mr{Hom}_{D(X_i)}(\ms{C}_i,\ms{C}_i)=0
 \end{equation*}
 for any $i<0$. Then there exists a unique $\ms{C}\in
 D^{\mr{b}}_{\mr{tot}}(X_\bullet)$ such that
 $\rho^*_i(\ms{C})\cong\ms{C}_i$ (cf.\ \ref{functorrhodef}) and the
 gluing isomorphism is equal to $\alpha_\phi$ via this isomorphism.
\end{thm*}
\begin{proof}
 For the uniqueness, use the spectral sequence (\ref{BDspecseq}).
 The existence is more difficult.
 We use a construction of Beilinson and Drinfeld. In \cite[7.4.10]{BD},
 they define an abelian category
 $\mathit{hot}_+(M(X_\bullet))$. This category is nothing but
 $\mr{tot}(\mc{A}^+)$ in the notation of \cite[3.2.7]{BBD} by taking
 $\mc{A}(n)$ to be $M(X_n)$. In \cite{BD}, they construct an equivalence
 of categories $s_+\colon D\mr{sec}_+(M(X_\bullet))\xrightarrow{\sim}
 D\mathit{hot}_+(M(X_\bullet))$ and characterize
 $D_{\mr{tot}}(X_\bullet)$ in terms of $D\mathit{hot}$. Even though the
 appearance is slightly different, this is the statement corresponding
 to \cite[3.2.17]{BBD}. Our task is, thus, to construct an object in
 $K(\mathit{hot}_+(M(X_\bullet)))$. For this, we can copy the argument
 of \cite[3.2.9]{BBD}.
\end{proof}

\begin{lem}
 \label{dualautozero}
 Let $p\colon X\rightarrow\mr{Spec}(k)$ be a morphism of spaces. Let
 $L_X:=p^+(L)$. Then we have $\mb{R}^i\mr{Hom}_{D(X)}(L_{X},L_{X})=0$
 for $i<0$.
\end{lem}
\begin{proof}
 Consider the first read case, namely the case where $X$ is a
 realizable scheme. Since $\mr{for}_L$ is conservative, we may assume
 that $L=K$. We have isomorphisms
 \begin{equation*}
  \mb{R}\mr{Hom}_{D(X)}(L_{X},L_{X})\cong
   \mb{R}\Gamma\,p_{+}\shom(L_{X},L_{X})\cong
   \mb{R}\Gamma\,p_{+}p^+(L)
 \end{equation*}
 where we used Proposition \ref{pushforcohrel} for the first
 isomorphism, and the second one follows since $p^+$ is monoidal and
 $L_X$ is the unit object.
 Now the lemma follows by the left c-t-exactness of $p_+$ (cf.\ Lemma
 \ref{constproplem}) and $\mb{R}\Gamma$.
 For the second read case, we can reduce
 to the first read case by using (\ref{BDspecseq}).
\end{proof}

\subsubsection{}
\label{defofconstanddual}
Let $\st{X}$ be a \stack, and $X_\bullet\rightarrow\st{X}$ be a
simplicial space presentation. Let us construct the
unit complex on $X_\bullet$. The unit complex $L_{X_i}$ in
$D^{\mr{b}}_{\mr{hol}}(X_i/L)$ has already been defined. Let
$\phi\colon[i]\rightarrow[j]$. Recall the notation \ref{reldimfunc} and
\ref{tateshiftconv}. We have a canonical isomorphism
\begin{equation*}
 X(\phi)^*\bigl(L_{X_i}[d_{X_i/\st{X}}]\bigr)
  \cong L_{X_j}[d_{X_j/\st{X}}].
\end{equation*}
By Lemma \ref{dualautozero}, the conditions in Theorem
\ref{BBDgluingthm} are satisfied. Thus the data
$\bigl\{L_{X_i}[d_{X_i/\st{X}}]\bigr\}_i$ yield an object
$L_{X_\bullet/\st{X}}$ in $D^{\mr{b}}_{\mr{hol}}(X_\bullet)$.

\begin{lem*}
 The object $L_{X_\bullet/\st{X}}$ does not depend on the choice of
 simplicial presentation up to canonical isomorphism.
\end{lem*}
\begin{proof}
 The proof is straightforward.
\end{proof}

\begin{dfn*}
 (i) We define the {\em unit complex} $L_{\st{X}}$ to be the
 object in $D^{\mr{b}}_{\mr{hol}}(\st{X})$ defined by
 $L_{X_\bullet/\st{X}}$ thanks to the lemma above.

 (ii) We define the {\em dualizing complex} $L^{\omega}_{\st{X}}$ to be
 $\mb{D}'_{\st{X}}(L_{\st{X}})$.
\end{dfn*}

\begin{rem*}
 We can construct $L^{\omega}_{\st{X}}$ similarly to $L_{\st{X}}$,
 without using the functor $\mb{D}'_{\st{X}}$.
\end{rem*}

\begin{lem}
 \label{externaltensstsh}
 (i) Let $\st{X}$ and $\st{Y}$ be \stacks. Then we have
 $L_{\st{X}\times\st{Y}}\cong L_{\st{X}}\boxtimes L_{\st{Y}}$, and
 $L^{\omega}_{\st{X}\times\st{Y}}\cong L^{\omega}_{\st{X}}\boxtimes
 L^{\omega}_{\st{Y}}$.

 (ii) Let $f\colon\st{X}\rightarrow\st{Y}$ be a
 finite morphism between \stacks. Then we have an isomorphism
 $\iota_f\colon f^+(L_{\st{Y}})\cong L_{\st{X}}$ such that given another
 finite morphism $g\colon\st{Y}\rightarrow\st{Z}$, the isomorphism is
 compatible with composition: the composition $(g\circ
 f)^+(L_{\st{Z}})\cong g^+(f^+(L_{\st{Z}}))\xrightarrow[\sim]{\iota_f}
 g^+L_{\st{Y}}\xrightarrow[\sim]{\iota_g}L_{\st{X}}$ is equal to
 $\iota_{g\circ f}$.
\end{lem}
\begin{proof}
 For (i), the first claim follows from the corresponding statement for
 spaces, and the second by Lemma \ref{compexttensoth}. Verification of
 (ii) is left to the reader.
\end{proof}

\begin{dfn}
 \label{dfnofpushpullstackproj}
 Let $\st{X}$ and $\st{Y}$ be \stacks.
 We put
 \begin{equation*}
  p_+:=p_{L_{\st{X}}+}\colon
   D^+_{\mr{hol}}(\st{X}\times\st{Y})\rightarrow
   D^+_{\mr{hol}}(\st{Y}),\qquad
   p^+:=p_{L_\st{X}}^+\colon D_{\mr{hol}}^{\mr{b}}(\st{Y})
   \rightarrow D_{\mr{hol}}^{\mr{b}}(\st{X}\times\st{Y}).
 \end{equation*}
\end{dfn}

Let $\st{X}\times\st{Y}\times\st{Z}\xrightarrow{f}\st{Y}\times
\st{Z}\xrightarrow{g}\st{Z}$ be projections. By using the canonical
isomorphism in Lemma \ref{externaltensstsh}, we have a canonical
isomorphisms
\begin{equation}
 \label{projcanisomtrans}
 f^+\circ g^+\cong L_{\st{X}}\boxtimes\bigl(L_{\st{Y}}
  \boxtimes(-)\bigr)\cong L_{\st{X}\times\st{Y}}\boxtimes(-)
  \cong (g\circ f)^+.
\end{equation}
By taking the adjoint, we also get a canonical isomorphism $(g\circ
f)_+\cong g_+\circ f_+$.

\subsubsection{}
\label{kunnethforprojstack}
Let $\st{X}^{(\prime)}$, $\st{Y}^{(\prime)}$ be \stacks, and take
$\ms{A}^{(\prime)}\in D^{\mr{b}}_{\mr{hol}}(\st{X}^{(\prime)})$. Let
$p^{(\prime)}\colon\st{X}^{(\prime)}\times\st{Y}^{(\prime)}
\rightarrow\st{Y}^{(\prime)}$. Let $q\colon(\st{X}\times\st{X}')
\times(\st{Y}\times\st{Y}')\rightarrow\st{Y}\times\st{Y}'$ be the
projection. By definition, we have a canonical isomorphism
\begin{equation}
 \label{commpullprojcase}
  p_{\ms{A}}^+(-)\boxtimes p'^+_{\ms{A}'}(-)\cong
  q_{\ms{A}\boxtimes\ms{A}'}^+\bigl(-\boxtimes -\bigr).
\end{equation}
Now, we have
\begin{equation}
 \label{Kunnethisomproj}
 q^+_{\ms{A}\boxtimes{\ms{A}'}}
  \bigl(p_{\ms{A}+}(-)\boxtimes p'_{\ms{A}'+}(-)\bigr)\cong
  p^+_{\ms{A}}p_{\ms{A}+}(-)\boxtimes
  p'^+_{\ms{A}'}p'_{\ms{A}'+}(-)\rightarrow
  (-)\boxtimes(-)
\end{equation}
where we used (\ref{commpullprojcase}) at the first isomorphism.

\begin{lem*}
 The homomorphism $p_{\ms{A}+}(-)\boxtimes p'_{\ms{A}'+}(-)
 \rightarrow q_{\ms{A}\boxtimes\ms{A}'+}\bigl((-)\boxtimes(-)\bigr)$
 defined by taking adjoint to (\ref{Kunnethisomproj}) is an
 isomorphism.
\end{lem*}
\begin{proof}
 We can easily reduce to the case where $\st{X}^{(\prime)}$ and
 $\st{Y}^{(\prime)}$ are spaces using the spectral sequences of
 Corollary \ref{sheafdevsseq} and Lemma \ref{specseqprojcase}. Using the
 same spectral sequences, we may assume further that $\st{X}^{(\prime)}$
 and $\st{Y}^{(\prime)}$ are realizable schemes in the second read
 case. Now, for realizable schemes $X$ and $X'$ and $\ms{M}^{(\prime)}$,
 $\ms{N}^{(\prime)}$ in $D^{\mr{b}}_{\mr{hol}}(X^{(\prime)})$, we have
 \begin{equation*}
  \shom\bigl(\ms{M}\boxtimes\ms{M}',\ms{N}\boxtimes\ms{N}'\bigr)
   \cong
   \shom(\ms{M},\ms{N})\boxtimes\shom(\ms{M}',\ms{N}').
 \end{equation*}
 This follows by the isomorphism
 $\shom(\ms{M},\ms{N})\cong\mb{D}\bigl(\ms{M}\otimes\mb{D}(\ms{N})\bigr)$
 (cf.\ \ref{bidualrealcase}) and the commutativity of $\mb{D}$ and
 $\boxtimes$ (cf.\ \cite[1.3.3 (i)]{AC}). Using the K\"{u}nneth formula
 for realizable schemes \ref{Kunnethsch}, the lemma follows.
\end{proof}

\subsubsection*{Before the second read}
\subsubsection{}
\label{secondreadsub}
Before moving on to the ``second read'' (cf.\ \ref{fixterminolspace}) to
establish the theory for more general algebraic stacks, we need the
following proposition, which is an analogue of Beilinson's equivalence
\cite{AC2} for Deligne-Mumford stacks.

\begin{prop*}
 Let $\st{X}$ be a Deligne-Mumford stack of finite type.
 Moreover, assume that $\st{X}$ is {\em separated}. Then the canonical
 functor $D^{\mr{b}}(\mr{Hol}(\st{X}))\rightarrow
 D^{\mr{b}}_{\mr{hol}}(\st{X})$ is an equivalence.
\end{prop*}
\begin{proof}
 First, we note that since the Deligne-Mumford stack is separated and of
 finite type, the diagonal morphism is finite, and in particular,
 schematic.
 Fully faithfulness is the only problem.
 Thus, since the canonical functors
 $D^{\mr{b}}(\mr{Hol}(\st{X}))\rightarrow
 D^{\mr{b}}_{\mr{hol}}(\mr{Ind}(\mr{Hol}(\st{X})))
 \rightarrow D^+_{\mr{hol}}(\mr{Ind}(\mr{Hol}(\st{X})))$ are fully
 faithful, it suffices to show that
 the canonical functor $D^{+}(\mr{Ind}(\mr{Hol}(\st{X})))\rightarrow
 D^{+}(\st{X})$ is fully faithful.
 Let $f\colon X\rightarrow Y$ be an affine \'{e}tale morphism of
 realizable schemes. Then the pair $(f_!,f^*)$ of functors between
 $M(X/L)$ and $M(Y/L)$ is an adjoint pair, and $f_!$ is exact by
 \cite[1.3.13]{AC}. In particular, $f^*$ sends injective objects to
 injective objects. Now, let $f\colon X\rightarrow\st{X}$ be a smooth
 morphism from an affine scheme. The functors in \ref{smoothpullback}
 define a pair of adjoint functors $(f^*,f_*)$ between $\mr{Hol}(X/L)$
 and $\mr{Hol}(\st{X}/L)$. By passing to the Ind categories, these
 functors induce a pair of adjoint functors $(f^{\odot},f_{\odot})$
 between $M(X/L)$ and $\mr{Ind}(\mr{Hol}(\st{X}/L))$.
 Then we can define functors
 \begin{equation*}
  \mb{R}f_{\odot}\colon D^{+}(X/L)\rightleftarrows
   D^{+}(\mr{Ind}(\mr{Hol}(\st{X}/L)))
   \colon f^{\odot}
 \end{equation*}
 as in the scheme case. Consider the following cartesian diagram:
 \begin{equation*}
  \xymatrix{
   Y'\ar[d]_{g'}\ar[r]^{f'}\ar@{}[rd]|\square&Y\ar[d]^g\\
  X\ar[r]_-{f}&\st{X}
   }
 \end{equation*}
 where $X$, $Y$ are affine schemes, and $f$, $g$ are \'{e}tale.
 We can easily check that the canonical
 homomorphism $g^*\circ f_*\rightarrow f'_*\circ g'^*$ is an
 isomorphism. This extends to an isomorphism $g^{\odot}\circ
 f_{\odot}\xrightarrow{\sim}f'_*\circ g'^*$.
 Since $\st{X}$ is assumed separated, $g'$
 is an affine \'{e}tale morphism, and thus $g'^*$ preserves injective
 objects. This implies that the canonical homomorphism
 \begin{equation}
  \label{basechanisomforalgsp}
   \tag{$\star$}
   g^{\odot}\circ\mb{R}f_{\odot}\rightarrow
   \mb{R}f'_*\circ g'^*\colon
   D^{+}(X/L)\rightarrow D^+(Y/L).
 \end{equation} 
 is an isomorphism.

 Now, let $f\colon X_{\bullet}\rightarrow\st{X}$ be a simplicial
 realizable scheme presentation such that $X_0\rightarrow\st{X}$ is
 \'{e}tale. Then we can define a pair of adjoint functors
 \begin{equation*}
  \mb{R}f_{\odot}\colon D^+(X_\bullet/L)\rightleftarrows
   D^+(\mr{Ind}(\mr{Hol}(\st{X}/L)))\colon f^{\odot}
 \end{equation*}
 in the similar way to \ref{constofpairfunct}.
 To conclude the proof, we need to show an analogue of
 Proposition \ref{smoothdesre} in this context.
 For this, take an \'{e}tale presentation $Y\rightarrow\st{X}$, use the
 base change (\ref{basechanisomforalgsp}) above, and we are reduced to
 the realizable scheme situation we have already treated.
\end{proof}

\begin{rem*}
 The proposition is false for general algebraic stacks. For example, if
 $G$ is a geometrically connected algebraic group over $k$, the category
 $\mr{Hol}(BG)$ is equivalent to the category of finite dimensional $L$-vector
 spaces (cf.\ Lemma \ref{connectinvcllsp}). However, extension computed
 in the category $D^{\mr{b}}_{\mr{hol}}(BG)$ is not trivial. For
 example, we can compute as in the classical case (cf.\
 \cite[18.3.3]{LM}) that $H^i(B\mb{G}_m,L)\cong H^i(\mb{P}^i,L)$ for
 $i\geq0$.
\end{rem*}

\subsubsection{}
\label{secondreadexpl}
For an algebraic stack of finite type over $k$, we may take a
presentation $X\rightarrow\st{X}$ such that $X$ is a separated algebraic
space of finite type (or even affine scheme of finite type). Then
$\mr{cosk}_0(X\rightarrow\st{X})$ consists of {\em separated} algebraic
spaces of finite type, since $\st{X}$ is assumed quasi-separated.
In the ``first read'', the starting point of the construction was
Theorem \ref{Beilequivcat}. In the ``second read'', Proposition
\ref{secondreadsub} plays the role of the theorem. Replacing the
definitions of the terminologies accordingly to the table of
\ref{fixterminolspace}, we can construct the theory for quasi-separated
algebraic stacks of finite type over $k$.

\begin{rem}
 \label{remarkendofsecondpar}
 (i) As one can see from the construction of the cohomology theory
 explained in \ref{secondreadexpl}, quasi-separatedness is
 important. Otherwise, non-separated algebraic spaces appear in the
 simplicial space $\mr{cosk}_0(X\rightarrow\st{X})$, and we are not able
 to apply Proposition \ref{secondreadsub}. Non quasi-separated stacks
 naturally appear in the work of \cite{Olog}.

 (ii) We may treat the non quasi-compact case without much
 effort. For a separated scheme $X$ locally of finite type over
 $k$, we take an affine covering $\{U_i\}$, and define $M(X/L)$
 by gluing $M(U_i/L)$. Note that this category is not equivalent to
 $\mr{Ind}(\mr{Hol}(X/L))$ in general. Even though we need to care
 about finiteness and so on, the constructions in
 \S\ref{Dmodforstack} and \S\ref{cohfunctorstack} can be carried out
 similarly.
\end{rem}

\begin{prob}
 Let $X$ be a separated scheme of finite type over $k$.
 We may construct a functor $D^{\mr{b}}(\mr{Con}(X/L))\rightarrow
 D^{\mr{b}}_{\mr{hol}}(X/L)$ as in \cite{Bei}. We ask if this is an
 equivalence of categories.
\end{prob}

\begin{rem*}
 (i) If we have a positive answer to this problem, we can define the
 six functor formalism exactly as in \cite{OL}, or more precisely, we
 may define the push-forward to be the derived functor of $\cH^0f_+$. If
 the proof of the question is ``motivic'', then suitable six functor
 formalisms for schemes can be extended to that for algebraic stacks
 automatically.

 (ii) The problem is solved by Nori \cite{No} when $k=\mb{C}$ and
 $\mr{Hol}(X/L)$ is replaced by the category of perverse sheaves.
\end{rem*}

\subsubsection*{Theory of weights}
\subsubsection{}
\label{theoryofweightsetup}
Let $k$ be a finite field with $q=p^{s}$ elements, and we consider the
arithmetic situation where the base tuple (cf.\ \ref{5tuplesover}) is
$\mf{T}_F:=(k,R,K,L,s,\sigma=\mr{id})$.
We fix an isomorphism $\iota\colon\overline{\mb{Q}}_p\cong\mb{C}$.
Let $X$ be a realizable scheme over $k$. We say that $\ms{M}\in
D^{\mr{b}}_{\mr{hol}}(X/L_F)$ is {\em $\iota$-mixed} (resp.\
{\em $\iota$-mixed of weight $\leq w$}, $\iota$-mixed of weight $\geq
w$) if $\mr{for}_L(\ms{M})\in D^{\mr{b}}_{\mr{hol}}(X/K_F)$ is.
The results \cite[4.1.3, 4.2.3]{AC} are automatically true also for
$D^{\mr{b}}_{\mr{hol}}(-/L_F)$ since cohomological operators
commute with $\mr{for}_L$ by \ref{Frobsetupp}. The same for
\cite[4.3]{AC}.

\begin{rem*}
 For $(V,\varphi)\in F\text{-}\mr{Vec}_L$, if $f_{\varphi}(x)$ is the
 characteristic polynomial of $\varphi$ and $\{\alpha_i\}$ is the set of
 eigenvalues, then the set of eigenvalues for
 $\mr{for}_L((V,\varphi))\in F\text{-}\mr{Vec}_K$ is
 \begin{equation*}
  \bigl\{x\in\overline{\mb{Q}}_p\mid
   \sigma(f_{\varphi})(x)=0\bigr\}_{\sigma\in
   \mr{Hom}_K(L,\overline{\mb{Q}}_p)}
   =
   \{\sigma(\alpha_i)\}_{\sigma\in\mr{Hom}_K(L,\overline{\mb{Q}}_p)}.
 \end{equation*} 
\end{rem*}

\subsubsection{}
Let $\st{X}$ be an algebraic stack over $k$. We say that
$\ms{M}\in\mr{Hol}(\st{X}/L_F)$ is {\em $\iota$-pure of weight $w$}
(resp.\ {\em $\iota$-mixed}, {\em $\iota$-mixed
of weight $\leq w$}, {\em $\iota$-mixed of weight $\geq w$}) if for any
$f\colon X\rightarrow\st{X}$ in $\st{X}_{\mr{sm}}$ (cf.\
\ref{smpresentnotasur}) of relative dimension
$d$, $f^*(\ms{M})$ is $\iota$-pure of weight $w+d$ (resp.\
$\iota$-mixed, $\iota$-mixed of weight $\leq w+d$, $\iota$-mixed of
weight $\geq w+d$). By the existence of weight filtration
\cite[4.3.4]{AC}, if $\ms{M}$ is $\iota$-mixed, then there exists an
increasing filtration $W$ such that $\mr{gr}^W_i(\ms{M})$ is
$\iota$-pure of weight $i$. A complex $\ms{C}\in
D^{\mr{b}}_{\mr{hol}}(\st{X}/L_F)$ is said to be 
{\em $\iota$-mixed complex of weight $\star\in\{\leq w,\geq
w,\emptyset\}$}, if $\H^i\ms{C}$ is $\iota$-mixed of weight
$\star+i$. We say that the complex $\ms{C}$ is {\em $\iota$-pure of
weight $w$} if it is $\iota$-mixed of weight both $\leq w$ and $\geq
w$. We can check that $\iota$-mixedness or other relevant notions
defined here are compatible with those for realizable scheme case in
\ref{theoryofweightsetup}.

\subsubsection{}
\label{theoweiarst}
Let $f\colon\st{X}\rightarrow\st{Y}$ be a morphism of algebraic
stacks of finite type over $k$. We have the following properties:
\begin{enumerate}
 \item For any algebraic stack $\st{X}$, the functor
       $\mb{D}'_{\st{X}}$ preserves $\iota$-mixed complexes, and
       exchanges $\iota$-mixed of weight $\leq w$ and $\geq -w$.

 \item Assume $f$ is finite. Then $f_+$, $f^!$ preserve
       $\iota$-mixedness. Moreover, $f_+$ preserves weights, and $f^!$
       preserves complexes of weight $\geq w$.

 \item Let $\ms{M}$ and $\ms{N}$ be $\iota$-mixed complexes in
       $D^{\mr{b}}_{\mr{hol}}(\st{X}/L_F)$ and
       $D^{\mr{b}}_{\mr{hol}}(\st{Y}/L_F)$ respectively. Then
       $\ms{M}\boxtimes\ms{N}$ is $\iota$-mixed as well. Moreover if
       $\ms{M}$ and $\ms{N}$ are of weight $\geq w$ and $\geq w'$
       (resp.\ $\leq w$ and $\leq w'$), then $\ms{M}\boxtimes\ms{N}$ is
       of weight $\geq w+w'$ (resp.\ $\leq w+w'$).

 \item The complex $L^\omega_{\st{X}}$ (resp.\ $L_{\st{X}}$) is
       $\iota$-mixed of weight $\geq0$ (resp.\ $\leq0$).

 \item Assume $f=:j$ is an open immersion, and
       $\ms{M}\in\mr{Hol}(\st{X}/L_F)$ be $\iota$-pure. Then
       $j_{!+}(\ms{M})$ (cf.\ \ref{openimmdefshri}) is $\iota$-pure with
       the same weight. (cf.\ \cite[4.2.4]{AC})

 \item Assume $f$ is a projection
       $\st{X}\times\st{Y}\rightarrow\st{Y}$, and $\ms{A}$ be an
       $\iota$-mixed complex of weight $\leq w'$ on $\st{X}$. Then
       $f_{\ms{A}+}$ sends $\iota$-mixed complexes of weight $\geq w$ to
       that of weight $\geq w-w'$.
\end{enumerate}
We think only the last property needs a proof. Let
$\ms{M}\in\mr{Hol}(\st{X}\times\st{Y})$ be $\iota$-mixed of weight
$\geq w$. We may assume $\ms{A}\in\mr{Hol}(\st{X})$.
Let us use the notation of \ref{compuprojcon}. We denote by
$d_i$ (resp.\ $d'_j$) the relative dimension of $X_i\rightarrow\st{X}$
(resp.\ $Y_j\rightarrow\st{Y}$). By definition, $\ms{M}_{i,j}$ is
$\iota$-mixed of weight $\geq w+d_i+d'_j$, and $\ms{A}_i$
is of weight $\leq w'+d_i$. Recall that
\begin{equation*}
 \mb{R}p_{i\ms{A}_i*}(\ms{M}_{i,j})\cong p_{i+}
  \shom(q_i^+\ms{A}_i,\ms{M}_{i,j}).
\end{equation*}
where $q_i\colon X_i\times Y_j\rightarrow Y_j$ is the projection. Using
\cite[4.1.3]{AC}, $\mb{R}p_{i\ms{A}_i*}(\ms{M}_{i,j})$ is of weight
$\geq (w+d_i+d'_j)-(d_i+w')=w-w'+d'_j$. Now, by the spectral sequence of
Lemma \ref{specseqprojcase}, the claim follows.

\subsection{Six functor formalism for admissible stacks}
\label{sixfuncadmstsec}
In this subsection, we construct the six functor formalism for
admissible stacks, namely algebraic stacks of finite type with finite
diagonal morphism.

\begin{dfn}
 \label{admissdfn}
 A morphism $f\colon\st{X}\rightarrow\st{Y}$ between algebraic stacks is
 said to be {\em admissible} if it is of finite type, and the diagonal
 morphism $\Delta_f\colon\st{X}\rightarrow\st{X}\times_{\st{Y}}\st{X}$
 is finite. An {\em admissible stack} ({\em over $k$}) is an
 algebraic stack over $k$ whose structural morphism is
 admissible.
\end{dfn}

\begin{rem*}
 (i) An admissible morphism is quasi-compact and separated by
 definition.

 (ii) For an admissible stack $\st{X}$, there exists a finite covering
 $\{\st{U}_i\}$ such that $\st{U}_i$ possesses a quasi-finite flat
 morphism $V_i\rightarrow\st{U}_i$ from a scheme. This is possible by
 [SGA 3, V, 7.2]\footnote{See also \cite[2.1]{Co2}.}. In particular,
 there is a dense open substack $\st{U}$ of $\st{X}$ such that there
 exists a finite locally free morphism $V\rightarrow\st{U}$ from a
 scheme.
\end{rem*}

\begin{lem}
 Let $\st{X}\xrightarrow{f}\st{Y}\xrightarrow{g}\st{Z}$ be morphisms
 between algebraic stacks.

 (i) If $f$ and $g$ are admissible, so is $g\circ f$.

 (ii) Let $\st{Y}'\rightarrow\st{Y}$ be a morphism between algebraic
 stacks. If $f$ is admissible, then the base change
 $\st{X}\times_{\st{Y}}\st{Y}'\rightarrow\st{Y'}$ is admissible as
 well.

 (iii) If $g\circ f$ is admissible, so is $f$.

 (iv) Separated representable morphisms of finite type between
 algebraic stacks are admissible. In particular, immersions are
 admissible.

 (v) Any morphism $X\rightarrow\st{Y}$ from a scheme to an admissible
 stack is schematic.
\end{lem}
\begin{proof}
 The proof for (i) and (ii) are the same as [EGA I, 5.5.1]. Let us
 show (iii). Consider the factorization of $\Delta_{g\circ f}$ into
 morphisms
 \begin{equation*}
  \st{X}
   \xrightarrow{\Delta_f}
   \st{X}\times_{\st{Y}}\st{X}
   \xrightarrow{p}
   \st{X}\times_{\st{Z}}\st{X}.
 \end{equation*}
 These morphisms are representable and separated by \cite[7.7]{LM}.
 We need to show that $\Delta_f$ is finite. By definition of algebraic
 stack, $\st{Y}\rightarrow\st{Y}\times\st{Y}$ is
 representable separated.
 Thus, by [EGA I, 5.5.1 (v)], $\Delta_g$ is separated as
 well. This implies that $p$ is separated since it is the base change of
 $\Delta_{g}$. Since the composition $p\circ\Delta_f$ is assumed finite
 and the morphisms are representable, we conclude that
 $\Delta_f$ is finite by [EGA II, 6.1.5 (v)].

 For (iv), assume $f$ is representable and separated. Then
 \cite[8.1.2]{LM} shows that the diagonal
 $\st{X}\rightarrow\st{X}\times_{\st{Y}}\st{X}$ is a monomorphism, and
 since $f$ is assumed separated, it is a closed immersion. For the
 latter assertion, use [EGA I, 5.5.1 (i)].

 For (v), factorize the morphism into $X\rightarrow
 X\times\st{Y}\rightarrow\st{Y}$. The first morphism is finite since
 $\st{Y}$ is admissible, so it is schematic. The second one is schematic
 as well since $X$ is a scheme.
\end{proof}

\subsubsection{}
The following variant of Chow's lemma for admissible stack is important
for showing fundamental properties of cohomological operations:

\begin{prop*}
 Let $\st{X}$ be an admissible stack. Then there exists a morphism
 $p\colon X\rightarrow\st{X}$ such that $X$ is a scheme and $p$ is a
 surjective generically finite proper morphism.
\end{prop*}
\begin{proof}
 We modify slightly the proof of \cite[(1.1)]{O}. From [{\it ibid.}, 2.1
 to 2.4], the argument is the same.
 In [{\it ibid.}, 2.5], he replaces $\st{X}$ by the closure of
 $\st{U}\hookrightarrow\st{X}\times\mb{P}(V)$. This replacement is not
 finite, but birational over $\st{U}$, so this replacement is harmless
 in our situation. In [{\it ibid.}, 2.6], it suffices to take $P'$ such
 that $\dim(P')=\dim(\st{X})$ in addition to the conditions there. If we
 have a surjective morphism $a\colon P'\rightarrow\st{X}$ then this is
 generically finite. Indeed, let $q\colon Q\rightarrow\st{X}$ be a
 smooth presentation, and $P'_Q:=P'\times_{\st{X}}Q$. By generically
 finiteness and some standard limit argument, there exists an open dense
 subscheme $V\subset Q$ such that $P'_Q\times_{Q}V\rightarrow V$ is
 finite. By fpqc descent, $a$ is finite over $q(V)\subset\st{X}$.

 We do not need [{\it ibid.}, 2.7, 2.8]. Take a quasi-finite flat
 covering $\bigl\{V_i\rightarrow\st{X}\bigr\}$ as in Remark
 \ref{admissdfn} (ii).
 Put $P'=\mb{P}^{r}_P$ where $r:=\dim(\st{X})-\dim(P)$. By copying
 the argument of [{\it ibid.}, 2.9] (taking $P_1$ to be $P'$),
 we can shrink $V_i$ and may assume that there exist morphisms
 $V_i\rightarrow P'$ factoring $V_i\rightarrow\st{X}\rightarrow
 P$ and the morphism $\coprod V_i\rightarrow P'$ is surjective.
 Indeed, in [{\it ibid.}, 2.9], he uses only the fact that
 $V_i\rightarrow P$ has equidimensional fibers. In our case, since
 $V_i\rightarrow\st{X}$ and $\st{X}\rightarrow P$ are flat, the
 equidimensionality holds. For [{\it ibid.}, 2.10 to 2.13], we just copy
 word by word. Since he only takes normalizations and blow-ups of $P'$,
 the dimension does not change, and we get the desired morphism.
\end{proof}

\begin{rem*}
 We are not able to take $p$ to be generically \'{e}tale in
 general. Indeed, let $k$ be an algebraically closed field of
 characteristic $p$, and let $G$ be a connected finite flat group scheme
 of dimension $0$ over $k$ which is not \'{e}tale ({\it e.g.}\
 $\alpha_p:=\mr{Spec}(k[T])/(T^p)$). Consider the admissible stack
 $\st{X}:=BG:=\left[\mr{Spec}(k)/G\right]$. Assume there exists a
 generically finite \'{e}tale proper surjective morphism $p\colon
 X\rightarrow BG$. Then since $\dim(BG)=0$, the dimension of $X$ would
 be $0$ as well.
 Since $BG$ is smooth over $\mr{Spec}(k)$ by \cite[5.1.2]{Behr}, $X$ is
 \'{e}tale over $\mr{Spec}(k)$. Thus by taking a connected component, we
 may assume that $X=\mr{Spec}(k)$, since $k$
 is assumed algebraically closed. Since any $G$-torsor on
 $\mr{Spec}(k)$ splits, the category $BG(\mr{Spec}(k))$ is a singleton,
 and $p$ would be nothing but the universal torsor. 
 The morphism $p$ cannot be \'{e}tale, since if it were, $G$ would be
 \'{e}tale.
\end{rem*}

\begin{cor}
 \label{alterationforadmisst}
 Let $\st{X}$ be an admissible stack. Then there exists a generically
 finite proper surjective morphism $X\rightarrow\st{X}$ such
 that $X$ is a smooth quasi-projective scheme.
\end{cor}
\begin{proof}
 Use Chow's lemma above, and then de Jong's alteration theorem.
\end{proof}

\subsubsection{}
Let $f\colon\st{X}\rightarrow\st{Y}$ be a morphism of admissible
stacks. When $f$ is a finite morphism (resp.\ projection), we denote by
$f_{\oplus}$ and $f^{\oplus}$ the functors $f_+$ and $f^+$ defined in
\ref{finitemorphfibercat} (resp.\ \ref{dfnofpushpullstackproj}) for
clarification.

Consider the canonical factorization $\st{X}\xrightarrow{i}
\st{X}\times\st{Y}\xrightarrow{p}\st{Y}$. Since $\st{Y}$ is admissible,
$i$ is finite, and the following definitions make sense:
\begin{equation*}
 f_+:=p_{\oplus}\circ i_{\oplus}\colon
  D^{+}_{\mr{hol}}(\st{X})\rightarrow
  D^+_{\mr{hol}}(\st{Y}),\qquad
  f^+:=i^{\oplus}\circ p^{\oplus}\colon
  D_{\mr{hol}}^{\mr{b}}(\st{Y})\rightarrow
  D^{\mr{b}}_{\mr{hol}}(\st{X}).
\end{equation*}
We have the adjoint pair $(f^+,f_+)$.

\begin{lem}
 \label{justofnotpu}
 Let $f\colon\st{X}\rightarrow\st{Y}$ be a finite morphism between
 admissible stacks. Then there are canonical isomorphisms
 $f_{\oplus}\cong f_+$, and $f^{\oplus}\cong f^+$.
\end{lem}
\begin{proof}
 Let $\st{X}\xrightarrow{i}\st{X}\times\st{Y}\xrightarrow{p}\st{Y}$ be
 the standard factorization. By adjointness property, it
 suffices to construct the isomorphism for the pull-back. We construct
 the isomorphism for their duals, namely $f^!$ and $i^!\circ p^!$ where
 $p^!:=L_{\st{X}}^\omega\boxtimes(-)$ (cf.\ Definition
 \ref{defofconstanddual}). Let
 $Y_\bullet\rightarrow\st{Y}$ be a simplicial realizable scheme
 presentation, and $X_\bullet\rightarrow\st{X}$ be its pull-back. We may
 assume that $X_0\rightarrow\st{X}$ and $Y_0\rightarrow\st{Y}$ are
 equidimensional. For $\ms{M}\in M(Y_\bullet)$, let us construct an
 isomorphism $\alpha\colon\H^0f^!(\ms{M})\xrightarrow{\sim}
 \H^0\bigl(i^!(L^\omega_{X_{\bullet}}\boxtimes\ms{M})\bigr)$.
 Put $Y=Y_n$ for $n\geq0$, and $X:=\st{X}\times_{\st{Y}}Y$. We have the
 following commutative diagram
 \begin{equation*}
  \xymatrix@C=40pt{
   X\ar[d]_{g}\ar[r]^-{i'}&
   X\times Y\ar[d]^{g'}\ar[r]^-{p'}&Y\ar[d]^{g''}\\
  \st{X}\ar[r]_-{i}&\st{X}\times\st{Y}\ar[r]_-{p}&\st{Y}.
   }
 \end{equation*}
 Let $d:=d_{g'}-d_{g''}$ where $d_{\star}$ denotes the relative
 dimension of $\star$. We have canonical isomorphisms
 \begin{align*}
  \label{schcaisocomp}
  g^*f^!(\ms{M})
  &\cong
  (p'\circ i')^!g''^*(\ms{M})\cong
  i'^!(L^\omega_{X}\boxtimes g''^*\ms{M})
  \\
  \notag
  &\cong
  i'^!g'^*(L^\omega_{X_{\bullet}}\boxtimes\ms{M})(d)[d]
  \cong
  g^*i^!(L^\omega_{X_{\bullet}}\boxtimes\ms{M}).
 \end{align*}
 Apply $\H^0$ to this isomorphism, and since it satisfies the cocycle
 condition, we have the desired isomorphism $\alpha$.
 This isomorphism moreover implies that
 \begin{equation}
  \label{vanishcohcommextpul}\tag{$\star$}
  \H^n\bigl(i^!(L^\omega_{X_{\bullet}}\boxtimes\ms{I})
 \bigr)=0\quad
 \mbox{for $n\neq0$}
 \end{equation}
 if $\ms{I}$ is an injective object in
 $M(Y_\bullet)$. Now, for $\ms{M}\in D^{+}(X_\bullet)$, take an
 injective resolution
 $\ms{M}\rightarrow\ms{I}^{\bullet}$. We denote by
 $\H^0\bigl(i^!(L^\omega_{X_\bullet}\boxtimes\ms{I}^{\bullet})\bigr)$
 the complex whose term in degree $n$ is
 $\H^0\bigl(i^!(L^\omega_{X_\bullet}\boxtimes\ms{I}^n)\bigr)$.
 Recall the notation $f^{\circ}:=\H^0f^!$ in
 \ref{defofpushfinitesimpl}. We have quasi-isomorphisms
 \begin{equation*}
  f^!(\ms{M})\cong f^{\circ}(\ms{I}^\bullet)\xrightarrow{\sim}
   \H^0\bigl(i^!(L^\omega_{X_\bullet}\boxtimes\ms{I}^{\bullet})\bigr)
   \xleftarrow{\sim}i^!(L^\omega_{X_\bullet}\boxtimes\ms{M})
 \end{equation*}
 where the first isomorphism holds since $f^!:=\mb{R}f^{\circ}$, the
 second one is induced by $\alpha$, and the third one follows by
 the vanishing (\ref{vanishcohcommextpul}). Thus, the lemma follows.
\end{proof}

\begin{lem}
 Let $\st{X}\xrightarrow{f}\st{Y}\xrightarrow{g}\st{Z}$ be morphisms of
 admissible stacks. Then we have canonical isomorphisms of functors
 \begin{alignat*}{2}
  \alpha&\colon\mr{id}^+\xrightarrow{\sim}\mr{id},&\qquad
  \beta&\colon\mr{id}\xrightarrow{\sim}\mr{id}_+\\
  c_{g,f}&\colon f^+\circ g^+\xrightarrow{\sim}(g\circ f)^+,&\qquad
   d^{g,f}&\colon (g\circ f)_+\xrightarrow{\sim}g_+\circ f_+.
 \end{alignat*}
 These homomorphisms are subject to the following conditions:
 1.\ We have identities $c_{f,\mr{id}}=\alpha(f^+)$,
 $c_{\mr{id},f}=f^+\alpha$. 2.\ Assume
 given another morphism of admissible stacks
 $h\colon\st{Z}\rightarrow\st{W}$. Then we have the equality
 \begin{equation*}
  c_{h,g\circ f}\circ c_{g,f}(h^+)=
   c_{h\circ g,f}\circ f^+(c_{h,g}).
 \end{equation*}
 We have the similar equalities for $\beta$ and $d^{g,f}$.
\end{lem}
\begin{proof}
 By the adjointness property, it suffices to show the lemma for the
 pull-back. First, we define $\alpha$ to be the isomorphism of Lemma
 \ref{justofnotpu}. Consider the following diagram:
 \begin{equation*}
  \xymatrix{
   \st{X}\ar[rr]^-{c}\ar[d]_-{a}\ar@{}[rd]|{\ccirc{1}}&&
   \st{X}\times\st{Z}
   \ar[dd]\ar[ld]^-{d}\ar@{}[ddl]|(.8){\ccirc{4}}\\
  \st{X}\times\st{Y}\ar[r]_-{b}\ar[d]\ar@{}[rd]|{\ccirc{3}}&
   \st{X}\times\st{Y}\times\st{Z}
   \ar[rd]\ar[d]\ar@{}[r]|-{\ccirc{2}}&\\
  \st{Y}\ar[r]&\st{Y}\times\st{Z}\ar[r]&\st{Z}.
   }
 \end{equation*}
 For a morphism $a$, we denote by $\Gamma_a$ the graph morphism.
 The ``transitivity isomorphism''
 for $\ccirc{1}$, namely the isomorphism
 $a^{\oplus}\circ b^{\oplus}\cong c^{\oplus}\circ d^{\oplus}$, is
 defined by \ref{finitemorphfibercat}.
 That for $\ccirc{4}$, use Lemma \ref{externaltensstsh}, and for
 $\ccirc{3}$, use Lemma \ref{compexttensoth} for finite morphisms
 $\mr{id}\colon\st{X}\rightarrow\st{X}$ and
 $\Gamma_g\colon\st{Y}\rightarrow\st{Y}\times\st{Z}$.
 Finally, for the transitivity isomorphism for $\ccirc{2}$,
 it suffices to construct
 $\Gamma_f^+(L_{\st{X}\times\st{Y}})\cong L_{\st{X}}$.
 This follows by Lemma \ref{externaltensstsh}.
 The verification of the compatibility conditions is straightforward, so
 we leave it as an exercise.
\end{proof}

Let $\mr{St}^{\mr{adm}}(k)$ be the full subcategory of the
category of algebraic stacks (we do not consider the 2-morphisms)
consisting of admissible stacks.
For an admissible stack $\st{X}$, we associate the triangulated
category $D^{+}_{\mr{hol}}(\st{X})$ (resp.\
$D^{\mr{b}}_{\mr{hol}}(\st{X})$). By the data of the lemma above, we
have a cofibered category $\ms{F}_+\rightarrow\mr{St}^{\mr{adm}}$ by
considering $f_+$ (resp.\ a fibered category
$\ms{F}^+\rightarrow\mr{St}^{\mr{adm}}$ by considering $f^+$).
Recall the notation of \ref{finitemorphfibercat} and Definition
\ref{commudualprimpushfin}.
The isomorphisms of Lemma \ref{justofnotpu} yield isomorphisms of
fibered and cofibered categories
$\ms{F}^{\oplus}\cong\ms{F}^+$ and $\ms{F}_{\oplus}\cong\ms{F}_+$ over
the category of admissible stack with finite morphism
$\mr{St}^{\mr{adm,fin}}$.

\begin{lem}
 \label{starplusequal}
 Let $f\colon \st{X}\rightarrow\st{Y}$ be a smooth morphism of
 admissible stacks, and let $d$ be the relative dimension of $f$. Then
 $f^*\cong f^+[d]$, and $\mb{R}f_*\cong f_+[-d]$ (cf.\ \ref{smoothpullback}).
\end{lem}
\begin{proof}
 By adjointness, it suffices to prove $f^*\cong f^+[d]$.
 Since the proof is similar to that of Lemma \ref{justofnotpu}, we only
 sketch. As the proof of the lemma, we also show the dual claim:
 $f^*(d)\cong f^![-d]$ (cf.\ Lemma \ref{smoothpullback}).
 We may take $X_{\bullet\bullet}$, $Y_{\bullet}$ as in
 \ref{smoothpullback}.
 Let $Y:=Y_i$ and $X:=X_{i,j}$.
 Consider the following cartesian diagram:
 \begin{equation*}
  \xymatrix{
   X\times_{\st{Y}}Y\ar[r]\ar[d]_{\alpha}\ar@{}[rd]|\square&
   X\times Y\ar[r]\ar[d]&
   Y\ar[d]\\
  \st{X}\ar[r]&
   \st{X}\times\st{Y}\ar[r]&
   \st{Y}.
   }
 \end{equation*}
 For $\ms{M}\in M(Y_\bullet)$, we have
 $\alpha^*f^*(\ms{M})(d)\cong\alpha^*f^![-d](\ms{M})$. This
 follows by using the fact that if $g$ is a smooth morphism of relative
 dimension $d_g$ between realizable schemes,
 then $g^*(d)\cong g^![-d_g]$ by definition.
 We finish the proof by descent argument.
\end{proof}

For a smooth morphism $\rho\colon X\rightarrow\st{X}$ from a realizable
scheme to an admissible scheme and $\ms{M}\in
D^{\mr{b}}_{\mr{hol}}(\st{X})$, we often denote $\rho^*(\ms{M})$ by
$\ms{M}_X$.
When $f$ is an open immersion, this lemma justifies the notation of
\ref{openimmdefshri}.
Recall that in such a case, we have an adjoint pair
$(f_!,f^+)$ as well as $(f^+,f_+)$.

\begin{prop}
 [Smooth base change]
 \label{smoothbcthmst}
 Consider the following cartesian diagram of admissible stacks:
 \begin{equation*}
  \xymatrix{
   \st{X}'\ar[r]^{g'}\ar[d]_{f'}\ar@{}[rd]|\square&
   \st{X}\ar[d]^{f}\\
  \st{Y}'\ar[r]_g&\st{Y}
   }
 \end{equation*}
 where $g$ is smooth. Then the base change homomorphism of functors
 $g^+f_+\rightarrow f'_+g'^+\colon
 D^{+}_{\mr{hol}}(\st{X})\rightarrow
 D^{+}_{\mr{hol}}(\st{Y}')$ is an isomorphism.
\end{prop}
\begin{proof}
 In the verification, it suffices to replace $g^+$ and $g'^+$ by $g^*$
 and $g'^*$ by Lemma \ref{starplusequal}.
 We use the standard factorization of $f$ into a finite morphism and a
 projection, and the verification is reduced to these cases
 separately. In both cases, the verification is straightforward from the
 definition, so we leave the details to the reader.
\end{proof}

\begin{prop}
 Let $f\colon\st{X}\rightarrow\st{Y}$ be a morphism between admissible
 stacks. Then $f_+$ preserves boundedness, namely
 induces a functor $D^{\mr{b}}_{\mr{hol}}(\st{X})
 \rightarrow D^{\mr{b}}_{\mr{hol}}(\st{Y})$.
\end{prop}
\begin{proof}
 It suffices to show that $f_+(\ms{M})$ is bounded for
 $\ms{M}\in\mr{Hol}(\st{X})$.
 Let us show  by the induction on the dimension of the support of
 $\ms{M}$. We may assume that the support of $\ms{M}$ is equal to
 $\st{X}$. Now, we may shrink $\st{X}$. Indeed, consider the
 localization triangle $i_+i^!\rightarrow\mr{id}\rightarrow
 j_+j^+\xrightarrow{+1}$ (cf.\ Lemma \ref{localtriangs}). We know that
 $i_+i^!$ and $j_+j^+$ preserve boundedness.
 By induction hypothesis, we are reduced to showing the
 proposition for $f_+j_+\cong(f\circ j)_+$, and the claim follows. 
 By shrinking $\st{X}$, we may assume that there exists a finite locally
 free morphism $X\rightarrow\st{X}$ from an affine scheme $X$ by Remark
 \ref{admissdfn} (ii).
 In this situation, since $\ms{M}$ is a direct factor of
 $h_+h^+(\ms{M})$ by Lemma \ref{directfactorlemmaeasy}, we may assume
 that $\st{X}$ is an affine scheme.
 Finally, let $g\colon Y\rightarrow\st{Y}$ be a smooth presentation
 from an affine scheme. Since the functor $g^*$ is conservative by
 Remark \ref{dfnofcat} (ii), it suffices to show the claim for the
 morphism $\st{X}\times_{\st{Y}}Y\rightarrow Y$. Since $\st{X}$ is
 assumed to be an affine scheme, this is a morphism of realizable
 schemes, and the boundedness is already known.
\end{proof}

\subsubsection{}
\label{dfninthomstack}
Let $\st{X}$ be an admissible stack. Then the diagonal morphism
$\Delta\colon\st{X}\rightarrow\st{X}\times\st{X}$ is finite.
For $i=1,2$, let $p_i\colon\st{X}\times\st{X}\rightarrow\st{X}$ be the
$i$-th projection. We define the {\em internal Hom functor} by
\begin{equation*}
 \shom(\ms{M},\ms{N}):=
  p_{1\ms{M}+}\bigl(\Delta_+(\ms{N})\bigr)
  \colon D^{\mr{b}}_{\mr{hol}}(\st{X})^{\circ}\times
  D^{\mr{b}}_{\mr{hol}}(\st{X})\rightarrow
  D^{\mr{b}}_{\mr{hol}}(\st{X}).
\end{equation*}
Let $\ms{L}\in D^{\mr{b}}_{\mr{hol}}(\st{X})$.
Since $(p_{1\ms{M}}^+,p_{1\ms{M}+})$ is an adjoint pair, with Remark
\ref{deradjlemm}, we have
\begin{equation*}
 \mb{R}\mr{Hom}_{D(\st{X})}
  \bigl(\ms{L},\shom(\ms{M},\ms{N})\bigr)
  \cong
  \mb{R}\mr{Hom}_{D(\st{X}\times\st{X})}
  \bigl(\ms{L}\boxtimes\ms{M},\Delta_+\ms{N}\bigr).
\end{equation*}
We can check easily that when $\st{X}$ is a realizable scheme, $\shom$
coincides with that in \ref{fundproprealsch}.
Now, let $f\colon\st{X}\rightarrow\st{Y}$,
$g\colon\st{X}'\rightarrow\st{Y}'$ be morphisms of admissible stacks.
Then, we have an isomorphism $(f\times g)^+\bigl((-)
\boxtimes(-)\bigr)\cong f^+(-)\boxtimes g^+(-)$. This follows by
combining Lemma \ref{compexttensoth} and (\ref{commpullprojcase}).
Using this, we have
\begin{align*}
 \mr{Hom}\bigl(\ms{L},f_+\shom&(f^+\ms{M},\ms{N})\bigr)
 \cong
 \mr{Hom}\bigl(f^+(\ms{L})\boxtimes f^+(\ms{M}),
 \Delta_{\st{X}*}(\ms{N})\bigr)\\
 &\cong
 \mr{Hom}\bigl((f\times f)^+(\ms{L}\boxtimes\ms{M}),
 \Delta_{\st{X}*}(\ms{N})\bigr)\\
 &\cong
 \mr{Hom}\bigl(\ms{L},\shom(\ms{M},f_+\ms{N})\bigr).
\end{align*}
Thus, we get an isomorphism
\begin{equation*}
 f_+\shom(f^+\ms{M},\ms{N})\cong\shom(\ms{M},f_+\ms{N}).
\end{equation*}

\begin{lem}
 \label{locdeschom}
 Let $\ms{M}$ and $\ms{N}$ be objects of
 $D^{\mr{b}}_{\mr{hol}}(\st{X})$.
 For a presentation $\rho\colon X\rightarrow\st{X}$ from a realizable
 scheme, there is a canonical isomorphism
 \begin{equation*}
  \rho^*\shom(\ms{M},\ms{N})\cong
   \shom\bigl(\rho^*(\ms{M}),\rho^*(\ms{N})\bigr)[d],
 \end{equation*}
 and $d$ denotes the relative dimension function of $\rho$.
\end{lem}
\begin{proof}
 Consider the following commutative diagram
 \begin{equation*}
  \xymatrix@C=40pt{
   &X\ar[dl]_{\delta}\ar[d]_{\Delta'}\ar@{=}[dr]&&\\
   X\times X\ar[r]_-{\rho\times\mr{id}}&
   \st{X}\times X\ar[r]_-q&
   X\ar[r]_{\rho}&\st{X},
   }
 \end{equation*}
 where $\delta$ is the diagonal morphism, $\Delta':=(\rho,\mr{id})$, and
 $q$ is the second projection. Put $\ms{N}':=\rho^*(\ms{N})$.
 By the definition of the
 functor $q_{\ms{M}+}$, we have
 \begin{equation}
  \label{firsisohomloc}
   \tag{$\star$}
  \rho^*\shom(\ms{M},\ms{N})\cong q_{\ms{M}+}(\Delta'_+(\ms{N}')).
 \end{equation}
 Let $\ms{L}$ be an object in $D^{\mr{b}}_{\mr{hol}}(X)$. For an
 algebraic stack $\st{Y}$, we denote $\mr{Hom}_{D(\st{Y})}$ by
 $\mr{Hom}_{\st{Y}}$.
 We have
 \begin{align*}
  \mr{Hom}_X&\bigl(\ms{L},q_{\ms{M}+}(\Delta'_+(\ms{N}'))\bigr)
  \cong
  \mr{Hom}_{\st{X}\times X}
  \bigl(\ms{M}\boxtimes\ms{L},\Delta'_+(\ms{N}')\bigr)\\
  &\cong
  \mr{Hom}_X\bigl(\delta^+(\rho\times\mr{id})^+(\ms{M}\boxtimes\ms{L}),
  \ms{N}'\bigr)
  \cong
  \mr{Hom}_X\bigl(\rho^+(\ms{M})\otimes\ms{L},\ms{N}'\bigr)
  \\
  &\cong
  \mr{Hom}_X\bigl(\ms{L},\shom(\rho^+(\ms{M}),\ms{N}')\bigr),
 \end{align*}
 where the first isomorphism follows by the adjunction
 $(q^+_{\ms{M}},q_{\ms{M}+})$, the second by the adjunction
 $(\Delta'^+,\Delta'_+)$, the third by Lemma \ref{compexttensoth}, and
 the last by the adjointness property of $\otimes$ and $\shom$ for
 realizable schemes.
 Thus, we have a canonical isomorphism
 $\shom\bigl(\rho^+(\ms{M}),\ms{N}'\bigr)\cong
 q_{\ms{M}+}\bigl(\Delta'_+(\ms{N}')\bigr)$.
 Combining with the isomorphism $\rho^+\cong\rho^*[-d]$ and
 (\ref{firsisohomloc}), the lemma follows.
\end{proof}

\subsubsection{}
\label{bidualityofstack}
Recalling Definition \ref{defofconstanddual}, we define the {\em dual
functor} to be
\begin{equation*}
 \mb{D}_{\st{X}}(\ms{M}):=\shom(\ms{M},L^\omega_{\st{X}})\colon
  D^{\mr{b}}_{\mr{hol}}(\st{X})^{\circ}\rightarrow
  D^{\mr{b}}_{\mr{hol}}(\st{X}).
\end{equation*}
If no confusion may arise, we often omit the subscript
$\st{X}$ from $\mb{D}_{\st{X}}$.

\begin{prop*}[Biduality]
 There exists a canonical isomorphism of functors
 \begin{equation*}
  \gamma\colon
   \mr{id}\xrightarrow{\sim}\mb{D}_{\st{X}}\circ\mb{D}_{\st{X}}
   \colon
   D^{\mr{b}}_{\mr{hol}}(\st{X})\rightarrow
   D^{\mr{b}}_{\mr{hol}}(\st{X}).   
 \end{equation*}
\end{prop*}
\begin{proof}
 We have
 isomorphisms
 \begin{align}
  \label{constdualhom}
  \mr{Hom}\bigl(\mb{D}(\ms{M}),\mb{D}(\ms{M})\bigr)&\cong
  \mr{Hom}\bigl(\mb{D}(\ms{M})\boxtimes\ms{M},\Delta_+
  (L^\omega_{\st{X}})\bigr)\\
  \notag
 &\cong\mr{Hom}\bigl(\ms{M}\boxtimes\mb{D}(\ms{M}),
 \Delta_+(L^\omega_{\st{X}})\bigr)\cong
  \mr{Hom}(\ms{M},\mb{D}\mb{D}(\ms{M})),
 \end{align}
 where the second isomorphism is induced by the morphism
 $\st{X}\times\st{X}\xrightarrow{\sim}\st{X}\times\st{X}$ exchanging the
 first and second factor.
 The image of the identity homomorphism induces $\gamma$ in the claim.
 Let $\rho\colon X\rightarrow\st{X}$ be a presentation from a
 realizable scheme. By Remark \ref{dfnofcat} (ii), it suffices to show
 that $\rho^*(\gamma)$ is an isomorphism.
 Since the dual functor is compatible with $\rho^*$ by
 Lemma \ref{locdeschom}, we are reduced to checking the biduality in the
 realizable scheme case, which is Lemma \ref{bidualrealcase}.
\end{proof}

\begin{rem*}
 (i) We have a canonical isomorphism
 $\H^i\mb{D}_{\st{X}}\cong\H^i\mb{D}'_{\st{X}}$ for any $i$. Indeed, for
 any presentation $\rho\colon X\rightarrow\st{X}$, we
 can check that $\rho^*\mb{D}_{\st{X}}$ and $\rho^*\mb{D}'_{\st{X}}$ are
 canonically isomorphic to $\mb{D}_X\rho^*$, thus we get the claim by
 gluing. However, even though there is no doubt to believe that
 $\mb{D}_{\st{X}}$ and $\mb{D}'_{\st{X}}$ coincide, we do not know how
 to construct a morphism $\mb{D}_{\st{X}}\rightarrow\mb{D}'_{\st{X}}$ in
 $D^{\mr{b}}_{\mr{hol}}(\st{X})$ compatible with the isomorphisms of
 $i$-th cohomologies.

 (ii) We have a canonical isomorphism
 $L^{\omega}_{\st{X}}\cong\mb{D}_{\st{X}}(L_{\st{X}})$. This follows
 since for any $X\in\st{X}_{\mr{sm}}$, we have
 $(L^{\omega})_X\cong\mb{D}(L)_X$, and we have the uniqueness of
 Theorem \ref{BBDgluingthm}.

 (iii) When $j$ is an open immersion, we have $j^+\circ\mb{D}\cong
 \mb{D}\circ j^+$. Thus, $j_!\cong\mb{D}\circ j_+\circ\mb{D}$.

 (iv) When $f$ is a finite morphism, we prove in Lemma
 \ref{uppershriekproj} that we have an isomorphism $f^+\cong\mb{D}\circ
 f^!\circ\mb{D}$.

\end{rem*}

\subsubsection{}
\label{basicprophomanddfn}
We define the {\em tensor product} by
\begin{equation*}
 (-)\otimes(-):=\Delta^+\bigl((-)
  \boxtimes(-)\bigr)\colon D^{\mr{b}}_{\mr{hol}}(\st{X})\times
  D^{\mr{b}}_{\mr{hol}}(\st{X})\rightarrow
  D^{\mr{b}}_{\mr{hol}}(\st{X}).
\end{equation*}

Let $\ms{M}$, $\ms{N}$, $\ms{L}$ be objects in
$D^{\mr{b}}_{\mr{hol}}(\st{X})$. We have
\begin{align}
 \notag
 \mb{R}\mr{Hom}\bigl(\ms{M},\shom(\ms{N},\ms{L})\bigr)
 &:=
 \mb{R}\mr{Hom}\bigl(\ms{M},p_{1\ms{N}+}(\Delta_+\ms{L})\bigr)
 \cong
 \mb{R}\mr{Hom}\bigl(\ms{M}\boxtimes\ms{N},\Delta_+\ms{L}\bigr)\\
 \label{adjointhomtensshom}
 &\cong
 \mb{R}\mr{Hom}\bigl(\Delta^+(\ms{M}\boxtimes\ms{N}),\ms{L}\bigr)
 =:
 \mb{R}\mr{Hom}\bigl(\ms{M}\otimes\ms{N},\ms{L}\bigr).
\end{align}
The identity homomorphism of $\shom(\ms{M},\ms{N})$ induces the {\em
evaluation homomorphism}
\begin{equation*}
 \shom(\ms{M},\ms{N})\otimes\ms{M}\rightarrow\ms{N}.
\end{equation*}

Now, let $p\colon\st{X}\rightarrow\mr{Spec}(k)$ denote the
structural morphism. Since $p_2\circ\Delta\cong\mr{id}$, we have
\begin{equation*}
 L_{\st{X}}\otimes\ms{M}=
  \Delta^+\bigl(L_{\st{X}}\boxtimes\ms{M}\bigr)\cong
  \Delta^+(p_2^+(\ms{M}))\cong(p_2\circ \Delta)^+(\ms{M})
  \cong\ms{M}
\end{equation*}
where the isomorphisms hold by the definition and results in
\ref{dfnofpushpullstackproj}.
Using these, we have
\begin{equation*}
 \mb{R}\Gamma\circ p_{+}\shom(\ms{M},\ms{N})\cong
  \mb{R}\mr{Hom}_{D(\st{X})}\bigl(L_{\st{X}},
  \shom(\ms{M},\ms{N})\bigr)\cong
  \mb{R}\mr{Hom}_{D(\st{X})}(\ms{M},\ms{N})
\end{equation*}
where the first isomorphism holds by Proposition \ref{pushforcohrel},
and the second by what we have just proven.

\begin{prop}
 \label{easypropfunc}
 Let $f\colon\st{X}\rightarrow\st{X}'$,
 $g\colon\st{Y}\rightarrow\st{Y}'$ be morphisms of admissible
 stacks.

\begin{enumerate}
 \item We have an isomorphism $(f\times g)^+\bigl((-)
 \boxtimes(-)\bigr)\cong f^+(-)\boxtimes g^+(-)$.

 \item We have $\mb{D}\bigl((-)\boxtimes(-)\bigr)\cong
 \mb{D}(-)\boxtimes\mb{D}(-)$.

 \item\label{unitcommute}
      We have $f^+\bigl((-)\otimes(-)\bigr)\cong f^+(-)\otimes
      f^+(-)$.

 \item\label{adjunchomtensshom}
      We have $\shom(\ms{M}\otimes\ms{N},\ms{L})\cong
      \shom\bigl(\ms{M},\shom(\ms{N},\ms{L})\bigr)$.

 \item\label{inthombidual}
      We have $\shom\bigl(\mb{D}(\ms{M}),\mb{D}(\ms{N})\bigr)\cong
      \shom(\ms{N},\ms{M})$.

 \item\label{calcofhomdual}
      We have $\shom(\ms{M},\ms{N})\cong\mb{D}\bigl(\ms{M}
      \otimes\mb{D}(\ms{N})\bigr)$.
\end{enumerate}
\end{prop}
\begin{proof}
 The first one is just a reproduction from \ref{dfninthomstack}. The
 second claim follows by combining Lemma \ref{compexttensoth} (the
 commutativity of $f_+$ and $\boxtimes$), Lemma
 \ref{kunnethforprojstack}, and Lemma \ref{externaltensstsh} (i).
 Let us show \eqref{unitcommute}. Let $\ms{M}'$, $\ms{N}'$ be objects in
 $D^{\mr{b}}_{\mr{hol}}(\st{X}')$. Consider the following commutative
 diagram:
 \begin{equation*}
  \xymatrix@C=50pt{
   \st{X}\ar[r]^-{\Delta_{\st{X}}}\ar[d]_f&
   \st{X}\times\st{X}\ar[d]^{f\times f}\\
  \st{X}'\ar[r]_-{\Delta_{\st{X}'}}&\st{X}'\times\st{X}'.
   }
 \end{equation*}
 Using this diagram, we have
 \begin{equation*}
  f^+\bigl(\ms{M}'\otimes\ms{N}'\bigr)\cong f^+\Delta^+_{\st{X}'}
   \bigl(\ms{M}'\boxtimes\ms{N}'\bigr)\cong
   \Delta^+_{\st{X}}(f\times f)^+\bigl(\ms{M}'\boxtimes\ms{N}'\bigr)
   \cong f^+(\ms{M}')\otimes f^+(\ms{N}').
 \end{equation*}
 Let us show \eqref{adjunchomtensshom}. For any $\ms{Q}\in
 D^{\mr{b}}(\st{X})$, we have a canonical isomorphism
 \begin{equation*}
  \mr{Hom}\bigl(\ms{Q},\shom(\ms{M}\otimes\ms{N},\ms{L})\bigr)
   \cong
  \mr{Hom}\bigl(\ms{Q},\shom(\ms{M},\shom(\ms{N},\ms{L}))\bigr)
 \end{equation*}
 by using (\ref{adjointhomtensshom}) twice, thus the claim follows.
 To show \eqref{inthombidual}, the isomorphisms (\ref{constdualhom})
 remain to hold even if we replace $\mr{Hom}$ by $\shom$.
 For the last claim \eqref{calcofhomdual}, it suffices to construct a
 canonical isomorphism $\mr{Hom}\bigl(\ms{Q},\shom(\ms{M},\ms{N})\bigr)
 \cong\mr{Hom}\bigl(\ms{Q},\mb{D}(\ms{M}\otimes\mb{D}(\ms{N}))\bigr)$.
 This can be shown by using \eqref{adjunchomtensshom} and
 \eqref{inthombidual}.
\end{proof}

\subsubsection{}
Let $f\colon\st{X}\rightarrow\st{Y}$ be a morphism of admissible
stacks, and let $\ms{M}$, $\ms{N}$ be objects in
$D^{\mr{b}}_{\mr{hol}}(\st{X})$. The adjunction homomorphism
$f^+f_+\rightarrow\mr{id}$ induces a homomorphism
\begin{equation*}
 f^+\bigl(f_+(\ms{M})\otimes f_+(\ms{N})\bigr)\cong
 f^+f_+(\ms{M})\otimes f^+f_+(\ms{N})\rightarrow
  \ms{M}\otimes\ms{N},
\end{equation*}
where the first isomorphism follows by Proposition \ref{easypropfunc}.
This induces the homomorphism
\begin{equation}
 \label{pushtenscomhom}
 f_+(\ms{M})\otimes f_+(\ms{N})\rightarrow
  f_+\bigl(\ms{M}\otimes\ms{N}\bigr).
\end{equation}
Using this, we have a homomorphism
\begin{equation*}
 f_+\shom(\ms{M},\ms{N})\otimes f_+(\ms{M})
  \rightarrow
  f_+\bigl(\shom(\ms{M},\ms{N})\otimes\ms{M}\bigr)
  \rightarrow f_+(\ms{N})
\end{equation*}
where the second homomorphism is induced by the evaluation
homomorphism. Taking the adjunction (\ref{adjointhomtensshom}), we get a
canonical homomorphism
\begin{equation}
 \label{canhomhompush}
 f_+\shom\bigl(\ms{M},\ms{N}\bigr)\rightarrow
  \shom(f_+\ms{M},f_+\ms{N}).
\end{equation}

\subsubsection*{Duality results}
\subsubsection{}
\label{constdualmapproprep}

First, let us construct the trace map for projective morphisms.
Let $f\colon\st{X}\rightarrow\st{Y}$ be a {\em projective} morphism
between admissible stacks. Let $Y_\bullet\rightarrow\st{Y}$ be a
simplicial quasi-projective scheme presentation, and since $f$ is
assumed projective, the cartesian product
$X_\bullet:=\st{X}\times_{\st{Y}}Y_\bullet$ is an admissible
simplicial quasi-projective scheme as well. Since
$f_+L^\omega_{X_\bullet},L^\omega_{Y_\bullet}$ are
in $D^{\mr{b}}_{\mr{hol}}(\st{Y}_\bullet)$, we have a spectral
sequence
\begin{equation*}
 E_1^{p,q}=\mr{Ext}^q_{D(Y_p)}(f_{p+}L^\omega_{X_p},
  L^\omega_{Y_p})\Rightarrow
  \mr{Ext}^{p+q}_{D(Y_\bullet)}(f_+L^\omega
  _{X_\bullet},L^\omega_{Y_\bullet})
\end{equation*}
by (\ref{BDspecseq}). The usual trace map
defines $\mr{Tr}_{f_p}\in\mr{Hom}_{D(X_p)}
(f_{p+}L^\omega_{X_p},L^\omega_{Y_p})$ for each $p$. By this spectral sequence, the
trace map for $p=1$ yields the desired trace map $\mr{Tr}_{f}$. We note
that $(\mr{Tr}_f)_p$ is nothing but $\mr{Tr}_{f_p}$ by the compatibility
of the trace map.

Now, using this trace map, we define
\begin{align*}
 f_+\circ\mb{D}_{\st{X}}\cong f_+\shom(-,L^\omega_{\st{X}})
  &\xrightarrow{\star}\shom(f_+(-),f_+L^\omega_{\st{X}})\\
  &\xrightarrow{\mr{Tr}_f}
  \shom(f_+(-),L^\omega_{\st{Y}})\cong\mb{D}_{\st{Y}}\circ f_+,
\end{align*}
where $\star$ is induced by (\ref{canhomhompush}).
This homomorphism is in fact an isomorphism. To check this, since the
verification is local, we may easily reduce to the realizable scheme
case, and this case follows by Lemma \ref{dualfunccommuform}.
Summing up, we have the following:

\begin{lem*}
 \label{uppershriekproj}
 Let $f\colon\st{X}\rightarrow\st{Y}$ be a projective morphism between
 admissible stacks. Then we have the canonical isomorphism
 $\mb{D}_{\st{Y}}\circ f_+\xleftarrow{\sim} f_+\circ\mb{D}_{\st{X}}$. In
 particular, putting $f^!:=\mb{D}_{\st{X}}\circ
 f^+\circ\mb{D}_{\st{Y}}$, the pair $(f_+,f^!)$ is an adjoint pair.
 Moreover, consider the following cartesian diagram of admissible
 stacks:
 \begin{equation}
  \label{cartdiagbch}
  \xymatrix{
   \st{X}'\ar[r]^{g'}\ar[d]_{f'}\ar@{}[rd]|\square&
   \st{X}\ar[d]^{f}\\
  \st{Y}'\ar[r]_g&\st{Y}.
   }
 \end{equation}
 Under the assumption that $f$ is projective, the canonical homomorphism
 $g^+f_+\rightarrow f'_+g'^+$ is an isomorphism. If, moreover, $g$ is an
 open immersion, we have the canonical isomorphism $g_!\circ f'_+\cong
 f_+\circ g'_!$.
\end{lem*}

\begin{prop}[Proper base change]
 \label{propbasechansta}
 Consider the cartesian diagram of admissible stacks
 (\ref{cartdiagbch}). Assume that $f$ is proper (where we do {\em not}
 assume $f$ to be projective). Then the canonical
 homomorphism $g^+f_+\rightarrow f'_+g'^+$ is an isomorphism.
\end{prop}
\begin{proof}
 It suffices to show that $g^+f_+(\ms{M})\xrightarrow{\sim}
 f'_+g'^+(\ms{M})$ for $\ms{M}\in\mr{Hol}(\st{X})$. We may assume
 $\st{X}$ to be reduced by Lemma \ref{directfactorlemmaeasy} and
 Lemma \ref{uppershriekproj}.
 By smooth base change theorem \ref{smoothbcthmst}, we may replace
 $\st{Y}$ by its smooth presentation. In particular, we may assume
 $\st{Y}=:Y$ to be a realizable scheme. Now, we use the induction on the
 dimension of the support of $\ms{M}$. We may assume
 $\mr{Supp}(\ms{M})=\st{X}$. By induction hypothesis, it
 suffices to show the equality for $\ms{M}=j_!(\ms{N})$ where
 $j\colon\st{U}\hookrightarrow\st{X}$ which is open dense.
 By Corollary \ref{alterationforadmisst}, there is a smooth
 quasi-projective scheme $X$ projective and generically finite over
 $\st{X}$, and projective over $Y$. Since
 $h\colon X\rightarrow \st{X}$ is projective, we already know the base
 change by Lemma \ref{uppershriekproj}.
 We may shrink $\st{U}$ so that $h$ is finite flat over $\st{U}$ since
 $\st{X}$ is assumed reduced.
 We denote $h^{-1}(\st{U})\rightarrow\st{U}$ by $h$ abusing the
 notation.
 Then we have $h_+j'_!h^+(\ms{N})\cong j_!h_+h^+(\ms{N})$, where
 $j'\colon h^{-1}(\st{U})\rightarrow X$, and this contains $j_!(\ms{N})$
 as a direct factor by Lemma \ref{directfactorlemmaeasy}. Thus, the
 verification is reduced to the case $\ms{M}=j'_!(h^+(\ms{N}))$.
 Indeed, let $\ms{F}\in D^{\mr{b}}_{\mr{hol}}(\st{X})$, and $\ms{E}$ be
 a direct factor of $\ms{F}$. For any integer $i$, we have the following
 commutative diagram:
 \begin{equation*}
  \xymatrix{
   \H^ig^+f_+\ms{E}\ar[r]\ar[d]&
   \H^ig^+f_+\ms{F}\ar[r]\ar[d]^{\star}&
   \H^ig^+f_+\ms{E}\ar[d]\\
  \H^if'_+g'^+\ms{E}\ar[r]&
   \H^if'_+g'^+\ms{F}\ar[r]&
   \H^if'_+g'^+\ms{E}
   }
 \end{equation*}
 where the compositions of the horizontal homomorphisms are the
 identities. If the homomorphism $\star$ is an injection (resp.\
 surjection), so is the left (resp.\ right) vertical homomorphism, so it
 suffices to show that $\star$ is an isomorphism.

 Thus we may replace $\st{X}$ by $X$. In this case,
 the verification is local with respect to $Y$, and may assume it to be
 an affine scheme. In this situation, $X$ is realizable as well since
 $X$ is projective over $Y$, and the proper base change theorem has
 already been known (cf.\ \ref{fundproprealsch} \eqref{basechangeprop}).
\end{proof}

\begin{dfn}
 \label{defcompactcadmmorph}
 A morphism of admissible stacks $f\colon\st{X}\rightarrow\st{Y}$ is
 said to be {\em compactifiable} if it can be factorized as
 \begin{equation*}
  \st{X}\xrightarrow{j}\overline{\st{X}}
   \xrightarrow{f'}\st{Y}
 \end{equation*}
 where $\overline{\st{X}}$ is admissible, $j$ is an open immersion, and
 $f'$ is proper. We say that
 $\st{X}\rightarrow\overline{\st{X}}$ is a {\em compactification of
 $f$}. An admissible stack $\st{X}$ is said to be {\em compactifiable}
 if the structural morphism is compactifiable. We abbreviate
 compactifiable admissible stack as {\em c-admissible stack}.
\end{dfn}

\subsubsection{}
\label{defofcadmcat}
In this subsection, we fix a subcategory $\cadm$ of the category of
admissible stacks satisfying the following conditions: 1.\ open
immersions and proper morphisms are morphisms in $\cadm$,
2.\ any morphism $f\colon\st{X}\rightarrow\st{Y}$ in $\cadm$ is
compactifiable by an object in $\cadm$, and 3.\ for a proper morphism
$\st{X}\rightarrow\st{Y}$ and any morphism $\st{Y}'\rightarrow\st{Y}$ in
$\cadm$, the fiber product
$\st{X}\times_{\st{Y}}\st{Y}'\rightarrow\st{Y}'$ is in $\cadm$. An
example of such category is the following:

\begin{lem*}
 The full subcategory of c-admissible stacks satisfies the conditions.
\end{lem*}
\begin{proof}
 We need to show that any morphism between c-admissible stacks is
 compactifiable. Let $f\colon \st{X}\rightarrow\st{Y}$ be a morphism
 between c-admissible stacks, and let $\overline{\st{X}}$ be a
 compactification of the structural morphism of $\st{X}$. Then $f$ is
 factorized as $\st{X}\xrightarrow{\Gamma}\overline{\st{X}}\times
 \st{Y}\xrightarrow{p}\st{Y}$ where $\Gamma$ is the graph morphism and
 $p$ is the projection. Since
 $\overline{\st{X}}$ is assumed proper, $p$ is proper. Thus, it
 suffices to show that $\Gamma$ is compactifiable. Since $\st{Y}$ is
 admissible, $\Gamma$ is a quasi-finite morphism. By \cite[16.5]{LM},
 any quasi-finite morphism between admissible stacks is compactifiable,
 and the claim follows.
\end{proof}

\begin{rem*}
 Any {\em algebraic space} separated of finite type over $k$ is known to
 be c-admissible by \cite{Co3}.
\end{rem*}

\subsubsection{}
An advantage of considering the category $\cadm$ is that it satisfies
the conditions of [SGA 4, XVII, 3.2.4] if we take $(S)$ to be $\cadm$,
$(S,i)$ to be the subcategory consisting of open immersions, and $(S,p)$
to be the subcategory consisting of proper morphisms.

Now, for $\st{X}\in\cadm$, we associate the category
$D^{\mr{b}}_{\mr{hol}}(\st{X})$. We shall further endow with data which
satisfy the conditions of [SGA 4, XVII, 3.3.1].
For a proper morphism $p$, we consider the push-forward $p_+$ and the
canonical isomorphism $(q\circ p)_+\cong q_+\circ p_+$ for composable
morphisms $p$, $q$.
For an open immersion $j$, we consider $j_!$ with canonical isomorphisms
for compositions. These are the data of [{\it ibid.}, (i), (i'), (ii),
(ii')]. These functors are subject to the conditions [{\it ibid.}, (a),
(a'), (b), (b')]. Finally, for [{\it ibid.}, (iii)], we use the proper
base change \ref{propbasechansta} and localization exact triangle
\ref{localtriangs}. This isomorphism is subject to the
conditions [{\it ibid.}, (c), (c')]. Thus, we may apply [{\it ibid.},
Proposition 3.3.2]. Summing up, we come to get the following
definition.

\begin{dfn*}
 Let $f\colon\st{X}\rightarrow\st{Y}$ be a morphism in $\cadm$. Take a
 compactification $j\colon\st{X}\hookrightarrow\overline{\st{X}}$, and
 $g\colon\overline{\st{X}}\rightarrow\st{Y}$ be the proper
 morphism. Then the functor $g_+\circ j_!$ does not depend on the choice
 of the factorization up to canonical equivalence. This functor is
 denoted by $f_!$. Given composable morphisms $f$ and $g$ in $\cadm$, we
 have a canonical equivalence $(f\circ g)_!\cong f_!\circ g_!$.
\end{dfn*}

\begin{prop}
 \label{basechforshrik}
 Consider the cartesian diagram (\ref{cartdiagbch}) (where we do {\em
 not} assume $f$ to be projective).
 We assume that the diagram is in $\cadm$. Then there exists a canonical
 isomorphism $g^+\circ f_!\cong f'_!\circ g'^+$.
\end{prop}
\begin{proof}
 By definition of $f_!$, it suffices to treat the case where $f$ is
 proper and an open immersion separately. When $f$ is proper, this is
 nothing but proper base change theorem \ref{propbasechansta}. When
 $f=:j$ is an open immersion, we have the canonical homomorphism
 $j'_!\circ g'^+\rightarrow g^+\circ j_!$. By definition, this
 homomorphism is an isomorphism if we take $j'^+$. Thus by the
 localization triangle (cf.\ Lemma \ref{localtriangs}), we get the
 isomorphism. Finally, we need to show that the resulting isomorphism
 does not depend on the choice of the factorization. Since the
 verification is standard, we leave it to the reader.
\end{proof}

\subsubsection{}
Let us construct a trace map, namely a map
$f_!L^{\omega}_{\st{X}}\rightarrow L^{\omega}_{\st{Y}}$ for any morphism
$f\colon\st{X}\rightarrow\st{Y}$ in $\cadm$. This will be achieved in
Theorem \ref{constoftracemap}. For this, we need to introduce a new
t-structure.

\begin{dfn*}
 Let $X$ be a realizable scheme over $k$.
 For $\star\in\{\geq0,\leq0\}$, let ${}^{\mr{c}}D^{\star}$ be the full
 subcategory of $D^{\mr{b}}_{\mr{hol}}(X/L_{\emptyset})$ consisting of
 $\ms{C}$ such that $\mr{for}_L(\ms{C})\in
 {}^{\mr{c}}D^{\star}_{\mr{hol}}(X/K_{\emptyset})$ using
 \ref{dfnofctstruc}. The
 pair $({}^{\mr{c}}D^{\leq},{}^{\mr{c}}D^{\geq})$ defines a t-structure
 on $D^{\mr{b}}_{\mr{hol}}(X/L_{\emptyset})$ called the
 {\em constructible t-structure}. We define
 ${}^{\mr{dc}}D^{\leq}:=\mb{D}({}^{\mr{c}}D^{\geq})$ and
 ${}^{\mr{dc}}D^{\geq}:=\mb{D}({}^{\mr{c}}D^{\leq})$. Then
 $({}^{\mr{dc}}D^{\leq},{}^{\mr{dc}}D^{\geq})$ defines a t-structure on
 $D^{\mr{b}}_{\mr{hol}}$. This is called the {\em dual constructible
 t-structure}.
 We also define t-structures
 $({}^{\mr{c}}D^{\leq},{}^{\mr{c}}D^{\geq})$ and
 $({}^{\mr{dc}}D^{\leq},{}^{\mr{dc}}D^{\geq})$ on
 $D^{\mr{b}}_{\mr{hol}}(X/L_F)$ such that $\ms{C}$ is in one of the full
 subcategories if and only if $\mr{for}_F(\ms{C})$ is in the
 corresponding one of $D^{\mr{b}}_{\mr{hol}}(X/L_{\emptyset})$.
\end{dfn*}

\begin{dfn}
 \label{Artinstackcdctstr}
 Let $\st{X}$ be an algebraic stack. Let
 $X_\bullet\rightarrow\st{X}$ be a simplicial realizable scheme
 presentation, and let $\ms{M}\in D^{\mr{b}}_{\mr{hol}}(X_\bullet)$. Put
 $d_i:=d_{X_i/\st{X}}$ (cf.\ \ref{reldimfunc}).
 \begin{enumerate}
  \item A complex $\ms{M}\in D^{\mr{b}}_{\mr{hol}}(\st{X})$ is in
	${}^{\mr{c}}D^{\star}$ ($\star\in{\leq0,\geq0}$) if and only if
	$\rho_i^*(\ms{M})\in{}^{\mr{c}}D^{\star-d_i}$.

  \item A complex $\ms{N}\in D^{\mr{b}}_{\mr{hol}}(\st{X})$ is in
	${}^{\mr{dc}}D^{\star}$ ($\star\in{\leq0,\geq0}$) if and only if
	$\rho_i^*(\ms{N})\in{}^{\mr{dc}}D^{\star+d_i}$.
 \end{enumerate}
 We leave the reader to check that
 $({}^{\mr{c}}D^{\leq0},{}^{\mr{c}}D^{\geq0})$ and
 $({}^{\mr{dc}}D^{\leq0},{}^{\mr{dc}}D^{\geq0})$ define t-structures,
 and do not depend on the choice of the simplicial schemes. These
 t-structures are called the {\em constructible
 t-structure} and {\em dual constructible t-structure}, and abbreviate
 as {\em c-t-structure} and {\em dc-t-structure} respectively. We denote
 the cohomology functor for the c-t-structure (resp.\ dc-t-structure) by
 $\cH^*$ (resp.\ $\dcH^*$), and objects in the heart
 are called {\em c-modules} (resp.\ {\em dc-modules}).
\end{dfn}

\begin{lem}
 \label{exactnessconsbasicfunc}
 (i) We have $\dcH^i\cong\mb{D}\circ
 \cH^{-i}\circ\mb{D}$. In particular, $\ms{M}\in
 D^{\mr{b}}_{\mr{hol}}(\st{X})$ is a c-module if and only if
 $\mb{D}(\ms{M})$ is a dc-module, and a homomorphism
 $f\colon\ms{M}\rightarrow\ms{N}$ of c-modules is c-injective (resp.\
 c-surjective) if and only if $\mb{D}(f)$ is dc-surjective (resp.\
 dc-injective).

 (ii) Let $f\colon\st{X}\rightarrow\st{Y}$ be a smooth morphism. Then
 the functor $f^+$ is c-t-exact.
\end{lem}
\begin{proof}
 For (ii), see Lemma \ref{constproplem}.
 The details are left to the reader.
\end{proof}

\begin{lem}
 \label{dctexactness}
 (i) For an admissible stack $\st{X}$, $L^\omega_{\st{X}}$ is a
 dc-module.

 (ii) Let $f\colon\st{X}\rightarrow\st{Y}$ be a morphism in
 $\cadm$. Then $f_!$ is right dc-t-exact.
\end{lem}
\begin{proof}
 To check (i), it suffices to show that $\mb{D}(L^\omega_{\st{X}})\cong
 L_{\st{X}}$ is a c-module. This follows from Lemma \ref{constproplem}
 (i). Let us check (ii). First we may assume $\st{Y}=:Y$ to be a
 realizable scheme. For a dc-module $\ms{M}$ on $\st{X}$, we need to
 show that $\dcH^if_!(\ms{M})=0$ for $i>0$.
 We use the induction on the support of $\ms{M}$. We may assume
 $\mr{Supp}(\ms{M})=\st{X}$. For an open dense substack
 $j\colon\st{U}\hookrightarrow\st{X}$, it suffices to check that
 $\dcH^if_!(j_!j^+\ms{M})=0$ for $i>0$. Indeed, consider the
 localization triangle \ref{localtriangs}:
 \begin{equation*}
  j_!j^+\ms{M}\rightarrow\ms{M}
   \rightarrow i_+i^+\ms{M}\xrightarrow{+1},
 \end{equation*}
 where $i$ is the closed immersion of the complement of $\st{U}$.
 Since $i_+i^+$ is right dc-t-exact by Lemma \ref{constproplem} and
 $i_+i^+\ms{M}$ is supported on the complement of $\st{U}$,
 we know that $\dcH^if_!(i_+i^+\ms{M})=0$ for
 $i>0$ by the induction hypothesis.
 Thus, if we know the vanishing for $j_!j^+(\ms{M})$, so do we for
 $\ms{M}$. By shrinking $\st{X}$, we can take a finite flat morphism
 $h\colon X\rightarrow\st{X}$ from a realizable scheme. By
 Lemma \ref{directfactorlemmaeasy}, $\ms{M}$ is a direct factor of
 $\mb{D}h_+h^+\mb{D}(\ms{M})\cong h_!h^!(\ms{M})$ (cf.\ Lemma
 \ref{uppershriekproj}). Since $h^!$ is dc-t-exact, it
 remains to prove the right dc-t-exactness of $f\circ h\colon
 X\rightarrow Y$. Since $\mb{D}\circ f_!\circ \mb{D}\cong f_+$ is left
 c-t-exact by Lemma \ref{constproplem}, we get the result.
\end{proof}

\begin{lem}
 \label{glusingtstrumod}
 Let $X_\bullet\rightarrow\st{X}$ be an admissible simplicial
 scheme. Assume given data $\bigl\{\ms{M}_i,\alpha_\phi\bigr\}$ where
 $\ms{M}_i$ is a dc-module on $X_i$, and $\alpha_\phi\colon
 X(\phi)^*(\ms{M}_i)\cong\ms{M}_j$ for $\phi\colon[i]\rightarrow[j]$
 satisfying the cocycle condition. Then there exists a unique dc-module
 $\ms{M}$ on $\st{X}$, the {\em descent}, such that
 $\rho_i^*(\ms{M})\cong\ms{M}_i$. Moreover, given other data
 $\bigl\{\ms{N}_i,\beta_\phi\bigr\}$ and its descent $\ms{N}$ on
 $\st{X}$, homomorphisms $\ms{M}\rightarrow\ms{N}$
 correspond bijectively to systems of homomorphisms
 $\ms{M}_i\rightarrow\ms{N}_i$ compatible in the obvious sense.
 We also have the similar results for c-modules.
\end{lem}
\begin{proof}
 To check this, it suffices to show that
 $\mb{R}^k\mr{Hom}(\ms{M}_i,\ms{M}_i)=0$ for $k<0$ by Theorem
 \ref{BBDgluingthm}. By the definition of dc-t-structure and biduality
 (cf.\ Proposition \ref{bidualityofstack}),
 we may assume that $\ms{M}_i$ is a c-module. In this case, since
 $\mb{R}\mr{Hom}(-,-)$ is left c-t-exact, the claim follows.
\end{proof}

\begin{lem}
 \label{genbasechange}
 Let $f\colon X\rightarrow Y$ be a smooth morphism of realizable
 schemes. Let $\ms{M}'$ be a complex in $D^{\mr{b}}_{\mr{hol}}(Y)$, and
 put $\ms{M}:=f^+(\ms{M}')$. Then there exists an open dense subscheme
 $V\subset Y$ such that $\mr{Supp}(\ms{M}')\cap V$ is dense in
 $\mr{Supp}(\ms{M}')$,
 and for any closed immersion from a point $g\colon\{y\}\rightarrow
 V\hookrightarrow Y$, the base change homomorphism
 $g^+f_+(\ms{M})\rightarrow f'_+g'^+(\ms{M})$ is an isomorphism, where
 $f'\colon X':=X\times_Y\{y\}\rightarrow\{y\}$ and $g'\colon
 X'\rightarrow X$ are the base change of $f$ and $g$.
\end{lem}
\begin{proof}
 We may assume $\ms{M}\in\mr{Hol}(X/L)$.
 By replacing $Y$ by the support of $\ms{M}$, we may
 assume that the support of $\ms{M}$ is equal to $Y$.
 We may assume $Y$ to be reduced, and by shrinking $Y$, we may moreover
 assume that $Y$ is smooth and $\ms{M}'$ is smooth on $Y$. We may
 take $V$ such that each cohomology of $f_+(\ms{M})$ is smooth. Let $c$
 be the codimension of $\{y\}$ in $Y$. In this case, we have
 $g^+f_+(\ms{M})\cong g^!f_+(\ms{M})(c)[2c]$ and $f'_+g'^+(\ms{M})\cong
 f'_+g'^!(\ms{M})(c)[2c]$ by Theorem \ref{purityforrealsch}, and the
 claim follows by \ref{fundproprealsch} \eqref{basechangeprop}.
\end{proof}

\begin{lem}
 \label{exactsequoftstr}
 Let $X_\bullet\rightarrow\st{X}$ be a simplicial realizable scheme
 presentation of a c-admissible stack $\st{X}$. Let $p_i\colon
 X_i\rightarrow\st{X}$ be the induced morphism.
 We put $p^0_{i+}:=\cH^0p_{i+}$ (resp.\
 $p^0_{i+}:=\dcH^0p_{i+}$), and similarly for $p^0_{i!}$.
 For a c-module $\ms{M}$ (resp.\ dc-module $\ms{N}$), we denote by
 $\ms{M}_i$ (resp.\ $\ms{N}_i$) the object $\rho^*_i\ms{M}$ (resp.\
 $\rho^*_i\ms{N}$) on $X_i$. We have the
 following exact sequences of c-modules (resp.\ dc-modules):
 \begin{align*}
  &{
  0\rightarrow\ms{M}\rightarrow p^0_{0+}(\ms{M}_0[-d_0])\rightarrow
  p^0_{1+}(\ms{M}_1[-d_1]),}
  \\
  &\qquad\qquad{
   (\text{resp.\ }
  0\rightarrow\ms{N}\rightarrow p^0_{0+}(\ms{N}_0[-d_0])\rightarrow
  p^0_{1+}(\ms{N}_1[-d_1])\quad),
  }\\
  &{
  p^0_{1!}(\ms{M}_1[d_1](d_1))\rightarrow
  p^0_{0!}(\ms{M}_0[d_0](d_0))\rightarrow\ms{M}\rightarrow0,
  }\\
  &\qquad\qquad{
  (\text{resp.\ }
  p^0_{1!}(\ms{N}_1[d_1](d_1))\rightarrow
  p^0_{0!}(\ms{N}_0[d_0](d_0))\rightarrow\ms{N}\rightarrow0
  ).}
 \end{align*}
\end{lem}
\begin{proof}
 Let us prove the first sequence for $\ms{M}$.
 The sequence is defined by the adjunction. We only need to
 show that it is exact. Thus, we may assume $\st{X}=:X$ to be a
 scheme. First, let us show that there exists an open subscheme
 $j\colon U\hookrightarrow X$ which is dense in the support of $\ms{M}$
 such that $j_+\ms{M}|_U$ satisfies the exactness property.
 We may take $U$ such that $p_0$ and $p_1$ possess the base change
 property by Lemma \ref{genbasechange}. Indeed, it suffices to show the
 exactness after restricting to $U$ since $j_+$
 is left c-t-exact by Lemma \ref{constproplem}. Then, we are
 reduced to the case where $X$ is a point by the base change.
 Then c-t-structure coincides with the usual t-structure, and the
 exactness follows by Proposition \ref{smoothdesre}.

 We show the exactness by the induction on the dimension of the support
 of $\ms{M}$. Take an open dense subscheme $U$ of the support of
 $\ms{M}$ such that the sequence is exact for $j_+\ms{M}|_U$ where
 $j\colon U\hookrightarrow X$. Consider the following
 diagram of c-modules where we omit shifts and twists:
 \begin{equation*}
  \xymatrix@R=15pt{
   &0\ar[d]&0\ar[d]&0\ar[d]&0\ar[d]&\\
   0\ar[r]&\ms{C}\ar[r]\ar[d]&\ms{M}\ar[r]\ar[d]&
    j_+\ms{M}\ar[r]\ar[d]&\ms{C}'\ar[d]\ar[r]&0\\
   0\ar[r]&p^0_{0+}\ms{C}_0\ar[r]\ar[d]&p^0_{0+}\ms{M}_0\ar[r]\ar[d]&
    p^0_{0+}j_+\ms{M}_0\ar[d]\ar@{.>}[r]&p^0_{0+}\ms{C}'_0\ar[d]&\\
   0\ar[r]&p^0_{1+}\ms{C}_1\ar[r]&p^0_{1+}\ms{M}_1\ar[r]&
    p^0_{1+}j_+\ms{M}_1\ar@{.>}[r]&p^0_{1+}\ms{C}'_1.&
   }
 \end{equation*}
 The horizontal sequences are complexes and the ones with solid arrows
 are exact. By induction hypothesis, the vertical sequences are known to
 be exact except for the one starting from $\ms{M}$.
 Then by diagram chasing, we get that the vertical sequence starting
 from $\ms{M}$ is exact as well, and we get the lemma.

 Now, for the exactness of the second sequence for $\ms{N}$, we just
 argue dually. For the second sequence for $\ms{M}$, the argument is
 similar, and even simpler: We may assume $\st{X}$ to be a scheme. We
 can check the exactness by taking the ``stalk'', and reduce to the case
 where $X$ is a point immediately, in which case we get the exactness by
 Proposition \ref{smoothdesre}.
 We can show dually for the first sequence for
 $\ms{N}$, and may finish the proof.
\end{proof}

\begin{thm}
 \label{constoftracemap}
 Let $f\colon\st{X}\rightarrow\st{Y}$ be a morphism in
 $\cadm$. Then there exists a unique homomorphism
 $\mr{Tr}^{\mr{p}}_f\colon f_!L^\omega_{\st{X}}\rightarrow
 L^\omega_{\st{Y}}$ satisfying the following conditions:

 (I) Transitivity: given
 $\st{X}\xrightarrow{f}\st{Y}\xrightarrow{g}\st{Z}$ in
 $\cadm$, the composition of the following homomorphisms is
 equal to $\mr{Tr}^{\mr{p}}_{g\circ f}$:
 \begin{equation*}
  (g\circ f)_!L^\omega_{\st{X}}\cong g_!(f_!L^\omega_{\st{X}})
   \xrightarrow{g_!\mr{Tr}^{\mr{p}}_f}g_!L^\omega_{\st{Y}}
   \xrightarrow{\mr{Tr}^{\mr{p}}_g}K^\omega_{\st{Z}}.
 \end{equation*}

 (II) When $\st{X}=:X$ and $\st{Y}=:Y$ are realizable schemes and $L=K$,
 then $\mr{Tr}^{\mr{p}}_f$ is the adjunction homomorphism
 $f_!K^\omega_{X}\cong f_!f^!K^\omega_Y\rightarrow
 K^\omega_Y$. Moreover, $\mr{Tr}^{\mr{p}}_f$ commutes with
 $\mr{for}_L$.

 (III) The trace map is compatible with smooth pull-back on
 $\st{Y}$. Namely, consider (\ref{cartdiagbch}) in
 $\cadm$ such that $g$ is smooth of relative dimension
 $d$. Then the composition
 \begin{equation*}
  f'_!L_{\st{X}'}^\omega
   \cong
   f'_!g'^*L_{\st{X}}^\omega(d)[d]
   \rightarrow
   g^*f_!L_{\st{X}}^\omega(d)[d]
   \xrightarrow{g^*\mr{Tr}_f^{\mr{p}}}
   g^*L_{\st{Y}}^\omega(d)[d]
   \cong
   L_{\st{Y}'}^\omega,
 \end{equation*}
 where the second map is the base change homomorphism,
 coincides with $\mr{Tr}^{\mr{p}}_{f'}$.
\end{thm}
\begin{proof}
 We put the dc-t-structure on $D^{\mr{b}}_{\mr{hol}}(\st{X})$ and
 $D^{\mr{b}}_{\mr{hol}}(\st{Y})$. Since $f_!$ is right dc-t-exact by
 Lemma \ref{dctexactness}, it suffices to construct a morphism of
 dc-t-modules $f^0_!L^\omega_{\st{X}}\rightarrow L^\omega_{\st{Y}}$
 where $f^0_!:=\dcH^0f_!$.
 When $\st{X}$ and $\st{Y}$ are realizable schemes, the trace map for
 $L=K$ extends uniquely to general $L$ by (II).
 Let us construct in the case where $\st{Y}=:Y$ is a realizable
 scheme. Let $X_0\rightarrow\st{X}$ be a presentation from a
 quasi-projective scheme, $X_1:=X_0\times_{\st{X}}X_0$, $p_0,p_1\colon
 X_1\rightarrow X_0$ be the first and second projection, and put
 $f_i\colon X_i\rightarrow\st{X}\xrightarrow{f}\st{Y}$. By the property
 of adjunction homomorphism, we have the following commutative diagram:
 \begin{equation*}
  \xymatrix@C=60pt@R=3pt{
   f^0_{1!}L^\omega_{X_1}
   \ar@<0.5ex>[dd]^{p_0^*}\ar@<-0.5ex>[dd]_{p_1^*}
   \ar[dr]^-{\mr{Tr}^{\mr{p}}_{f_1}}&\\
   &L^\omega_Y\\
  f^0_{0!}L^\omega_{X_0}\ar[ru]_-{\mr{Tr}^{\mr{p}}_{f_0}}&
   }
 \end{equation*}
 Thus by the second exact sequence for $\ms{N}$ in Lemma
 \ref{exactsequoftstr},
 we have a homomorphism $\mr{Tr}^{\mr{p}}_f\colon
 f_!^0L^\omega_{\st{X}}\rightarrow L^\omega_Y$ as required. By condition
 (I), this map is uniquely determined.
 It is straightforward to check that this map does not depend on the
 choice of the smooth presentation and satisfies (II).

 Finally consider the case where $\st{Y}$ is not a realizable
 scheme. Take an simplicial realizable scheme presentation
 $Y_\bullet\rightarrow\st{Y}$. By Lemma \ref{glusingtstrumod}, it
 suffices to construct a homomorphism
 $(f^0_!L^\omega_{\st{X}})_{Y_i}\rightarrow L^\omega_{Y_i}$
 with compatibility conditions. By condition (III), this map should be
 the one we have already constructed, and we conclude the proof.
\end{proof}

\subsubsection{}
 Let $f\colon\st{X}\rightarrow\st{Y}$ be a {\em proper} morphism between
 admissible stacks. Then we have the homomorphism
 $f_+L^\omega_{\st{X}}\xleftarrow{\sim}f_!L^\omega_{\st{X}}
 \xrightarrow{\mr{Tr}^{\mr{p}}_f}L^\omega_{\st{Y}}$.
 This homomorphism induces $f_+\circ\mb{D}_{\st{X}}\rightarrow
 \mb{D}_{\st{Y}}\circ f_+$ as in \ref{constdualmapproprep}.
 Now, let $f$ be a morphism in $\cadm$. Let
 $\st{X}\xrightarrow{j}\overline{\st{X}}\xrightarrow{\overline{f}}\st{Y}$
 be a compactification of $f$ in $\cadm$. We have the
 homomorphism
 \begin{equation}
  \label{poincardualhom}\tag{$\star$}
  f_+\circ\mb{D}_{\st{X}}\cong \overline{f}_+\circ
  j_+\circ\mb{D}_{\st{X}}
  \xleftarrow{\sim}
  \overline{f}_+\circ\mb{D}_{\st{X}'}\circ j_!
  \rightarrow
  \mb{D}_{\st{Y}}\circ\overline{f}_+\circ j_!\cong
  \mb{D}_{\st{Y}}\circ f_!
 \end{equation}
 where the second isomorphism follows by Remark \ref{bidualityofstack}
 (iii). We may check that this homomorphism does not depend on the
 choice of the factorization up to canonical equivalence.

\begin{thm*}[Duality]
 \label{relativedualprop}
 For any morphism $f$ in $\cadm$, the homomorphism
 (\ref{poincardualhom}) is, in fact, an isomorphism.
\end{thm*}
\begin{proof}
 First, we may assume $\st{Y}$ is a scheme by (III) of Theorem
 \ref{constoftracemap}.
 Let us show the isomorphism for $\ms{M}\in\mr{Hol}(\st{X})$.
 We use the induction on $\dim\mr{Supp}(\ms{M})$. Assume the theorem
 holds for $\dim\mr{Supp}(\ms{M})<k$. Let $\st{Z}$ be the support of
 $\ms{M}$. We may shrink $\st{X}$ so that $\st{Z}$ is shrunk by its open
 dense subscheme. Indeed, let $j\colon \st{U}\hookrightarrow\st{X}$ be
 an open immersion such that $\st{Z}\cap\st{U}$ is dense in $\st{Z}$,
 and $i\colon\st{W}\hookrightarrow\st{X}$ is its complement. The
 proposition holds for $i_+i^+(\ms{M})$ by the induction
 hypothesis. Thus, it suffices to show for $\ms{M}=j_!j^+(\ms{M})$.
 Since the theorem holds for $f=j$, we may replace $\st{X}$ by
 $\st{U}$.

 Shrinking $\st{X}$ by its open dense substack, we may assume that there
 exists a finite flat morphism $g\colon X\rightarrow\st{X}$ from a
 realizable scheme. Since $\ms{M}$ is a direct factor of
 $g_+g^+\ms{M}$, by arguing as in the proof of Proposition
 \ref{propbasechansta}, it suffices to show that the homomorphism
 $f_+\circ\mb{D}_{\st{X}}(g_+g^+\ms{M})\rightarrow \mb{D}_{\st{Y}}\circ
 f_!(g_+g^+\ms{M})$ is an isomorphism. By Lemma \ref{uppershriekproj},
 we are reduced to the realizable scheme case.
\end{proof}

\begin{dfn}
 \label{dfnoflowshrista}
 Let $f\colon\st{X}\rightarrow\st{Y}$ be a morphism in $\cadm$. We
 define $f^!:=\mb{D}_{\st{X}}\circ f^+\circ\mb{D}_{\st{Y}}$. The couple
 $(f_!,f^!)$ is an adjoint pair. Transitivity holds since it holds for
 $f^+$.
\end{dfn}

\begin{lem}
 \label{cohdimcalc}
 Let $f\colon\st{X}\rightarrow\mr{Spec}(k)$ be the structural
 morphism of a c-admissible stack of dimension $d$. Then for
 $\ms{M}\in\mr{Con}(\st{X})$, we have $\H^if_!(\ms{M})=0$ for $i>2d$.
\end{lem}
\begin{proof}
 We may use the induction on the dimension of $\ms{M}$. By standard {\it
 d\'{e}vissage} using the induction hypothesis, we may shrink $\st{X}$
 by its open dense substack. By shrinking $\st{X}$ and taking a finite
 flat morphism from a realizable scheme, we may assume that $X$ is a
 realizable scheme. Then the proposition is reduced to Lemma
 \ref{cohdimlem}.
\end{proof}

\begin{thm}[Relative Poincar\'{e} duality]
 \label{poincaredualrela}
 Let $\mf{M}^{\mr{st}}_d$ be the set of morphisms
 $f\colon\st{X}\rightarrow\st{Y}$
 of $\cadm$ such that there exists an open substack
 $\st{U}\subset\st{Y}$ such that
 $\st{X}\times_{\st{Y}}\st{U}\rightarrow\st{U}$ is flat of relative
 dimension $d$, and the dimension of any fiber of
 $\st{Y}\setminus\st{U}$ is $<d$.
 Then for $f\in\mf{M}^{\mr{st}}_d$ there is a unique trace map
 $\mr{Tr}^{\mr{sm}}_f\colon f_!f^+(d)[2d]\rightarrow\mr{id}$ satisfying
 the following properties:

 (I) When $\st{X}$ and $\st{Y}$ are realizable schemes and $L=K$, then
 it coincides with the trace map in Theorem
 \ref{traceexitstate}. Moreover, it commutes with $\mr{for}_L$.

 (II) It commutes with base change in the sense of (Var 2) of
 \ref{traceexitstate} if we replace the diagram of realizable schemes by
 that in $\cadm$ and $f\in\mf{M}^{\mr{st}}_d$.

 (III) It is transitive with respect to the composition of morphisms in
 $\mf{M}^{\mr{st}}_d$ in the sense of (Var 3) of \ref{traceexitstate}.

 Taking the adjoint, we have a homomorphism $f^+(d)[2d]\rightarrow f^!$,
 which is an isomorphism when $f$ is smooth.
\end{thm}

\begin{rem*}
 The superscript of $\mr{Tr}^{\mr{sm}}$ stands for ``smooth''. This is
 because the trace map is used to show that $f^+(d)[2d]\cong f^!$ for
 {\em smooth} morphisms. On the other hand, that for the trace map
 $\mr{Tr}^{\mr{p}}$ in Theorem \ref{constoftracemap} stands for
 ``proper'', since this trace map is related to the isomorphism
 $f_!\xrightarrow{\sim}f_+$ when $f$ is proper.
 These two trace maps have {\it a priori} no relation.
 Finally, the superscript of $\mf{M}^{\mr{st}}_d$ stands for ``stack''.
\end{rem*}

\begin{proof}
 First, we need to construct the trace map $\mr{Tr}^{\mr{sm}}_f\colon
 f_!f^+L_{\st{Y}}(d)[2d]\rightarrow L_{\st{Y}}$. Put c-t-structure on
 $D^{\mr{b}}_{\mr{hol}}(\st{X})$. By Lemma \ref{cohdimcalc}, it suffices
 to construct a homomorphism
 $\cH^{2d}f_!f^+L_{\st{Y}}(d)\rightarrow L_{\st{Y}}$ of c-modules. When
 $\st{X}$ and $\st{Y}=:Y$ are schemes, this is the trace map of Theorem
 \ref{traceexitstate} when $L=K$, and in general is defined by extending
 the scalar. For careful reader, we remark that, when $\base=F$,
 in {\it ibid.}, we used the category
 $F\text{-}D^{\mr{b}}_{\mr{hol}}(Y/K)$ to define the trace map.
 However, the isomorphism defining the Frobenius structure in
 $D^{\mr{b}}_{\mr{hol}}(Y/K)$ induces an isomorphism in
 $\mr{Con}(Y/K)$, which defines an object in $\mr{Con}(Y/\mf{T}_F)$, so
 {\it ibid.}\ is enough to get a trace map in $D^{\mr{b}}(Y/\mf{T}_F)$.

 For the construction of trace map in the general case, let
 $Y_\bullet\rightarrow\st{Y}$ be an admissible simplicial scheme. By
 Lemma \ref{glusingtstrumod}, it suffices to construct the trace map for
 $\st{X}\times_{\st{Y}}Y_i\rightarrow Y_i$ for each $i$ compatible with
 each other. The construction is similar to that of Theorem
 \ref{constoftracemap} using Lemma \ref{exactsequoftstr}, so we leave
 the details to the reader.

 The trace map defines a morphism $f^+(d)[2d]\rightarrow f^!$. Let
 us show that this is an isomorphism when $f$ is smooth. By base change
 property, we may assume $\st{Y}$ to be a scheme. Moreover, it suffices
 to show the identity after pulling back to schemes which are smooth
 over $\st{X}$. Then we are reduced to the scheme case we have already
 treated in Theorem \ref{Poindual}.
\end{proof}

\subsubsection{}
Finally, we have the projection formula, whose proof is similar to the
proper base change theorem, and left to the reader:

\begin{prop*}
 Let $f\colon\st{X}\rightarrow\st{Y}$ be a morphism in
 $\cadm$. Then for $\ms{M}\in
 D^{\mr{b}}_{\mr{hol}}(\st{X})$ and $\ms{N}\in
 D^{\mr{b}}_{\mr{hol}}(\st{Y})$, we have a canonical isomorphism:
 \begin{equation*}
  f_!\ms{M}\otimes\ms{N}\cong
   f_!\bigl(\ms{M}\otimes f^+\ms{N}\bigr).
 \end{equation*}
\end{prop*}

\subsubsection*{K\"{u}nneth formula}
\begin{prop}
 \label{Kunnethextprod}
 Consider morphisms of admissible stacks
 $f\colon\st{X}\rightarrow\st{X}'$ and
 $g\colon\st{Y}\rightarrow\st{Y}'$. Let $\ms{M}\in
 D^{\mr{b}}_{\mr{hol}}(\st{X})$ and $\ms{N}\in
 D^{\mr{b}}_{\mr{hol}}(\st{Y})$. Then there exists a canonical
 isomorphism
 \begin{equation*}
  f_+(\ms{M})\boxtimes g_+(\ms{N})
   \xrightarrow{\sim}
   (f\times g)_+\bigl(\ms{M}\boxtimes\ms{N}\bigr).
 \end{equation*}
 Moreover, if $f$ and $g$ are in $\cadm$, we get an
 isomorphism
 \begin{equation*}
  f_!(\ms{M})\boxtimes g_!(\ms{N})
   \xrightarrow{\sim}
   (f\times g)_!\bigl(\ms{M}\boxtimes\ms{N}\bigr).  
 \end{equation*}
\end{prop}
\begin{proof}
 Let us construct the first homomorphism. We have the following
 homomorphism
 \begin{equation*}
  (f\times g)^+\bigl(f_+(\ms{M})\boxtimes g_+(\ms{N})\bigr)\cong
   f^+f_+(\ms{M})\boxtimes g^+g_+(\ms{N})\rightarrow
   \ms{M}\boxtimes\ms{N}
 \end{equation*}
 where the first isomorphism follows by Proposition \ref{easypropfunc},
 and the second homomorphism is by adjunction. By taking the adjunction,
 we get the homomorphism we are looking for. To check that this
 homomorphism is an isomorphism, it suffices to treat the finite
 morphism case and projection case separately.
 The finite morphism case follows by Proposition
 \ref{compexttensoth}, and projection case follows by Lemma
 \ref{kunnethforprojstack} and Lemma \ref{externaltensstsh} (i).
 By Theorem \ref{relativedualprop}, we get
 that $f_!\cong\mb{D}_{\st{X}'}\circ f_+\circ\mb{D}_{\st{X}}$. Thus, the
 second isomorphism holds by the first one and the commutativity of
 $\boxtimes$ and $\mb{D}$ by Proposition \ref{easypropfunc}.
\end{proof}

\subsubsection{}
\label{Kunnethformulastform}
Let $f\colon\st{X}\rightarrow\st{Y}$ and
$g\colon\st{X}'\rightarrow\st{Y}$ be morphisms between admissible
stacks. Consider the following cartesian diagram:
\begin{equation*}
 \xymatrix@C=40pt{
  \st{X}\times_{\st{Y}}\st{X}'\ar[r]^{i}\ar[d]_{h}\ar@{}[rd]|\square&
  \st{X}\times\st{X}'\ar[d]^{f\times g}\\
 \st{Y}\ar[r]_-{\Delta_{\st{Y}}}&\st{Y}\times\st{Y}
  }
\end{equation*}
For $\ms{M}\in D^{\mr{b}}_{\mr{hol}}(\st{X})$, $\ms{N}\in
D^{\mr{b}}_{\mr{hol}}(\st{X}')$, we put
\begin{equation*}
 \ms{M}\boxtimes_{\st{Y}}\ms{N}:=
  i^+\bigl(\ms{M}\boxtimes\ms{N}\bigr).
\end{equation*}
When $f$ and $g$ are the identities, $(-)\boxtimes_{\st{Y}}(-)$ is
nothing but $(-)\otimes(-)$.

\begin{cor*}
 Assume $f$ and $g$ are $\cadm$. Then, we have a canonical
 isomorphism
 \begin{equation*}
   f_!(\ms{M})\otimes g_!(\ms{N})\cong
    h_!\bigl(\ms{M}\boxtimes_{\st{Y}}\ms{N}\bigr).
 \end{equation*}
\end{cor*}
\begin{proof}
 Use Proposition \ref{basechforshrik}.
\end{proof}

\subsubsection*{Theory of weights revisited}
\subsubsection{}
\label{theoryofweightadmiss}
The theory of six functors for c-admissible stacks fits perfectly
with the theory of weights. Consider the situation as in
\ref{theoryofweightsetup}. The following is a direct consequence of
\ref{theoweiarst}:

\begin{thm*}
 Let $f\colon\st{X}\rightarrow\st{Y}$ be a morphism between
 admissible stacks.

 (i) Then the functors $f_+$, $f^+$, $\mb{D}$, $\otimes$ preserve
 $\iota$-mixed complexes. Moreover, $f_+$ (resp.\ $f^+$) preserves
 complexes of weight $\geq w$ (resp.\ $\leq w$), $\mb{D}$ exchanges
 complexes of weight $\leq w$ and $\geq-w$, and $\otimes$ sends
 complexes of weight $(\leq w,\leq w')$ to $\leq w+w'$.

 (ii) Assume $f=:j$ is an immersion, and $\ms{M}$ is $\iota$-pure of
 weight $w$ in $\mr{Hol}(\st{X}/L_F)$. Then $j_{!+}(\ms{M})$ is
 $\iota$-pure of weight $w$.
 
\end{thm*}
In particular, by using the duality \ref{relativedualprop}, if $f$ is
proper, $f_+$ sends a complex pure of weight $w$ to a pure complex of
weight $w$.

\subsection{Miscellaneous on the cohomology theory}
\label{miscstsubsec}
Before proceeding to the next section, we pose a little and collect some
miscellaneous results which are used in the proof of the Langlands
correspondence. So far, we have established the theory for admissible
stacks. However, in the proof of the Langlands correspondence, we
sometimes need to deal with non-admissible stacks. For this, we employ
{\it ad hoc} constructions of cohomological operations and prove some
basic properties.
The next theme of this subsection is to show smoothness results using,
again, {\it ad hoc} construction of the nearby cycle functor.
Finally, we collect some properties of Tannakian fundamental groups of
isocrystals.

\subsubsection*{Cohomology theory for algebraic stacks}
\subsubsection{}
\label{descentconstcat}
We denote by $\mr{St}^{\mr{lft}}(k)$ the category of algebraic
stacks locally of finite type over $k$.
Let $\st{X}$ be in $\mr{St}^{\mr{lft}}(k)$. To
$X\in\st{X}_{\mr{sm}}$ (cf.\ \ref{smpresentnotasur}), we associate the
category $\mr{Con}(X/L)$.
For $f\colon X\rightarrow Y$ in
$\st{X}_{\mr{sm}}$, we have the pull-back
functor $f^+\colon\mr{Con}(Y/L)\rightarrow\mr{Con}(X/L)$. This functor
is exact by Lemma \ref{exactnessconsbasicfunc} (ii).
With the isomorphism $(f\circ g)^+\cong g^+\circ f^+$, these data form a
fibered category
$\mr{Con}_{\st{X}/L}\rightarrow\st{X}_{\mr{sm}}$. Now, we have

\begin{lem*}
 When $\st{X}$ is {\em quasi-compact}, the category of c-modules
 $\mr{Con}(\st{X}/L)$ (cf.\ \ref{Artinstackcdctstr}) we have
 defined so far is equivalent to the category of cartesian sections of
 the fibered category $\mr{Con}_{\st{X}/L}$ over $\st{X}_{\mr{sm}}$.
\end{lem*}
\begin{proof}
 We denote by $\Gamma\mr{Con}_{\st{X}/L}$ the category of cartesian
 sections of the fibered category $\mr{Con}_{\st{X}/L}$.
 We construct a functor
 $F\colon\mr{Con}(\st{X}/L)\rightarrow\Gamma\mr{Con}_{\st{X}/L}$.
 Let $\ms{M}\in\mr{Con}(\st{X}/L)$, and $X\in\st{X}_{\mr{sm}}$.
 Let $\rho\colon X\rightarrow\st{X}$ be the smooth
 morphism, then since $X$ and $\st{X}$ are both \stack in the sense of
 2nd read case, the functor $\rho^+$ (which is isomorphic to
 $\rho^*[-d_{X/\st{X}}]$ by Lemma \ref{starplusequal}) is defined, and
 put $F(\ms{M})_X:=\rho^+(\ms{M})\in\mr{Con}(X/L)$.
 Defining the gluing isomorphism by the transitivity of
 pull-backs, we have $F(\ms{M})\in\Gamma\mr{Con}_{\st{X}/L}$.

 Let $\algsp{X}_\bullet\rightarrow\st{X}$ be a simplicial algebraic
 space presentation.
 By associating the category $\mr{Con}(\algsp{X}_i/L)$ to $\algsp{X}_i$
 and considering the pull-back, we have the cofibered category
 $\mr{Con}(\algsp{X}_\bullet/L)_{\bullet}$ over $\Delta^+$ similarly to
 Definition \ref{dfncofibcatrel}.
 By the construction of $F$, there exists the canonical functor
 $\mr{Con}(\st{X}/L)\rightarrow
 (\mr{Con}(\algsp{X}_\bullet/L)_{\bullet})_{\mr{tot}}$. This induces an
 equivalence of categories: we can check the fully faithfulness by using
 (\ref{BDspecseq}), and it is essentially surjective by Theorem
 \ref{BBDgluingthm}.

 Let us construct the functor $G\colon\Gamma\mr{Con}_{\st{X}/L}
 \rightarrow\mr{Con}(\st{X}/L)$.
 First, let $\st{X}=:\algsp{X}$ be an algebraic space. Take
 $X\in\algsp{X}_{\mr{sm}}$, and
 $X_\bullet:=\mr{cosk}_0(X\rightarrow\algsp{X})$. Then since
 $X_n\in\algsp{X}_{\mr{sm}}$, we have the ``restriction'' functor
 $G\colon\Gamma\mr{Con}_{\algsp{X}/L}\rightarrow
 (\mr{Con}(X_\bullet/L)_{\bullet})_{\mr{tot}}\cong
 \mr{Con}(\algsp{X}/L)$. This does not depend on the auxiliary
 choices, and it is straightforward to check that this is quasi-inverse
 to $F$, thus the lemma is shown when $\st{X}$ is an algebraic space.
 Now, for a smooth morphism $f\colon\st{X}\rightarrow \st{Y}$ of
 algebraic stacks, we have the faithful functor
 $\st{X}_{\mr{sm}}\rightarrow\st{Y}_{\mr{sm}}$ sending
 $X\rightarrow\st{X}$ to the composition
 $X\rightarrow\st{X}\xrightarrow{f}\st{Y}$. This induces a functor
 $f^+\colon\Gamma\mr{Con}_{\st{Y}/L}\rightarrow
 \Gamma\mr{Con}_{\st{X}/L}$. Let $\rho\colon\algsp{X}\rightarrow\st{X}$
 be a smooth morphism. Then we have the functor $\Gamma
 \mr{Con}_{\st{X}/L}\xrightarrow{\rho^+}\Gamma\mr{Con}_{\algsp{X}/L}
 \cong\mr{Con}(\algsp{X}/L)$.
 Take a simplicial presentation
 $\algsp{X}_\bullet$ of $\st{X}$, then $\algsp{X}_n$ is an algebraic
 space so we argue as the case where $\st{X}$ is an algebraic space to
 construct the quasi-inverse $G$ of $F$.
\end{proof}

This lemma enables us to define $\mr{Con}(\st{X}/L)$, for $\st{X}$ in
$\mr{St}^{\mr{lft}}(k)$ not necessarily quasi-compact, to
be the category of cartesian sections of $\mr{Con}_{\st{X}/L}$.
We often denote $\mr{Con}(\st{X}/L)$ by $\mr{Con}(\st{X})$ for
simplicity. For a smooth morphism
$f\colon\st{X}\rightarrow\st{Y}$ in $\mr{St}^{\mr{lft}}(k)$, we
have a faithful functor $\st{X}_{\mr{sm}}\rightarrow\st{Y}_{\mr{sm}}$,
which induces the pull-back functor
$f^+\colon\mr{Con}(\st{Y}/L)\rightarrow\mr{Con}(\st{X}/L)$. Assume
$\st{X}$ and $\st{Y}$ are of finite type, and let $d$ be the relative
dimension of $f$. Then $f^+$ is canonically equivalent to $f^*[-d]$
using the functor in \ref{smoothpullback}. For
$\ms{M}\in\mr{Con}(\st{Y}/L)$, we sometimes denote
$f^+\ms{M}\in\mr{Con}(\st{X}/L)$ by $\ms{M}_{\st{X}}$.

\subsubsection{}
\label{representpushfor}
Let $f\colon\st{X}\rightarrow\st{Y}$ be a {\em representable morphism of
finite type} in $\mr{St}^{\mr{lft}}(k)$. Let us
construct $f^i_+$ and $f^i_!$. To construct these, let
$Y\in\st{Y}_{\mr{sm}}$. This defines the morphisms
$f_Y\colon\st{X}_Y:=\st{X}\times_{\st{Y}}Y\rightarrow Y$, and
$\rho\colon\st{X}_Y\rightarrow\st{X}$. Note that
$\st{X}_Y$ is an algebraic space by assumption, thus $f_Y$ is
compactifiable by Remark \ref{defofcadmcat}.
Now, for $\ms{M}\in\mr{Con}(\st{X})$, put
$\bigl(f^i_+\ms{M}\bigr)_Y:=\cH^if_{Y+}\rho^*(\ms{M})$, and similarly for
$\bigl(f^i_!\ms{M}\bigr)_Y$.
For a smooth morphism $\phi\colon Y'\rightarrow Y$ in
$\st{Y}_{\mr{sm}}$, the base change theorem \ref{smoothbcthmst} and
Proposition \ref{basechforshrik} give us the isomorphisms
\begin{equation*}
  \phi^+\bigl(f^i_+\ms{M}\bigr)_Y\cong
   \bigl(f^i_+\ms{M}\bigr)_{Y'},\qquad
   \phi^+\bigl(f^i_!\ms{M}\bigr)_Y\cong
   \bigl(f^i_!\ms{M}\bigr)_{Y'}.
\end{equation*}
We can check the transitivity of these isomorphisms easily, and define
objects $\bigl\{(f^i_+\ms{M})_Y\bigr\}_{Y\in\st{Y}_{\mr{sm}}}$ and
$\bigl\{(f^i_!\ms{M})_Y\bigr\}_{Y\in\st{Y}_{\mr{sm}}}$ in
$\mr{Con}(\st{Y})$. We denote them by $f^i_+\ms{M}$ and $f^i_!\ms{M}$
respectively. When $f$ is an immersion, $f^0_!$ is c-t-exact and
$f^i_!=0$ for $i\neq0$ by Lemma \ref{constproplem}.
Moreover, $f^0_!$ coincides with that induced by
\ref{openimmdefshri}. In this case, we sometimes denote $f^0_!$ by
$f_!$.

\subsubsection{}
\label{gluinglemmaarti}
Let $\mr{St}^{\mr{lft,sm}}(k)$ be the subcategory of
$\mr{St}^{\mr{lft}}(k)$ such that the objects are the same, and
for morphisms, we only consider smooth morphisms.
By associating $\mr{Con}(\st{X})$ to
$\st{X}\in\mr{St}^{\mr{lft,sm}}(k)$, and considering the
pull-back $f^*$ for smooth morphisms, we have the fibered category
$\mr{Con}\rightarrow\mr{St}^{\mr{lft,sm}}(k)$.
Then we have the following descent result:

\begin{lem*}
 Smooth surjective representable morphisms are universal effective
 descent morphisms in the fibered category
 $\mr{Con}\rightarrow\mr{St}^{\mr{lft,sm}}(k)$.
\end{lem*}
\begin{proof}
 Let $\mr{St}^{\mr{lft,sm,rep}}(k)$ be the subcategory of
 $\mr{St}^{\mr{lft,sm}}(k)$ such that the objects are the same and
 we only consider representable ones for morphisms between algebraic
 stacks. We may consider $\mr{Con}$ as a fibered category over
 $\mr{St}^{\mr{lft,sm,rep}}(k)$.
 We already proved in the proof of Lemma \ref{descentconstcat} that a
 smooth surjective morphism {\em from an algebraic space} is an effective
 descent morphism, thus it is a universal effective descent morphism in
 $\mr{St}^{\mr{lft,sm,rep}}(k)$.
 Let $\st{Y}\rightarrow\st{X}$ be a smooth surjective representable
 morphism. Take a smooth representable morphism
 $\algsp{X}\rightarrow\st{X}$ from an algebraic space, then we know that
 $\algsp{X}\rightarrow\st{X}$ and
 $\st{Y}\times_{\st{X}}\algsp{X}\rightarrow\algsp{X}$ are universal
 effective descent.
 Universal effective descent morphisms form a topology by
 \cite[6.23]{Gi}, $\st{Y}\rightarrow\st{X}$ is universal effective
 descent by {\it caract\`{e}re local} (cf.\ [SGA 4, II, 1.1])
 of Grothendieck topology.
\end{proof}

\subsubsection{}
\label{defpullbackconst}
We may extend the pull-back functor to arbitrary morphism between
algebraic stacks. Let $f\colon\st{X}\rightarrow\st{Y}$ be a morphism
in $\mr{St}^{\mr{lft}}(k)$, and take $Y\in\st{Y}_{\mr{sm}}$. Put
$f'\colon\st{X}':=\st{X}\times_{\st{Y}}Y\rightarrow Y$. Now, let
$X'\in\st{X}'_{\mr{sm}}$. Then we have the morphism $f'_{X'}\colon
X'\rightarrow Y$. We put
\begin{equation*}
 f'^+(\ms{M}_Y)_{X'}:=f'^+_{X'}(\ms{M}_Y).
\end{equation*}
We can check easily that the collection of these modules
satisfies the compatibility condition, and defines a cartesian section
of the fibered category
$\mr{Con}_{\st{X}'}$, which we define to be the module $f'^+(\ms{M}_Y)$
in $\mr{Con}(\st{X}')$.
These modules yield a descent data with respect to the representable
smooth morphism $\st{X}'\rightarrow\st{X}$. By using Lemma
\ref{gluinglemmaarti}, we get $f^{+}(\ms{M})\in\mr{Con}(\st{X})$.
The pull-back is exact by Lemma \ref{exactnessconsbasicfunc} (ii), and
satisfies the transitivity property: for morphisms
$\st{X}\xrightarrow{f}\st{Y}\xrightarrow{g}\st{Z}$ in
$\mr{St}^{\mr{lft}}$, we have a canonical equivalence $(g\circ f)^+\cong
f^+\circ g^+$.

Finally, let $\st{C}^{(\prime)}\rightarrow\st{D}$ be a morphism in
$\mr{St}^{\mr{lft}}(k)$ between smooth stacks. Let $\ms{M}^{(\prime)}$
be in $\mr{Sm}(\st{C}^{(\prime)})$. Let
$\Delta\colon\st{C}\times_{\st{D}}\st{C}'\rightarrow
\st{C}\times\st{C}'$ be the canonical morphism. We define
$\ms{M}\boxtimes_{\st{D}}\ms{M}':=\Delta^{+}(\ms{M}\boxtimes\ms{M}')$
for $\ms{M}^{(\prime)}\in\mr{Con}(\st{C}^{(\prime)})$.

\begin{lem}
 \label{Artinstackpushfor}
 (i) If $f\colon\st{X}\rightarrow\st{Y}$ is a representable morphism in
 $\cadm$, then $\cH^if_!\cong f^i_!$.

 (ii) Consider the following cartesian diagram in
 $\mr{St}^{\mr{lft}}(k)$:
 \begin{equation*}
  \xymatrix{
   \st{X}'\ar[r]^-{g'}\ar[d]_{f'}\ar@{}[rd]|\square&
   \st{X}\ar[d]^{f}\\
  \st{C}'\ar[r]_-{g}&\st{C}.
   }
 \end{equation*}
 Assume $f$ is representable. Then there exists a canonical
 isomorphism $g^+\circ f^i_!\cong f'^i_!\circ g'^+$.
\end{lem}
\begin{proof}
 The first claim follows from the definition, and the verification for
 (ii) is easy from the base change \ref{basechforshrik}.
\end{proof}

\subsubsection{}
\label{invariantclassflrdgr}
A morphism in $\mr{St}^{\mr{lft}}(k)$ is
said to be {\em gerb-like} if locally with respect to fppf-topology the
morphism can be written as the canonical morphism $B(G/\algsp{X})(=:BG)
\rightarrow\algsp{X}$  (cf.\ \cite[9.6]{LM}) for some flat group space
$G$ of finite presentation over $\algsp{X}$. Recall that such morphisms
are smooth by \cite[5.1.3, 5.1.5]{Behr}.

\begin{lem*}
 Let $\algsp{X}$ be a algebraic space of finite type, and $G$
 be a flat algebraic group space finite radicial surjective over
 $\algsp{X}$. Let $\rho\colon BG\rightarrow \algsp{X}$ be the canonical
 morphism. Note that $\rho$ is proper and is in $\cadm$.
 Then $\rho^+$ and $\rho_!\cong\H^0\rho_!$ induce the equivalence of
 categories between $\mr{Con}(\algsp{X})$ and $\mr{Con}(BG)$.
\end{lem*}
\begin{proof}
 Note that $\rho$ is a proper morphism between admissible stacks since
 $G$ is finite, thus $\rho\in\cadm$.
 We have the following commutative diagram
  \begin{equation*}
   \xymatrix{
    \algsp{X}\ar[r]_-{u}\ar@/^3ex/[rr]^{\mr{id}}&
    BG\ar[r]_-{\rho}&\algsp{X}
    }
  \end{equation*}
 where $u$ is the universal $G$-torsor, which is a
 universal homeomorphism by assumption. Thus we can use Lemma
 \ref{directfactorlemmaeasy} (i) to conclude.
\end{proof}

\begin{cor*}
 Let $f\colon\st{X}\rightarrow\st{Y}$ be a gerb-like morphism in
 $\mr{St}^{\mr{lft}}(k)$ whose structural group is flat
 finite and radicial. Then $f^+$ induces an equivalence of categories
 $\mr{Con}(\st{X})\cong\mr{Con}(\st{Y})$. Moreover, when $f\in\cadm$,
 $\cH^if_!=0$ for $i\neq0$, and $\cH^0f_!$ can be taken as a
 quasi-inverse to $f^+$.
\end{cor*}
\begin{proof}
 Since the structural group is flat, there exists a smooth surjective
 morphism from an algebraic space
 $P\colon\algsp{Y}_0\rightarrow\st{Y}$ such that
 $f_0\colon\st{X}_0:=\st{X}\times_\st{Y}\algsp{Y}_0\rightarrow\algsp{Y}_0$
 is a neutral gerb by Lemma \ref{neutralgerbsmpres}. Let
 $\algsp{Y}_1:=\algsp{Y}_0\times_{\st{Y}}\algsp{Y}_0$, 
 $\algsp{Y}_2:=\algsp{Y}_0\times_{\st{Y}}\algsp{Y}_0
 \times_{\st{Y}}\algsp{Y}_0$,
 and $f_i\colon\st{X}_i:=\algsp{Y}_i\times_{\st{Y}}\st{X}
 \rightarrow\algsp{Y}_i$ be the projection.
 We have the following diagram:
 \begin{equation*}
  \xymatrix{
   \mr{Con}(\st{Y})\ar[r]^{P^+}\ar[d]_{f^+}&
   \mr{Con}(\algsp{Y}_0)\ar@<0.5ex>[r]\ar@<-0.5ex>[r]
   \ar[d]^{f_0^+}&
   \mr{Con}(\algsp{Y}_1)\ar[d]^{f_1^+}
   \ar[r]\ar@<0.7ex>[r]\ar@<-0.7ex>[r]&
   \mr{Con}(\algsp{Y}_2)\ar[d]^{f_2^+}\\
  \mr{Con}(\st{X})\ar[r]^{P^+}&
   \mr{Con}(\st{X}_0)\ar@<0.5ex>[r]\ar@<-0.5ex>[r]&
   \mr{Con}(\st{X}_1)\ar[r]\ar@<0.7ex>[r]\ar@<-0.7ex>[r]&
   \mr{Con}(\st{X}_2)
   }
 \end{equation*}
 By the assumption on the structural group, $f_0^+$, $f_1^+$, $f_2^+$
 are equivalence of categories by Lemma
 \ref{invariantclassflrdgr}. Since $P$ is a presentation, we may use
 Lemma \ref{gluinglemmaarti} to conclude.
\end{proof}

\begin{lem}
 \label{connectinvcllsp}
 Let $\algsp{X}$ be an algebraic space of finite type over $k$
 such that $\algsp{X}_{\mr{red}}$ is smooth,
 and $G$ be a smooth fiberwise connected algebraic group over
 $\algsp{X}$. Then the pull-back by the structural morphism induces
 $\mr{Sm}(\algsp{X})\xrightarrow{\sim}\mr{Sm}(BG)$.
\end{lem}
\begin{proof}
 We may replace $\algsp{X}$ by $\algsp{X}_{\mr{red}}$ since the derived
 category do not change, and we may assume that $\algsp{X}$ is smooth.
 The canonical morphism $\algsp{X}\rightarrow BG$ is a smooth
 presentation and $\algsp{X}\times_{BG}\algsp{X}\cong G$ such that the
 $i$-th projection
 $p_i\colon\algsp{X}\times_{BG}\mc{X}\rightarrow BG$ is the
 structural morphism $p\colon G\rightarrow\algsp{X}$ by
 \cite[4.6.1]{LM}. By Lemma \ref{gluinglemmaarti}, taking an object of
 $\mr{Sm}(BG)$ is equivalent to taking
 $\mc{E}\in\mr{Sm}(\algsp{X})$ endowed with an isomorphism
 $\alpha\colon p^+\mc{E}\xrightarrow{\sim}p^+\mc{E}$ satisfying the
 cocycle condition. Let
 $\mc{K}:=\mr{Ker}\bigl(\alpha-\mr{id}\colon p^+\mc{E}\rightarrow
 p^+{\mc{E}}\bigr)$. Since $p^+\mc{E}\in\mr{Hol}(G)$ is smooth, $\mc{K}$
 is smooth as well. Let $e\colon\algsp{X}\rightarrow G$ be the unit
 morphism. Since $p^+\mc{E}$ is smooth, we get the exact sequence
 \begin{equation*}
  0\rightarrow e^+(\mc{K})\rightarrow e^+p^+\mc{E}
   \xrightarrow{e^+(\alpha-\mr{id})}e^+p^+\mc{E}.
 \end{equation*}
 By the cocycle condition, $e^+(\alpha-\mr{id})$ is $0$, and thus the
 rank of $\mc{K}$ is equal to that of $p^+\mc{E}$ since $G$ is
 connected. Thus $\alpha$ is the identity, and we get the lemma.
\end{proof}

\begin{rem*}
 The assumption that $\algsp{X}_{\mr{red}}$ is smooth is made only for
 the simplicity. In fact, with a little more argument on {\it
 d\'{e}vissage}, the lemma remains true even if we replace
 $\mr{Sm}$ by $\mr{Con}$, but we do not need this much.
\end{rem*}

\begin{lem}
 \label{diagconngeblike}
 Let $f\colon\st{X}\rightarrow\st{Y}$ be a diagonally connected
 gerb-like morphism (cf.\ \cite[5.1.3]{Behr}).
 Shrinking $\st{Y}$ by its open dense substack if necessary, the
 functor $f^+\colon\mr{Sm}(\st{Y})\rightarrow\mr{Sm}(\st{X})$
 induces an equivalence.
\end{lem}
\begin{proof}
 Take a presentation $P\colon Y\rightarrow\st{Y}$ from a scheme. There
 exists an open dense subscheme $U\subset Y$ such that $U_{\mr{red}}$ is
 smooth. By replacing $\st{Y}$ by $P(U)\subset\st{Y}$, we may assume
 that $Y_{\mr{red}}$ is smooth. Now, by using smooth descent, we may
 assume that $\st{Y}=:Y$ is a scheme and $\st{X}=BG$ with a
 connected flat algebraic group space $G$ over $Y$. Since the category
 is stable under universal homeomorphism, we may replace $Y$ by
 $Y_{\mr{red}}$, and assume that $BG$ and $Y$ are smooth.

 When $G$ is smooth, the lemma follows from Lemma
 \ref{connectinvcllsp}. In the general case, we use the argument of
 \cite[5.1.17]{Behr}. Take the relative Frobenius $G\twoheadrightarrow
 G'\hookrightarrow G^{(p)}$. Then $\mr{Ker}(G\rightarrow G')$ is of
 height $\leq1$, so this is flat finite radicial by definition (cf.\
 [SGA 3, $\text{VII}_{\text{A}}$, 4.1.3]).
 By Corollary \ref{invariantclassflrdgr},
 $\mr{Sm}(BG')\xrightarrow{\sim}\mr{Sm}(BG)$, so we may
 replace $G$ by $G'$. Repeating this, we come down to the case where $G$
 is smooth over a dense open subscheme of $Y$ by [SGA 3,
 $\text{VII}_{\text{A}}$, 8.3].
 Thus, by shrinking $Y$, we are reduced to the case where $G$ is
 smooth.
\end{proof}

\subsubsection*{A smoothness criterion}
\subsubsection{}
In the proof of the Langlands correspondence, we need smoothness of
certain holonomic modules. For this, we need to use the functors of
Beilinson. Let $k'$ be a finite extension of $k$, and put
$\mb{A}^1_{k'}(=\mb{A}^1):=\mr{Spec}(k'[x])$.
Let $f\colon\st{X}\rightarrow\mb{A}^1_{k'}$ be a morphism from a
c-admissible stack. We put $i_f\colon\st{Z}_f:=
f^{-1}(0)\hookrightarrow\st{X}$, and
$j_f\colon\st{U}_f:=\st{X}\setminus\st{Z}_f\hookrightarrow\st{X}$.
Then for any $\ms{M}\in\mr{Hol}(\st{U}_f)$ and integers $a\leq b$,
holonomic modules
$\Pi^{a,b}_{!+}(\ms{M}),\Phi^{\mr{un}}_f(\ms{M})\in
\mr{Hol}(\st{Z}_f)$ are defined in \cite[\S2]{AC2} using a technique of
Beilinson. To clarify $f$, we denote it by $\Pi^{a,b}_{f}(\ms{M})$.
Put $\Pi^{0,0}_f:=\Psi^{\mr{un}}_f$, the {\em unipotent nearby cycle
functor}. Explicitly, we can compute using the notation of
\cite[2.5]{AC2} that
\begin{equation}
 \label{interpnearbycyc}
 \Psi^{\mr{un}}_f(\ms{M})\cong
  \indlim_s\mr{Ker}\bigl(j_{f!}(\ms{M}^{-s,0})\rightarrow
  j_{f+}(\ms{M}^{-s,0})\bigr)
  \cong
  \indlim_s\H^{-1}i^{+}_fj_{f+}(\ms{M}^{-s,0}).
\end{equation}

\subsubsection{}
Let us recall the local theory very briefly. See \cite[2.1]{AM}
for more detailed review and references of the theory.
Let $\mathbf{l}$ be a complete discrete valuation field over $k$.
Then the Robba ring (with coefficients in
$K\otimes_{W(k)}W(\mr{res}(\mathbf{l}))$) denoted by
$\mc{R}_{\mathbf{l}}$ is defined. When $\mathbf{l}=k\cc{x}$, then
\begin{equation*}
 \mc{R}_{\mathbf{l}}=\Biggl\{f=\sum_{n\in\mb{Z}}a_nx^n
  \in K\dd{x,x^{-1}}
  \,\Bigg\arrowvert\,
  \parbox{15em}{There exists $0\leq\varepsilon<1$ such that $f$
  converges on $\varepsilon<|x|<1$}
  \Biggr\},
\end{equation*}
where $K\dd{x,x^{-1}}$ is the {\em $K$-vector space} of formal series.
The Robba ring is endowed with derivation. A differential module over
$\mathbf{l}$ is a finite free $\mc{R}_{\mathbf{l}}$-module
endowed with connection. We have the notion of {\em solvable
differential module} which is an analogue of overconvergent isocrystal
for differential modules, whose category is denoted by
$\mr{Sol}(\mc{R}_{\mathbf{l}})$.
We have the $s$-th Frobenius endomorphism of
$\mc{R}_{\mathbf{l}}$, and the pull-back defines a functor
$F^*\colon\mr{Sol}(\mc{R}_{\mathbf{l}})\rightarrow
\mr{Sol}(\mc{R}_{\mathbf{l}})$, which is known to be an equivalence of
categories. Thus, we may apply the construction of \ref{extofsclarsub}.
The category $F\mbox{-}\mr{Sol}(\mc{R}_{\mathbf{l}})_L$ is denoted by
$\mr{Hol}(\mathbf{l}/\mf{T}_F)$. As in \ref{revDmodtheory}, we denote by
$\mr{Hol}(\mathbf{l}/\mf{T}_{\emptyset})$ the
thick full subcategory of the category of differential modules over
$\mathbf{l}$ generated by differential modules which can be endowed with
$s'$-th Frobenius structure for some positive integer $s'$ divisible by
$s$ (but we do not consider Frobenius structure).
Since we only use $\mr{Hol}(\mathbf{l}/\mf{T}_{\emptyset})$ in the
following, we denote this simply by $\mr{Hol}(\mathbf{l})$.
We call the objects of
$\mr{Hol}(\mathbf{l})$ {\em holonomic modules on $\mathbf{l}$}.
For a separable finite extension $\mathbf{l}'/\mathbf{l}$,
we are able to define the {\em push-forward functor}
$\mr{Hol}(\mathbf{l}')\rightarrow
\mr{Hol}(\mathbf{l})$ (cf.\ \cite[at the end of 2.1.4]{AM}).

\subsubsection{}
\label{locmonodromytheminter}
For a Galois extension $\mathbf{l}/k'\cc{x}$,
we define $\Psi_{\mathbf{l},f}(\ms{M})$ as
follows: Let us denote by $\mc{L}_{\mathbf{l}}$ the holonomic module on
$k'\cc{x}$ defined by taking the push-forward of the trivial module on
$\mathbf{l}$ along the extension $\mathbf{l}/k'\cc{x}$.
Let $\ms{L}_{\mathbf{l}}$ be the canonical extension of
$\mc{L}_{\mathbf{l}}$ on $\mb{G}_{\mr{m},k'}$ in the sense of Crew and
Matsuda (cf.\ \cite[2.1.9]{AM}). Then put
$\Psi_{\mathbf{l},f}(\ms{M}):=\Psi^{\mr{un}}_{f}(\ms{M}\otimes
f^+\ms{L}_{\mathbf{l}})$. We remark that $\ms{M}\otimes
f^+\ms{L}_{\mathbf{l}}$, which is {\it a priori} defined in
$D^{\mr{b}}_{\mr{hol}}(\st{X})$, is in $\mr{Hol}(\st{X})$.
Indeed, it suffices to check this when $\st{X}=:X$ is a realizable
scheme. In this case, we may take a closed immersion
$i\colon X\hookrightarrow P$ to a smooth scheme. By shrinking $X$, we
may assume that there exists $g\colon P\rightarrow\mb{A}^1$ such that
$g\circ i=f$. Now, by the projection formula $i_+(\ms{M}\otimes
f^+\ms{L}_{\mathbf{l}})\cong i_+(\ms{M})\otimes g^+\ms{L}_{\mathbf{l}}$,
and the latter is in $\mr{Hol}(P)$. Since we have the action of
$\mr{Gal}(\mathbf{l}/k'\cc{x})$ on $\ms{L}_{\mathbf{l}}$, it induces the
{\em Galois action} on $\Psi_{\mathbf{l},f}$.

\begin{lem*}
 Let $f\colon C\rightarrow\mb{A}^1_{k'}$ be an \'{e}tale morphism from a
 curve such that $f^{-1}(0)=\{s\}$ and $k(s)=k'$.
 Let $\ms{M}\in\mr{Hol}(C)$. Assume that $\ms{M}$ is smooth outside of
 $s$. If $\Phi^{\mr{un}}_{f}(\ms{M})=0$
 and the actions of $\mr{Gal}(\mathbf{l}/k'\cc{x})$ and monodromy
 operator on $\Psi_{\mathbf{l},f}(\ms{M})$ are trivial for any
 $\mathbf{l}$, then $\ms{M}$ is smooth.
\end{lem*}
\begin{proof}
 Since $\Phi^{\mr{un}}_f(\ms{M})=0$, we get that
 $i^+_f(\ms{M})[-1]\cong\Psi^{\mr{un}}_f(\ms{M})$. Since the rank of
 $i^!_f(\ms{M})$ and $i^+_f(\ms{M})$ are the same, it suffices to show
 that the rank of $\Psi^{\mr{un}}_f(\ms{M})$ is equal to that of
 $\ms{M}$ by \cite[4.1.4]{AM}.
 By (\ref{interpnearbycyc}) and \cite[1.5.9
 (iii)]{AC}, $\Psi_{\mathbf{l},f}$ depends only on the differential
 module on the Robba ring around $s$ defined by restricting $\ms{M}$,
 and we can compute $\Psi_{\mathbf{l},f}(\ms{M})$ by using the local
 monodromy theorem. Since the argument is standard, we leave the details
 to the reader.
\end{proof}

\begin{lem}
 \label{propercommupspi}
 Let $f\colon\st{X}\xrightarrow{h}\st{Y}\xrightarrow{g}\mb{A}^1$ be
 morphisms between c-admissible stacks. Assume that $h$ is proper.
 Then we have
 \begin{equation*}
  \Pi^{a,b}_{g}(\H^ih_+\ms{M})\cong
   \H^ih_+\Pi^{a,b}_{f}(\ms{M}).
 \end{equation*}
 The same isomorphism holds if we replace $\Pi^{a,b}_{\star}$ by
 $\Psi_{\mathbf{l},\star}$ or $\Phi^{\mr{un}}_{\star}$.
\end{lem}
\begin{proof}
 Since $h$ is proper, we have $h_+j_{f\star}\cong j_{g\star}h_+$ for
 $\star\in\{!,+\}$. Thus, by projection formula, we have $\twolim\,
 j_{g\star}(\H^ih_+\ms{M})^{\bullet,\bullet}\cong
 \twolim\,\H^ih_+ j_{f\star}(\ms{M}^{\bullet,\bullet})$ where
 $\star\in\{!,+\}$. Thus, by construction, we get the commutativity for
 $\Pi^{a,b}$. Now, let us define $K_f(\ms{M}):=\mr{Ker}\bigl(
 \Xi_f(\ms{M})\oplus\ms{M}\twoheadrightarrow j_{f+}(\ms{M})\bigr)$,
 where $\Xi_f:=\Pi^{0,1}_f$.
 Since $\Xi_f$ and $j_{f+}$ commute with $\H^ih_+$, we get that
 $K_g(\H^ih_+(\ms{M}))\cong\H^ih_+K_f(\ms{M})$. Similarly,
 $I_f(\ms{M}):=\mr{Im}\bigl(j_{f!}(\ms{M})\rightarrow\Xi_f(\ms{M})
 \oplus\ms{M}\bigr)\,(\cong j_{f!}(\ms{M}))$ commutes with $\H^ih_+$ as
 well, and the lemma for $\Phi^{\mr{un}}:=K_f/I_f$
 follows by definition.
\end{proof}

\begin{lem}[{\cite[A.9 (i)]{Laf}}]
 \label{propsmoobcspeci}
 Let $p_{\st{X}}\colon\st{X}\rightarrow S$ be a proper morphism from a
 c-admissible stack to a smooth scheme, and
 $\mr{Res}\colon\st{X}\rightarrow\st{C}$ be a morphism to an algebraic
 stack locally of finite type over $k$.
 Assume that $(p_{\st{X}},\mr{Res})\colon\st{X}\rightarrow
 S\times\st{C}$ is smooth. Then for any $\ms{M}\in\mr{Con}(\st{C})$, the
 complex $p_{\st{X}+}\mr{Res}^+(\ms{M})$ is smooth.
\end{lem}
\begin{proof}
 We put $\ms{H}^i_{\st{X}}:=\H^ip_{\st{X}+}\mr{Res}^+(\ms{M})$.
 Assume given a smooth morphism $f\colon S\rightarrow\mb{A}^1$.
 Put $g:=f\circ p_{\st{X}}$. Since $f$ and $(p_{\st{X}},\mr{Res})$ are
 smooth, we get that $\Phi^{\mr{un}}_{g}(\mr{Res}^+(\ms{M}))=0$ and the
 Galois and monodromy action on
 $\Psi_{\mathbf{l},g}(\mr{Res}^+(\ms{M}))$ are trivial. 
 Now, since $p_{\st{X}}$ is assumed proper,
 $\Phi^{\mr{un}}_{f}(\H_{\st{X}}^i)=0$ and the Galois and monodromy
 action on $\Psi_{\mathbf{l},f}(\ms{H}^i_{\st{X}})$ are trivial by Lemma
 \ref{propercommupspi}. This, in particular, implies that
 $j_{f!+}(\ms{H}^i_{\st{X}})\cong\ms{H}^i_{\st{X}}$. Moreover,
 when $S$ is a curve, the lemma holds. Indeed, take an open subscheme
 $U\subset S$ such that $\ms{H}^i_{\st{X}}$ is smooth on $U$. We may
 replace $S$ by $S\otimes_kk'$, and may assume that
 $S\setminus U$ is $k'$-rational. Since the verification is local, we
 may assume that we are given an \'{e}tale morphism $f\colon
 S\rightarrow\mb{A}^1_{k'}$ such that $f^{-1}(0)=S\setminus U$ consists
 of one point, and then apply Lemma \ref{locmonodromytheminter}.

 Let us treat the general case. By {\it d\'{e}vissage}, we may assume
 that $\ms{M}$ has a Frobenius structure. Let $c\colon C\hookrightarrow
 S$ be an immersion from a smooth curve $C$.
 By base change and
 purity, we have $c^!(\ms{H}_{\st{X}}^i)\cong\H^i_{\st{X}\times_SC}[-r]$
 where $r$ is the codimension of $C$ is $S$. By the curve case we have
 already treated, this is smooth.
 Let $U\subset S$ be an open dense subscheme on which
 $p_{\st{X}+}\mr{Res}^+(\ms{M})$ is smooth.
 By Shiho's cut
 by curve theorem \cite[Thm 0.1]{Sh}\footnote{
 In {\it ibid.}, there is an assumption that $k$ is
 uncountable. However, Shiho pointed out to the author that this
 assumption is not needed if $\mc{E}$ in [{\it ibid.}, Thm 0.1] is
 endowed with Frobenius structure. Indeed, let us use the notation of
 the proof of the theorem in [{\it ibid.}, \S2.3]. By using [{\it
 ibid.}, Thm 2.5] instead of Thm 2.10, the slope of $E_{\mc{E},L}$
 is $0$. Since we have a Frobenius structure, the exponents are in
 $\mb{Q}$, and we can use [{\it ibid.}, Prop 1.20] to conclude.
 } and \cite[5.2.1]{KeSS1},
 the smooth module $\ms{H}^i_{\st{X}}|_{U}$ on
 $U$ can be extended to a smooth module on $S$. Since we showed that
 $j_{f!+}(\ms{H}_{\st{X}}^i)\cong\ms{H}_{\st{X}}^i$ for any $f$,
 we get the lemma.
\end{proof}

\subsubsection*{Isocrystals and its Tannakian fundamental group}
\subsubsection{}
\label{algcloscoefftheory}
Let us introduce $\overline{\mb{Q}}_p$-coefficient cohomology
theory.
From now on, by saying a base tuple, we also allow $L$ to be an
algebraic extension of $K$ which may {\em not} be finite, in the
definition in \ref{5tuplesover}.
In the definition of arithmetic tuple, $\sigma\colon L\rightarrow L$
should moreover satisfy the following:
\begin{quote}
 the automorphism $\sigma$ is an extension of a lifting of $s$-th
 Frobenius automorphism on $k$ to $K$,
 and there exists a sequence of finite extensions $M_n$ of $K$
 in $L$ such that $\sigma(M_n)\subset M_n$ and $\bigcup_nM_n=L$.
\end{quote}
We use the 2-inductive limit method of
Deligne \cite[1.1.3]{De} to construct the $L$-theory. For an algebraic
stack $\st{X}$ (resp.\ scheme $X$) of finite type  over $k$, we
define
\begin{gather*}
 D^{\mr{b}}(\st{X}/L_{\base}):=
  2\text{-}\indlim_{M\supset K}D^{\mr{b}}(\st{X}/M_{\base}),\qquad
  \mr{Sm}(\st{X}/L_{\base}):=
  2\text{-}\indlim_{M\supset K}\mr{Sm}(\st{X}/M_{\base}),\\
  \mr{Isoc}^\dag(X/L_{\base}):=
  2\text{-}\indlim_{M\supset K}\mr{Isoc}^\dag(X/M_{\base}),
\end{gather*}
where $M=M_n$ in the case of $\base=F$.
By taking the limits, the results we get in this paper can be
generalized automatically to these categories since cohomological
operators we have defined so far commute with $\iota_L$ by
\ref{commuiotaderiv}.
Further detail is left to the reader.
Let $f$ be the structural morphism of $\st{X}$. We denote by
$L_{\st{X}}:=f^+(L)$ as usual.

\subsubsection{}
\label{relbetourrig}
We have often used the category $\mr{Sm}(X/L)$, but the
category of overconvergent isocrystals $\mr{Isoc}^\dag(X/L)$ is
more standard in the literature. Let us clarify the relation between
these categories.
Consider the situation of \ref{baseicgeneralsetup}, or
\ref{algcloscoefftheory} if $L/K$ is not finite. Let $X$ be a
smooth scheme separated of finite type of dimension $d$ over $k$. Recall
the functor $\mr{sp}_+$ in \ref{fundproprealsch}
\eqref{carodagdagcat}. By extending the scalar and gluing, we have the
following functor
\begin{equation*}
 \widetilde{\mr{sp}}_+:=\mr{sp}_+(-d)[-d]\colon
  \mr{Isoc}^\dag(X/L)
  \xrightarrow{\sim}\mr{Sm}(X/L)
  \subset D^{\mr{b}}_{\mr{hol}}(X/L).
\end{equation*}
For a morphism $f\colon X\rightarrow Y$ between smooth schemes separated
of finite type, let $d:=\dim(X)-\dim(Y)$. Then, we have a canonical
equivalence $\mr{sp}_+\circ f^*\cong(f^![-d])\circ\mr{sp}_+$ compatible
with the composition of morphisms between smooth schemes by
\cite[6.1.9]{Cafai}. Thus, by
Theorem \ref{Poindual} and Theorem \ref{purityforrealsch}, we
have a canonical equivalence $\widetilde{\mr{sp}}_+\circ f^*\cong
f^+\circ\widetilde{\mr{sp}}_+$. Via this equivalence, we identify $f^*$
and $f^+$. We also know that $\widetilde{\mr{sp}}_+$ commutes with
tensor products by \cite[3.3.5]{Catens}. By taking the adjoint, the
commutation of $\shom$ follows as well.
Finally, let $f\colon X\rightarrow\mr{Spec}(k)$ be the structural
morphism of a {\em smooth realizable} scheme, and
$M\in\mr{Isoc}^\dag(X/L_\emptyset)$
(resp.\ $M\in\mr{Isoc}^\dag(X/L_F)$).
When $X$ is a scheme which has a compactification
$\overline{X}$ such that $\overline{X}$ possesses a smooth lifting over
$R$ and that $\overline{X}\setminus X$ is a divisor, then by
\cite[5.9]{A}, we have canonical isomorphisms
\begin{equation*}
 H_{\mr{rig}}^*(X,M)\cong
  \H^*f_+\bigl(\widetilde{\mr{sp}}_+(M)\bigr),
  \qquad
  H_{\mr{rig},\mr{c}}^*(X,M)\cong
  \H^*f_!\bigl(\widetilde{\mr{sp}}_+(M)\bigr)
\end{equation*}
as objects in $\mr{Vec}_L$ (resp.\ $F\text{-}\mr{Vec}_L$).
Here, $H_{\mr{rig}}$ and $H_{\mr{rig},\mr{c}}$ denotes the rigid
cohomology extended to $L$-coefficients in the obvious manner.

\begin{prob}
 \label{comparisonisocdmod}
 Unify the rigid cohomology theory into the framework of arithmetic
 $\ms{D}$-modules. Namely, let $X$ be a separated scheme. Define the
 category of smooth objects $\mr{Sm}(X/L)$ in $\mr{Con}(X/L)$, and
 establish an equivalence of categories
 $\mr{Isoc}^\dag(X/L)\rightarrow\mr{Sm}(X/L)$.
 This equivalence should coincide with the one in
 \ref{relbetourrig} when $X$ is smooth.
 Finally, compare the rigid cohomology and the push-forward in the sense
 of $\ms{D}$-modules in the style of \ref{relbetourrig}.
\end{prob}

\subsubsection{}
In this paragraph, we fix an algebraic closure $\overline{K}$ of $K$. We
denote by $\overline{k}$ the residue field of $\overline{K}$, which is
an algebraic closure of $k$.
Now, let $X$ be a smooth scheme of finite type over $k$, and assume it
to be {\em geometrically connected}.
Take a geometric point $\overline{x}\in X(\overline{k})$. Let $x$ be the
closed point of $X$ defined by $\overline{x}$, and denote by
$i_x\colon x\hookrightarrow X$ the closed immersion.
We denote by $K_x$ be the unramified extension of $K$ induced by the
finite extension $k(x)$ of $k$. Then we have the fiber functor
\begin{equation*}
 \omega_x\colon\mr{Isoc}^\dag(X/K)\xrightarrow{i^+_x}
  \mr{Isoc}^\dag(k(x)/K)\cong\mr{Vec}_{K_x}.
\end{equation*}
Let $L$ be a finite extension of $K_x$.
Since $\mr{End}(K_X)\cong K$, by
\cite[3.10.1]{Mi}, $\omega_x$ induces the fiber functor
\begin{equation*}
 \omega_{x/L}\colon\mr{Isoc}^\dag(X/L)\rightarrow
  \mr{Vec}_{L},
\end{equation*}
by sending $\mc{E}$ to $i^+_x(\mc{E})\otimes_{i^+_xL_X}L$.
This fiber functor $\omega_{x/L}$ is compatible with extension of
scalar, and we may take the 2-inductive limit to define $\omega_{x/L}$
for any algebraic extension $L$ of $K_x$. Now, the geometric point
$\overline{x}$ determines the embedding
$K_x\hookrightarrow\overline{K}$ with which we may regard $\overline{K}$
as an extension of $K_x$, thus, the fiber functor
$\omega_{x/\overline{K}}$ makes sense. This
fiber functor is denoted by $\omega_{\overline{x}}$.

Let $\pi_1^{\mr{isoc}}(X,\overline{x})$ be the
{\em isocrystal fundamental group}
$\mr{Aut}^{\otimes}(\omega_{\overline{x}})$, which is an affine
group scheme over $\overline{K}$. For an algebraic group $G$ over
$\overline{K}$, denote by $\mr{Rep}_{\overline{K}}(G)$ the category of
finite dimensional representation of $G$.
By \cite[3.11]{Mi}, and taking the 2-inductive limit, we have the
following equivalence of tensor categories
\begin{equation*}
 \mr{Isoc}^\dag(X/\overline{K})\xrightarrow{\sim}
  \mr{Rep}_{\overline{K}}(\pi_1^{\mr{isoc}}(X,\overline{x})).
\end{equation*}

\begin{rem*}
 If $X\rightarrow\mr{Spec}(k)$ is not geometrically connected, then
 $\mr{Isoc}^\dag(X/\overline{K})$ is {\em not} a Tannakian
 category over $\overline{K}$. Indeed, $\overline{K}_X$ is an unit
 object of the tensor category $\mr{Isoc}^\dag(X/\overline{K})$ but we
 have $\mr{End}(\overline{K}_{X})\cong\overline{K}^{\times
 c}$ where $c$ is the number of connected components of
 $X\otimes_k\overline{k}$. Compare also with Lemma
 \ref{behaviorofbasech} (i).
\end{rem*}

\subsubsection{}
\label{defofisocweilgroup}
From now on, till the end of this subsection, we consider the case where
$k$ is a finite field with $q=p^s$ elements. We fix an arithmetic base
tuple
$\mf{T}_F:=(k,R:=W(k),K:=\mr{Frac}(R),\overline{\mb{Q}}_p,s,\mr{id})$,
where $\overline{\mb{Q}}_p$ is an algebraic closure of $K$.
Let $\mf{T}_{\emptyset}$ be the associated geometric base tuple.
As in the last paragraph, $\overline{k}$ denotes the residue field of
$\overline{\mb{Q}}_p$, which naturally contains $k$.
To make the notations compatible with \cite{Laf}, we denote
the relative $s$-th Frobenius endomorphism on $X$ by $\mr{Frob}_X$
instead of $F_X$. Let $X$ be a geometrically connected smooth scheme of
finite type over $k$. Take a geometric point $\overline{x}\in
X(\overline{k})$, and $i_x\colon x\hookrightarrow X$ denotes the induced
closed immersion.
Let $\mc{E}\in\mr{Isoc}^\dag(X/\overline{\mb{Q}}_{p,\emptyset})$.
Since $K_x\cong W(k(x))\otimes_{W(k)}K$, the $s$-th Frobenius
automorphism on $W(k(x))$
induces an automorphism $\mr{Frob}_x^*\colon K_x\rightarrow K_x$.
The fiber $i_x^+\mc{E}$ can be seen as an
$i^+_x\overline{\mb{Q}}_{p,X}$-module, where the latter
ring is isomorphic to $K_x\otimes_K\overline{\mb{Q}}_p$
since $X$ is geometrically connected. Thus, we have isomorphisms
\begin{equation*}
 \omega_{\overline{x}}(\mr{Frob}_X^*\mc{E})\cong
  \bigl(K_x\otimes_{\mr{Frob}_x^*\nwarrow K_x}
  i^+_x\mc{E}\bigr)
  \otimes_{K_x\otimes_K\overline{\mb{Q}}_p}\overline{\mb{Q}}_p
  \xleftarrow[\alpha]{\sim}
  i^+_x\mc{E}\otimes_{K_x\otimes_K\overline{\mb{Q}}_p}
  \overline{\mb{Q}}_p
  \cong
  \omega_{\overline{x}}(\mc{E}),
\end{equation*}
where the homomorphism $\alpha$ sends $e\otimes a$ to $1\otimes e\otimes
a$. Thus, we get the following 2-commutative diagram:
\begin{equation}
 \label{TannakiandualityFisoc}
 \xymatrix@C=70pt@R=2pt{
  \mr{Isoc}^\dag(X/\overline{\mb{Q}}_{p,\emptyset})
  \ar[dd]_{\mr{Frob}_X^*}\ar[dr]^-{\omega_{\overline{x}}}&\\
 &\mr{Vec}_{\overline{\mb{Q}}_p}.\\
 \mr{Isoc}^\dag(X/\overline{\mb{Q}}_{p,\emptyset})
  \ar[ur]_-{\omega_{\overline{x}}}&
  }
\end{equation}
This diagram induces a homomorphism
$\mr{Frob}_X^*\colon\pi^{\mr{isoc}}_1(X,\overline{x})
\rightarrow\pi^{\mr{isoc}}_1(X,\overline{x})$. This
homomorphism is in fact an isomorphism, since $\mr{Frob}_X^*$ gives an
equivalence of categories by Remark \ref{fundproprealsch}. We define
$\rho\colon\mb{Z}\rightarrow\mr{Aut}(\pi_1^{\mr{isoc}}(X,\overline{x}))$
to be the homomorphism sending $1$ to $\mr{Frob}_X^*$. Using this
homomorphism, we put
$\WI(X,\overline{x}):=\pi^{\mr{isoc}}_1(X,\overline{x})\rtimes\mb{Z}$,
and call it the {\em isocrystal Weil group of $X$}. By construction,
we have the equivalence of tensor categories
\begin{equation}
 \label{equivcattannfrob}
 \mr{Isoc}^\dag(X/\overline{\mb{Q}}_{p,F})\xrightarrow{\sim}
  \mr{Rep}_{\overline{\mb{Q}}_p}(\WI(X,\overline{x})).
\end{equation}
induced by $\omega_{\overline{x}}$.

In general, let $X\rightarrow\mr{Spec}(k)$ be a smooth connected scheme
of finite type, which may not be geometrically connected, and take a
closed point $x$. The structural morphism factors as
$X\rightarrow\mr{Spec}(k')\rightarrow\mr{Spec}(k)$ where $k'$ is a
finite field extension of $k$ of degree $d$ and $X$ is geometrically
connected over $k'$. Consider the base tuple
$\mf{T}'_F:=(k',R':=W(k')\otimes_{W(k)}R,
K',\overline{\mb{Q}}_p,ds,\mr{id})$. Then
we define $\WI(X,\overline{x})$ to be the isocrystal Weil group of $X$
over $\mf{T}'_F$. Despite the base tuple being changed, by Corollary
\ref{behaviorofbasech},
the equivalence (\ref{equivcattannfrob}) remains true. We
note that, by definition, $\WI(X,\overline{x})$ does not depend on the
choice of the base field $k$.

Assume that $X$ is geometrically connected, and let $k'$ be a Galois
extension of $k$. Take a geometric point $\overline{x}'$ of
$X\otimes_kk'$, and let $\overline{x}$ be the projection to $X$. Then we
have the following exact sequence:
\begin{equation*}
 1\rightarrow\WI(X\otimes_kk',\overline{x}')\rightarrow
  \WI(X,\overline{x})\rightarrow\mr{Gal}(k'/k)\rightarrow1.
\end{equation*}

\begin{rem*}
 Let $X$ be a geometrically connected smooth scheme of finite type over
 $k$. Assume moreover that we have a $k$-rational point
 $i_x\colon\mr{Spec}(k)\rightarrow X$ for simplicity.
 Since $\mr{Isoc}^\dag(X/\overline{\mb{Q}}_{p,F})$ is a neutral
 Tannakian category over $\overline{\mb{Q}}_p$ by using the fiber
 functor $\omega:=i_x^+$, we could have used
 $\mr{Aut}^{\otimes}(\omega)$ as the fundamental group. However, this
 algebraic group is complicated to handle, and we used the simpler
 substitute $\WI$ following Crew \cite{Cr}.
\end{rem*}

\subsubsection{}
\label{surjectwisweilgroup}
Let $X'$, $X''$ be smooth schemes of finite type and geometrically
connected over $k$.
Put $X:=X'\times X''$ which is geometrically connected over $k$ as
well. Let $U\subset X$ be an open subscheme such that
$(\mr{Frob}_{X'}\times\mr{id}_{X''})(U)\subset U$ where
$\mr{Frob}_{X'}\times\mr{id}_{X''}\colon X'\times
X''\rightarrow X'\times X''$. Take a geometric point
$\overline{x}\in U(\overline{k})$.
Arguing as in \ref{defofisocweilgroup}, the pull-back
$(\mr{Frob}_{X'}\times\mr{id})^+$ induces an outer automorphism of
$\WI(U,\overline{x})$, and yields a homomorphism
$\mb{Z}\rightarrow\mr{Out}(\WI(U,\overline{x}))$ sending
$1$ to $(\mr{Frob}_{X'}\times\mr{id})^+$.
We put $\mb{Z}\WI(U,\overline{x}):=\WI(U,\overline{x})\rtimes\mb{Z}$.
Representations of $\mb{Z}\WI(U,\overline{x})$ correspond to pairs
$(\mc{E},\alpha)$ where
$\mc{E}\in\mr{Isoc}^\dag(U/\overline{\mb{Q}}_{p,F})$
and $\alpha\colon(\mr{Frob}_{X'}\times\mr{id})^+(\mc{E})\cong\mc{E}$.

\begin{lem*}
 [{\cite[VI.13]{Laf}}]
 Take geometric points $\overline{x}'$ and $\overline{x}''$ of $X'$
 and $X''$, and put $\overline{x}:=(\overline{x}',\overline{x}'')$.
 Then the canonical homomorphism
 $\mb{Z}\WI(X,\overline{x})\rightarrow\WI(X',\overline{x}')
 \times\WI(X'',\overline{x}'')$ is surjective (or more precisely,
 faithfully flat).
\end{lem*}
\begin{proof}
 Let $x'$ and $x''$ be the closed points of $X'$ and $X''$ induced by
 $\overline{x}'$ and $\overline{x}''$. Let $k'$ be a Galois extension of
 $k$, and put $G:=\mr{Gal}(k'/k)$. Consider the following diagram, where
 we omit the base points of the Weil groups and $\WI$ is abbreviated
 as $W$:
 \begin{equation*}
  \xymatrix{
   1\ar[r]&
   W(X'\otimes{k'})\times W(X''\otimes{k'})\ar[r]\ar[d]&
   W(X')\times W(X'')\ar[r]\ar[d]&
   G\times G\ar[r]\ar@{=}[d]&
   1\\
  1\ar[r]&
   \mb{Z}W(X\otimes{k'})\ar[r]&
   \mb{Z}W(X)\ar[r]&
   G\times G\ar[r]&
   1.
   }
 \end{equation*}
 Thus, we may replace $k$ by $k'$, and may assume that $x'$ and $x''$
 are rational points of $X'$ and $X''$.
 These rational points define morphisms $s'\colon X'\rightarrow X$,
 $s''\colon X''\rightarrow X$. Let us show that the canonical
 homomorphism $\alpha\colon\pi^{\mr{isoc}}_1(X,\overline{x})
 \rightarrow\pi^{\mr{isoc}}_1(X',\overline{x}')\times
 \pi^{\mr{isoc}}_1(X'',\overline{x}'')$ is surjective. 
 To check this, it suffices to show that for any $\overline{K}$-algebra
 $A$, the homomorphism of groups
 $\alpha(A)\colon\pi^{\mr{isoc}}_1(X,\overline{x})(A)\rightarrow
 \pi^{\mr{isoc}}_1(X',\overline{x}')(A)\times
 \pi^{\mr{isoc}}_1(X'',\overline{x}'')(A)$ is surjective.
 Now, the morphism $s'$ induces the homomorphism
 $s'_*\colon\pi^{\mr{isoc}}_1(X',\overline{x}')
 \rightarrow\pi^{\mr{isoc}}_1(X,\overline{x})$, and the image
 of $(\alpha\circ s'_*)(A)$ is
 $\pi^{\mr{isoc}}_1(X',\overline{x}')(A)\times\{1\}$. Using $s''$, the
 image of $\alpha(A)$ contains
 $\{1\}\times\pi^{\mr{isoc}}_1(X'',\overline{x}'')(A)$ as well, and the
 surjectivity of $\alpha(A)$ follows as required.
 Finally, by definition of the Weil groups, the lemma follows.
\end{proof}

\begin{lem}
 \label{surjofisocpione}
 Let $X$ be a smooth connected scheme, and $U\subset X$ be an open
 subscheme. Take a geometric point $x\in U(\overline{k})$.
 Then the homomorphism
 $\WI(U,\overline{x})\rightarrow\WI(X,\overline{x})$ is surjective.
\end{lem}
\begin{proof}
 It suffices to show that the homomorphism
 $\pi^{\mr{isoc}}_1(U,\overline{x})\rightarrow
 \pi^{\mr{isoc}}_1(X,\overline{x})$ induced by the open immersion is
 surjective. Let $j\colon U\hookrightarrow X$ be the open
 immersion. By \cite[2.21]{Mi}, this is equivalent to showing
 that $j^+$ is fully faithful and any subobject of $j^+\mc{E}$ for an
 overconvergent isocrystal $\mc{E}$ on $X$ is in the image of $j^+$.
 The fully faithfulness follows by purity (cf.\ Theorem
 \ref{purityforrealsch}).
 It remains to show that if $\mc{E}$ is an irreducible overconvergent
 isocrystal on $X$, then $\mc{E}|_U$ is irreducible. This follows by
 \cite[1.4.6]{AC}.
\end{proof}

\section{Cycle classes, correspondences, and $\ell$-independence}
\label{lindepsec}
The aim of this section is to prove an $\ell$-independence result. This
is a key tool to compute the trace of the action of Hecke algebra on the
cohomology of certain moduli spaces.
In this section, we fix $\base\in\{\emptyset,F\}$ and a base tuple as
usual. The algebraic extension $L/K$ can be infinite as in
\ref{algcloscoefftheory}. For simplicity, smooth admissible stacks over
$k$ are assumed equidimensional.

\subsection{Generalized cycles and correspondences}
\label{Genecyclcorres}
\subsubsection{}
Let $p\colon\st{X}\rightarrow\mr{Spec}(k)$ be the structural
morphism of a c-admissible stack $\st{X}$ (cf.\ Definition
\ref{defcompactcadmmorph}). If no confusion may arise, we
denote $L_{\st{X},\base}$ by $L$. For $\ms{M}\in
D^{\mr{b}}_{\mr{hol}}(\st{X})$, we put
\begin{align*}
 H^i(\st{X},\ms{M})&:=\mr{Hom}_{D(\mr{Spec}(k)/L_{\base})}
  \bigl(L,p_+(\ms{M})[i]\bigr),\\
  H^i_{\mr{c}}(\st{X},\ms{M})&:=
  \mr{Hom}_{D(\mr{Spec}(k)/L_{\base})}
  \bigl(L,p_!(\ms{M})[i]\bigr).
\end{align*}
Note that when $\base=\emptyset$,
we have $H^*(\st{X},\ms{M})\cong\H^*p_+(\ms{M})$ and
$H_{\mr{c}}^*(\st{X},\ms{M})\cong\H^*p_!(\ms{M})$ as vector spaces over
$L$. For a morphism $i\colon\st{Z}\rightarrow\st{X}$, we define the {\em
local cohomology} to be
\begin{equation*}
 H^i_{\st{Z}}(\st{X},\ms{M}):=\mr{Hom}_{D(\mr{Spec}(k)/L_{\base})}
  \bigl(L,p_+i_+i^!(\ms{M})[i]\bigr).
\end{equation*}
Furthermore, we put
$H^*_{\heartsuit}(\st{X}):=H^*_{\heartsuit}(\st{X},L_{\st{X}})$ where
$\heartsuit\in\{\emptyset,\mr{c},\st{Z}\}$.

Consider the following commutative diagram of c-admissible stacks:
\begin{equation*}
 \xymatrix{
  \st{Z}\ar[r]^-{f'}\ar[d]_{i'}&
  \st{W}\ar[d]^{i}\\
 \st{X}\ar[r]_-{f}&\st{Y}.
  }
\end{equation*}
If this diagram is cartesian, then we have the base change isomorphism
$i^!f_+\cong f'_+i'^!$ by (the dual of) Proposition
\ref{basechforshrik}, which induces the homomorphism
$H^*_{\st{W}}(\st{Y})\rightarrow H^*_{\st{Z}}(\st{X})$. By abuse of
notation, we also denote this homomorphism by $f^*$.
If the diagram is merely commutative, $f$ is the identity, and $f'$ is
proper, the adjunction
$f'_+f'^!\rightarrow\mr{id}$ induces the push-forward homomorphism
$H^*_{\st{Z}}(\st{X})\rightarrow H^*_{\st{W}}(\st{X})$.

Finally, given a {\em proper} morphism of c-admissible stacks
$f\colon\st{X}\rightarrow\st{Y}$, we have a homomorphism
$f^*\colon H_{\mr{c}}^*(\st{Y})\rightarrow H_{\mr{c}}^*(\st{X})$ induced
by the adjunction homomorphism.

\subsubsection{}
\label{cupproddef}
First, let $\st{S}$ be a c-admissible stack. Let $\ms{M}$, $\ms{N}$ be
objects in $D^{\mr{b}}_{\mr{hol}}(\st{S}/L)$. Denote by
$\Delta\colon\st{S}\rightarrow\st{S}\times\st{S}$ the diagonal
morphism, and by $p$ the structural morphism of $\st{S}$.
By identifying $\mr{Spec}(k)\times\st{S}$ and $\st{S}$, we have a
canonical homomorphism
\begin{equation*}
  p_+(\ms{M})\boxtimes\ms{N}\cong
  (p\times \mr{id})_+\bigl(\ms{M}\boxtimes\ms{N}\bigr)
  \xrightarrow{\mr{adj}}
  (p\times\mr{id})_+\Delta_+\Delta^+
  \bigl(\ms{M}\boxtimes\ms{N}\bigr)
  \cong
  \ms{M}\otimes\ms{N},
\end{equation*}
where the first isomorphism is induced by Proposition
\ref{Kunnethextprod}, $\mr{adj}$ is the adjunction homomorphism, and
the last isomorphism follows since
$(p\times\mr{id})\circ\Delta=\mr{id}$.

Now, let $f\colon\st{X}\rightarrow\st{S}$ be a morphism of c-admissible
stacks. We take a factorization
$\st{X}\xrightarrow{j}\overline{\st{X}}\xrightarrow{\overline{f}}\st{S}$
in $\cadm$ such that $j$ is an open immersion and $\overline{f}$ is
proper. We have
\begin{align*}
 f_+(\ms{M})\otimes f_!(\ms{N})&\cong
  \overline{f}_+j_+(\ms{M})\otimes\overline{f}_+j_!(\ms{N})
  \rightarrow\overline{f}_+
  \bigl(j_+(\ms{M})\otimes j_!(\ms{N})\bigr)\\
  &\xleftarrow{\sim}
  \overline{f}_+ j_!(\ms{M}\otimes\ms{N})
  \cong
  f_!(\ms{M}\otimes\ms{N}),
\end{align*}
where the second homomorphism is (\ref{pushtenscomhom}).
This homomorphism does not depend on the choice of compactification.

Finally, let $i\colon\st{Z}\rightarrow\st{X}$ be a morphism of
c-admissible stacks. Then we have the homomorphism
$i_!\bigl(i^!\ms{M}\otimes i^+\ms{N}\bigr)\cong
i_!i^!\ms{M}\otimes\ms{N}\rightarrow\ms{M}\otimes\ms{N}$. Taking the
adjoint, we get a homomorphism
\begin{equation*}
 i^!\ms{M}\otimes i^+\ms{N}\rightarrow i^!(\ms{M}\otimes\ms{N}).
\end{equation*}

\begin{dfn}
 \label{dfngencyccorres}
 (i) Let $\st{X}$ be a c-admissible stack of dimension
 $d$. A {\em generalized cycle of codimension $c$} is a
 proper morphism $g\colon\Gamma\rightarrow\st{X}$
 between c-admissible stacks such that $\dim(\st{X})-\dim(\Gamma)=c$.

 (ii) Let $S$ be a scheme of finite type over $k$, $\varphi$ be
 a {\em proper} endomorphism of $S$, and
 $f^{(\prime)}\colon\st{X}^{(\prime)}\rightarrow S$ be a c-admissible
 $S$-stack. Let
 \begin{equation*}
  c_{\Gamma}\colon\Gamma\rightarrow\st{X}\times_{\varphi,S}\st{X}'
 \end{equation*}
 be a morphism between c-admissible stacks,
 where the fiber product is taken for $\varphi\circ
 f\colon\st{X}\rightarrow S$ and $f'\colon\st{X}'\rightarrow S$. For $i=1,2$, put
 $p_i:=\pi_i\circ c_{\Gamma}$ where $\pi_i$ denotes the $i$-th
 projection. The morphism $c_\Gamma$ is said to be a
 {\em correspondence over $\varphi$} if
 $\Gamma$ is of equidimensional of dimension $\dim(\st{X})$ and
 $p_2$ is proper, or $\Gamma$ is the empty stack.
 We sometimes denote the correspondence by
 $\Gamma\colon\st{X}\rightsquigarrow\st{X}'$. Note that $c_\Gamma$ is a
 generalized cycle of codimension $\dim(\st{X}')$ of
 $\st{X}\times_{\varphi,S}\st{X}'$. From now on in this subsection, we
 fix $S$ and $\varphi$ as above, and we use them freely without
 referring to this paragraph.
\end{dfn}

\subsubsection{}
\label{tracemapmeanscons}
In this section, we denote $\mr{Tr}^{\mr{sm}}_f$ in Theorem
\ref{poincaredualrela} simply by $\mr{Tr}_f$.
Let $\alpha\colon\st{X}\rightarrow\mr{Spec}(k)$ be the
structural morphism of a c-admissible stack $\st{X}$. We often denote
$\mr{Tr}_\alpha$ by $\mr{Tr}_{\st{X}}$. Now, let
$f\colon\st{Y}\rightarrow\st{X}$ be a morphism of c-admissible stacks,
and {\em assume that $\st{X}$ is smooth}.
We construct $\mr{fTr}_f\colon f_!f^+L_{\st{X}}(d)[2d]\rightarrow
L_{\st{X}}$ where $d:=\dim(\st{Y})-\dim(\st{X})$, which is called the
{\em fake trace map of $f$}, as follows.
Let $p\colon\st{Y}\rightarrow\mr{Spec}(k)$ be the structural
morphism. We have the isomorphisms
\begin{align*}
 \mr{Hom}(f_!f^+L_{\st{X}}(d)[2d],L_{\st{X}})&\cong
  \mr{Hom}(f^+L_{\st{X}}(d)[2d],f^!L_{\st{X}})\\
 &\cong
  \mr{Hom}(p^+L(d_Y)[2d_Y],p^!L)\cong
  \mr{Hom}(p_!p^+L(d_Y)[2d_Y],L)
\end{align*}
where $d_Y:=\dim(\st{Y})$, and we used the Poincar\'{e} duality
\ref{poincaredualrela} for the second isomorphism. The trace map
$\mr{Tr}_{\st{Y}}$ is an element on the right hand side of the
isomorphisms. We define $\mr{fTr}_f$ to be the homomorphism defined
by sending this $\mr{Tr}_{\st{Y}}$ to the left hand side of the
isomorphisms. This homomorphism induces a homomorphism
\begin{equation}
 \label{pushhombasesm}
 f_*\colon H^{*+2d}_{\mr{c}}(\st{Y})(d)\rightarrow
  H^*_{\mr{c}}(\st{X}).
\end{equation}

(i) Let $i\colon\st{Z}\rightarrow\st{X}$ be a generalized cycle of
codimension $c$ on a smooth c-admissible stack $\st{X}$. Then by taking
the adjoint, $\mr{fTr}_i$ induces a homomorphism
\begin{equation*}
 c_{\st{Z}}\colon i^+L_X\rightarrow i^!L_X(c)[2c].
\end{equation*}

(ii) Let us construct a similar homomorphism when we are given a
correspondence. We use the notation of Definition \ref{dfngencyccorres}
(ii). We assume further that $\st{X}$ is smooth. When $\Gamma$ is
non-empty, we have
$\mr{fTr}_{p_1}\colon p_{1!}p_1^+L_{\st{X}}\rightarrow L_{\st{X}}$,
where we used the assumption that $\dim(\Gamma)=\dim(\st{X})$. Thus, we
have
\begin{equation*}
 \iota_{\Gamma}\colon
  p_2^+L_{\st{X}'}\cong p_1^+L_{\st{X}}\rightarrow p_1^!L_{\st{X}}
\end{equation*}
where we used the adjoint of $\mr{fTr}_{p_1}$ for the second
homomorphism. When $\Gamma$ is empty, we simply put $\iota_\Gamma:=0$.

\subsubsection{}
\label{sga4halfcharact}
Let us characterize the fake trace map in the style of [SGA
$4\frac{1}{2}$, Cycle]. Let us consider the following diagram of
c-admissible stacks
\begin{equation*}
 \xymatrix@C=40pt@R=20pt{
  \st{Y}
  \ar[rr]^-{i}\ar[rd]_-{f}&&
  \st{X}
  \ar[ld]^-{g}\\
 &\st{S},&
  }
\end{equation*}
where $\st{X}$ is smooth, and $\st{X}$ and $\st{Y}$ are of dimension $N$
and $d$ respectively. Put $c:=N-d$.
Combining homomorphisms in \ref{cupproddef}, we have
\begin{equation*}
 p_{\st{Y}+}i^!L_{\st{X}}\boxtimes f_!i^+L_{\st{X}}
  \rightarrow
  f_+i^!L_{\st{X}}\otimes f_!i^+L_{\st{X}}
  \rightarrow
  f_!\bigl(i^!L_{\st{X}}\otimes i^+L_{\st{X}}\bigr)
  \rightarrow
  f_!i^!L_{\st{X}}
  \xrightarrow{\mr{adj}_i}
  g_!L_{\st{X}},
\end{equation*}
where $p_{\st{Y}}$ is the structural morphism of $\st{Y}$, and
$\mr{adj}_i\colon i_!i^!\rightarrow\mr{id}$ is the adjunction morphism.
Since $\H^kp_{\st{Y}+}i^!L_{\st{X}}=0$ for $k<2c$ (cf.\ Lemma
\ref{cohdimcalc}), this yields a coupling called the {\em cup product}
\begin{equation*}
 \cup\colon
  H_{\st{Y}}^{2c}(\st{X})(c)\otimes f_!L_{\st{Y}}(d)[2d]
  \rightarrow
  g_!L_{\st{X}}(N)[2N].
\end{equation*}

Now, by taking the adjoint, we may regard $\mr{fTr}_i$ as an element in
$H^{2c}_{\st{Y}}(\st{X})(c)$. We put $\st{S}=\mr{Spec}(k)$, and
we have the following characterization of the fake trace map:

\begin{lem*}
 The class $\mr{fTr}_i\in H^{2c}_{\st{Y}}(\st{X})(c)$ is the
 unique element such that
 for any $u\in H^{2d}_{\mr{c}}(\st{Y})(d)$,
 \begin{equation*}
  \mr{Tr}_f(u)=\mr{Tr}_g\bigl(\mr{fTr}_i\cup u\bigr).
 \end{equation*}
\end{lem*}
\begin{proof}
 As in [SGA $4\frac{1}{2}$, Cycle, 2.3], we have the following
 commutative diagram:
 \begin{equation*}
  \xymatrix@C=25pt{
   H^{2c}_{\st{Y}}(\st{X})
   \ar[r]
   \ar[rd]+L+<0pt,4pt>_{\ccirc{2}}&
   \mr{Hom}\bigl(f_!L_{\st{Y}}(d)[2d],f_!i^!L_{\st{X}}(N)[2N]\bigr)
   \ar[r]^-{\ccirc{1}}\ar[d]^{\mr{adj}_i}&
   \mr{Hom}\bigl(f_!L_{\st{Y}}(d)[2d],L_{\st{S}}\bigr)\ar@{=}[d]\\
  &
  \mr{Hom}\bigl(g_!i_!L_{\st{Y}}(d)[2d],g_!L_{\st{X}}(N)[2N]\bigr)
  \ar[r]_-{\mr{Tr}_g}&
  \mr{Hom}\bigl(g_!i_!L_{\st{Y}}(d)[2d],L_{\st{S}}\bigr).
   }
 \end{equation*}
 Here, $\ccirc{1}$ is the homomorphism induced by
 the adjunction using the assumption that $\st{X}$ is smooth. The
 homomorphism $\mr{adj}_i\colon i_!i^!\rightarrow\mr{id}$ is the
 adjunction. The homomorphism $\ccirc{2}$ is induced by $\cup$ defined
 above. By the definition of fake trace map, the upper horizontal
 homomorphism maps $\mr{fTr}_i$ to $\mr{Tr}_f$. Moreover, the
 composition of the upper horizontal maps is an isomorphism. Thus we can
 conclude the proof.
\end{proof}

\begin{rem*}
 The assumption that $\st{S}=\mr{Spec}(k)$ is used only for the
 existence of the fake trace over $\mf{S}$. It is not hard to generalize
 the definition of the fake trace to relative situations as SGA
 $4\frac{1}{2}$, and we may prove the lemma in this generality, although
 we are not sure if it is meaningful for our purpose.
\end{rem*}

\begin{cor*}
 Assume that $f\colon\st{X}\rightarrow\st{Y}$ is a flat morphism of
 c-admissible stacks such that $\st{Y}$ is smooth. Then
 $\mr{Tr}_f=\mr{fTr}_f$.
\end{cor*}
\begin{proof}
 This follows readily from the characterization lemma of
 $\mr{fTr}_f$ above.
\end{proof}

\subsubsection{}
(i) Consider the situation as in \ref{tracemapmeanscons}
(i). We have an isomorphism
\begin{equation*}
 \mr{Hom}(i^+L_{\st{X}},i^!L_{\st{X}}(c)[2c])\cong
  \mr{Hom}(L_{\st{X}},i_+i^!L_{\st{X}}(c)[2c])=:
  H_{\st{Z}}^{2c}(\st{X})(c).
\end{equation*}
The image of $c_{\st{Z}}$ is denoted by $\mr{cl}_{\st{X}}(\st{Z})$, and
called the {\em cycle class of $\st{Z}$}. Since the homomorphism
$\st{Z}\rightarrow\st{X}$ is proper, we have the homomorphism
$H^*_{\st{Z}}(\st{X})\rightarrow H^*(\st{X})$. The image of
$\mr{cl}_{\st{X}}(\st{Z})$ in $H^{2c}(\st{X})(c)$ is also called the
cycle class. Note that if the morphism $\st{Z}\rightarrow g(\st{Z})$ is
not generically finite, then $H^{2c}_{g(\st{Z})}(\st{X})=0$, and in
particular the cycle class in $H^{2c}(\st{X})(c)$ is $0$.

(ii) Consider the situation as in \ref{tracemapmeanscons} (ii). Recall
that $\varphi$ is assumed proper. We have the action of
correspondence on the cohomology, which is the composition of the
homomorphisms
\begin{equation*}
 \Gamma^*\colon
 f'_!L_{\st{X}'}\rightarrow
 f'_!p_{2+}p_2^+L_{\st{X}'}
 \xleftarrow{\sim}
  (\varphi\circ f\circ p_1)_!\,p_2^+L_{\st{X}'}
  \xrightarrow{\iota_\Gamma}
  \varphi_+f_!p_{1!}p_1^!L_{\st{X}}\rightarrow
  \varphi_+f_!L_{\st{X}}.
\end{equation*}
When $S$ is a point, this is nothing but the following composition using
(\ref{pushhombasesm}):
\begin{equation*}
 H^*_{\mr{c}}(\st{X}')\xrightarrow{p_2^*}
  H^*_{\mr{c}}(\Gamma)\xrightarrow{p_{1*}}
  H^*_{\mr{c}}(\st{X}).
\end{equation*}

\begin{lem}
 \label{commutabasechancorre}
 Consider the following cartesian diagram of c-admissible stacks on the
 left:
 \begin{equation*}
  \xymatrix{
   \st{Y}'\ar[d]_{f'}\ar[r]^-{g'}\ar@{}[rd]|\square&
   \st{Y}\ar[d]^{f}\\
  \st{X}'\ar[r]_-{g}&
   \st{X},
   }
   \hspace{50pt}
   \xymatrix@C=50pt{
   g'^+f^+L_{\st{X}}(d)[2d]\ar[r]^-{g'^+\mr{fTr}_{f}}\ar[d]_{\sim}&
   g'^+f^!L_{\st{X}}\ar[d]\\
  f'^+L_{\st{X}'}(d)[2d]\ar[r]_-{\mr{fTr}_{f'}}&
   f'^!L_{\st{X}'},
   }
 \end{equation*}
 where $\st{X}$ and $\st{X}'$ are smooth. Assume moreover that there
 exists an open substack $\st{V}\subset\st{Y}$ such that the
 morphism $\st{V}\rightarrow\st{X}$ is flat of relative dimension $d$,
 and $g'^{-1}(\st{V})\subset\st{Y}'$ is dense. Then the above diagram on
 the right is commutative. In particular, if $g$ is proper, we have an
 equality
 \begin{equation*}
  g^*f_*=f'_*g'^*\colon H_{\mr{c}}^*(\st{Y})\rightarrow
   H_{\mr{c}}^{*-2d}(\st{X}')(-d).
 \end{equation*}
\end{lem}
\begin{proof}
 Since the commutativity of the diagram on the right can be interpreted
 as coincidence of two elements in $H^{2a}_{\mr{c}}(\st{Y'})(a)^\vee$,
 where $a$ denotes the dimension of $\st{Y}'$, we may shrink $\st{Y'}$
 by its open dense substack by Lemma \ref{cohdimcalc}. Thus we may
 replace $\st{Y}$ by $\st{V}$, and may assume that $f$ is flat of
 relative dimension $d$. Now the lemma follows by Corollary
 \ref{sga4halfcharact} and the base change property of the trace map
 (cf.\ Theorem \ref{poincaredualrela} (II)).
\end{proof}

\begin{cor*}[Projection formula]
 Let $f\colon\st{X}\rightarrow\st{Y}$ be a proper morphism of
 c-admissible stacks and $\st{Y}$ is smooth. Assume that there exists an
 open dense substack $\st{U}\subset\st{X}$ such that the morphism
 $\st{U}\rightarrow\st{Y}$ is flat of relative dimension $d$. Then for
 $\alpha\in H^i_{\mr{c}}(\st{X})$ and $\beta\in
 H^j_{\mr{c}}(\st{Y})$, we have the following equality in
 $H^{i+j-2d}_{\mr{c}}(\st{Y})(-d)$.
 \begin{equation*}
  f_*\bigl(\alpha\cup f^*\beta\bigr)=f_*\alpha\cup\beta.
 \end{equation*}
\end{cor*}
\begin{proof}
 Consider the following commutative diagram of proper morphisms:
 \begin{equation*}
  \xymatrix@C=40pt{
   \st{X}\ar[r]\ar@/^1.5pc/[rr]^-{\mr{id}}&
   \st{X}\times_{\st{Y}}\st{X}
   \ar[r]\ar[d]\ar@{}[rd]|\square&
   \st{X}\ar[r]^{f}\ar[d]_{\Delta'}\ar@{}[rd]|\square&
   \st{Y}\ar[d]^{\Delta}\\
  &
   \st{X}\times\st{X}\ar[r]&
   \st{X}\times\st{Y}\ar[r]_{f\times\mr{id}}&
   \st{Y}\times\st{Y}.
   }
 \end{equation*}
 By the hypothesis on $f$, we can apply the lemma above to the right
 cartesian diagram (take $f$, $g$ in the lemma to be $\mr{id}\times f$
 and $\Delta$ respectively). We have
 \begin{equation*}
  f_*\alpha\cup\beta
   =\Delta^*\bigl((f\times\mr{id})_*(\alpha\boxtimes\beta)\bigr)
   =f_*\Delta'^*(\alpha\boxtimes\beta)
 \end{equation*}
 where we used the lemma for the second equality.
 The diagram above and the transitivity of the pull-back show that
 $\Delta'^*(\alpha\boxtimes\beta)=\alpha\cup f^*\beta$, and the corollary
 follows.
\end{proof}

\begin{lem}
 \label{fibercorres}
 Let $S$ be a scheme of finite type over $k$, $\st{X}^{(\prime)}$ be a
 c-admissible stack, and
 $f^{(\prime)}\colon\st{X}^{(\prime)}\rightarrow S$ be a morphism.
 Assume that $\varphi=\mr{id}$.
 Let $\Gamma\colon\st{X}\rightsquigarrow\st{X}'$ over $S$.
 Assume there is an open substack $\Gamma'\subset\Gamma$ such that
 the first projection $\Gamma'\subset\Gamma\xrightarrow{p_1}\st{X}$ is
 flat. For a closed point $i_s\colon s\hookrightarrow S$, we denote by
 $\st{X}_s$, $\st{X}'_s$, $\Gamma_s$, $\Gamma'_s$ the fibers over $s$.
 If $\Gamma'_s\subset\Gamma_s$ is {\em dense}, $\Gamma_s$ is a
 correspondence $\st{X}_s\rightsquigarrow\st{X}'_s$, and the
 the following diagram is commutative:
 \begin{equation*}
  \xymatrix@C=40pt{
   i_s^+f'_!L_{\st{X}'}\ar[r]^-{\Gamma^*}\ar[d]&
   i_s^+f_!L_{\st{X}}\ar[d]\\
  f'_{s!}L_{\st{X}'_s}\ar[r]_-{\Gamma_s^*}&
   f_{s!}L_{\st{X}_s}.
   }
 \end{equation*}
 Here the vertical homomorphisms are the base change maps.
\end{lem}
\begin{proof}
 Since $\Gamma'$ is flat over $S$, $\Gamma_s$ is a correspondence for
 any $s\in S$. Now, to show the commutativity, it suffices to show the
 commutativity of the following diagrams:
 \begin{equation*}
  \xymatrix{
   i_s^+p_1^+L_{\st{X}}\ar[r]\ar[d]_{i_s^+(\iota_\Gamma)}&
   p_{1s}^+L_{\st{X}_s}\ar[d]^{\iota_{\Gamma_s}}\\
  i_s^+p_1^!L_{\st{X}}\ar[r]&p_{1s}^!L_{\st{X}_s},
   }
   \hspace{50pt}
   \xymatrix{
   i_s^+L_{\st{X}'}\ar[r]^-{\sim}\ar[d]&
   L_{\st{X}'_s}\ar[d]\\
  i_s^+p_{2+}p_2^+L_{\st{X}'}\ar[r]_-{\sim}&
   p_{2s+}p_{2s}^+L_{\st{X}'_s}.
   }
 \end{equation*}
 The commutativity of the left diagram follows by Lemma
 \ref{commutabasechancorre}. The commutativity of the right one
 follows by the fact that $(g^+,g_+)$ is an adjoint pair when
 $g$ is a morphism of admissible stacks.
\end{proof}

\begin{rem*}
 In general, there exists an open dense subscheme $U\subset S$ such that
 the condition of the lemma holds for any $s\in U$.
\end{rem*}

\begin{lem}
 Let $\rho\colon\st{X}\rightarrow\st{X}'$ be a proper morphism over
 $\varphi$ such that $\st{X}$ is smooth. Let $\Gamma_\rho$ denote the
 graph of $\rho$ and regard it as a correspondence
 $\st{X}\rightsquigarrow\st{X}'$. Then $\rho^*=\Gamma_\rho^*$.
\end{lem}
\begin{proof}
 We have $\st{X}\xleftarrow[\sim]{p_1}\Gamma_{\rho}\xrightarrow{p_2}
 \st{X}'$. Via the identification $p_{1!}K_{\Gamma_\rho}
 \xrightarrow{\sim} K_{\st{X}}$, the fake trace $\mr{fTr}_{p_1}\colon 
 p_{1!}p_1^+K_{\st{X}}\rightarrow K_{\st{X}}$ is the identity. Thus the
 lemma follows.
\end{proof}

\subsubsection{}
\label{tracepropmuln}
Let $\st{X}$, $\st{Y}$ are c-admissible stacks of
dimension $d$. Let $g\colon\st{Y}\rightarrow\mr{Spec}(k)$,
$h\colon\st{X}\rightarrow\st{Y}$ and $f:=g\circ h$. Assume that $h$ is
proper. Then, we have the canonical homomorphism
$(h^*)^\vee\colon(H^{2d}_{\mr{c}}(\st{X})(d))^\vee\rightarrow
(H^{2d}_{\mr{c}}(\st{Y})(d))^\vee$ where $(-)^\vee$ denotes
$\mr{Hom}(-,L)$. Note that $\mr{Tr}_{\st{X}}\in
H^{2d}_{\mr{c}}(\st{X})(d)^\vee$.

The morphism $h$ is said to be {\em generically locally free} if there
exists a dense open substack $\st{V}\subset\st{Y}$ such that the induced
morphism $h^{-1}(\st{V})\rightarrow\st{V}$ is finite and locally free.
Furthermore, $h$ is said to be {\em generically of constant degree} if
all the degrees of $h|_{h^{-1}(\st{V})}$ over irreducible components of
$\st{V}$ are the same.

\begin{lem*}
 Assume that $h$ is a generically locally free morphism
 of constant degree. Then $(h^*)^\vee$ sends the trace map $\mr{Tr}_f$
 to $\deg(h)\cdot\mr{Tr}_{g}$.
\end{lem*}
\begin{proof}
 We have the following diagram:
 \begin{equation*}
  \xymatrix{
   H^{2d}_{\mr{c}}(\st{X})(d)^\vee\ar[r]^-{(h^*)^\vee}\ar[d]&
   H^{2d}_{\mr{c}}(\st{Y})(d)^\vee\ar[d]^{\sim}\\
  H^{2d}_{\mr{c}}(h^{-1}(\st{V}))(d)^\vee\ar[r]&
   H^{2d}_{\mr{c}}(\st{V})(d)^\vee,
   }
 \end{equation*}
 where the right vertical homomorphism is an isomorphism by Lemma
 \ref{cohdimcalc}. Thus, we may replace $\st{Y}$ by $\st{V}$, and in
 particular, we may assume that $h$ is locally free. We have the
 following commutative diagram:
 \begin{equation*}
  \xymatrix{
   g_!h_!h^+g^+(L)(d)[2d]\ar[r]^-{\mr{Tr}_h}&
   g_!g^+(L)(d)[2d]\ar[r]^-{\mr{Tr}_g}&L\\
  g_!g^+L(d)[2d]\ar[ru]_-{\deg(h)\cdot}.\ar[u]^{\mr{adj}_h}&&
   }
 \end{equation*}
 The composition of the first row is $\mr{Tr}_f$. By definition,
 $(h^*)^\vee(\mr{Tr}_f)=\mr{adj}_h\circ\mr{Tr}_f$. Thus
 the lemma follows.
\end{proof}

\begin{cor*}
 \label{genconstdegrel}
 Let $\st{X}$ be a smooth c-admissible stack.
 
 (i) Let $\st{Z}$, $\st{Z}'$ be generalized cycles of codimension $c$ of
 $\st{X}$. Assume given a generically locally free morphism of constant
 degree $\rho\colon\st{Z}'\rightarrow\st{Z}$ over $\st{X}$. Then
 by the push-forward homomorphism $\rho_*\colon
 H^{2c}_{\st{Z}'}(\st{X})(c)\rightarrow H^{2c}_{\st{Z}}(\st{X})(c)$,
 $\mr{cl}_{\st{X}}(\st{Z}')$ is sent to
 $\deg(\rho)\cdot\mr{cl}_{\st{X}}(\st{Z})$.

 (ii) Let $\Gamma,\Gamma'\colon\st{X}\rightsquigarrow\st{X}'$ be
 correspondences. Assume that there exists a morphism
 $\rho\colon\Gamma\rightarrow\Gamma'$ such that
 $c_{\Gamma'}\circ\rho=c_{\Gamma}$, and $\rho$ is generically locally
 free of constant degree. Then
 $\Gamma'^*\cong\deg(\rho)^{-1}\cdot\Gamma^*$.
\end{cor*}
\begin{proof}
 They are straightforward from the lemma.
\end{proof}

\begin{dfn}
 \label{corresnotdef}
 We denote by $\mr{Corr}_{\varphi}(\st{X},\st{X}')$ the $\mb{Q}$-vector
 space freely generated by the set
 $\bigl\{\Gamma\colon\st{X}\rightsquigarrow\st{X}'\mid
 \mbox{correspondence over $\varphi$}\bigr\}$. We denote
 $\mr{Corr}_{\mr{id}}(\st{X},\st{X}')$ by $\mr{Corr}_S(\st{X},\st{X}')$,
 and we denote $\mr{Corr}_{\star}(\st{X},\st{X})$ by
 $\mr{Corr}_{\star}(\st{X})$ ($\star=\varphi,S$) for short.
 We have a homomorphism
 \begin{equation*}
  \mr{Corr}_{\varphi}(\st{X},\st{X}')\rightarrow
   \mr{Hom}_S\bigl(\varphi^+f'_!L_{\st{X}'},
   f_!L_{\st{X}}\bigr)
 \end{equation*}
 by sending $\Gamma$ to $\Gamma^*$. Let $I$ be the $\mb{Q}$-vector
 subspace of $\mr{Corr}_{\varphi}(\st{X},\st{X}')$ generated by
 $(\Gamma'-\deg(\rho)^{-1}\cdot\Gamma)$ where $\Gamma$, $\Gamma'$ are
 correspondences and $\rho$ is a generically locally free morphism of
 constant degree $\Gamma\rightarrow\Gamma'$.
 When $\st{X}$ is smooth, by Corollary \ref{genconstdegrel} (ii), the
 homomorphism above factors through
 $\mr{Corr}_{\varphi}(\st{X},\st{X}')/I$.
\end{dfn}

Let $\mr{Corr}^{\star}_{\varphi}(\st{X},\st{X}')$ for
$\star=\mr{et}$ (resp.\ $\mr{fin}$) be the
$\mb{Q}$-vector subspace of $\mr{Corr}_{\varphi}(\st{X},\st{X}')$
generated by integral correspondences $\Gamma$ ({\it i.e.}\ $\Gamma$ is
integral) such that the first projection $\Gamma\rightarrow\st{X}$ is
\'{e}tale (resp.\ finite). There exists the composition map
\begin{equation*}
 \circ\colon\mr{Corr}^{\mr{et}}_{\psi}(\st{X}',\st{X}'')
  \times
  \mr{Corr}^{\mr{et}}_{\varphi}(\st{X},\st{X}')
  \rightarrow\mr{Corr}^{\mr{et}}_{\psi\circ\varphi}
  (\st{X},\st{X}'')
\end{equation*}
defined by sending $(\Gamma',\Gamma)$ to
$\Gamma'\circ\Gamma:=\Gamma\times_{\st{X}'}\Gamma'$.

\begin{lem*}
 Let $\Gamma\colon\st{X}\rightsquigarrow\st{X}'$,
 $\Gamma'\colon\st{X}'\rightsquigarrow\st{X}''$ be correspondences over
 $\varphi$ and $\psi$. Assume further that the second projection of
 $\Gamma$ or the first projection of $\Gamma'$ is smooth.
 When $\st{X}$ and $\st{X}'$ are smooth, we have
 $(\Gamma'\circ\Gamma)^*=\Gamma^*\circ\Gamma'^*$.
\end{lem*}
\begin{proof}
 The verification is standard. See, for example, \cite[A.7]{Laf}.
\end{proof}

\subsubsection{}
\label{complafourcorre}
Given a correspondence, the results in this subsection hold in exactly
the same manner for $\ell$-adic \'{e}tale cohomology. However, in
\cite[\S A]{Laf}, he uses slightly different definition of the actions
of correspondences on the cohomology, and we need to compare these.

For a smooth admissible stack $\st{X}$, we denote by $\st{X}^{\mr{gr}}$
the associated coarse moduli algebraic space of Keel and Mori (cf.\
\cite[A.2]{Laf}). Now, let $\Gamma\colon\st{X}\rightsquigarrow\st{X}'$
be an integral correspondence over $\varphi$ such that the first
projection
$p\colon\Gamma\rightarrow\st{X}$ is {\em generically finite} and
dominant. Then the morphism $\rho\colon\Gamma\rightarrow
\widetilde{\Gamma}^{\mr{gr}}:=\Gamma^{\mr{gr}}{\times}
_{p\searrow\st{X}^\mr{gr}}\st{X}$ is generically finite locally free.
Indeed, let $P\rightarrow\st{X}$ be a presentation. Then
$|\Gamma\times_{\st{X}}P|\rightarrow
|\widetilde{\Gamma}^{\mr{gr}}\times_{\st{X}}P|=
|\Gamma^{\mr{gr}}|\times_{|\st{X}^{\mr{gr}}|}|P|$ is surjective.
The last equality holds since the fiber product is taken in the
category of algebraic spaces. Thus the
morphism $\Gamma\times_{\st{X}}P\rightarrow
\widetilde{\Gamma}^{\mr{gr}}\times_{\st{X}}P$ is surjective, and
$\Gamma\rightarrow\widetilde{\Gamma}^{\mr{gr}}$ is
surjective. This implies that
$\widetilde{\Gamma}^{\mr{gr}}$ is irreducible.
Now, since $\pi_*\mc{O}_{\Gamma}\cong\mc{O}_{\Gamma^{\mr{gr}}}$ (cf.\
\cite[1.1]{Co2}), $\Gamma^{\mr{gr}}$ is reduced. By \cite[5.1.12, 13,
14]{Behr}, the morphism $\st{X}\rightarrow\st{X}^{\mr{gr}}$ is a gerb
generically over $\st{X}^{\mr{gr}}$, thus the morphism is generically
smooth by \cite[5.1.5]{Behr}. This implies that
$\widetilde{\Gamma}^{\mr{gr}}\rightarrow\Gamma^{\mr{gr}}$ is smooth
generically, and thus $\widetilde{\Gamma}^{\mr{gr}}$ is reduced
generically over $\st{X}$. Since $\widetilde{\Gamma}^{\mr{gr}}$ is
irreducible and dominant over $\st{X}$,
we conclude that $\widetilde{\Gamma}^{\mr{gr}}$ is
integral. Thus, $\rho$ is generically flat. Since $p$ is assumed
generically finite, $\rho$ is generically finite as well, and thus
generically finite locally free.

The generic degree of $\rho$ is denoted by $d_{\Gamma}$.
We put $\mr{norm}(\Gamma):=(d_\Gamma)^{-1}\cdot\Gamma$.
We can check easily that for composable correspondences $\Gamma$,
$\Gamma'$, we have
$d_{\Gamma'\circ\Gamma}=d_{\Gamma}\cdot d_{\Gamma'}$.
Then we have a ring homomorphism
\begin{equation*}
 \mr{norm}\colon
 \mr{Corr}^{\mr{fin,et}}_{\varphi}(\st{X})\rightarrow
  \mr{Corr}^{\mr{fin,et}}_{\varphi}(\st{X});\
  \Gamma\mapsto (d_{\Gamma})^{-1}\cdot\Gamma.
\end{equation*}

Now, let $\st{Y}$ be a c-admissible stack, and
$q_{\st{Y}}\colon\st{Y}\rightarrow\st{Y}^{\mr{gr}}$ be the canonical
morphism, which is known to be proper (cf.\ \cite[1.1]{Co2}). By
\cite[A.3]{Laf}, the adjunction homomorphism
$\mb{Q}_\ell\rightarrow\mb{R}q_{\st{Y}*}q_{\st{Y}}^*(\mb{Q}_\ell)$ is an
isomorphism. For
$\Gamma\in\mr{Corr}_{\varphi}(\st{X},\st{X}')$, we define
\begin{equation*}
 \Gamma^{*\mr{Laf}}\colon
  \mc{H}^*_{\mr{c}}(\st{X}')
  \xleftarrow{\sim}
  \mc{H}^*_{\mr{c}}(\st{X}'^{\mr{gr}})
  \xrightarrow{\Gamma^{\mr{gr*}}}
  \mc{H}^*_{\mr{c}}(\st{X}^{\mr{gr}})
  \xrightarrow{\sim}
  \mc{H}^*_{\mr{c}}(\st{X})
\end{equation*}
where, for a c-admissible stack $f\colon\st{Y}\rightarrow S$,
$\mc{H}^*_{\mr{c}}(\st{Y})$ denotes $\H^*f_!\mb{Q}_{\ell}$. This is
nothing but the action of correspondence defined in \cite[A.6]{Laf}.

\begin{lem*}
 Let $\Gamma\in\mr{Corr}_{\varphi}(\st{X},\st{X}')$ which is integral
 and the first projection $\Gamma\rightarrow\st{X}$ is generically
 finite. Then we have
 $\Gamma^{*\mr{Laf}}=(\mr{norm}(\Gamma))^*$.
\end{lem*}
\begin{proof}
 Consider the following commutative diagram on the left:
 \begin{equation*}
  \xymatrix@C=40pt@R=20pt{
   &\Gamma\ar[r]\ar[d]&\st{X}'\ar[d]\\
  \st{X}\ar[d]\ar@{}[rd]|\square&
   \widetilde{\Gamma}^{\mr{gr}}
   \ar[l]_{\widetilde{p}}\ar[r]^{\widetilde{p}'}\ar[d]^q&
   \st{X}'^{\mr{gr}}\ar@{=}[d]\\
  \st{X}^{\mr{gr}}&
   \Gamma^{\mr{gr}}
   \ar[l]^{p^{\mr{gr}}}\ar[r]_{p'^{\mr{gr}}}&
   \st{X}'^{\mr{gr}},
   }\hspace{40pt}
   \xymatrix@C=40pt{
   q^*p'^{\mr{gr}*}\mb{Q}_\ell
   \ar[r]^{\iota_{\Gamma^{\mr{gr}}}}
   \ar[d]_{\sim}&
   q^*p^{\mr{gr}!}\mb{Q}_\ell\ar[d]\\
  \widetilde{p}'^*\mb{Q}_\ell
   \ar[r]_{\iota_{\widetilde{\Gamma}^{\mr{gr}}}}&
   \widetilde{p}^!\mb{Q}_\ell.
   }
 \end{equation*}
 Then the right diagram is commutative by using Lemma
 \ref{commutabasechancorre}. Thus, the lemma follows by Lemma
 \ref{tracepropmuln}. Note that even though $\st{X}^{\mr{gr}}$ is not
 smooth, it is cohomological smooth by \cite[A.5]{Laf}, and arguments in
 \ref{tracepropmuln} works without any change.
\end{proof}

\subsection{Independence of $\ell$}
We show an $\ell$-independence result of the trace of the action of
correspondence on the cohomology of a c-admissible stack.
In the scheme and $\ell$-adic case, the $\ell$-independence result is
the one proven in \cite{KS}.

\subsubsection{}
The main result of this subsection is as follows:

\begin{thm*}
 \label{indepofltrace}
 Let $\st{X}$ be a smooth c-admissible stack over a finite
 field $k$, and $\Gamma\in\mr{Corr}_S(\st{X})$. Then we
 have
 \begin{equation*}
  \mr{Tr}\bigl(\Gamma^*: H^*_{\mr{c}}(\st{X}\otimes
   \overline{k},\mb{Q}_\ell)\bigr)=
   \mr{Tr}\bigl(\Gamma^*: H^*_{\mr{c}}(\st{X},
   L_{\st{X},\emptyset})\bigr).
 \end{equation*}
\end{thm*}

\begin{rem*}
 We assumed $k$ to be finite since some delicate arguments might be
 needed in the $p$-adic situation to reduce to the finite field case as
 in the proof of \cite[2.3.6.1]{KS}.
 However, there is no reason to doubt that the theorem holds without the
 assumption on the base field.
\end{rem*}

The idea of the proof is essentially the same as \cite[2.3.6]{KS}, and
the proof takes the whole subsection.
In the application, we use the following form: Let $k$ be a finite field
with $q=p^s$ elements, $R=W(k)$, $\sigma=\mr{id}$.
Let $X$ be a separated scheme of finite type over $k$, and let
$i_x\colon x\hookrightarrow X$ be a closed point of $X$.
We take an algebraic extension $L$ of $K_x:=\mr{Frac}(W(k(x)))$, and
take $\sigma_L:=\mr{id}$ for an extension of $\sigma$.
Put $s':=[k(x):k]\cdot s$, and let
$\mf{T}_{x,F}:=(k(x),W(k(x)),K_x,L,s',\mr{id})$.
We have the following functor
\begin{equation*}
 \iota^n_x\colon
 D^{\mr{b}}_{\mr{hol}}(X/L_F)
 \xrightarrow{\H^ni_x^+}
  \mr{Hol}(k(x)/L_F)\cong
  \mr{Hol}(k(x)/\mf{T}_{x,F})
\end{equation*}
where the last equivalence follows by Corollary \ref{behaviorofbasech}.
The last category is equivalent to the category of finite dimensional
$L$-vector spaces with automorphism denoted by
$F_x$. Let $\ms{E}\in
D^{\mr{b}}_{\mr{hol}}(X/L_F)$, and assume given an
automorphism $\alpha$. This induces an automorphism on
$\iota^n_x(\ms{E})$ denoted by $\alpha_x$. Note that $\alpha_x$ and
$F_x$ commute. For $n\in\mb{Z}$, we put
\begin{equation*}
 \mr{Tr}\bigl(\alpha\times F_x^n:\ms{E}\bigr):=
  \sum_{i}(-1)^i\cdot\mr{Tr}\bigl(\alpha_x\circ F_x^n:
  \iota^i_x(\ms{E})\bigr).
\end{equation*}

\begin{cor*}
 Let $S$ be a smooth connected scheme of finite type over $k$, and
 $f\colon\st{X}\rightarrow S$ be a smooth morphism between c-admissible
 stacks. Assume given $\Gamma\in\mr{Corr}_S(\st{X})$.
 For a closed point $x\in S$, let $\st{X}_x$ and
 $\Gamma_x$ denote the fibers over $x$. Then
 there exists an open dense subscheme $U\subset S$ such that for any
 closed point $x\in U$ such that $L\supset K_x$ and any integer $n$, we
 have
 \begin{equation*}
  \mr{Tr}\bigl(\Gamma^*\times F^n_x:
   f_!\mb{Q}_\ell\bigr)=
   \mr{Tr}\bigl(\Gamma^*\times F_x^n:
   f_!L_{\st{X}}\bigr).
 \end{equation*}
\end{cor*}
\begin{proof}
 Let $f_x$ denote the fiber of $f$ on $x$.
 By Lemma \ref{fibercorres} and its remark, we can take $U\subset S$
 such that the action of $\Gamma$ on $f_!\mb{Q}_\ell$ and
 $f_!L_{\st{X}}$ at the fiber $x\in U$ are equal to the action of
 $\Gamma_x$ on $f_{x!}\mb{Q}_\ell$ and $f_{x!}L_{\st{X}_x}$. Then use
 our theorem to get the corollary.
\end{proof}

\subsubsection{}
\label{lemonpullcycle}
We do not assume $k$ to be finite here, and we fix
$\base\in\{\emptyset,F\}$.

\begin{lem*}[{[SGA $4\frac{1}{2}$, Cycle, 2.3.8 (ii)]}]
 Consider the following cartesian diagram of
 c-admissible stacks over $k$
 \begin{equation*}
  \xymatrix@C=40pt{
   \Gamma'\ar[r]^-{g'}\ar[d]_{f'}\ar@{}[rd]|\square&
   \st{X}'\ar[r]^-{h'}\ar[d]^{f}\ar@{}[rd]|\square&
   S'\ar[d]\\
  \Gamma\ar[r]_-{g}&\st{X}\ar[r]_-{h}&S
   }
 \end{equation*}
 where $\st{X}^{(\prime)}$ is smooth (over $k$),
 $S^{(\prime)}$ is a scheme, $g$ is a generalized
 cycle of codimension $c$, and $h$ and $h\circ g$ are flat and
 equidimensional. Then $g'$ is a generalized cycle of codimension $c$ as
 well, and we have $f^*\mr{cl}(\Gamma)=\mr{cl}(\Gamma')\in
 H^{2c}_{\Gamma'}(\st{X}')(c)$.
\end{lem*}
\begin{proof}
 We have the following commutative diagram:
 \begin{equation*}
  \xymatrix{
   \mr{Hom}(g^+L_{\st{X}},g^!L_{\st{X}})
   \ar@{.>}[d]\ar@{-}[r]^-{\sim}&
   H^{2c}_{\Gamma}(\st{X})(c)\ar[d]^{f^*}\\
   \mr{Hom}(g'^+L_{\st{X}},g'^!L_{\st{X}})\ar@{-}[r]_-{\sim}&
   H^{2c}_{\Gamma'}(\st{X}')(c)
   }
 \end{equation*}
 where the dotted arrow can be described as follows: Let
 $\phi\in\mr{Hom}(g^+L_{\st{X}},g^!L_{\st{X}})$. Then the image of
 $\phi$ is the unique dotted homomorphism, in the following diagram on
 the left, which makes the diagram commutative:
 \begin{equation*}
  \xymatrix@C=40pt{
   g'^+L_{\st{X}'}
   \ar@{.>}[r]&
   g'^!L_{\st{X}'}(c)[2c]\\
  f'^+g^+L_{\st{X}}
   \ar[r]^-{\sim}_-{f'^+(\phi)}\ar[u]^-{\sim}&
   f'^+g^!L_{\st{X}}(c)[2c],\ar[u]
   }\qquad
   \xymatrix@C=40pt{
   g'^+L_{\st{X}'}
   \ar[r]_-{\sim}^-{\mr{Tr}_{\Gamma'}}&
   g'^!L_{\st{X}'}(c)[2c]\\
   f'^+g^+L_{\st{X}}
   \ar[r]^-{\sim}_-{f'^+\mr{Tr}_{\Gamma}}\ar[u]^{\sim}&
   f'^+g^!L_{\st{X}}(c)[2c].\ar[u]
   }
 \end{equation*}
 Here the right vertical homomorphism is the base change homomorphism.
 Thus, by taking $\phi=\mr{Tr}_\Gamma$, the problem is reduced to
 showing the commutativity of the right diagram above.
 Now, $\mr{Tr}_\Gamma$ and $\mr{Tr}_{\Gamma'}$ can be regarded as
 $\mr{Tr}_{h\circ g}$ and $\mr{Tr}_{h'\circ g'}$ by the transitivity of
 trace. By the base change property of trace, we get the lemma.
\end{proof}

\begin{lem}
 \label{chowgroupfacto}
 Let $X$ be a smooth scheme of dimension $d$ over $k$, and let
 $Z\rightarrow X$ be a generalized cycle. Then the cycle class map
 induces a homomorphism $\mr{cl}_X\colon\mr{CH}_i(Z)\rightarrow
 H^{2d-2i}_Z(X)(d-i)$.
\end{lem}
\begin{proof}
 Let $W$ be a closed integral subscheme of dimension $i+1$ of $Z$
 and $W\rightarrow\mb{P}^1$ be a dominant morphism (hence flat by
 [EGA, IV.2.8.2]). Let $W_i$ be the fiber of $i\in\mb{P}^1$.
 By \cite[Ch.I, Prop 1.6]{Fu}, it suffices to show that
 $\mr{cl}_X(W_0)=\mr{cl}_X(W_1)$. Let $W^\circ$ be the pull-back of
 $\mb{A}^1\subset\mb{P}^1$ by the morphism
 $W\rightarrow\mb{P}^1$. Note that $W^\circ\rightarrow Z\times\mb{A}^1$
 is a closed immersion.
  Consider the following commutative diagram:
 \begin{equation*}
  \xymatrix@R=15pt{
   W_0\ar[r]\ar[d]\ar@{}[rd]|\square&
   W^\circ\ar[d]\ar@{}[rd]|\square&
   W_1\ar[l]\ar[d]\\
  Z\ar[r]\ar[d]\ar@{}[rd]|\square&
   Z\times\mb{A}^1\ar[d]\ar@{}[rd]|\square&
   Z\ar[l]\ar[d]\\
  \{0\}\ar[r]&
   \mb{A}^1&
   \{1\}.\ar[l]
   }
 \end{equation*}
 This induces the commutative diagram
 \begin{equation*}
     \xymatrix{
   H^{2j}_{W_0}(X)\ar[d]&
   H^{2j}_{W^\circ}(X\times\mb{A}^1)
   \ar[r]^-{1^*}\ar[l]_-{0^*}\ar[d]&
   H^{2j}_{W_1}(X)\ar[d]\\
  H^{2j}_Z(X)&
   H^{2j}_{Z\times\mb{A}^1}(X\times\mb{A}^1)
   \ar[l]_-{\sim}^-{0^*}
   \ar[r]^-{\sim}_-{1^*}&
   H^{2j}_Z(X),
   }
 \end{equation*}
 where $j:=d-i$ and we omit Tate twists.
 The bottom arrows are isomorphisms since $H^i(\mb{A}^1)=0$ for $i\neq0$
 and $H^0(\mb{A}^1)\cong L$. By Lemma \ref{lemonpullcycle} and the
 flatness of $X^{\circ}\rightarrow\mb{A}^1$,
 $\mr{cl}_{X\times\mb{A}^1}(W^\circ)$
 on the top middle is sent to $\mr{cl}_X(W_0)$ and $\mr{cl}_X(W_1)$ via
 $0^*$ and $1^*$, so the lemma follows.
\end{proof}

\begin{rem*}
 When $Z$ is a cycle in $X$, we believe that the method of \cite{Gil}
 can be applied to construct the cycle class map. However, the author
 does not know how we define the Zariski sheaves
 $\underline{\Gamma}^*(i)$ in [{\it ibid.}, 1.1].
\end{rem*}

\begin{lem}[{\cite[2.1.1]{KS}}]
 \label{2.1.1ofKS}
 Let $\st{X}$ be a c-admissible stack\footnote{
 In the corresponding statement of \cite{KS}, they assume $X$ to be
 smooth. We think that this is a typo and, in fact, is too strong for
 their and our purposes.
 Indeed, in the proof of [{\it ibid.}, 2.3.2],
 they apply [{\it ibid.}, 2.1.1] in the situation where $X$ is not
 necessarily smooth.},
 $\st{U}$ be an open substack of $\st{X}$ which is smooth, and $\Gamma$
 be a generalized cycle of codimension $d$ on $\st{U}$.
 Consider the following commutative diagram
 \begin{equation*}
  \xymatrix{
   \Gamma\ar[r]_i\ar@/^3ex/[rr]^{i'}&
  \st{U}\ar@{^{(}->}[r]_j&\st{X},}
 \end{equation*}
 where $j$ is the open immersion, and $i$ is the generalized
 cycle. Assume $i'$ is proper. Recall the homomorphism
 (\ref{pushhombasesm}). Using this homomorphism, the composition
 \begin{equation*}
  c\colon H^*_{\mr{c}}(\st{X},j_+L_{\st{U}})\xrightarrow{i'^*}
   H^*_{\mr{c}}(\Gamma)\xrightarrow{i_*}
   H^{*+2d}_{\mr{c}}(\st{U})(d)
 \end{equation*}
 sends $u$ to $u\cup r(\mr{cl}_{\st{U}}(\Gamma))$
 (cf.\ \ref{cupproddef} for $\cup$). Here, $r$ is the homomorphism
 $H^{2d}_{\Gamma}(\st{U})(d)\cong H^{2d}(\st{X},j_!i_+i^!L_{\st{U}})(d)
 \rightarrow H^{2d}(\st{X},j_!L_{\st{U}}(d))$,
 where the first isomorphism follows since $i'$ is assumed proper
 as well.
\end{lem}
\begin{proof}
 We can copy the proof of {\it ibid.}.
\end{proof}

\subsubsection{}
Let $f\colon X\rightarrow\st{Y}$ be a representable l.c.i.\ morphism
from a scheme\footnote{This assumption is made for simplicity. In fact,
with suitable changes, similar results can be obtained when $X$ is a
c-admissible stack.}
of finite type purely of dimension $n$ to a c-admissible stack purely of
dimension $m$ over $k$. For example, a representable morphism
between smooth c-admissible stacks is l.c.i.. We note that $f$ is
schematic since $\st{Y}$ is admissible.
Let $\Gamma\rightarrow\st{Y}$ be a generalized cycle of codimension $d$
which is a scheme, and put $\Gamma':=\Gamma\times_{\st{Y}}X$, which is a
generalized cycle of $X$ and is a scheme as well since $f$ is
schematic. Let us briefly recall the construction of
$f^!([\Gamma])\in\mr{CH}_{n-d}(\Gamma')$ by Kresch \cite{Kr}.

Let $f'\colon\st{X}'\rightarrow\st{Y}'$ be a representable separated
morphism between algebraic stacks. In \cite[5.1]{Kr}, Kresch
constructs\footnote{
In fact, he constructs over $\mb{P}^1$ instead of $\mb{A}^1$. However,
for convenience, we restrict his construction over $\mb{A}^1$ and, by
abusing the notation, we still denote it by
$M^{\circ}_{\st{X}'}\st{Y}'$.
}
$\rho\colon M^{\circ}_{\st{X}'}\st{Y}'\rightarrow\mb{A}^1$, whose fiber
over $0$ is called the {\em normal cone} denoted by
$C_{\st{X}'}\st{Y}'\rightarrow\st{X}'$, and the general fiber is just
$\st{Y}'$. When $f'\colon X'\hookrightarrow Y'$ is a closed immersion of
schemes, $M^{\circ}_{X'}Y'$ is nothing but the one introduced in
\cite[Ch.5]{Fu}. We remark that by construction, there is a canonical
morphism $\st{X}'\times\mb{A}^1\rightarrow M^{\circ}_{\st{X}'}\st{Y}'$
defined by the strict transform, and $\rho$ is flat by using
\cite[B.6.7]{Fu}. When $f'$ is l.c.i.,
$C_{\st{X}'}\st{Y}'$ is known to be a vector bundle over $\st{X}$, in
which case we denote it by $N_{f'}$.

Now, assume that $X$ and $\st{Y}$ are smooth. In this situation,
$M^\circ_X\st{Y}$ is smooth. We put
$\Gamma':=\Gamma\times_{\st{Y}}X$. By definition, we have the diagram
below on the left:
\begin{equation}
 \label{cartdiagfulton}
  \tag{$\star$}
 \xymatrix{
  M^\circ_{\Gamma'}\Gamma\ar@{^{(}->}[r]_-{\alpha}&
  M_{X}^\circ\st{Y}\times_{\st{Y}}
  \Gamma\ar[r]\ar[d]\ar@{}[rd]|\square&
  M^{\circ}_X\st{Y}\ar[d]\\
 &\Gamma\ar[r]&\st{Y},
  }\hspace{50pt}
  \xymatrix{
  \Gamma'\times\mb{A}^1\ar[r]\ar[d]\ar@{}[rd]|\square&
  X\times\mb{A}^1\ar[d]\\
 M^\circ_{\Gamma'}\Gamma\ar[r]&M^\circ_{X}\st{Y}.
  }
\end{equation}
This diagram induces the cartesian diagram on the right.

Let us define $f^!([\Gamma])\in\mr{CH}_{n-d}(\Gamma')$. For the detail,
see [{\it ibid.}, 3.1, 5.1]. By taking the pull-back by the morphism
$X\rightarrow\st{Y}$ of the diagram above on the left, we have the
following diagram of {\em schemes}:
\begin{equation}
 \label{realvarfulintdef}
 \xymatrix{
  C_{\Gamma'}\Gamma\ar@{^{(}->}[r]&
  N'\ar[r]\ar[d]\ar@{}[rd]|\square&N_f\ar[d]\\
 &\Gamma'\ar[r]&X\ar@/_2ex/[u]_g
  }
\end{equation}
The closed immersion on the upper left is induced by $\alpha$ in
(\ref{cartdiagfulton}).
By definition, $f^!(\Gamma)$ is the image of $[C_{\Gamma'}\Gamma]$ by
the homomorphism
\begin{equation*}
 \mr{CH}_{m-d}(C_{\Gamma'}\Gamma)\rightarrow\mr{CH}_{m-d}(N')
  \xrightarrow[\sim]{g^!}\mr{CH}_{n-d}(\Gamma')
\end{equation*}
where the last isomorphism\footnote{This isomorphism holds even when $X$
is an admissible stack. Indeed, by \cite[3.5.7]{Kr}, an admissible stack
admits ``a stratification by global quotients''. Then by [{\it ibid.},
4.3.2], we have the required homotopy invariance property.}
follows by \cite[Theorem 3.3]{Fu}.

\begin{lem}[{\cite[2.1.2]{KS}}]
 \label{pullcalccyccl}
 We preserve the notation.
 Let $f^*\colon H^{2d}_{\Gamma}(\st{Y})\rightarrow
 H^{2d}_{\Gamma'}(X)$ be the pull-back. Then the class
 $\mr{cl}(f^!([\Gamma]))\in H^{2d}_{\Gamma'}(X)$ is equal to
 $f^*\mr{cl}(\Gamma)$.
\end{lem}
\begin{proof}
 The verification is essentially the same as \cite{KS}. We have the
 following commutative diagram:
 \begin{equation*}
  \xymatrix@C=40pt{
   H^{2d}_{\Gamma}(\st{Y})(d)\ar[d]_{f^*}&
   H^{2d}_{Z}(M^{\circ}_X\st{Y})(d)
   \ar[r]^{0^*}\ar[l]_{1^*}\ar[d]&
   H^{2d}_{C'}(N_f)(d)\ar[d]^{g^*}\\
  H^{2d}_{\Gamma'}(X)(d)&
   H^{2d}_{\Gamma'\times\mb{A}^1}(X\times\mb{A}^1)(d)
   \ar[r]^-{0^*}_-{\sim}\ar[l]_-{1^*}^-{\sim}&
   H^{2d}_{\Gamma'}(X)(d)
   }
 \end{equation*}
 where $Z:=M^{\circ}_{\Gamma'}\Gamma$, $C':=C_{\Gamma'}\Gamma$, and the
 middle vertical homomorphism is induced by the morphism
 $X\times\mb{A}^1\rightarrow M^{\circ}_X\st{Y}$ defined by the strict
 transform. At the upper row, the image of the cycle class $\mr{cl}(Z)$
 is sent to $\mr{cl}(\Gamma)$ and $\mr{cl}(C')$ by $1^*$ and $0^*$
 respectively by Lemma \ref{lemonpullcycle} and the flatness of
 $Z\rightarrow\mb{A}^1$.

 Recall the diagram of schemes (\ref{realvarfulintdef}).
 It is reduced to showing that $\mr{cl}(g^![C'])=g^*\mr{cl}(C')$. Using
 Lemma \ref{chowgroupfacto}, this amounts to proving the commutativity
 of the following diagram on the left:
 \begin{equation*}
  \xymatrix{
   \mr{CH}_{m-d}(N')\ar[d]_{g^!}\ar[r]&
   H^{2d}_{N'}(N_f)(d)\ar[d]^{g^*}\\
  \mr{CH}_{n-d}(\Gamma')\ar[r]&
   H^{2d}_{\Gamma'}(X)(d),
   }\qquad
   \xymatrix{
   \mr{CH}_{n-d}(\Gamma')\ar[r]\ar[d]_{p^*}&
   H^{2d}_{\Gamma'}(X)(d)\ar[d]^{p^*}\\
   \mr{CH}_{m-d}(N')\ar[r]&
    H^{2d}_{N'}(N_f)(d).
   }
 \end{equation*}
 Note that the stacks appearing in these diagrams are, in fact,
 schemes. Consider the diagram on the
 right above, where $p$ denotes the projection $N_f\rightarrow X$.
 Since $g^*\circ p^*$ is the identity on $H^{2d}_{\Gamma'}(X)(d)$, and
 $g^!$ is an isomorphism whose inverse is $p^*$ on the Chow groups,
 the verification of the commutativity of the left diagram is reduced to
 that of the right one. There exists an open dense subscheme $U\subset
 X$ such that $U\cap\Gamma'\subset\Gamma'$ is dense, and
 $N_f\times_XU\rightarrow U$ is a trivial bundle so that we can write
 $N_f\times_XU\cong U\times\mb{A}^n$. Since $H^{2d}_{\Gamma'}(X)(d)\cong
 H^{2d}_{\Gamma'\cap U}(U)(d)$, we may replace $X$ by $U$ and $\Gamma'$
 by $\Gamma'\cap U$. Thus, the claim follows by Lemma
 \ref{lemonpullcycle}.
\end{proof}

\subsubsection{}
\label{identificationofcycl}
Now, we only consider the case where $\base=\emptyset$
(but $k$ is still not necessarily finite).
Let us recall the construction of \cite[after Lemma 2.3.1]{KS}. Let
$\st{U}$ be a smooth c-admissible stack of dimension $d$, and let
$\Gamma$ be a correspondence on $\st{U}$ (over $\mr{Spec}(k)$). Let
$j\colon\st{U}\hookrightarrow\st{X}$ be a compactification. Since the
second projection $p_2\colon\Gamma\rightarrow\st{U}$ is assumed
proper, the morphism $(j\circ
p_1,p_2)\colon\Gamma\rightarrow\st{X}\times\st{U}$ is proper, and we
have isomorphisms
\begin{equation*}
 H^{2d}_{\Gamma}\bigl(\st{X}\times\st{U},
  (j\times\mr{id})_!L(d)\bigr)
  \xrightarrow{\sim}
  H^{2d}_{\Gamma}\bigl(\st{X}\times\st{U},L(d)\bigr)
  \xrightarrow{\sim}
  H^{2d}_{\Gamma}\bigl(\st{U}\times\st{U},L(d)\bigr).
\end{equation*}
Thus, the cycle class $\mr{cl}(\Gamma)$ defined in
$H^{2d}_{\Gamma}\bigl(\st{U}\times\st{U},L(d)\bigr)$
induces an element in
\begin{equation*}
 H^{2d}_{!,*}(\st{U}\times\st{U})(d):=H^{2d}\bigl(\st{X}\times\st{U},
  (j\times\mr{id})_!L(d)\bigr).
\end{equation*}
Since $\base=\emptyset$, this cohomology is isomorphic to
$\prod_{i}\mr{End}\bigl(H^i_{\mr{c}}(\st{U})\bigr)$ as in {\it ibid.}\
using the K\"{u}nneth formula (cf.\ Corollary
\ref{Kunnethformulastform}) and the Poincar\'{e} duality (cf.\ Theorem
\ref{relativedualprop}).

\begin{lem*}[{\cite[2.3.2]{KS}}]
 \label{identifcyclend}
 The action $\Gamma^*$ can be identified with the class
 $\mr{cl}(\Gamma)$ via this isomorphism.
\end{lem*}
\begin{proof}
 Using Lemma \ref{2.1.1ofKS}, the proof of {\it ibid.}\ works exactly in
 the same manner.
\end{proof}

\subsubsection{Proof of Theorem \ref{indepofltrace}}\mbox{}\\
By Corollary \ref{alterationforadmisst}, we can take a proper
generically finite surjective morphism $f\colon X\rightarrow\st{X}$ such
that $X$ is a smooth scheme. Using the same corollary and Corollary
\ref{tracepropmuln}, we may assume $\Gamma$ to be a scheme.
Let $H^*_{\mr{c}}(\st{X})$ be
$H^*_{\mr{c}}(\st{X}\otimes\overline{k},\mb{Q}_\ell)$ or
$H^*_{\mr{c}}(\st{X},L_{\st{X}})$. Using Lemma
\ref{identifcyclend} and Corollary \ref{commutabasechancorre}, by
arguing as \cite[2.3.3]{KS}, we have
\begin{equation}
 \label{altrtrrelcalc}
  \tag{$\star$}
 \mr{Tr}\bigl(\Gamma^*: H^*_{\mr{c}}(\st{X})\bigr)=
  \deg(f)^{-1}\cdot\mr{Tr}\bigl((f\times f)^*
  \mr{cl}_{\st{X}\times\st{X}}(\Gamma):
  H^*_{\mr{c}}(X)\bigr),
\end{equation}
where we regard classes in $H^{2d}_{!,*}(\st{X}\times\st{X})(d)$
($d:=\dim(\st{X})$) as endomorphisms of $H^*_{\mr{c}}(\st{X})$ using
Lemma \ref{identificationofcycl}.
Consider the following commutative diagram:
\begin{equation*}
 \xymatrix{
  \Gamma'\ar[r]\ar@{}[rd]|{\square}\ar[d]&X\times X\ar[r]^-{\pi_2}
  \ar[d]^{f\times f}&X\ar[d]^f\\
 \Gamma\ar[r]&\st{X}\times\st{X}\ar[r]_-{\pi_2}&\st{X}.
  }
\end{equation*}
Since the composition of the horizontal morphisms below is proper by
assumption, the composition morphism from $\Gamma'$ to $\st{X}$ is
proper. Thus the composition of the horizontal morphisms of the first
row is proper.
By Lemma \ref{pullcalccyccl}, $(f\times
f)^*\mr{cl}_{\st{X}\times\st{X}}(\Gamma)$ is equal to
$\mr{cl}_{X\times X}\bigl((f\times f)^!(\Gamma)\bigr)$.
Let $\widetilde{\Gamma}'$ be the correspondence defined by $(f\times
f)^!(\Gamma)$. By applying Lemma
\ref{identifcyclend} once again, the trace of the right hand side of
(\ref{altrtrrelcalc}) is equal to
$\mr{Tr}\bigl(\widetilde{\Gamma'}\colon H^*_{\mr{c}}(X)\bigr)$. This
implies that it suffices to show the theorem in the case where
$\st{X}=:X$ and $\Gamma$ are schemes. By Corollary \ref{genconstdegrel},
we may replace $\Gamma$ by its image in $X\times X$, and assume that
$\Gamma\subset X\times X$. In this case, we just repeat the argument of
\cite[Prop 2.3.6]{KS}. Further details are left to the reader.
\qed

\section{Langlands correspondence}
\label{langcorrsec}
In this final section, we establish the Langlands correspondence, and in
particular, prove the existence of {\it petits camarades cristallins}
for curves. We try this section to be as independent as possible.

\subsection{Preliminaries}
First, let us very briefly recall basic notions of the $p$-adic
cohomology theory and notations of this paper for the convenience of the
reader. We do not have new inputs, and those who have read the previous
sections may skip.

\subsubsection{}
Let $k$ be a perfect field, $R$ be a complete discrete valuation rings
whose residue field is $k$, and $K$ be the field of fractions of $R$.
We assume further that the $s$-th Frobenius automorphism of $k$ can be
lifted to an automorphism $\sigma\colon R\xrightarrow{\sim}R$. The
induced automorphism between $K$ is also denoted by $\sigma$.
With this setup, let $X$ be a scheme of finite type over $k$. Berthelot
defined the category of {\em overconvergent isocrystals} (resp.\ {\em
overconvergent $F$-isocrystals}) denoted by
$\mr{Isoc}^\dag(X/K)$ (resp.\ $F\mbox{-}\mr{Isoc}^\dag(X/K)$).
We do not try to recall the definition here, but a standard
reference is \cite{Berrigcoh}. We may find other references in
\cite{Kesurv} and \cite{Kenote}.
This category is a $p$-adic analogue of the category of smooth
$\mb{Q}_\ell$-sheaves over $X\otimes_k\overline{k}$ (resp.\ over
$X$). In this paper, we denote $F\mbox{-}\mr{Isoc}^\dag(X/K)$
by $\mr{Isoc}^\dag(X/K_F)$.
The description of these categories when $X=\mr{Spec}(k)$ is
simple but important: $\mr{Isoc}^\dag(\mr{Spec}(k)/K)$
is canonically equivalent to the category of finite dimensional
$K$-vector spaces, and $\mr{Isoc}^\dag(\mr{Spec}(k)/K_F)$
is canonically equivalent to that of finite dimensional $K$-vector
spaces $V$ equipped with an isomorphism
$K\otimes_{\sigma,K}V\xrightarrow{\sim}V$. In particular, the trivial
vector space $K$ with trivial isomorphism determines an object of
$\mr{Isoc}^\dag(\mr{Spec}(k)/K_F)$, which is denoted by $K$ abusing the
notation.

Let $f\colon X\rightarrow Y$ be a morphism between schemes of
finite type over $k$. Then the pull-back functor
$f^+\colon\mr{Isoc}^\dag(Y/K)\rightarrow\mr{Isoc}^\dag(X/K)$
(resp.\ $f^+\colon\mr{Isoc}^\dag(Y/K_F)
\rightarrow\mr{Isoc}^\dag(X/K_F)$)
is defined in \cite[2.3.2 (iv)]{Berrigcoh}
\footnote{The pull-back functor is denoted by $f^*$ in {\it ibid.}.
}. If $p$ is the structural
morphism of $X$, we put $K_X:=p^+K$, which is also denoted by $K$.
The category is equipped with tensor product $\otimes$, with which
$\mr{Isoc}^\dag(X/K)$ (resp.\ $\mr{Isoc}^\dag(X/K_F)$) forms a tensor
category (cf.\ \cite[2.3.3 (iii)]{Berrigcoh}). The unit object of the
tensor category is $K_X$. It also
possesses an internal hom functor $\shom$. Finally, we have the notion
of ranks (cf.\ \cite[2.3.3 (ii)]{Berrigcoh}). This implies that any
object in $\mr{Isoc}^\dag(X/K)$ (resp.\ $\mr{Isoc}^\dag(X/K_F)$) is of
finite length.

\begin{rem*}
 The category $\mr{Isoc}^\dag(X/K)$ used in
 \S\ref{firstsection}--\ref{lindepsec} is smaller than the one recalled
 here, but the category with Frobenius $\mr{Isoc}^\dag(X/K_F)$ is the
 same. See, the paragraph right after
 \ref{fundproprealsch} \eqref{frobdesrecall}.
 In the following, we only use $\mr{Isoc}^\dag(X/K_F)$, so this
 does not cause any problem.
\end{rem*}

\subsubsection{}
In Langlands correspondence, we consider the case where $k$ is a
finite field with $q=p^s$ elements, $R:=W(k)$, and $\sigma\colon
R\xrightarrow{\sim}R$ is the identity.
Then we may consider the category $\mr{Isoc}^\dag(X/K_F)$ which is
$K$-abelian. Note that this would be $K^{\sigma=1}$-abelian if $\sigma$
were not the identity. Now, we need to extend the scalar from $K$ to
$\overline{\mb{Q}}_p$. This was done in
\ref{defofcatwithcoeff} and \ref{algcloscoefftheory}, and let us
recall the idea briefly\footnote{Actually the ``definition'' presented
here is the one in Remark \ref{Frobsetupp}.}:
For a finite extension $L/K$, we define $\mr{Isoc}^\dag(X/L_F)$ to be
the category of pairs $(\mc{E},\rho)$ where
$\mc{E}\in\mr{Isoc}^\dag(X/K_F)$ and $\rho\colon
L\rightarrow\mr{End}(\mc{E})$ be a homomorphism of
$K$-algebras, and the morphisms are defined in the obvious
way. To define $\mr{Isoc}^\dag(X/\overline{\mb{Q}}_{p,F})$, we take the
2-inductive limit of $\mr{Isoc}^\dag(X/L_F)$ over all finite extension
$L$ of $K$.
We have the scalar extension functor $\otimes\overline{\mb{Q}}_p\colon
\mr{Isoc}^\dag(X/K_F)\rightarrow\mr{Isoc}^\dag(X/\overline{\mb{Q}}_{p,F})$.
We remark that even though $\mc{E}\in\mr{Isoc}^\dag(X/K_F)$ is
irreducible, $\mc{E}\otimes\overline{\mb{Q}}_p$ may not be irreducible
in general, and this is why we needed to extend the scalar.
For a morphism $f\colon X\rightarrow Y$, the pull-back functor can
formally be extended to
$f^+\colon\mr{Isoc}^\dag(Y/\overline{\mb{Q}}_{p,F})\rightarrow
\mr{Isoc}^\dag(X/\overline{\mb{Q}}_{p,F})$, and similarly for $\otimes$,
$\shom$. The data $\mf{T}:=(k,R,K,L,s,\sigma=\mr{id})$ (where $\sigma$
is an extension to $L$ of a lifting of $s$-th Frobenius automorphism on
$k$ to $K$) we used to define $\mr{Isoc}^\dag(X/L_F)$ is called the base
tuple. To clarify the base, we also use the notation
$\mr{Isoc}^\dag(X/\mf{T})$ for $\mr{Isoc}^\dag(X/L_F)$.
See \ref{5tuplesover}, \ref{algcloscoefftheory} for details.

\subsubsection{}
Before going to the next section, let us briefly recall what we have
done so far. Let $\st{X}$ be a scheme, or more generally,
algebraic stack of finite type over $k$. We constructed a triangulated
category $D^{\mr{b}}_{\mr{hol}}(\st{X}/\overline{\mb{Q}}_{p,F})$  over
$\overline{\mb{Q}}_p$ with a natural t-structure $\H$ in Definition
\ref{dfnofcat}. When $\st{X}$ is a smooth separated scheme over $k$,
$\mr{Isoc}^\dag(\st{X}/\overline{\mb{Q}}_{p,F})$ is fully faithfully
embedded into $D^{\mr{b}}_{\mr{hol}}(\st{X}/\overline{\mb{Q}}_{p,F})$ by
\ref{relbetourrig}. There is another t-structure on
$D^{\mr{b}}_{\mr{hol}}(\st{X}/\overline{\mb{Q}}_{p,F})$: the
constructible t-structure $\cH$ defined in Definition
\ref{Artinstackcdctstr}. Philosophically, $\H$ corresponds
to the perverse t-structure in the $\ell$-adic setting, and $\cH$
corresponds to the standard (constructible) t-structure.
In this section, we mostly use $\cH$
since its heart, denoted by $\mr{Con}(\st{X}/\overline{\mb{Q}}_{p,F})$,
contains $\mr{Isoc}^\dag(\st{X}/\overline{\mb{Q}}_{p,F})$ when $\st{X}$
is a smooth separated scheme. Objects of $\mr{Con}$ are called
{\em constructible objects}.
Over the category of ``compactifiable admissible stacks (c-admissible
stacks)'' , we have six functor formalism:
$D^{\mr{b}}_{\mr{hol}}(\st{X}/\overline{\mb{Q}}_{p,F})$ is endowed with
tensor and dual functors. Given a morphism
$f\colon\st{X}\rightarrow\st{Y}$, we have
\begin{equation*}
 f_+,f_!\colon D^{\mr{b}}_{\mr{hol}}(\st{X}/\overline{\mb{Q}}_{p,F})
  \rightarrow
  D^{\mr{b}}_{\mr{hol}}(\st{Y}/\overline{\mb{Q}}_{p,F}),\quad
  f^+,f^!\colon D^{\mr{b}}_{\mr{hol}}(\st{Y}/\overline{\mb{Q}}_{p,F})
  \rightarrow
  D^{\mr{b}}_{\mr{hol}}(\st{X}/\overline{\mb{Q}}_{p,F}),
\end{equation*}
and satisfy standard properties. Here, $f_+$ and $f^+$ are analogues of
$f_*$ and $f^*$ in the $\ell$-adic theory, and we adopted these
notations in order to follow the tradition of the theory of
$\ms{D}$-modules. We do not need this far in the statement of Langlands
correspondence, but these techniques are required in the proof.

\subsection{Langlands correspondence}
The aim of this subsection is to state the main theorem of this paper,
namely the Langlands correspondence, and overview the strategy of
the proof.

\label{langcorresstate}
\subsubsection{}
\label{situationfix}
First of all, let us fix the basis (cf.\ \ref{5tuplesover},
\ref{algcloscoefftheory}).
We assume $k$ to be a finite field with $q=p^s$ elements.
We fix an algebraic closure $\overline{\mb{Q}}_p$ of $K$, and denote by
$\overline{k}$ the residue field of $\overline{\mb{Q}}_p$ which is
algebraically closed as well.
We define an arithmetic base tuple
$\mf{T}_k:=(k,R:=W(k),K:=\mr{Frac}(R),
\overline{\mb{Q}}_p,s,\sigma:=\mr{id})$.
Likewise, for a finite extension $k'$ of $k$ in $\overline{k}$, we put
$\mf{T}_{k'}:=(k',W(k'),\mr{Frac}(R'),
\overline{\mb{Q}}_p,[k':k]\cdot s,\mr{id})$.
With these data, we may consider the
$\overline{\mb{Q}}_p$-coefficient
cohomology theory (cf.\ \ref{algcloscoefftheory}), which we mainly use.
For any finite extension $k'$ of $k$, the category
$\mr{Isoc}^\dag(k'/\mf{T}_{k'})$ is equivalent to 
the category of finite dimensional $\overline{\mb{Q}}_p$-vector spaces
$V$ endowed with isomorphism $\mr{id}^*(V)\xrightarrow{\sim}V$ (cf.\
Definition \ref{defofcatwithcoeff}).
By identifying $\mr{id}^*(V)$ with $V$, we view the isomorphism as an
automorphism of $V$.

Let $X$ be a smooth scheme over $k$, and let $i_x\colon
x\hookrightarrow X$ be a closed point of $X$. Then
$\mf{T}_{k(x)}=(k(x),W(k(x)),K_x,\overline{\mb{Q}}_p,s',\mr{id})$ where
$s':=[k(x):k]\cdot s$.
Let us define the ``linearized Frobenius automorphism at $x$''.
Choose a geometric point $\overline{x}\in X(\overline{k})$ lying above
$x$. This defines an embedding $K_x\hookrightarrow\overline{\mb{Q}}_p$,
and we have the following functor
\begin{equation*}
 \iota_{\overline{x}}\colon
 \mr{Isoc}^\dag(X/\mf{T}_k)\xrightarrow{i_x^+}
  \mr{Isoc}^\dag(k(x)/\mf{T}_k)\cong
  \mr{Isoc}^\dag(k(x)/\mf{T}_{k(x)})
\end{equation*}
where the equivalence follows by Corollary \ref{behaviorofbasech} using
the embedding.
For $\mc{E}\in\mr{Isoc}^\dag(X/\mf{T}_k)$, the equipped
automorphism on $\iota_{\overline{x}}(\mc{E})$ is called the
{\em linearized geometric Frobenius automorphism at $x$ of $\mc{E}$}.
The inverse of the linearized geometric Frobenius automorphism is
denoted by $\mr{Frob}_x$ and is simply called the
{\em Frobenius automorphism at $x$}.
The multiset of eigenvalues of $\mr{Frob}_x$ acting on
$\iota_{\overline{x}}(\mc{E})$ depends only on the choice of $x$ and not
on $\overline{x}$. By abuse of language, we call this multiset the {\em
set of Frobenius eigenvalues at $x$}.

\begin{rem*}
 We defined $\mr{Frob}_x$ so that the notation is compatible with that
 of Lafforgue. In the $\ell$-adic theory, the corresponding automorphism
 is sometimes called the arithmetic Frobenius.
\end{rem*}

\begin{thm}[Langlands correspondence for isocrystals]
 \label{main}
 We fix an isomorphism $\overline{\mb{Q}}_p\cong\mb{C}$.
 Let $X$ be a geometrically connected proper smooth curve over
 $k$. Denote by $F$ the function field of $X$, and let
 $\mb{A}_{F}$ be the ring of ad\`{e}les. For an integer $r\geq1$,
 consider the following two sets:
 \begin{center}
  \begin{tabular}{lp{30em}}
   $\mc{I}_r$:& The set of isomorphism classes of irreducible
       isocrystals of rank $r$ in
       $2\text{-}\indlim\,\mr{Isoc}^\dag(U/\mf{T}_{k})$, where the
       limit runs over open subschemes $U\subset X$, such that the
       determinant is of finite order.\\
   $\mc{A}_r$:& The set of isomorphism classes of cuspidal automorphic
       representations $\pi$ of $\mr{GL}_r(\mb{A}_{F})$ such that
       the order of the central character of $\pi$ is finite.\\
  \end{tabular}
 \end{center}
 \begin{enumerate}
  \item\label{maincorres}
       There exist maps
       \begin{equation*}
	\mc{E}_\bullet\colon\mc{A}_r\rightleftarrows
	 \mc{I}_r\colon \pi_{\bullet}
       \end{equation*}
       with which $\mc{A}_r$ and $\mc{I}_r$ correspond in the sense of
       Langlands:
       for $\pi\in\mc{A}_r$ (resp.\ $\mc{E}\in\mc{I}_r$),
       the sets of unramified places of $\pi$ (resp.\ $\mc{E}$) and
       $\mc{E}_\pi$ (resp.\ $\pi_{\mc{E}}$) coincide, which we denote
       by $U$, and for any $x\in |U|$, the set of Frobenius
       eigenvalues of $\mc{E}_\pi$ (resp.\ $\mc{E}$) at $x$ and that
       of Hecke eigenvalues of $\pi$ (resp.\ $\pi_{\mc{E}}$) at
       $x$ coincide.

  \item\label{epsilonlocalL}
       Assume that $\pi^{(\prime)}\in\mc{A}_{r^{(\prime)}}$ and
       $\mc{E}^{(\prime)}\in\mc{I}_{r^{(\prime)}}$ correspond in the
       sense of Langlands. Then the local $L$-functions and local
       $\varepsilon$-factors of pairs $(\pi,\pi')$ and
       $(\mc{E},\mc{E}')$ coincide for any point $x\in |X|$.
       (cf.\ \cite{ALang} and \cite[VI.9 (ii)]{Laf})
 \end{enumerate}
\end{thm}

\begin{rem*}
 The correspondence is unique if it exists.
\end{rem*}

\subsubsection{}
\label{thmdeligneprinciple}
For a proof of the theorem, we follow the program of Drinfeld and
Lafforgue. We shortly recall the outline of the proof to introduce some
notations we use in the next subsection. The idea is explained
in detail and clearly in the introduction of \cite{Laf}, so we
encourage the reader who are not familiar with Lafforgue's proof to read
through it before entering our proof.

In \cite{ALang}, with the help of the product formula proven in
\cite{AM}, we have the following theorem, which is nothing but the
$p$-adic version of ``{\it principe de r\'{e}currence}'' by Deligne:

\begin{thm*}[{\cite[\S5]{ALang}}]
 Let $n$ be a positive integer, and assume Theorem \ref{main} is known
 for $r,r'\leq n$, then we have the map
 $\mc{I}_{n+1}\rightarrow\mc{A}_{n+1}$ in the sense of Langlands such
 that the corresponding cuspidal representation is unramified at the
 places where the isocrystal is. Moreover, if we have a map
 $\mc{A}_{n+1}\rightarrow\mc{I}_{n+1}$ in the sense of Langlands such
 that the corresponding isocrystal is unramified at the places where the
 cuspidal representation is, then Theorem \ref{main} holds for $r',r\leq
 n+1$. In other words, \eqref{epsilonlocalL} of the theorem holds
 automatically once we prove \eqref{maincorres}.
\end{thm*}

\subsubsection{}
\label{tablechoutsatis}
Thanks to the theorem above, our task is only to construct a map
$\mc{A}_r\rightarrow\mc{I}_r$ such that the corresponding isocrystal is
unramified at the places where the cuspidal representation
is. A rough idea is to realize this as the relative
cohomologies of moduli spaces of ``shtukas'' {\em \'{a} la} Drinfeld.
Even though the moduli spaces we use here are the same as that of
Lafforgue, we take $p$-adic cohomologies that we have developed in the
preceding sections instead of $\ell$-adic cohomologies to carry this
out.

In this paper, we use the following various types of moduli spaces of
shtukas. We remind that the notation is the same as that of Lafforgue.
Let $r$ be a positive integer. Let $N$ be a {\em level}, {\it
i.e.}\ a closed subscheme $N=\mr{Spec}(\mc{O}_N)\hookrightarrow X$
which is not equal to $X$, $p\colon[0,r]\rightarrow\mb{R}$ be a convex
polygon, and $a\in\mb{A}^{\times}_{F}$ of degree $1$.
Given these data, we have c-admissible stacks
(cf.\ Definition \ref{defcompactcadmmorph})
over the surface $(X-N)\times(X-N)$ as follows:
\begin{center}
 \begin{threeparttable}
 \begin{tabular}{|lc||c|c|c|c|}
  \hline
  &Moduli space&smooth?&proper?&correspondence&Reference\\
  \hline\hline
  $\ccirc{1}$&
      \parbox[c][0.9cm][c]{0cm}{}
      $\Ch^{r,\overline{p}\leq p}_N/a^{\mb{Z}}$&\Circle&\Cross&
		  \TriangleUp\tnote{1}&
		      \parbox[c]{8em}{\cite[right after Prop I.3]{Laf}}
		      \\\hline
  $\ccirc{2}$&
      \parbox[c][0.9cm][c]{0cm}{}
      $\Chtprop{r,\overline{p}\leq p}{N}/a^{\mb{Z}}$&
	  \Cross&\Circle&$-$&
		      [{\it ibid.}, Def III.8]
		      \\\hline
  $\ccirc{3}$&
      \parbox[c][0.9cm][c]{0cm}{}
      $\Chtpr{r,\overline{p}\leq p}{N}/a^{\mb{Z}}$&
	  \Circle&\Cross&\TriangleUp\tnote{2}&
		      \parbox[c]{8em}{[{\it ibid.}, Cor III.14, Thm
		      V.14]}
		      \\\hline
 \end{tabular}
  \begin{tablenotes}[flushleft]
   \footnotesize
   \item[1] We have the action of Hecke algebra
   only after taking the inductive limit over the convex polygon $p$.

   \item[2] We have the action of Hecke algebra for each element, but the
   action may not be compatible with the product structure.
  \end{tablenotes}
 \end{threeparttable}
\end{center}
In this table, the second column (resp.\ the third column) refers to the
smoothness (resp.\ properness) over $(X-N)\times(X-N)$ of the
corresponding moduli spaces in the first column, and definitions
and proofs of the properties listed here can be found in
``Reference''. These stacks are c-admissible by \cite[V.1]{Laf}, or more
precisely, in
the first line of its proof, it is said that these stacks are
quasi-projective over $\Chtprop{r,d,\overline{p}\leq p}{}$, and the
last stack is serene (cf.\ \cite[Appendix A]{Laf}) which is proper over
$X\times X$.
The components of these spaces are indexed by integers $1\leq d\leq r$
called the {\em degree}. The component corresponding to $d$ is denoted
by $\Ch^{r,d,\overline{p}\leq p}_N$, $\Chtprop{r,d,\overline{p}\leq
p}{N}$, $\Chtpr{r,d,\overline{p}\leq p}{N}$.

\subsubsection{}
Let $f\colon\st{X}\rightarrow(X-N)\times(X-N)$ be one of the three
moduli spaces of shtukas. Then
$\mc{H}_{\mr{c}}^*:=f_!\overline{\mb{Q}}_{p,\st{X}}$ contains the
isocrystals which correspond to cuspidal representations in the set
$\{\pi\}^r_N$ (cf.\ \ref{defofsomenotepi}). However, it also contains a
lot of ``junk'' which have already appeared in the Langlands
correspondence of lower ranks and we need to throw these away.
The junk is called the ``{\em $r$-negligible part}'',
and the part we need for the correspondence is called the
``{\em essential part}''.
We first need to show that the essential part is concentrated at a
certain degree of $\mc{H}_{\mr{c}}^*$.
For this, we need to use the purity of intersection cohomology, and we
need the compact space $\ccirc{2}$. Still, the essential part is
a mixture of isocrystals corresponding to $\{\pi\}^r_N$, and we need to
extract the particular isocrystal which corresponds to a given cuspidal
representation $\pi\in\{\pi\}^r_N$.
For this, we need to define an action of the Hecke algebra
$\mc{H}^r_N$. We have ring homomorphism from $\mc{H}^r_N$ to the ring of
correspondences on the moduli space $\ccirc{1}$ if we pass to the limit
of $p$. Since we are passing to the limit to define the action,
the resulting stack is not of finite type anymore. For the calculation
of the trace of the action of correspondences, we use $\ccirc{3}$.
We note that even though we have the correspondences associated to
elements of the Hecke algebra on $\ccirc{3}$, this map might not be a
homomorphism of rings.
Finally, we use the $\ell$-independence result to calculate the trace,
and extract exactly the information we need.

\subsection{Proof of the theorem}
\subsubsection{}
In this subsection, the base tuple $\mf{T}_k$ is fixed as
\ref{situationfix}, and we also fix an isomorphism
$\iota\colon\overline{\mb{Q}}_p\cong\mb{C}$ as in Theorem \ref{main}.
Let $Y$ be a smooth scheme of finite type and
geometrically connected over $k$.
We usually omit ``$/\mf{T}_k$'' ({\it e.g.}\
$\mr{Isoc}^\dag(Y)$ instead of
$\mr{Isoc}^\dag(Y/\mf{T}_k)$).
We denote the category
$\mr{Isoc}^\dag(Y)$ of
overconvergent $\overline{\mb{Q}}_p$-isocrystals with Frobenius
structure by $\mc{I}(Y)$ to shorten the notations.
We identify $\mc{I}(Y)=\mr{Isoc}^\dag(Y)$ and
$\mr{Sm}(Y)\subset D^{\mr{b}}_{\mr{hol}}(Y)$ via
$\widetilde{\mr{sp}}_+$ defined in \ref{relbetourrig}.
Because of this identification, we often say $\mc{E}\in\mr{Con}(Y)$ is
{\em smooth} if it comes from $\mc{I}(Y)$.
Let $p\colon\st{X}\rightarrow\mr{Spec}(k)$ be the structural morphism of
a c-admissible stack. For $\mc{E}\in\mr{Isoc}^\dag(\st{X})$ or more
generally an object of $D^{\mr{b}}_{\mr{hol}}(\st{X})$, we put
$H^*(\st{X},\mc{E}):=\H^*p_+(\mc{E})$ and
$H^*_{\mr{c}}(\st{X},\mc{E}):=\H^*p_!(\mc{E})$ and regard these as
$\overline{\mb{Q}}_p$-vector spaces with automorphism.
We abbreviate $\iota$-pure (resp.\ $\iota$-mixed, $\iota$-weight, {\it
etc.})\ simply by pure (resp.\ mixed, weight, {\it etc.}).

When we say $\ell$-adic sheaf, it refers to $\ell$-adic Weil sheaf (cf.\
\cite[1.1.10]{De}). The category of smooth Weil
sheaves is denoted by $\mc{W}_\ell(Y)$.
For a scheme $X$ over $k$, we denote by $\mr{Frob}_X\colon X\rightarrow
X$ the absolute Frobenius endomorphism; $f\in\mc{O}_X$ is sent to
$f^q$.
For an abelian category $\mc{A}$, we denote by $\Gr(\mc{A})$
the Grothendieck group of $\mc{A}$, and
$\mb{Q}\Gr(\mc{A}):=\Gr(\mc{A})\otimes\mb{Q}$. For an object
$X\in\mc{A}$ of finite length, we denote by $X^{\mr{ss}}$ the
semi-simplification of $X$, namely the direct sum of constituents of
$X$.

\subsubsection{}
\label{Lfuncreveiw}
Let $X$ be a smooth scheme of finite type over $k$, and $\mc{E}$ be in
$\mc{I}(X)$. Take a closed point $x\in |X|$.
We take a geometric point $\overline{x}\in X(\overline{k})$ which lies
above $x$, and recall the functor $\iota_{\overline{x}}$ defined in
\ref{situationfix}. The {\em local $L$-function at $x$} is defined to be
\begin{equation*}
 L_x(\mc{E},Z):=\det\bigl(1-Z^{\deg(x)}\mr{Frob}^{-1}_x;
  \iota_{\overline{x}}(\mc{E})
  \bigr)^{-1}
\end{equation*}
in $\overline{\mb{Q}}_p\dd{Z}$, which does not depend on the choice of
$\overline{x}$.
Using the fixed isomorphism $\overline{\mb{Q}}_p\cong\mb{C}$, we usually
consider this series as a series in $\mb{C}\dd{Z}$.
The {\em global $L$-function} is defined as
\begin{equation*}
 L_X(\mc{E},Z):=\prod_{x\in|X|}L_x(\mc{E},Z).
\end{equation*}
Analogous to Grothendieck's formula, this $L$-function has the following
cohomological interpretation:
\begin{equation}
 \label{LTFforpcoh}
  L_X(\mc{E},Z)=\prod_{\nu=0}^{2\dim(X)}
  \det\bigl(1-Z\cdot\mr{Frob}_X;
  H^\nu_{\mr{c}}(X,\mc{E})
  \bigr)^{(-1)^{\nu+1}},
\end{equation}
which is an identity of formal power series (cf.\ Corollary
\ref{cohintGroth}).
We refer to \S\ref{LFTforsch} for further details.

\subsubsection{}
\label{conseqweighidenti}
We use the theory of weights. We refer to
\ref{theoryofweightsetup}--\ref{theoweiarst},
\ref{theoryofweightadmiss} for more details.
Let $t\in\mb{C}$. Using $\iota$, we may consider $q^t$ as an element
in $\overline{\mb{Q}}_p$. We have the automorphism of
$\overline{\mb{Q}}_p$, considered as a $\overline{\mb{Q}}_p$-vector
space, sending $1$ to $q^{-t}$, which defines an object
in $\mc{I}(\mr{Spec}(k))$ denoted by
$\overline{\mb{Q}}_p(t)$. This is of weight $-2\,\mr{Re}(t)$.
When $t$ is an integer, the notation is compatible with Tate twists
(cf.\ \ref{tatetwistdfn}).
The following standard consequences of the theory of weights are
important tools in the proof of Langlands correspondence:

\begin{prop*}[{\cite[VI.3]{Laf}}]
 Let $Y$ be a smooth geometrically connected scheme of finite type
 over $k$. Let $\mc{E}\in\mc{I}(Y)$ be mixed of weight
 $\leq n$, and $\mc{E}'$ be an irreducible object in $\mc{I}(Y)$ pure of
 weight $m$.

 (i) The rational function $L_Y(\mc{E}\otimes\mc{E}'^\vee,Z)$ does not
 have zeros in the region $|Z|<q^{\frac{m-n}{2}-\dim(Y)+\frac{1}{2}}$.

 (ii) Let $t\in\mb{C}$ such that $\mr{Re}(t)=(m-n)/2$, and assume that
 the multiplicity of $\mc{E}'(t)$ in the semi-simplification of $\mc{E}$
 is $\mu\geq0$. Then the $L$-function appeared in (i) has a pole of
 order $\mu$ at $Z=q^{t-\dim(Y)}$. Moreover, it does not have any pole
 in $|Z|<q^{\frac{m-n}{2}-\dim(Y)}$.
\end{prop*}
\begin{proof}
 Use Theorem \ref{theoryofweightadmiss} and (\ref{LTFforpcoh}) to show
 (i). We have
 \begin{align*}
  \mr{Hom}_{\mc{I}(Y)}(\mc{E},\mc{E}'(t))&\cong
  \bigl(H^0\bigl(Y,\mc{E}^\vee\otimes\mc{E}'(t)
  \bigr)\bigr)^{\mr{Frob}_Y}\\
   &\cong
   \bigl(H_{\mr{c}}^{2\dim(Y)}\bigl(Y,\mc{E}\otimes
   (\mc{E}'(t))^\vee\bigr)(\dim(Y))\bigr)^{\mr{Frob}_Y},
 \end{align*} 
 where $(-)^{\mr{Frob}_Y}$ denotes the fixed part by the action of
 $\mr{Frob}_Y$, the first isomorphism is by \ref{basicprophomanddfn},
 and the second by Theorem \ref{poincaredualrela}.
 If $\mc{E}$ is semi-simple, then the action of $\mr{Frob}_Y$ on
 $H^0\bigl(Y,\mc{E}^\vee\otimes\mc{E}'(t)\bigr)$ is semi-simple, thus
 the fixed part of $\mr{Frob}_Y$ is equal to the generalized eigenspace
 of $\mr{Frob}_Y$ with eigenvalue $1$, and we get (ii).
\end{proof}

We usually use this proposition in the following form:

\begin{cor*}
 Let $Y$ be as in the proposition.
 Let $\mc{E}$ be in $\Gr(\mc{I}(Y))$, and $\mc{E}'$ be an
 irreducible object in $\mc{I}(Y)$.
 Assume that any component of $\mc{E}$ is mixed of weight $\leq n$. For
 $t\in\mb{C}$ such that $\mr{Re}(t)=(m-n)/2$, the multiplicity of
 $\mc{E}'(-t)$ in $\mc{E}$ is exactly the order of pole of
 $L_Y(\mc{E}\otimes\mc{E}'^\vee,Z)$ at $Z=q^{t-\dim(Y)}$.
\end{cor*}

\begin{dfn}
 Let $Y$ be a smooth scheme of finite type and geometrically connected
 over $k$, and let $U\subset Y\times Y$ such that
 $(\mr{Frob}_Y\times\mr{id}_Y)(U)\subset U$.
 We denote by $\mb{Z}\mc{I}(U)$ the category of overconvergent
 $\overline{\mb{Q}}_p$-isocrystals on $U$ with Frobenius structure
 equipped with an isomorphism
 $(\mr{Frob}_Y\times\mr{id}_Y)^+\mc{E}\xrightarrow{\sim}\mc{E}$.
 Let $F$ be the function field of $Y$. We put
 \begin{equation*}
  \mc{I}(F):=2\text{-}\indlim_{U\subset Y}\mc{I}(U),\qquad
   \mb{Z}\mc{I}(F^2):=2\text{-}
   \indlim_{U\subset Y\times Y}\mb{Z}\mc{I}(U).
 \end{equation*}
\end{dfn}

\begin{rem*}
 Take a geometric point $\overline{y}\in Y(\overline{k})$.
 Using the notation of \ref{defofisocweilgroup}, $\mc{I}(Y)$ is
 equivalent to $\mr{Rep}_{\overline{\mb{Q}}_p}(\WI(Y,\overline{y}))$,
 the category of finite dimensional representations of the algebraic
 group $\WI(Y,\overline{y})$ over
 $\overline{\mb{Q}}_p$. Similarly, we have $\mb{Z}\mc{I}(U)\cong
 \mr{Rep}_{\overline{\mb{Q}}_p}(\mb{Z}\WI(U,\overline{y}))$ by
 \ref{surjectwisweilgroup}.
 We often ignore the base points of $\WI$ and $\mb{Z}\WI$.
\end{rem*}

\subsubsection{}
We preserve the notation, and $q', q''\colon Y\times Y\rightarrow Y$ be
the first and second projection respectively.

\begin{dfn*}[{\cite[VI.14]{Laf}}]
 Let $r\geq1$ be an integer.
 An object $\mc{E}\in\mc{I}(F^2)$
 (resp.\ element of $\mb{Q}\Gr\mc{I}(F^2)$)
 is said to be {\em $r$-negligible} if
 any of its subquotient (resp.\ any of its component)
 is a direct factor of an object of the form
 $q'^+\mc{E}'\otimes q''^+\mc{E}''$ where $\mc{E}'$ and $\mc{E}''$ are
 objects of rank $<r$ in $\mc{I}(F)$. It is said to be {\em essential}
 if all the subquotients are not $r$-negligible. A semi-simple
 $r$-negligible object of $\mc{I}(F^2)$ (resp.\ an $r$-negligible
 element of $\mb{Q}\Gr\mc{I}(F^2)$) is said to be {\em
 complete} if it is a direct sum (resp.\ sum) of objects of the form
 $q'^+\mc{E}'\otimes q''^+\mc{E}''$.
\end{dfn*}

\begin{lem}
 (i) Let $\mc{E}'$, $\mc{E}''$ be irreducible objects in
 $\mc{I}(F)$. Then $q'^+\mc{E}'\otimes q''^+\mc{E}''$ is irreducible as
 an object in $\mb{Z}\mc{I}(F^2)$.

 (ii) A semi-simple $r$-negligible object, or an $r$-negligible element
 of $\mb{Q}\Gr\mc{I}(F^2)$, $\mc{E}$ is
 complete if it is invariant under the action of
 $(\mr{Frob}_Y\times\mr{id}_Y)^+$, namely if there exists an isomorphism
 $(\mr{Frob}_Y\times\mr{id}_Y)^+(\mc{E})\cong\mc{E}$. In particular, for a
 semi-simple $r$-negligible object $\mc{E}$,
 $\bigoplus_{n=1}^{r^2!}(\mr{Frob}^n_Y\times\mr{id}_Y)^+(\mc{E})$ is
 complete.
\end{lem}
\begin{proof}
 Let us prove (i). We may assume $\mc{E}'$ and $\mc{E}''$ are defined on
 $U\subset Y$. By Lemma \ref{surjofisocpione}, it suffices to show
 that $q'^+\mc{E}'\otimes q''^+\mc{E}''$ is irreducible as an object in
 $\mb{Z}\mc{I}(U\times U)$. This follows by Lemma
 \ref{surjectwisweilgroup}.
 Let us check (ii). There exists $U\subset Y$ such that $\mc{E}$ is an
 isocrystal on $U\times U$. Since $\mc{E}$ is assumed
 negligible, there exists a complete $r$-negligible object
 $\widehat{\mc{E}}\in\mb{Z}\mc{I}(U\times U)$ such that
 $\mc{E}\subset\widehat{\mc{E}}$ in $\mc{I}(U\times U)$ by definition.
 Since $\mc{E}$ is invariant under the pull-back by
 $\mr{Frob}_Y\times\mr{id}$, we may assume that $\mc{E}$ is invariant
 under the equipped isomorphism
 $\alpha\colon(\mr{Frob}_Y\times\mr{id}_Y)^+
 (\widehat{\mc{E}})\cong
 \widehat{\mc{E}}$, by changing the inclusion if necessarily.
 Let $\rho_{\widehat{\mc{E}}}$ be the
 corresponding representation of $\mb{Z}\WI(U\times U)$,
 and $\rho_{\mc{E}}$ be the subrepresentation of
 $\rho_{\widehat{\mc{E}}}$ corresponding to $\mc{E}$ defined since
 $\mc{E}$ is invariant under the isomorphism $\alpha$.
 Let $K:=\mr{Ker}\bigl(\mb{Z}\WI(U\times U)\rightarrow
 \WI(U)\times\WI(U)\bigr)$. Then since $\widehat{\mc{E}}$ is assumed
 complete, $\rho_{\widehat{\mc{E}}}(K)=\mr{id}$. Thus,
 $\rho_{\mc{E}}(K)=\mr{id}$, which implies that $\rho_{\mc{E}}$ is the
 pull-back of a representation of $\WI(U)\times\WI(U)$, and the first
 claim follows.
 To check the last claim, we note that, for any
 irreducible objects $\mc{E}'$, $\mc{E}''$ in $\mc{I}(F)$ of rank $r'$
 and $r''$ respectively, $q'^+\mc{E}'\otimes q''^+\mc{E}''$ is
 semi-simple in $\mc{I}(F^2)$ by (i), and the number of constituents $N$
 of $q'^+\mc{E}'\otimes q''^+\mc{E}''$ is $\leq r'r''$. This
 implies that for any constituent $\mc{F}$ of $q'^+\mc{E}'\otimes
 q''^+\mc{E}''$ and any integer $k>0$,
 $\bigoplus_{n=1}^{kN}(\mr{Frob}^n_Y\times\mr{id}_Y)^+(\mc{F})$ is
 invariant by the action of $(\mr{Frob}_Y\times\mr{id}_Y)^+$. Since
 $N\leq r'r''\leq r^2$, $N$ divides $r^2!$, and the claim follows by the
 first part of (ii).
\end{proof}

\begin{rem*}
 In (ii) of the corollary, we may take $\bigoplus_{n=1}^{r!}$ as in
 \cite{Laf}, but for our purpose, $\bigoplus_{n=1}^{r^2!}$ is enough.
\end{rem*}

\subsubsection{}
\label{defofsomenotepi}
From now on, we use the notations of \S\ref{langcorresstate} freely.
In the following, we fix $a\in\mb{A}_F^{\times}$ of degree $1$, and a
level $N=\mr{Spec}(\mc{O}_N)\hookrightarrow X$. Let $p$ be a large
enough convex function.  For $\pi\in\mc{A}^r(F)$, we denote by
$\chi_\pi$ the central character of $\pi$. We define a set by
\begin{equation*}
 \{\pi\}^r_N:=\bigl\{\pi\in\mc{A}^r(F)\mid
  \chi_{\pi}(a)=1 \text{ and }
  \pi\cdot\mathbf{1}_N\neq0
  \bigr\}.
\end{equation*}
where $\mathbf{1}_N$ is the quotient of the characteristic function of
$K_N:=\mr{Ker}\bigl(\mr{GL}_r(\mb{A}_F)\rightarrow
\mr{GL}_r(\mc{O}_N)\bigr)$ by its volume. It suffices to construct
isocrystals corresponding to the cuspidal representations belonging to
$\{\pi\}^r_N$. We put $S_N:=(X-N)\times (X-N)$.

Let $q',q''\colon X\times X\rightarrow X$ be the first and second
projection respectively.
For a morphism of c-admissible stacks $f\colon\st{X}\rightarrow
S_{N}$, we denote the relative cohomology
$\cH^{\nu}f_!\overline{\mb{Q}}_{p,\st{X},F}$ by
$\mc{H}^\nu_{\mr{c}}(\st{X})$, where $\overline{\mb{Q}}_{p,\st{X},F}$
denotes the unit object in $\mr{Con}(\st{X})$. This is an object in
$\mr{Con}(S_N)$.
Assume $f$ is proper. There exists an open
dense substack $j\colon\st{U}\hookrightarrow\st{X}$ such that
$\overline{\mb{Q}}_{p,\st{U},F}$ is pure of weight $0$.
Then we denote by
$\IH^{\nu}(\st{X}):=\cH^{\nu}f_+j_{!+}
\overline{\mb{Q}}_{p,\st{U},F}$, which is pure of weight $\nu$.

We also use $\ell$-adic cohomologies. We denote by
$\mc{H}^\nu_{\mr{c}}(\st{X},\overline{\mb{Q}}_\ell):=
\H^\nu f_!\overline{\mb{Q}}_\ell$ and
$\IH^\nu(\st{X},\overline{\mb{Q}}_\ell):=\H^\nu
f_+j_{!+}\overline{\mb{Q}}_\ell$, where $\H^\nu$ denotes the standard
(constructible) t-structure.
In most cases in this paper, these $\ell$-sheaves are smooth,
namely an object of $\mc{W}_{\ell}(S_N)$ (cf.\ \cite[p.165]{Laf}).

\subsubsection{}
\label{starproof}
{\em We start the proof of the theorem from here}.
We prove slightly stronger statement than the theorem.
For an integer $n\geq1$, we call the following two statements
$(\mr{S})_n$:
\begin{enumerate}
 \item\label{prodfoumindhyp}
      Theorem \ref{main} is true for $r,r'\leq n$. As a consequence, we
      have $\mc{I}_{n+1}\rightarrow\mc{A}_{n+1}$ as well by Theorem
      \ref{thmdeligneprinciple}.

 \item\label{rneghyp}
      For $r'\leq n$, the constructible object
      $\mc{H}_{\mr{c}}^\nu(\Ch^{r',\overline{p}\leq
       p}_N/a^{\mb{Z}})$ is smooth and $(n+1)$-negligible for any $\nu$,
      level $N\colon\mr{Spec}(\mc{O}_N)\hookrightarrow X$,
      $a\in\mb{A}^{\times}_F$ such that $\deg(a)=1$, and convex polygon
      $p$ large enough with respect to $X$ and $N$.
\end{enumerate}
We know that $(\mr{S})_1$ holds.
Indeed, \eqref{prodfoumindhyp} is nothing but the class field theory
together with the theorem of Tsuzuki \cite[Thm 7.1.1]{T}. See
\cite[6.5]{ALang} for more details. Let us check
\eqref{rneghyp}. We know that $f\colon\Ch^{1,\overline{p}\leq p}_N
=\Ch^{1}_N/a^{\mb{Z}}\rightarrow S_N$
is an abelian covering and $f_!\overline{\mb{Q}}_\ell=\bigoplus
q'^*\chi'\otimes q''^*\chi''$ where $\chi'$ and $\chi''$ are smooth
sheaf on $X-N$ of rank $1$ by \cite[remark of VI.15]{Laf}. Using
\eqref{prodfoumindhyp}, let $\chi'_p$ and $\chi''_p$ be the
corresponding isocrystals of rank $1$ in $\mc{I}(X-N)$. By Proposition
\ref{basechforshrik} (or Theorem \ref{Gabberfujiwarateh} if one
prefers), $f_!\overline{\mb{Q}}_\ell$ and $f_!\overline{\mb{Q}}_p$ have
the same Frobenius eigenvalues. Thus, by \v{C}ebotarev density
\ref{Cebdensity}, we get
$f_!\overline{\mb{Q}}_p\cong\bigoplus q'^+\chi'_p\otimes q''^+\chi'_p$.

In the following, we fix an integer $r>1$, and assume that
$(\mr{S})_{r-1}$ holds. Our goal is to show $(\mr{S})_{r}$ under this
assumption, which is attained at the very end of this subsection.

\begin{dfn}
 \label{defofneglibasch}
 Let $k'$ be the extension of $k$ of degree $d_0\geq1$. For a smooth
 scheme $U$ over $k'$, we put
 $\mc{I}_{d_0}(U):=\mr{Isoc}^\dag(U/\mf{T}_{k'})$ (cf.\
 \ref{situationfix} for the notation of base tuple).
 We denote by $F_{d_0}$ and $F^2_{d_0}$
 the function fields of $X\otimes_kk'$ and
 $(X\times X)\otimes_{k}k'$, and we define
 $\mc{I}_{d_0}(F_{d_0})$, $\mc{I}_{d_0}(F^2_{d_0})$ accordingly.
 An irreducible object in $\mc{I}(F)$ is said to be {\em $r$-negligible}
 if it is of rank $<r$. An irreducible object $\mc{E}$ in
 $\mc{I}_{d_0}(F_{d_0})$ (resp.\ $\mc{I}_{d_0}(F^2_{d_0})$) is said to
 be {\em $r$-negligible} if there exists an irreducible $r$-negligible
 object $\mc{E}'$ in $\mc{I}(F)$ (resp.\ $\mc{I}(F^2)$) such that the
 pull-back $\mc{E}'\otimes k'$ contains $\mc{E}$. Sums of
 $r$-negligible objects are said to be $r$-negligible as well.
\end{dfn}

\begin{lem}[{\cite[VI.16]{Laf}}]
 \label{bachnegok}
 Let $d_0\geq1$ be an integer. An irreducible object $\mc{E}$ in
 $\mc{I}(F^2)$ (resp.\ $\mc{I}(F)$) is $r$-negligible if $\mc{E}\otimes
 k'$ contains an $r$-negligible object in $\mc{I}_{d_0}(F^2_{d_0})$
 (resp.\ $\mc{I}_{d_0}(F_{d_0})$).
\end{lem}
\begin{proof}
 Let $\varphi\in\mr{Gal}(k'/k)$ be a generator, and $\varphi^*\colon
 U\otimes k'\rightarrow U\otimes k'$ be the automorphism {\em
 over $k$} induced by $\varphi$ for some smooth scheme $U$ over
 $k$. Giving an object in $\mc{I}_{d_0}(U\otimes k')$ is equivalent to
 giving $\mc{F}\in\mc{I}(U\otimes k')$ with isomorphism
 $\varphi^*(\mc{F})\cong\mc{F}$. This observation implies that if
 $\mc{E},\mc{E}'\in\mc{I}(F^2)$ (resp.\ $\mc{I}(F)$)
 are irreducible objects such that
 $\mc{E}\otimes k'\cong\mc{E}'\otimes k'$, then there exists a character
 $\chi$ of $\mr{Gal}(k'/k)\cong\mb{Z}/d_0\mb{Z}$ (which can be seen as a
 rank $1$ object of $\mc{I}(\mr{Spec}(k))$) such that
 $\mc{E}\otimes\chi\cong\mc{E}'$, thus the lemma follows.
\end{proof}

\subsubsection{}
\label{keylemmLaff}
We need to show the following technical proposition.
In the statement and the proof, the algebraic stack $\st{C}^{r,N}$ and
its variants\footnote{In \cite{Laf}, Lafforgue uses script fonts ({\it
e.g.}\ $\mc{C}^{r,N}$).} are used. We remark that these stacks are used
only in this proposition and its corollary \ref{impcorrnegbound}.
In \cite[III 3a)]{Laf}, Lafforgue defined a morphism 
$\mr{Res}\colon\overline{\Ch^{r,d,\overline{p}\leq p}}
\times_{X\times X}S_N\rightarrow\st{C}^{r,N}$ between algebraic
stacks locally of finite type over $k$.
We remind that all the three stacks in the table of
\ref{tablechoutsatis} are defined over the source of $\mr{Res}$. The
open substack $\st{C}^{r,N}_{\emptyset}$ of $\st{C}^{r,N}$ is defined in
{\it ibid.}\ as well, and the complement
$\st{C}^{r,N}-\st{C}^{r,N}_{\emptyset}$ is called the {\em
boundary}. See Step \ref{recalldifflem} of the proof of
the following proposition for some review of these stacks.

\begin{prop*}[{\cite[VI.17]{Laf}}]
 Let $p$ be a convex polygon large enough with respect to $X$, $N$, and an
 integer $1\leq d\leq r$. Let $\st{C}$ be an algebraic stack
 representable and quasi-projective (cf.\ \cite[14.3.4]{LM}) over
 $\st{C}^{r,N}$, and consider the following cartesian diagram:
 \begin{equation*}
  \xymatrix@C=50pt{
   \st{X}
   \ar[r]^{\mr{Res}}\ar[d]_{g'}\ar@{}[rd]|\square&
   \st{C}\ar[d]^{g}\\
  \overline{\Ch^{r,d,\overline{p}\leq p}}
   \times_{X\times X}S_N\ar[r]_-{\mr{Res}}&
   \st{C}^{r,N}.
   }
 \end{equation*}
 We denote by $p_{\st{X}}\colon\st{X}\rightarrow S_N$ the projection.

 (i) Let $\ms{M}\in\mr{Con}(\st{C})$. Then
 $\cH^\nu p_{\st{X}!}\bigl(\mr{Res}^+\ms{M}\bigr)$ is smooth on $S_N$
 for any $\nu$.

 (ii) Assume, moreover, that $\ms{M}$ is supported on the boundary
 of $\st{C}$, namely
 $g^{-1}(\st{C}^{r,N}-\st{C}^{r,N}_\emptyset)$. Then
 $\cH^\nu p_{\st{X}!}\bigl(\mr{Res}^+\ms{M}\bigr)$ is
 $r$-negligible as an object in $\mc{I}(F^2)$ for any $\nu$.
\end{prop*}
\begin{proof}
 Before beginning the proof, let us remark that $\st{X}$ is
 c-admissible since it is quasi-projective over the proper admissible
 stack $\overline{\Ch^{r,d,\overline{p}\leq p}}$ thus we are able to use
 full six functor formalism. However, $\st{C}$ is locally of finite type
 but not even admissible. Currently, we only have partial formalism in
 such a situation, so we need slightly to be careful.
 Nevertheless, most of the cohomological functors are available
 on the level of the category $\mr{Con}(-)$.
 See the first part of \ref{miscstsubsec} for the details.
 The proof is divided into several steps.

\begin{step}[Proof of (i) and the first reduction of (ii)]
 \label{firststepdifflem}
 The morphism $(p_{\st{X}},\mr{Res})\colon\st{X}\rightarrow
 S_N\times\st{C}$ is known to be smooth by \cite[Prop III.7
 (ii)]{Laf}, which implies that $\mr{Res}$ is smooth as well since $S_N$
 is smooth.
 Using Lemma \ref{Artinstackpushfor}, we have
 \begin{equation*}
  \cH^\nu g'_!\mr{Res}^+(\ms{M})\cong
   \mr{Res}^+ g^\nu _!(\ms{M}).
 \end{equation*}
 Since $p_{\st{X}!}\cong\mr{pr}_{2!}\circ g'_!$, for $\ms{N}\in
 D^{\mr{b}}_{\mr{hol}}(\st{X})$, $\cH^\nu p_{\st{X}!}(\ms{N})$ can be
 expressed as extensions of subquotients of
 $\bigoplus_{i+j=\nu}\cH^i\mr{pr}_{2!}\,\cH^jg'_!(\ms{N})$. As a
 consequence, it suffices to show (i) and (ii) for
 $\st{C}=\st{C}^{r,N}$. In this case, (i) is an easy consequence of
 Lemma \ref{propsmoobcspeci}: the algebraic stack
 $\overline{\Ch^{r,d,\overline{p}\leq p}}\times_{X\times X}S_N$ is
 admissible by \cite[Prop V.I]{Laf}, proper over $S_N$ by \cite[Prop
 III.7 (i)]{Laf}, and we already recalled that $(p_{\st{X}},\mr{Res})$
 is smooth, so the lemma is applicable. We concentrate on proving (ii)
 from now on. In the following, we initialize the notation $\mf{C}$,
 and use it for other stacks.
\end{step}

\begin{step}[Induction hypothesis]
 \label{reductiondifflem}
 Let $c_0:=\mr{sup}\bigl\{\dim(I_x)\mid x\in|\mr{Res}(\st{X})|\bigr\}$
 where $I_x$ is the inertia algebraic group space of $\st{C}^{r,N}$ at
 $x$. By the quasi-compactness of $\st{X}$, we have $c_0>-\infty$, and
 the dimension of any locally closed substack of $\st{C}^{r,N}$ in
 the image of $\mr{Res}$ can be bounded below by $-c_0$.
 We use the induction on $k:=c_0+\dim(\mr{Supp}(\ms{M}))\,(\geq0)$
 (cf.\ \ref{isocsuppdef} for the definition of support).
 Assume that the proposition holds for
 $k=k_0\geq-1$. We will show the proposition for modules whose support
 is of dimension $k=k_0+1$. Now, we frequently use the the following
 reduction,
 which follows by using the localization exact sequence Lemma
 \ref{localtriangs} and the induction hypothesis:
 \begin{quote}
  ($\clubsuit$)\,
  Let $j\colon\mf{U}\subset\st{S}:=\mr{Supp}(\ms{M})$ be an open
  substack such that $\dim(\st{S}\setminus\mf{U})<\dim(\st{S})=k_0+1$.
  Let $\ms{M}'\in\mr{Con}(\st{S})$ such that $j^+\ms{M}\cong
  j^+\ms{M}'$. Then proving the proposition for $\ms{M}$ and $\ms{M}'$
  are equivalent.
 \end{quote}
\end{step}

\begin{step}[Recall of geometry of moduli spaces]
 \label{recalldifflem}
 Now, we put aside the cohomology theory for a while, and summarize
 the complicated geometry of moduli spaces in {\it ibid.}\ very
 briefly.
 Let $\st{A}^{r,1}:=\mb{A}^{r-1}$ the toric variety, and
 $\st{A}^{r,1}_{\emptyset}$ be the torus $\mb{G}_m^{r-1}$. Let
 $\underline{r}=(r_1,\dots,r_k)$ be a partition $r_1+\dots+r_k=r$. We
 define $\st{A}^{r,1}_{\underline{r}}$ to be the locally closed
 subscheme of $\st{A}^{r,1}$ consisting of points whose coordinates
 indexed by $r_1+\dots+r_e$ for $1\leq e<k$ are zero and invertible
 for other coordinates.
 These are orbits of $\st{A}^{r,1}_\emptyset$. See
 \cite[III 1a)]{Laf} for more information on these schemes.
 Now, in \cite[III 2a), 3a)]{Laf}, the sequence of algebraic stacks
 locally of finite type
 $\st{C}^{r,N}\subset\widetilde{\st{C}}^{r,N}\subset
 \overline{\st{C}}^{r,N}$ and the morphism
 $\overline{\st{C}}^{r,N}\rightarrow\st{A}^{r,1}/\st{A}^{r,1}_\emptyset$
 are defined. For a partition $\underline{r}$, the pull-back of
 $\st{A}^{r,1}_{\underline{r}}/\st{A}^{r,1}_\emptyset$ to the three
 stacks are denoted by
 $\st{C}^{r,N}_{\underline{r}}\subset\widetilde{\st{C}}^{r,N}
 _{\underline{r}}\subset\overline{\st{C}}^{r,N}_{\underline{r}}$, and
 similarly for $\st{A}^{r,1}_{\emptyset}/\st{A}^{r,1}_\emptyset$. These
 are stratifications of the boundary.

 Let $\mb{A}^N$ and $\mb{G}_m^N$ be schemes induced from $\mb{A}^1$ and
 $\mb{G}_m$ by the Weil restriction from $\mc{O}_N$ to $k$.
 We also use a variant ${}^{r}\overline{\st{C}}^{N}_\emptyset$, and both
 ${}^{r}\overline{\st{C}}^{N}_\emptyset$
 and $\overline{\st{C}}^{r,N}_\emptyset$ have natural morphisms to
 $(\mb{A}^N/\mb{G}_m^N)^2$ (cf.\ \cite[II 1a)]{Laf}).
 We have the finite surjective radicial morphism
 $\mr{sw}\colon
 {}^{r}\overline{\st{C}}^{N}_\emptyset\rightarrow
 \overline{\st{C}}^{r,N}_\emptyset$
 over the endomorphism $\mr{id}\times\mr{Frob}$ of
 $(\mb{A}^N/\mb{G}_m^N)^2$ sending
 $(\mc{F}\leftarrow\mc{F}'\rightarrow{}^{\tau}\mc{F})$ to
 $(\mc{F}'\rightarrow{}^{\tau}\mc{F}\leftarrow{}^{\tau}\mc{F}')$ using
 the notation of {\it ibid.}.

 We have the morphism
 $\overline{\Ch^{r,d,\overline{p}\leq p}}\rightarrow
 \overline{\st{C}}^{r,N}$, which factorizes through
 $\widetilde{\st{C}}^{r,N}$ (cf.\ \cite[III 3a)]{Laf}).
 The morphism $\mr{Res}$ in the statement of the proposition is induced
 by this. The pull-back of $\st{A}^{r,1}_{\underline{r}}$ is denoted by
 $\Ch^{r,d,\overline{p}\leq p}_{\underline{r}}$ following the notation
 of \cite[III 1c)]{Laf}.
\end{step}

\begin{step}[Geometric construction of Lafforgue]
 Consider the following commutative diagram of solid arrows which
 appears in \cite[right before III.6]{Laf}:
 \begin{gather*}
  \xymatrix{
   \Ch^{r,d,\overline{p}\leq p}_{\underline{r}}\ar[r]^<>(.5){\beta}
   \ar[d]&
   \widetilde{\st{C}}^{r,N}_{\underline{r}}\ar[d]_{\gamma}
   &\st{D}\ar@{_{(}.>}[l]\ar@{.>}[ld]^-{\gamma_{\st{D}}}\\
  \Ch^{\underline{r},d,\overline{p}\leq p}\ar[r]_<>(.5){}&
   \overline{\st{C}}^{\underline{r},N},&
   }\\
  \hspace{13em}
  {\scriptstyle (\overline{\st{C}}^{\underline{r},N}:=
  \overline{\st{C}}^{r_1,N}_\emptyset\times_{\mb{A}^N/\mb{G}_m^N}
  {}^{r_2}\overline{\st{C}}^{N}_\emptyset
  \times_{\mb{A}^N/\mb{G}_m^N}\dots\times_{\mb{A}^N/\mb{G}_m^N}
  {}^{r_k}\overline{\st{C}}^{N}_\emptyset)}
 \end{gather*}
 where $\gamma$ is defined in \cite[III 2b)]{Laf},
 $\Ch^{\underline{r},d,\overline{p}\leq p}$ is defined in
 \cite[Prop III.3]{Laf}. Lafforgue proved the following two claims in
 the proof of \cite[VI.17]{Laf}:
 \medskip

 \noindent
 (i) Assume that we are given a locally closed substack $\st{D}$ of
 $\mr{Im}(\beta)$ in $\widetilde{\st{C}}^{r,N}_{\underline{r}}$. Then
 there exists an open dense substack $\st{D}'$ of $\st{D}$
 such that $\gamma_{\st{D}'}$ can be written as the composition of the
 morphisms of the following three types;
 1.\ a gerb-like morphism whose structural group is flat fiber-wise
 geometrically connected, 2.\ a finite flat radicial
 morphism, 3.\ a finite \'{e}tale morphism.

 \noindent
 (ii) Consider the compositions
 \begin{equation}
  \label{composiimfinitept}
  \tag{$\star$}
  \st{C}^{r,N}_{\underline{r}}\xrightarrow{\gamma}
   \overline{\st{C}}^{\underline{r},N}
   \xrightarrow{\mr{pr}_1}\overline{\st{C}}^{r_1,N}
   _{\emptyset},\qquad
   \st{C}^{r,N}_{\underline{r}}\xrightarrow{\gamma}
   \overline{\st{C}}
   ^{\underline{r},N}
   \xrightarrow{\mr{pr}_k}
   {}^{r_k}\overline{\st{C}}^{N}
   _{\emptyset}
   \xrightarrow{\mr{sw}}
   \overline{\st{C}}^{r_k,N}_{\emptyset}.
 \end{equation}
 The images of these morphisms consist of finitely many points.

 A proof of (i) is written at the end of \cite[p.171]{Laf}, and that
 for (ii) is in the second paragraph of \cite[p.172]{Laf}. Let us add a
 little explanation on the proof of (i). By dimension counting,
 Lafforgue proves that the preimage of a point of
 $\mr{Im}(\gamma\circ\beta)$ by $\gamma$ consists of finitely many
 points. To conclude, use Lemma \ref{factorrepdimsu}.
\end{step}

\begin{step}[Cohomological reduction]
 Now, let us come back to the cohomologies.
 Let
 $j\colon\st{C}\hookrightarrow\mf{S}:=\mr{Supp}(\ms{M})\,(\subset\st{C}^{r,N})$
 be a dense open immersion.
 By ($\clubsuit$) of Step \ref{reductiondifflem}, we may replace
 $\ms{M}$ by $j_!j^+(\ms{M})$ (cf.\ see \ref{representpushfor} for the
 definition of functors $j_!$ and $j^+$ in this situation).
 By shrinking $\st{C}$ using ($\clubsuit$), we may assume $\st{C}$ is of
 dimension $k_0+1$ in the stratum $\st{C}^{r,N}_{\underline{r}}$ for
 some non-trivial partition $\underline{r}$ such that
 $\ms{M}':=j^+\ms{M}$ is smooth.
 Using (i) above and shrinking $\st{C}$ always using ($\clubsuit$),
 we may assume $\gamma':=\gamma_{\st{C}}$ is the composition of
 morphisms of the three types.
 We may change $\ms{M}'$ by a smooth object
 which contains $\ms{M}'$ as a direct factor, so by using Lemma
 \ref{diagconngeblike}, Corollary \ref{invariantclassflrdgr},
 and Lemma \ref{directfactorlemmaeasy} (ii), we may assume that there exists
 a smooth object $\ms{N}'$ on a locally closed substack of
 $\overline{\st{C}}^{\underline{r},N}$ such that
 $\ms{M}'=\gamma'^+(\ms{N}')$.
 By (ii), shrinking $\st{C}$ if needed, we may assume that the images of
 $\mf{C}$ by the two morphisms of (\ref{composiimfinitept}) consist of
 unique locally closed points
 $\widetilde{\st{B}}_\star\cong
 \left[\mr{Spec}(\mb{F}_{q_0})/\mr{Aut}(
 \widetilde{\st{B}}_\star)\right]\in
 \overline{\st{C}}^{r_{\star},N}_{\emptyset}$
 where $\star=1$ or $k$, and $\mb{F}_{q_0}$ is the field extension of
 $\mb{F}_q$ of degree $d_0$. We define the Galois coverings
 $\widetilde{\st{B}}'_{\star}$ of $\widetilde{\st{B}}_{\star}$ defined
 by the ``discrete part''
 $\mr{Aut}(\widetilde{\st{B}}_{\star})/
 (\mr{Aut}(\widetilde{\st{B}}_{\star}))^{\circ}$
 of $\mr{Aut}(\widetilde{\st{B}}_{\star})$.
 We denote by $\alpha\colon\st{C}'\rightarrow\st{C}$ the Galois
 covering induced by
 $\widetilde{\st{B}}'_1\times\widetilde{\st{B}}'_k\rightarrow
 \widetilde{\st{B}}_1\times\widetilde{\st{B}}_k$. We may replace
 $\ms{M}'$ by $\alpha_+\alpha^+\ms{M}'$.

 Now, consider the finite surjective radicial morphism
 \begin{equation}
  \label{morphstackreduclem}
  \tag{$\star\star$}
  \mr{id}\times\mr{sw}\times\dots\times\mr{sw}\colon
   \overline{\st{C}}^{\underline{r},N}\rightarrow
   \overline{\st{C}}^{r_1,N}_\emptyset\times_{\mb{A}^N/\mb{G}_m^N}
   \overline{\st{C}}^{r_2,N}_\emptyset
   \times_{\mb{A}^N/\mb{G}_m^N,\mr{Frob}}\dots
   \times_{\mb{A}^N/\mb{G}_m^N,\mr{Frob}}
   \overline{\st{C}}^{r_k,N}_\emptyset.
 \end{equation}
 We may and do identify the holonomic modules on
 these stacks by Lemma \ref{directfactorlemmaeasy}.
 By considering Lemma \ref{diagconngeblike}, we may assume that
 $\ms{N}'\cong\ms{M}_1\boxtimes_{\mb{A}^N/\mb{G}_m^N}
 \ms{M}''\boxtimes_{\mb{A}^N/\mb{G}_m^N,\mr{Frob}}\ms{M}_{k}$ (cf.\
 \ref{defpullbackconst} for the notation) where
 \begin{itemize}
  \item $\ms{M}_\star$ ($\star=1,k$) is the push-forward of
	the trivial smooth object on $\widetilde{\st{B}}'_\star$ by the
	finite \'{e}tale covering $\widetilde{\st{B}}'_i\rightarrow
	\widetilde{\st{B}}_i$.

  \item $\ms{M}''$ is a smooth object on a locally
	closed substack of
	$\overline{\st{C}}^{r_2,N}_\emptyset
	\times_{\mb{A}^N/\mb{G}_m^N,\mr{Frob}}\dots
	\times_{\mb{A}^N/\mb{G}_m^N,\mr{Frob}}
	\overline{\st{C}}^{r_{k-1},N}_\emptyset$.
 \end{itemize}
\end{step}

\begin{step}
 We denote by $\mr{pr}_i$ the $i$-th projection, and put
 \begin{align*}
  X^2&:=X\times X,\qquad
  \mr{Cht}^\star:=\mr{Cht}^{r_\star,d_\star,\overline{p}\leq
  p_\star},\\
  (\mr{Cht})'&:=\mr{Cht}^{r_2,d_2,\overline{p}\leq p_2}
  \times_{X,\mr{Frob}}\dots\times_{X,\mr{Frob}}
  \mr{Cht}^{r_{k-1},d_{k-1},\overline{p}\leq p_{k-1}},\\
  (\overline{\st{C}}^{\underline{r}',N}_{\emptyset})'&:=
  \overline{\st{C}}^{r_2,N}_\emptyset
  \times_{\mb{A}^N/\mb{G}_m^N,\mr{Frob}}\dots\times
  _{\mb{A}^N/\mb{G}_m^N,\mr{Frob}}\overline{\st{C}}
  ^{r_{k-1},N}_\emptyset.
 \end{align*}
 Consider the following commutative diagram:
 \begin{equation*}
  \xymatrix@C=15pt{
   \overline{\Ch^{r,d,\overline{p}\leq p}}
   \ar[d]_{\mr{Res}}\ar@{}[rd]|\square&
   \mr{Cht}^{r,d,\overline{p}\leq p}_{\underline{r}}
   \ar@{_{(}->}[l]\ar[r]^-{g}\ar[d]_{\beta}&
   \mr{Cht}^{1}\times_{X}(\mr{Cht})'\times_{X,\mr{Frob}}
   \mr{Cht}^{k}
   \ar@/^50pt/[]+R+<-3pt,-3pt>;[dd]+R+<3pt,3pt>
   |(0.8){q=q_1\times q''\times q_k}
   \ar[d]_{\beta'}\\
  \overline{\st{C}}^{r,N}&
   \widetilde{\st{C}}^{r,N}_{\underline{r}}\ar[r]^-{g'}\ar[l]&
   \overline{\st{C}}^{r_1,N}_\emptyset\times_{\mb{A}^N/\mb{G}_m^N}
   (\overline{\st{C}}^{\underline{r}',N}_{\emptyset})'
   \times_{\mb{A}^N/\mb{G}_m^N,\mr{Frob}}
   \overline{\st{C}}^{r_k,N}_\emptyset\\
  &
   X^2\times X^2&
   X^2\times_X X^2\times_{X,\mr{Frob}}X^2
   \ar[l]_-{r=\mr{pr}_1\times\mr{pr}_3}^-{\sim}\\
   }
 \end{equation*}
 Moreover, $g'$ is the composition of $\gamma$ and
 (\ref{morphstackreduclem}), and the morphism $g$ is the composition of
 a gerb-like morphism whose structural group is flat finite radicial and
 a representable universal homeomorphism defined in \cite[Cor
 III.4]{Laf}. Thus, we get $g_!g^+\cong\mr{id}$ by Lemma
 \ref{directfactorlemmaeasy} (i) and Corollary
 \ref{invariantclassflrdgr}.
 Put $p':=r\circ q\circ g$, and we can compute
 \begin{align}
  \tag{$\star\!\star\!\star$}
  \label{comptcohdiflem}
  \cH^\nu p'_!\beta^+\ms{M}&\cong
  \cH^\nu p'_!g^+\beta'^+(\ms{N}')\\
  \notag&\cong
  \cH^\nu r_!q_!\,\beta'^+\bigl(
  \ms{M}_1\boxtimes_{\mb{A}^N/\mb{G}_m^N}
  \ms{M}''\boxtimes_{\mb{A}^N/\mb{G}_m^N,\mr{Frob}}\ms{M}_{k}
  \bigr)\\
  \notag&\cong
  \cH^\nu r_!q_!\bigl(\mr{Res}^+(\ms{M}_1)\boxtimes_X
  \mr{Res}^+(\ms{M}'')\boxtimes_{X,\mr{Frob}}
  \mr{Res}^+\ms{M}_k\bigr)\\
  \notag&\cong
  \cH^\nu r_!\bigl(
  q_{1!}\mr{Res}^+(\ms{M}_1)\boxtimes_Xq''_{!}\mr{Res}^+(\ms{M}'')
  \boxtimes_{X,\mr{Frob}}q_{k!}\mr{Res}^+\ms{M}_k
  \bigr)\\
  \notag&\hspace{15em}
  (\because\mbox{by K\"{u}nneth \ref{Kunnethextprod}})\\
  \notag&\cong
  \cH^\nu\bigl(
  (q_{1!}\mr{Res}^+(\ms{M}_1)\boxtimes
  q_{k!}\mr{Res}^+\ms{M}_k)\otimes
  r'^+(q''_{!}\mr{Res}^+(\ms{M}''))
  \bigr),
 \end{align}
 where we identify $\ms{N}'$ and its zero extension, and similarly for
 other objects to simplify the notation. Moreover, $r'$ is the composition
 \begin{align*}
  (X\times X)\times
  (X\times X)\xrightarrow{\mr{pr}_2\times\mr{pr}_3}
  X\times X
  &\xleftarrow[\sim]{\mr{pr}_1\times\mr{pr}_4}
  X\times_X(X\times
  X)\times_{X,\mr{Frob}}X\\
  &\xrightarrow{\mr{pr}_2\times\mr{pr}_3}
  (X\times X).
 \end{align*}
\end{step}

\begin{step}
 \label{twoclaimsdifflem}
 Let $\star=1$ or $k$.
 Let $f_{\star}\colon\mr{Cht}^{r_\star,d_\star,\overline{p}\leq p_\star}
 \times_{\overline{\st{C}}^{r_\star,N}_\emptyset}
 \widetilde{\st{B}}'_\star\rightarrow X\times X$ be the canonical
 morphism. By definition \cite[II 1a)]{Laf}, the image of this morphism
 is contained in $X\times(X-N)\cup(X-N)\times X$.
 We say that {\em $\widetilde{\st{B}}_1$ hits a zero} (resp.\ {\em
 $\widetilde{\st{B}}_k$ hits a pole}) if the image of
 $f_1$ (resp.\ $f_k$) is contained in $(X-N)\times\{0\}$ for some $0\in
 N$ (resp.\ $\{\infty\}\times(X-N)$ for some $\infty\in N$), and
 {\em does not hit a zero} (resp.\ {\em does not hit a pole})
 otherwise. We are reduced to the following two claims:
 \begin{itemize}
  \item If $\widetilde{\st{B}}_1$ hits a zero (resp.\ does not hit a
	zero), then the relative cohomology $\H^\nu(\mr{pr}_1\circ
	f_1)_!\overline{\mb{Q}}_p$ (resp.\
	$\H^\nu f_{1!}\overline{\mb{Q}}_p$)
	over the generic point of $X$ (resp.\ $X\times X$) is
	$r$-negligible in $\mc{I}_{d_0}(F_{d_0})$ (resp.\
	$\mc{I}_{d_0}(F^2_{d_0})$).
 
  \item If $\widetilde{\st{B}}_k$ hits a pole (resp.\ does not hit a
	pole), then the relative cohomology $\H^\nu(\mr{pr}_2\circ
	f_k)_!\overline{\mb{Q}}_p$ (resp.\
	$\H^\nu f_{k!}\overline{\mb{Q}}_p$)
	over the generic point of $X$ (resp.\ $X\times X$) is
	$r$-negligible in $\mc{I}_{d_0}(F_{d_0})$ (resp.\
	$\mc{I}_{d_0}(F^2_{d_0})$.)
 \end{itemize}
 Indeed, if these hold, any constituent of
 $q_{\star!}\mr{Res}^+(\ms{M}_\star)$ is a direct factor of an object of
 the form $\ms{E}_{\star}\boxtimes\ms{F}_{\star}$
 such that $\ms{E}_{\star}$ and $\ms{F}_{\star}$ are irreducible of rank
 $<r$ or supported on a point.
 Since $p_{\st{X}!}\mr{Res}^+\ms{M}\cong
 (\mr{pr}_1\times\mr{pr}_4)_!p'_!\beta^+\ms{M}$, using
 (\ref{comptcohdiflem}) and K\"{u}nneth formula again, we see that any
 constituent of $\H^\nu p_{\st{X}!}\mr{Res}^+\ms{M}$ is a
 direct factor of
 \begin{equation*}
  (\ms{E}_1\boxtimes\ms{F}_k)\otimes H^\nu
   \bigl(X\times X,r'^+q''_{!}\mr{Res}^+(\ms{M}'')\otimes
   (\ms{F}_1\boxtimes\ms{E}_k)\bigr).
 \end{equation*}
 Thus, by considering Lemma \ref{bachnegok}, the proposition follows.
\end{step}

\begin{step}[Use of induction hypothesis]
 \label{laststepdifflem}
 Let us show the claims in Step \ref{twoclaimsdifflem}. We only consider
 $\star=1$ case. When $\widetilde{\st{B}}_1$ does not hit a zero,
 $\mr{Cht}^{r_1,d_1,\overline{p}\leq p_1}\times
 _{\overline{\st{C}}^{r_1,N}_\emptyset}\widetilde{\st{B}}'_1$ and
 $\mr{Cht}^{r_1,d_1,\overline{p}\leq p_1}_N\otimes_{\mb{F}_q}
 \mb{F}_{q_0}$ are isomorphic as written in \cite[p.173]{Laf}.
 Since $r_1<r$, by the induction hypothesis
 \ref{starproof} \eqref{rneghyp}, we get the desired $r$-negligibility.
 Let us treat the case where $\widetilde{\st{B}}_1$ hits a zero
 $\{0\}\in N$ defined over $\mb{F}_{q_0}$. We put
 $\st{X}^{d_1}_1:=\overline{\mr{Cht}^{r_1,d_1,\overline{p}\leq p_1}}
 \times_{X\times X}\bigl((X-N)\times\{0\}\bigr)$.
 Let $(p_{\st{X}^{d_1}_1},\mr{Res}_1)\colon
 \st{X}^{d_1}_{1}\rightarrow (X-N)\times
 \bigl(\overline{\st{C}}^{r_1,N}
 \times_{\mb{A}^N/\mb{G}^N_m}\{0\}\bigr)$ be the
 canonical morphism which is smooth by \cite[III.5]{Laf}.
 Since $\overline{\mr{Cht}^{r_1,d_1,\overline{p}\leq p_1}}
 \times_{\overline{\st{C}}^{r_1,N}}
 \overline{\st{C}}^{r_1,N}_\emptyset\cong
 \mr{Cht}^{r_1,d_1,\overline{p}\leq p_1}$ by \cite[III.2]{Laf}, we have
 the following cartesian diagram:
 \begin{equation*}
  \xymatrix{
   \mr{Cht}^{r_1,d_1,\overline{p}\leq p_1}\times
   _{\overline{\st{C}}^{r_1,N}_\emptyset}\widetilde{\st{B}}'_1
   \ar[r]\ar[d]\ar@{}[rd]|\square&
   \widetilde{\st{B}}'_1\ar[d]\\
  \st{X}^{d_1}_1\ar[r]^-{\mr{Res}_1}&
   \overline{\st{C}}^{r_1,N}
   \times_{\mb{A}^N/\mb{G}^N_m}\{0\}.
   }
 \end{equation*}
\end{step}

\begin{step}
 We are now reduced to the following claim:
 \begin{cl}[{\cite[VI.18]{Laf}}]
  For any constructible object $\ms{M}_1$ on
  $\overline{\st{C}}^{r_1,N}\times_{\mb{A}^N/\mb{G}^N_m}\{0\}$,
  the relative cohomology $\H^\nu
  p_{\st{X}^{d_1}_1}\mr{Res}_1^+\ms{M}_1$
  is a smooth object and $r$-negligible as an object in
  $\mc{I}_{d_0}(F_{d_0})$ for any $\nu$.
 \end{cl}
 \begin{proof}
  Since $(p_{\st{X}^{d_1}_1},\mr{Res}_1)$ is smooth, the relative
  cohomology is smooth by Lemma \ref{propsmoobcspeci}.
  We show the $r$-negligibility by the induction on $r_1$ (called the
  {\em rank}) and the dimension of the support of $\ms{M}_1$. Assume
  that the result is
  known for rank $<r_1$ and for the dimension of the support being
  $<c$. It suffices to show the claim for $\ms{M}_1$ which is the zero
  extension of a smooth object defined on a locally closed substack
  $\st{C}$ in
  $\overline{\st{C}}^{r_1,N}\times_{\mb{A}^N/\mb{G}^N_m}\{0\}$ of
  dimension $c$. By induction hypothesis, we may assume that $\st{C}$ is
  in the stratification associated to a partition $\underline{r}_1$ of
  $r_1$. When $\underline{r}_1$ is non-trivial, we can reduce to the
  claim for smaller ranks by arguments similar to Step
  \ref{firststepdifflem} to \ref{laststepdifflem}. Thus, this case is
  true by the induction hypothesis, and we only need to treat the case
  where $\underline{r}_1$ is trivial.
 In this case, we may assume that $\st{C}$ is a locally
  closed point $\widetilde{\st{B}}_1$ of
  $\overline{\st{C}}^{r_1,N}_{\emptyset}\times_{\mb{A}^N/\mb{G}^N_m}
  \{0\}$. Denote by $\widetilde{\st{B}}'_1$ the finite \'{e}tale
  covering of $\widetilde{\st{B}}_1$ constructed as before, and let
  $\widetilde{\st{B}}'_1\xrightarrow{\alpha}
  \widetilde{\st{B}}_1\xrightarrow{j}
  \overline{\st{C}}^{r_1,N}_{\emptyset}\times_{\mb{A}^N/\mb{G}^N_m}
  \{0\}$. Recalling the intermediate extension from
  \ref{openimmdefshri}, put
  $\ms{M}'_1:=j_{!+}\alpha_+(\overline{\mb{Q}}_p)$, which is pure of
  weight $0$ by the result of the same paragraph.
  It suffices to show the claim for
  $\ms{M}'_1$ by the induction hypothesis. Since $\mr{Res}_1$ is smooth
  and $p_{\st{X}^{d_1}_1}$ is proper,
  $\cH^\nu p_{\st{X}^{d_1}_1!}\mr{Res}^+_1(\ms{M}'_1)$ is pure of weight
  $\nu$ by Theorem \ref{theoryofweightadmiss}. Thus, it suffices to show
  that the alternating sum
  \begin{equation}
   \label{padicsumneedneg}
    \tag{$\star$}
    \sum(-1)^\nu
    \left[\cH^\nu p_{\st{X}^{d_1}_1!}\mr{Res}^+_1(\ms{M}'_1)
    \right]^{\mr{ss}}
  \end{equation}
  as an object in $\Gr\mc{I}_{d_0}(F_{d_0})$ is $r$-negligible. Let
  $\mc{F}'_1:=j_{!+}\alpha_+(\overline{\mb{Q}}_\ell)$, the $\ell$-adic
  counterpart.
  Since $\st{X}_1$ is smooth over $(X-N)\otimes_{\mb{F}_q}\{0\}$
  and admissible by \cite[II.4]{Laf}, we can use
  Theorem \ref{indepofltrace} to see that the Frobenius eigenvalues of
  each point of $(X-N)\otimes0$ coincide with those of
  \begin{equation*}
   \sum(-1)^\nu
    \left[\H^\nu p_{\st{X}^{d_1}_1!}\mr{Res}^*_1(\mc{F}'_1)\right]
    ^{\mr{ss}}.
  \end{equation*}
  This sum is known to be $r$-negligible exactly by the corresponding
  claim for $\ell$-adic situation, namely \cite[VI.18]{Laf}.
  Thus, there exists a finite collection $\{\sigma'_\iota\}$ in
  $\mc{W}_{\ell}((X-N)\otimes\{0\})$ such that $\sigma'_\iota$ is a
  constituent of the pull-back of an irreducible object
  $\sigma_\iota\in\mc{W}_{\ell}(X-N)$ of rank $<r$
  and integers $a_\iota$ such that the sum is equal to
  $\sum a_\iota\sigma'_\iota$. By the
  induction hypothesis \ref{starproof} \eqref{prodfoumindhyp}, there
  exists $\mc{E}'_\iota\in\mc{I}((X-N)\otimes\{0\})$ which corresponds
  to $\sigma'_\iota$ in the sense of Langlands.
  By \v{C}ebotarev density \ref{Cebdensity}, the sum
  (\ref{padicsumneedneg}) is equal to $\sum a_\iota\mc{E}'_\iota$. Let
  $\mc{E}_\iota\in\mc{I}(X-N)$ which corresponds to $\sigma_\iota$ in
  the sense of Langlands, whose existence is assured by using the
  induction hypothesis once again.
  Since $\mc{E}'_\iota$ is a constituent of the pull-back of
  $\mc{E}_\iota$, we conclude that (\ref{padicsumneedneg}) is
  $r$-negligible as required.
  \renewcommand{\qedsymbol}{$\square\blacksquare$}
 \end{proof}
\end{step}
 \renewcommand{\qedsymbol}{}
\end{proof}
\vspace{-6.5ex}

\subsubsection{}
\label{impcorrnegbound}
Recall the notation of \ref{defofsomenotepi}. The proposition is used in
the following form:
\begin{cor*}
 For any $\nu$, consider the canonical homomorphisms:
 \begin{equation*}
  \mc{H}^{\nu}_{\mr{c}}(\Ch^{r,\overline{p}\leq p}_N/a^{\mb{Z}})
   \rightarrow
   \mc{H}^{\nu}_{\mr{c}}(\Chtpr{r,\overline{p}\leq
   p}{N}/a^{\mb{Z}}),\quad
   \mc{H}^{\nu}_{\mr{c}}(\Ch^{r,\overline{p}\leq p}_N/a^{\mb{Z}})
   \rightarrow
   \IH^\nu(\Chtprop{r,\overline{p}\leq p}{N}
   /a^{\mb{Z}}).
 \end{equation*}
 These four objects are smooth on $S_N$, and the kernels and cokernels
 of these homomorphisms are $r$-negligible.
\end{cor*}
\begin{proof}
 To show the claim for the first homomorphism, take
 $\st{C}:=\st{C}'^r_N$ (cf.\ \cite[III 3b)]{Laf})
 in Proposition \ref{keylemmLaff}, then $\st{X}$
 is $\Chtpr{r,d,\overline{p}\leq p}{N}$ by \cite[Cor
 III.14]{Laf}.
 Taking $\ms{M}$ to be the trivial object of the boundary, we get the
 claim for the first homomorphism by the proposition and the
 localization sequence \ref{localtriangs}. To check the
 second one, we take $\st{C}$ to be $\st{C}^r_N$, then $\st{X}$ is
 $\Chtprop{r,d,\overline{p}\leq p}{N}$ by \cite[Def III.8]{Laf}. Take
 $\ms{M}$ to be the intersection complex of $\st{C}^r_N$ restricted to
 the boundary, and we get the claim.
\end{proof}

\subsubsection{}
\label{essenpartdefcond}
Let us extract the ``essential part'' from the relative cohomology
object $\mc{H}_{\mr{c}}^*(\Ch^{r,\overline{p}\leq
p}_N/a^{\mb{Z}})$. Note that
this object is smooth on $S_N$ by Corollary \ref{impcorrnegbound}.
We fix a prime number $\ell\neq p$ from now on, till the end of this
subsection. Let $\mc{H}^*_{N,\mr{ess},\ell}$ be the object in
$\mc{W}_{\ell}(S_N)$ denoted by $H^*_{N,\mr{ess}}$ in
\cite[VI.19]{Laf}.

\begin{lem*} 
 There exists a unique element $\mc{H}^*_{N,\mr{ess}}$ in
 $\mb{Q}\Gr\mc{I}(S_N)$ such that for any closed point $x\in S_N$,
 \begin{equation}
  \label{tracecondonHess}
   \mr{Tr}_{\mc{H}^*_{N,\mr{ess}}}\bigl(\mr{Frob}_x^{s}\bigr)=
   \mr{Tr}_{\mc{H}^*_{N,\mr{ess},\ell}}
   \bigl(\mr{Frob}_x^{s}\bigr).
 \end{equation}
 Now, put $\mc{H}^*_{\mr{c}}
 \bigl(\mr{Cht}^{r,\overline{p}\leq p}_N/a^{\mb{Z}}
 \bigr)^{\mr{ss}}:=\sum(-1)^{\nu}\,\mc{H}^{\nu}_{\mr{c}}
 \bigl(\mr{Cht}^{r,\overline{p}\leq p}_N/a^{\mb{Z}}
 \bigr)^{\mr{ss}}$ in $\Gr\mc{I}(S_N)$.
 Then, the formal difference
 \begin{equation}
  \tag{$\star$}
  \label{complenegcondonHess}
   \mc{H}_{N,\mr{ess}}^*-\frac{1}{r^2!}\sum_{n=1}^{r^2!}
   \bigl(\mr{Frob}^n_X\times\mr{id}_X\bigr)^+\,\mc{H}^*_{\mr{c}}
   \bigl(\mr{Cht}^{r,\overline{p}\leq p}_N/a^{\mb{Z}}
   \bigr)^{\mr{ss}},
 \end{equation}
 considered as an element of $\mb{Q}\Gr\mc{I}(F^2)$, is complete
 $r$-negligible.
\end{lem*}
\begin{proof}
 Since $\mc{H}^*_{N,\mr{ess},\ell}$ is pure by \cite[VI.20 (i)]{Laf},
 $\mc{H}^*_{N,\mr{ess}}$ is pure as well if it exists. Thus, the
 uniqueness readily follows from \v{C}ebotarev density
 \ref{Cebdensity}. Let us show the existence. Since
 $\mc{H}^*_{N,\mr{ess},\ell}$ is stable under the pull-back by
 $\mr{Frob}_X\times\mr{id}_X$ (cf.\ \cite[right after VI.22]{Laf}) and
 using \cite[VI.19]{Laf},
 $(r!)^{-1}\,\sum_{n=1}^{r!}
 \bigl(\mr{Frob}^n_X\times\mr{id}_X\bigr)^*\,\mc{H}^*_{\mr{c}}
 \bigl(\mr{Cht}^{r,\overline{p}\leq p}_N/a^{\mb{Z}},\mb{Q}_\ell
 \bigr)^{\mr{ss}}$ is stable under the pull-back as well.
 This implies that
 $(r!)^{-1}\sum_{n=1}^{r!}(\dots)=(r^2!)^{-1}\sum_{n=1}^{r^2!}(\dots)$.
 Thus, there exist complete $r$-negligible $\ell$-adic sheaves
 $\sigma_\iota$ and rational constants $c_\iota$ such that
 \begin{equation*}
 \mc{H}^*_{N,\mr{ess},\ell}-\frac{1}{r^2!}\,\sum_{n=1}^{r^2!}
  \bigl(\mr{Frob}^n_X\times\mr{id}_X\bigr)^*\,\mc{H}^*_{\mr{c}}
  \bigl(\mr{Cht}^{r,\overline{p}\leq p}_N/a^{\mb{Z}},\mb{Q}_\ell
  \bigr)^{\mr{ss}}
  =-\sum_{\iota}c_\iota\cdot\sigma_\iota,
 \end{equation*}
 by \cite[VI.19 (i)]{Laf}.
 Since $\sigma_\iota$ is complete $r$-negligible for any $\iota$, there
 exist $\ell$-adic sheaves $\sigma'$ and $\sigma''$ on $X-N$ of rank
 $<r$ such that $\sigma_\iota\cong q'^*\sigma'\otimes q''^*\sigma''$ by
 definition.
 By the induction hypothesis \ref{starproof} \eqref{prodfoumindhyp},
 there exist $\mc{E}'$ and $\mc{E}''$ in $\mc{I}(X-N)$ which correspond
 to $\sigma'$ and $\sigma''$ in the sense of Langlands. The object
 $\mc{E}_\iota:=q'^+\mc{E}'\otimes q''^+\mc{E}''$ in $\mc{I}(S_N)$
 corresponds to $\sigma_\iota$ in the sense of
 Langlands. Note that this $\mc{E}_\iota$ is complete $r$-negligible by
 construction. Put
 \begin{equation*}
 \mc{H}^*_{N,\mr{ess}}:=\frac{1}{r^2!}\,\sum_{n=1}^{r^2!}
  \bigl(\mr{Frob}^n_X\times\mr{id}_X\bigr)^+\,\mc{H}^*_{\mr{c}}
  \bigl(\mr{Cht}^{r,\overline{p}\leq p}_N/a^{\mb{Z}}
  \bigr)^{\mr{ss}}
  -\sum_{\iota}c_\iota\cdot\mc{E}_{\iota}.
 \end{equation*}
 Now, the difference (\ref{complenegcondonHess}) is
 $\sum_{\iota}c_\iota\cdot\mc{E}_{\iota}$, thus it is complete
 $r$-negligible.
 Since the c-admissible stack
 $\mr{Cht}^{r,\overline{p}\leq p}_N/a^{\mb{Z}}$ is
 smooth over $S_N$ (cf.\ \ref{tablechoutsatis}), we may use the base
 change \ref{basechforshrik} and Theorem \ref{indepofltrace} to show
 that it satisfies (\ref{tracecondonHess}), and the element meets our
 need.
\end{proof}

\begin{prop}[{\cite[VI.20]{Laf}}]
 \label{fundpropnegH}
 (i) None of the irreducible components of $\mc{H}^*_{N,\mr{ess}}$ are
 $r$-negligible. All the components have positive multiplicity, and pure
 of weight $2r-2$.

 (ii) The object $\mc{H}^{\nu}_{\mr{c}}
   \bigl(\mr{Cht}^{r,\overline{p}\leq p}_N/a^{\mb{Z}}
   \bigr)^{\mr{ss}}$ is $r$-negligible for $\nu\neq 2r-2$. Moreover, the
 following difference is $r$-negligible:
 \begin{equation*}
  \mc{H}^*_{N,\mr{ess}}-
   \frac{1}{r^2!}\sum_{n=1}^{r^2!}
   \bigl(\mr{Frob}^n_X\times\mr{id}_X\bigr)^+\,\mc{H}^{2r-2}_{\mr{c}}
   \bigl(\mr{Cht}^{r,\overline{p}\leq p}_N/a^{\mb{Z}}
   \bigr)^{\mr{ss}}.
 \end{equation*}
\end{prop}
\begin{proof}
 We have two proofs, both of which use Proposition \ref{keylemmLaff}
 substantially. The first one is just to copy the proof of {\it
 ibid.}. The details are left to the reader.
 The second one is to make use of the Lafforgue's $\ell$-adic result.
 First, we prove (i).
 For an object $A$ in $\mb{Q}\Gr\mc{W}_\ell(S_N)$ or
 $\mb{Q}\Gr\mc{I}(S_N)$, let $\{A\}$ be the set of constituents
 of $A$, and $\{A\}_{\mr{neg}}$ be the subset consisting of
 $r$-negligible objects. Writing $A=\sum_{B\in\{A\}}c_BB$ with
 $c_B\in\mb{Q}$, we put
 $A_{\mr{neg}}:=\sum_{B\in\{A\}_{\mr{neg}}}c_BB$.
 For objects $A$ and $B$ in $\mb{Q}\Gr\mc{W}_\ell(S_N)$ or
 $\mb{Q}\Gr\mc{I}(S_N)$, we write $A\eqtr B$ if
 $\mr{Tr}_{A}(\mr{Frob}^s_x)=\mr{Tr}_{B}(\mr{Frob}^s_x)$
 for any integer $s$, $x\in|S_N|$. Denoting $\Chtprop{}{}:=
 \Chtprop{r,\overline{p}\leq p}{N}/a^{\mb{Z}}$, we put
 \begin{align*}
  \mc{IH}^\nu&:=\frac{1}{r^2!}\sum_{n=1}^{r^2!}
   (\mr{Frob}_X^n\times\mr{id}_X)^+
   \IH^\nu(\Chtprop{}{})^{\mr{ss}},\\
   \mc{IH}^\nu_{\ell}&:=
   \frac{1}{r^2!}
   \sum_{n=1}^{r^2!}(\mr{Frob}_X^n\times\mr{id}_X)^*
   \IH^\nu(\Chtprop{}{},\mb{Q}_\ell)^{\mr{ss}},
 \end{align*}
 and $\mc{IH}^*_{(\ell)}:=\sum_{\nu=1}^{2(2r-2)}(-1)^\nu
 \mc{IH}^{\nu}_{(\ell)}$. Summarizing [{\it ibid.}, VI.15, 20], we
 already know the following:
 \begin{quote}
  The difference $\mc{IH}^{2r-2}_\ell-\mc{H}^*_{N,\mr{ess},\ell}$
  as well as $\mc{IH}^{\nu}_\ell$ for $\nu\neq 2r-2$ are
  $r$-negligible, and any irreducible component of
  $\mc{H}^*_{N,\mr{ess},\ell}$ is not
  $r$-negligible.
 \end{quote}  
 This implies that
 \begin{equation*}
  \mc{IH}^{2r-2}_\ell-\mc{H}^*_{N,\mr{ess},\ell}=
   (\mc{IH}^{2r-2}_\ell)_{\mr{neg}}
 \end{equation*}
 in $\mb{Q}\Gr\mc{W}_\ell(S_N)$.
 Now, let us show the following three equalities for any $\nu$:
 \begin{gather*}
  \mc{IH}_\ell^\nu\eqtr\mc{IH}^\nu,\qquad
  (\mc{IH}^{2r-2}_\ell)_{\mr{neg}}\eqtr
  (\mc{IH}^{2r-2})_{\mr{neg}},\\  
  \mc{IH}^{2r-2}_\ell-\mc{H}^*_{N,\mr{ess},\ell}\eqtr
  \mc{IH}^{2r-2}-\mc{H}^*_{N,\mr{ess}}.
 \end{gather*}
 Indeed, by $\ell$-independence Theorem \ref{Gabberfujiwarateh}, we know
 that $\mc{IH}^*_\ell\eqtr\mc{IH}^*$. By purity (cf.\ Theorem
 \ref{theoryofweightadmiss} (ii)), $\mc{H}^\nu_\ell$
 and $\mc{H}^\nu$ are pure of weight $\nu$, thus we get the first
 equality. The third one follows by (\ref{tracecondonHess}) and the
 first equality, and it remains to show the second one.
 Let $\sigma\in\mb{Q}\Gr\mc{W}_\ell(S_N)$ and
 $\mc{E}\in\mb{Q}\Gr\mc{I}(S_N)$ such that both are
 invariant under pull-back by $\mr{Frob}_X\times\mr{id}_X$,
 $\sigma\eqtr\mc{E}$, and both are positive and pure.
 The second equality follows if we can show
 $\sigma_{\mr{neg}}\eqtr\mc{E}_{\mr{neg}}$. Let us check this.
 Multiplying $\sigma$ and $\mc{E}$ by some integer, we may assume that
 they are semi-simple objects in $\mc{W}_\ell(S_N)$ and $\mc{I}(S_N)$.
 Let $\sigma'$ be an irreducible $r$-negligible constituent of
 $\sigma$. Since $\sigma$ is invariant under the pull-back by
 $\mr{Frob}_X\times\mr{id}$, $\sigma$ contains a complete $r$-negligible
 sheaf $\widetilde{\sigma}'$ which is the external tensor product of
 irreducible sheaves on $X-N$ and contains $\sigma'$ as a constituent.
 By the induction hypothesis \ref{starproof} \eqref{prodfoumindhyp},
 there exists an $r$-negligible object $\widetilde{\mc{E}}'$ in
 $\mc{I}(S_N)$ such that $\widetilde{\sigma}'\eqtr\widetilde{\mc{E}}'$.
 By Corollary \ref{conseqweighidenti}, at least one constituent of
 $\widetilde{\mc{E}}'$ is contained in $\mc{E}$, and since $\mc{E}$ is
 stable under pull-back by $\mr{Frob}_X\times\mr{id}$, $\mc{E}''$ is
 contained in $\mc{E}$. Similarly, if we are given $\widetilde{\mc{E}}'$
 which is complete $r$-negligible and can be written as the external
 tensor product of irreducible objects on $X-N$, there exists
 $\widetilde{\sigma}'$ in $\sigma$ such that
 $\widetilde{\mc{E}}'\eqtr\widetilde{\sigma}'$. Thus the claim follows.

 Combining all of these three equalities, we get
 \begin{equation*}
  \mc{IH}^{2r-2}-\mc{H}^*_{N,\mr{ess}}
   \eqtr
   (\mc{IH}^{2r-2})_{\mr{neg}},
 \end{equation*}
 thus $\mc{IH}^{2r-2}-\mc{H}^*_{N,\mr{ess}}=
 (\mc{IH}^{2r-2})_{\mr{neg}}$ in $\mb{Q}\Gr\mc{I}(S_N)$ by
 \v{C}ebotarev density \ref{Cebdensity}.

 Now, let us prove (ii). Since $\mc{IH}^{\nu}_\ell$ is $r$-negligible
 and $\mc{IH}^\nu_\ell\eqtr\mc{IH}^\nu$ as we have already seen,
 $\mc{IH}^\nu$ is $r$-negligible as well. Since we know that the
 difference of $\mc{IH}^\nu$ and $\mc{H}^{\nu}_{\mr{c}}
 \bigl(\mr{Cht}^{r,\overline{p}\leq p}_N/a^{\mb{Z}}\bigr)
 ^{\mr{ss}}$ is $r$-negligible for any $\nu$ by Corollary
 \ref{impcorrnegbound}, (ii) follows from (i).
\end{proof}

\begin{cor}[{\cite[VI.21]{Laf}}]
 \label{kercokerneg}
 Let $p\leq q$ be large enough convex polygons. The kernel and cokernel of
 the induced homomorphisms
 \begin{align*}
  \alpha&\colon
   \mc{H}^{2r-2}_{\mr{c}}(\Ch^{r,\overline{p}\leq p}_N/a^{\mb{Z}})
   \rightarrow
  \mc{H}^{2r-2}_{\mr{c}}(\Chtpr{r,\overline{p}\leq
  p}{N}/a^{\mb{Z}}),\\
  \beta&\colon
  \mc{H}^{2r-2}_{\mr{c}}(\Ch^{r,\overline{p}\leq p}_N/a^{\mb{Z}})
  \rightarrow
  \mc{H}^{2r-2}_{\mr{c}}(\Ch^{r,\overline{p}\leq q}_N/a^{\mb{Z}})
 \end{align*}
 by the inclusions $\Ch^{r,\overline{p}\leq
 p}_N/a^{\mb{Z}}\hookrightarrow\Chtpr{r,\overline{p}\leq p}{N}
 /a^{\mb{Z}}$ and $\Ch^{r,\overline{p}\leq
 p}_N/a^{\mb{Z}}\hookrightarrow\Ch^{r,\overline{p}\leq
 q}_N/a^{\mb{Z}}$ are $r$-negligible.
\end{cor}
\begin{proof}
 We may prove similarly to {\it ibid.}:
 The claim for $\alpha$ is just a reproduction of Corollary
 \ref{impcorrnegbound}. Let $\Gamma$ be the correspondence from
 $\Chtpr{r,\overline{p}\leq p}{N}/a^{\mb{Z}}$ to
 $\Chtpr{r,\overline{p}\leq q}{N}/a^{\mb{Z}}$ defined in
 \cite[V.14 (i)]{Laf}. This yields the following commutative diagram
 by Lemma \ref{corresnotdef}:
 \begin{equation*}
  \xymatrix{
   \mc{H}^{2r-2}_{\mr{c}}(\Ch^{r,\overline{p}\leq q}_N/a^{\mb{Z}})
   \ar[r]&
   \mc{H}^{2r-2}_{\mr{c}}(\Chtpr{r,\overline{p}\leq q}{N}/a^{\mb{Z}})
   \ar[d]^{\Gamma^*}\\
  \mc{H}^{2r-2}_{\mr{c}}(\Ch^{r,\overline{p}\leq p}_N/a^{\mb{Z}})
   \ar[u]^{\beta}\ar[r]_-{\alpha}&
   \mc{H}^{2r-2}_{\mr{c}}(\Chtpr{r,\overline{p}\leq p}{N}/a^{\mb{Z}}),
  }
 \end{equation*}
 where $\Gamma^*$ is
 the action of $\Gamma$ as in Definition \ref{corresnotdef}, and the
 other three arrows are defined by the inclusions. Thus, we
 have $\mr{Ker}(\beta)\subset\mr{Ker}(\alpha)$, and
 $\mr{Ker}(\beta)$ is $r$-negligible since we already know that
 $\mr{Ker}(\alpha)$ is. This implies that, to show the $r$-negligibility
 of $\mr{Coker}(\beta)$, it suffices to show that the ``essential
 parts'', namely the set of constituents which are not $r$-negligible,
 of the source and the target of $\beta$ are the same. This follows by
 the previous proposition, which concludes the proof.
\end{proof}

\subsubsection{}
\label{exteritensprodcat}
We digress a little, and recall some general nonsense.
Let $F$ be an algebraically closed field, and $\mc{A}$ an abelian
category over $F$ such that any object in $\mc{A}$ has finite
length. We assume, moreover, that for any
$X,X'$ in $\mc{A}$, $\mr{Hom}_{\mc{A}}(X,X')$ is a finite
dimensional $F$-vector space.
We note that for an irreducible object $X$, $\mr{End}(X)\cong F$. For
$X\in\mc{A}$, we have the functor
\begin{equation*}
 \mr{Hom}_{\mc{A}}(X,-)\colon
  \mc{A}\rightarrow\mr{Vec}^{\mr{fin}}_F,
\end{equation*}
where $\mr{Vec}^{\mr{fin}}_F$ is the category of finite dimensional
$F$-vector spaces. We can check that it has a left adjoint,
denoted by $(-)\boxtimes X$. For an integer $n>0$, we have
$F^{\oplus n}\boxtimes X\cong X^{\oplus n}$.

Let $G$ be a group, and we denote by $G\mbox{-}\mc{A}$ the category of
objects of $\mc{A}$ equipped with $G$-action, namely couples
$(X,\rho)$ where $X\in\mc{A}$ and a group homomorphism
$\rho\colon G\rightarrow\mr{Aut}(X)$. Let $\rho\colon
G\rightarrow\mr{GL}(V)$ be a finite dimensional $F$-representation of
$G$. Then $V\boxtimes X$ is equipped with action of $G$ determined
by $\rho$, and thus defines an object of $G\mbox{-}\mc{A}$. We sometimes
denote this by $\rho\boxtimes X$. The following lemma is
well-known (cf.\ \cite[4.15.8]{eting}):

\begin{lem*}
 (i) For any irreducible $F$-representation $\rho$ of $G$ and
 irreducible object $X$ of $\mc{A}$, $\rho\boxtimes X$ is an irreducible
 object of $G\mbox{-}\mc{A}$.

 (ii) Conversely, any irreducible object of $G\mbox{-}\mc{A}$ can be
 written in a such form.
\end{lem*}
In practice, we take $\mc{A}=\mc{I}(S_N)$. Of course, in this case
$F=\overline{\mb{Q}}_p$. As we have already recalled, any object of
$\mc{I}(S_N)$ has finite length, and $\mr{Hom}_{\mc{A}}$ is finite
dimensional by the existence of six functor
formalism. Thus, the results of this paragraph is applicable to this
category.

\subsubsection{}
\label{heckeactiondefined}
Next thing we need to do is to define a Hecke action on
$\mc{H}^*_{N,\mr{ess}}$. We put $(\Ch^r_N/a^{\mb{Z}})_\eta=\indlim
\Ch^{r,\overline{p}\leq p}_N/a^{\mb{Z}}\times_{S_N}U$ where
the inductive limit runs over $p$ and open dense subschemes $U\subset
X\times X$. Lafforgue constructed a homomorphism of
algebras (cf.\ \ref{corresnotdef})
\begin{equation*}
 \varrho\colon\mc{H}^r_N/a^{\mb{Z}}\rightarrow
  \mr{Corr}^{\mr{fin,et}}_{\eta}
  \bigl((\Ch^r_N/a^{\mb{Z}})_{\eta}\bigr)
\end{equation*}
sending $f$ to $\Gamma_f$, which is a finite \'{e}tale correspondence on
a certain open dense subscheme of $X\times X$ (cf.\ \cite[I 1c), V
2a)]{Laf} or \cite[I.4, Theorem 5]{La1}).
To be compatible with \cite{Laf}, we consider
$\varrho':=\mr{norm}\circ\varrho$ for the action on the relative
cohomology (cf.\ \ref{complafourcorre}).
Moreover, we have partial Frobenius endomorphisms $\mr{Frob}_\infty$ and
$\mr{Frob}_0$ on $\Ch^r_N/a^{\mb{Z}}$ over $\mr{Frob}_X\times\mr{id}_X$
and $\mr{id}_X\times\mr{Frob}_X$ respectively such that
$\mr{Frob}_0\circ\mr{Frob}_{\infty}=
\mr{Frob}_{\infty}\circ\mr{Frob}_{0}=\mr{Frob}$ 
(cf.\ \cite[I 1b)]{Laf}).
The action of correspondence on $\mc{H}^*_{N,\mr{ess}}$ commutes with
$\mr{Frob}_{\infty}$ and $\mr{Frob}_0$ (cf.\ \cite[I 1c)]{Laf}).
For any convex polygon $p$,
$f\in\mc{H}^r_N/a^{\mb{Z}}$, $s,u\in\mb{N}$, there exists a convex
polygon $q\geq p$ such that the correspondence
$f\times\mr{Frob}_\infty^s\times\mr{Frob}_0^u$ sends
$\Ch^{r,\overline{p}\leq p}_N/a^{\mb{Z}}$ to
$\Ch^{r,\overline{p}\leq q}_N/a^{\mb{Z}}$ over some open subscheme
$U\subset S_N$. Thus, via $\varrho'$, we have a homomorphism
\begin{align*}
 (f\times\mr{Frob}_\infty^s\times\mr{Frob}_0^u)^*\colon
 (\mr{Frob}_\infty^s\times\mr{Frob}_0^u)^+
 \mc{H}^{2r-2}_{\mr{c}}&(\Ch^{r,\overline{p}\leq q}
 _N/a^{\mb{Z}}\times_{S_N}U)\\
 &\rightarrow
 \mc{H}^{2r-2}_{\mr{c}}(\Ch^{r,\overline{p}\leq p}
 _N/a^{\mb{Z}}\times_{S_N}U)
\end{align*}
compatible with compositions by Lemma \ref{corresnotdef}.
Note that since $\mr{Frob}^*=(\mr{Frob}_\infty\times\mr{Frob}_0)^*$ is
an isomorphism, $(\mr{Frob}_\infty\times\mr{id})^*$ and
$(\mr{id}\times\mr{Frob}_0)^*$ are isomorphisms.

Let us define a filtration as in \cite[VI 3a)]{Laf}.
Let $\mc{E}\in\mc{I}(S_N)$.
Then we may define a canonical filtration $F^{\bullet}$ as follows:
Put $F^0\mc{E}=0$. Assume $F^{2i}\mc{E}$ has already been
constructed. We define $F^{2i+1}\mc{E}$ to be the largest submodule such
that $F^{2i+1}\mc{E}/F^{2i}\mc{E}$ is {\em $r$-negligible}. Then, we put
$F^{2i+2}\mc{E}$ to be the largest submodule such that
$F^{2i+2}\mc{E}/F^{2i+1}\mc{E}$ is {\em essential}.
This filtration is functorial, namely, given a morphism
$\mc{E}\rightarrow\mc{E}'$ in $\mc{I}(F^2)$, then it induces a
morphism $F^i\mc{E}\rightarrow F^i\mc{E}'$. We put the ``essential
part'' of $\mc{E}$
\begin{equation*}
 \mc{E}_{\mr{even}}:=\bigoplus_{i\geq0}F^{2i+2}/F^{2i+1}(\mc{E}).
\end{equation*}

For a convex polygon $p$, we put $\mc{H}^{\leq
p}:=\mc{H}^{2r-2}_{\mr{c}}(\Ch^{r,\overline{p}\leq p}
_N/a^{\mb{Z}})$, which is an object in $\mc{I}(S_N)$.
For a large enough convex polygon $p$ and $p'\geq p$, Corollary
\ref{kercokerneg} tells us that the homomorphism induced
by the canonical homomorphism $\mc{H}^{\leq p}\rightarrow\mc{H}^{\leq
p'}$
\begin{equation*}
 \mc{H}^{\leq p}_{\mr{even}}\rightarrow\mc{H}^{\leq p'}_{\mr{even}}
\end{equation*}
is an isomorphism in $\mc{I}(S_N)$. Because of the
functoriality of the filtration, the action of the correspondence
$(f\times\mr{Frob}^s_{\infty}\times\mr{Frob}^u_{0})^*$ induces an action
of $\mc{H}^r_N/a^{\mb{Z}}$ as well as the invariance by the pull-back
$(\mr{Frob}_X\times\mr{id}_X)^+$ on $\mc{H}^{\leq p}_{\mr{even}}$ as an
object of $\mc{I}(F^2)$.
Summing up, $\mc{H}^{\leq p}_{\mr{even}}$, for large enough $p$, can
be seen as an object of $\mc{H}^r_N/a^{\mb{Z}}\text{-}\mb{Z}\mc{I}(F^2)$
(not only of $\mc{I}(F^2)$!).

Now, the invariance by $(\mr{Frob}_X\times\mr{id}_X)^+$ shows that
\begin{equation*}
 \frac{1}{r^2!}\sum_{n=1}^{r^2!}(\mr{Frob}_X^n\times\mr{id}_X)^+
  (\mc{H}^{\leq p}_{\mr{even}})
  =\mc{H}^{\leq p}_{\mr{even}}
\end{equation*}
in $\Gr\mc{I}(F^2)$.
Thus by Proposition \ref{fundpropnegH} (ii), it turns out that
$\mc{H}^{\leq p}_{\mr{even}}$ considered in $\mb{Q}\Gr\mc{I}(F^2)$
coincides with $\mc{H}^*_{N,\mr{ess}}$. This implies that the
coefficients of $\mc{H}^*_{N,\mr{ess}}$, which is {\it a priori}
rational numbers by construction, are positive integers.
Now, recall that for a smooth scheme $Y$ over $k$
and an open dense subscheme $U\subset Y$, the restriction functor
$\mc{I}(Y)\rightarrow\mc{I}(U)$ is fully faithful (cf.\ \cite[Thm
5.2.1]{KeSS1}).
This implies that, since $\mc{H}^{\leq p}_{\mr{even}}\in\mc{I}(S_N)$,
the action of $\mc{H}^r_N/a^{\mb{Z}}$ and the invariance by
$(\mr{Frob}_X\times\mr{id}_X)^+$ of
$\mc{H}^{\leq p}_{\mr{even}}$ as an object of $\mc{I}(F^2)$ can be
extended uniquely to
an action of $\mc{H}^r_N/a^{\mb{Z}}$ and an invariance by
$(\mr{Frob}_X\times\mr{id}_X)^+$ as an object of $\mc{I}(S_N)$. Thus,
$\mc{H}^{\leq p}_{\mr{even}}$ can be seen as an element of
$\mc{H}^r_N/a^{\mb{Z}}\text{-}\mb{Z}\mc{I}(S_N)$. We denote by
$\mc{H}_{N,\mr{ess}}$ the semi-simplification of $\mc{H}^{\leq
p}_{\mr{even}}$ (for $p$ large enough needless to say)
as an object of $\mc{H}^r_N/a^{\mb{Z}}\text{-}\mb{Z}\mc{I}(S_N)$, which
is equal to $\mc{H}_{N,\mr{ess}}^*$ considered as elements of
$\Gr\mc{I}(S_N)$.

\subsubsection{}
We need to calculate the trace of the Hecke action on
$\mc{H}_{N,\mr{ess}}$. Let $f\in\mc{H}^r_N/a^{\mb{Z}}$.
Recall that $\Gamma_f$ is a correspondence
\begin{equation*}
 \Ch^{r,\overline{p}\leq p}_N/a^{\mb{Z}}\times_{S_N}U
  \rightsquigarrow
  \Ch^{r,\overline{p}\leq
  q}_N/a^{\mb{Z}}\times_{S_N}U
\end{equation*}
for some open dense subscheme $U\subset S_N$ and $q\geq p$.
We denote by $\Gamma'_f$ the pull-back of $\Gamma_f$ by the open
immersion
\begin{align*}
 \bigl(\Ch^{r,\overline{p}\leq p}_N/a^{\mb{Z}}\times
  \Ch^{r,\overline{p}\leq p}_N/a^{\mb{Z}}\bigr)
  &\times_{(S_N\times S_N)}(U\times U)\\
  &\hookrightarrow
  \bigl(\Ch^{r,\overline{p}\leq p}_N/a^{\mb{Z}}\times
  \Ch^{r,\overline{p}\leq q}_N/a^{\mb{Z}}\bigr)
  \times_{(S_N\times S_N)}(U\times U).
\end{align*}
Now, we take the normalization of the morphism
\begin{equation*}
 \Gamma'_f\rightarrow
  \Ch^{r,\overline{p}\leq p}_N/a^{\mb{Z}}\times
  \Ch^{r,\overline{p}\leq p}_N/a^{\mb{Z}}\hookrightarrow
  \Chtprop{r,\overline{p}\leq p}{N}/a^{\mb{Z}}\times
  \Chtprop{r,\overline{p}\leq p}{N}/a^{\mb{Z}}.
\end{equation*}
This normalization certainly defines a correspondence on
$\Chtprop{r,\overline{p}\leq p}{N}/a^{\mb{Z}}$.
A marvelous thing is that, in fact, the correspondence stabilizes
$\Chtpr{r,\overline{p}\leq p}{N}/a^{\mb{Z}}$ as Lafforgue
shows in \cite[V.14]{Laf}, and defines a correspondence on it.
Using the filtration of \ref{heckeactiondefined},
the action $\Gamma_f^*$ of the correspondence $\Gamma_f$ on
$\mc{H}^*_{\mr{c}}\bigl(\Chtpr{r,\overline{p}\leq
p}{N}/a^{\mb{Z}}\bigr)$ induces an action on
$\bigl(\mc{H}^*_{\mr{c}}\bigl(\Chtpr{r,\overline{p}\leq
p}{N}/a^{\mb{Z}}\bigr)\bigr)_{\mr{even}}$.

Let us introduce a notation. Let $Y$ be a smooth scheme over $k$,
$\mc{E}\in\mc{I}(Y)$, and $\alpha$ be an endomorphism of
$\mc{E}$. For a closed point $x$ of $Y$, take a geometric point
$\overline{x}$ above $x$. Recalling the notation of
\ref{situationfix},
$\alpha$ induces an endomorphism of $\iota_{\overline{x}}(\mc{E})$ which
commutes with $\mr{Frob}_x$. We denote
$\mr{Tr}\bigl(\alpha\circ\mr{Frob}_x^n:\iota_{\overline{x}}(\mc{E})\bigr)$
by $\mr{Tr}(\alpha\times\mr{Frob}^n_x:\mc{E})$ or
$\mr{Tr}_{\mc{E}}(\alpha\times\mr{Frob}_x^n)$, which does not depend on
the choice of $\overline{x}$.
Using this notation, we put
\begin{equation*}
 \mr{Tr}^{\leq p}_{\mc{H}^*_{N,\mr{ess}}}
  (f\times\mr{Frob}_x^n):=
  \mr{Tr}
  \Bigl(\Gamma^*_f\times\mr{Frob}_x^n:
   \bigl(\mc{H}^{2r-2}_{\mr{c}}\bigl(
   \Chtpr{r,\overline{p}\leq p}{N}/a^{\mb{Z}}
   \bigr)\bigr)_{\mr{even}}
  \Bigr).
\end{equation*}
We need to compare the trace of the action of correspondences on
$\mc{H}_{N,\mr{ess}}$ defined using $\Ch$ in the previous paragraph and
$\mr{Tr}^{\leq p}_{\mc{H}^*_{N,\mr{ess}}}$.

\begin{lem}[{\cite[VI.23, 24]{Laf}}]
 \label{calccorresprime}
 (i) Let $\mc{E}$ be a mixed object in $\mc{I}(S_N)$, and $f$
 be an endomorphism on $\mc{E}$. Assume that $\mc{E}$ has a filtration
 $F^{\bullet}\mc{E}$ compatible with $f$. Take $I\subset\mb{Z}$, and put
 $\mc{F}:=\bigoplus_{i\in I}\mr{gr}^i(\mc{E})$.
 Let $\{\mc{F}\}$ denote the set of irreducible objects in
 $\mc{I}(S_N)$ appearing in $\mc{F}$. Then
 there exists a unique set of complex numbers $\{c_{\mc{E}'}\}$ such
 that, for any $x\in|S_N|$ and $n\in\mb{Z}$, we have
 \begin{equation}
  \label{traceequaldecomp}
  \tag{$\star$}
  \mr{Tr}_{\mc{F}}(f\times\mr{Frob}_x^n)=
   \sum_{\mc{E}'\in\{\mc{F}\}}c_{\mc{E}'}\cdot
   \mr{Tr}_{\mc{E}'}(\mr{Frob}_x^n).
 \end{equation}

 (ii)  Let $f\in\mc{H}^r_N/a^{\mb{Z}}$, and take a large enough convex polygon
 $p$. We have
 \begin{equation*}
  \mr{Tr}_{\mc{H}_{N,\mr{ess}}}(f\times\mr{Frob}_x^n)=
   \mr{Tr}^{\leq p}_{\mc{H}^*_{N,\mr{ess}}}(f\times\mr{Frob}_x^n)
 \end{equation*}
 where the action on the left hand side is the one defined in
 \ref{heckeactiondefined}.
\end{lem}
\begin{proof}
 Let us show (i). Since $\mc{E}$ is assumed to be mixed, any object in
 $\{\mc{F}\}$ is pure by \cite[4.2.3, 4.3.4]{AC}, thus the uniqueness
 follows by the \v{C}ebotarev density \ref{Cebdensity}.
 We can see $\mc{F}$ as an object of $\mb{Z}\mbox{-}\mc{I}(S_N)$ where
 the action of $\mb{Z}$ is defined by $f$. Let $\mc{G}$ be an
 irreducible object in $\mb{Z}\mbox{-}\mc{I}(S_N)$, then by Lemma
 \ref{exteritensprodcat}, it can be written as $V\boxtimes\mc{G}'$ where
 $V$ is an irreducible representation of $\mb{Z}$ and $\mc{G}'$ is an
 irreducible object in $\mc{I}(S_N)$. Since $\mb{Z}$ is abelian,
 $\dim(V)=1$ and $f$ act as multiplication by $c$. This implies that
 $\mr{Tr}_{\mc{G}}(f\times\mr{Frob}^n_x)=
 c\cdot\mr{Tr}_{\mc{G}'}(\mr{Frob}^n_x)$, and (\ref{traceequaldecomp})
 follows.

 For a proof of (ii), we copy the argument of {\em ibid.}. Let us
 sketch. The commutative diagram in the proof of Corollary
 \ref{kercokerneg} yields the following commutative diagram for $p''\geq
 p'$ sufficiently large (cf.\ \cite[p.185]{Laf})
 \begin{equation*}
  \xymatrix{
   \mc{H}^{2r-2}_{\mr{c}}\bigl(\Ch^{r,\overline{p}\leq
   p''}_N/a^{\mb{Z}}\bigr)\ar[rr]&&
   \mc{H}^{2r-2}_{\mr{c}}\bigl(\Chtpr{r,\overline{p}\leq
   p''}{N}/a^{\mb{Z}}\bigr)\ar[d]\\
  \mc{H}^{2r-2}_{\mr{c}}\bigl(\Ch^{r,\overline{p}\leq
   p}_N/a^{\mb{Z}}\bigr)\ar[u]^{f}\ar[r]&
   \mc{H}^{2r-2}_{\mr{c}}\bigl(\Chtpr{r,\overline{p}\leq
   p}{N}/a^{\mb{Z}}\bigr)\ar[r]_f&
   \mc{H}^{2r-2}_{\mr{c}}\bigl(\Chtpr{r,\overline{p}\leq
   p}{N}/a^{\mb{Z}}\bigr),
   }
 \end{equation*}
 where the homomorphisms marked as ``$f$'' are induced by the
 correspondence associated to $f$, and the others are canonical ones.
 Thus the claim follows by Corollary \ref{kercokerneg}.
\end{proof}

\subsubsection{}
\label{computoftrauselp}
For an unramified irreducible admissible representation $\pi$ of
$\mr{GL}_r(F_x)$ at some place $x$ of $F$ and $t\in\mb{Z}$, we put
\begin{equation*}
 z_{\bullet}^{t}(\pi):=z_1(\pi)^{t}+\dots+z_r(\pi)^{t},
\end{equation*}
where $z_i(\pi)$ denotes the Hecke eigenvalue of $\pi$.
Take a closed point $x$ in $X\times X$. We denote by $\infty_x$ (resp.\
$0_x$) the image of $x$ by the first (resp.\ second) projection. We note that
$\deg(x)=\mr{lcm}\bigl(\deg(\infty_x),\deg(0_x)\bigr)$.

\begin{lem*}[{\cite[VI.25]{Laf}}]
 Let $f\in\mc{H}^r_N/a^{\mb{Z}}$. Then there exists an open dense
 subscheme $U_f\subset S_N$ such that for any $x\in U_f$, we have
 \begin{equation*}
  \mr{Tr}_{\mc{H}_{N,\mr{ess}}}\bigl(f\times
   \mr{Frob}_x^{-s/\deg(x)}\bigr)=
   q^{(r-1)s}\sum_{\pi\in\{\pi\}^r_N}\mr{Tr}_\pi(f)\cdot
   z_{\bullet}^{-s'}(\pi_{\infty_x})\cdot
   z_{\bullet}^{u'}(\pi_{0_x}),
 \end{equation*}
 where $s=\deg(\infty_x)s'=\deg(0_x)u'\in\mb{Z}\cdot\deg(x)$.
\end{lem*}
\begin{proof}
 By Corollary \ref{indepofltrace}, there exists an open dense subscheme
 $U'_f\subset S_N$ such that
 \begin{align}
  \tag{$\star$}
  \label{complandpappl}
  \mr{Tr}\Bigl(f\times\mr{Frob}_x^{-s/\deg(x)}&:
  \mc{H}^*_{\mr{c}}\bigl(\Chtpr{r,\overline{p}\leq
  p}{N}/a^{\mb{Z}}\bigr)\Bigr)\\
  \notag
  &=
  \mr{Tr}\Bigl(f\times\mr{Frob}_x^{-s/\deg(x)}:
  \mc{H}^*_{\mr{c}}\bigl(\Chtpr{r,\overline{p}\leq
  p}{N}/a^{\mb{Z}},\overline{\mb{Q}}_\ell\bigr)\Bigr)
 \end{align}
 for any $x\in U'_f$.
 Let $\mc{H}_{N,\mr{ess},\ell}$ be $H_{N,\mr{ess}}$ defined in
 \cite[after VI.22]{Laf}. Note that since
 $\mc{H}^*_{\mr{c}}(\Chtpr{r,\overline{p}\leq p}{N}
 /a^{\mb{Z}})$ is mixed by Theorem \ref{theoryofweightadmiss},
 Lemma \ref{calccorresprime} (i) is applicable.
 Let $\{\mc{E}\}$ (resp.\ $\{\sigma\}$) be the set of $r$-negligible
 objects in $\mc{I}(S_N)$ (resp.\ in $\mc{W}_{\ell}(S_N)$).
 We have
 \begin{align}
  \label{compfirstuseind}
  \notag
  \mr{Tr}\bigl(f\times\mr{Frob}&_x^{-s/\deg(x)}:
  \mc{H}_{N,\mr{ess}}\bigr)=
  \mr{Tr}^{\leq p}\bigl(f\times\mr{Frob}_x^{-s/\deg(x)}:
   \mc{H}^*_{N,\mr{ess}}\bigr)
  \\
  \tag{$\star\star$}
  =&\mr{Tr}\Bigl(f\times\mr{Frob}_x^{-s/\deg(x)}:
   \mc{H}^*_{\mr{c}}\bigl(\Chtpr{r,\overline{p}\leq
   p}{N}/a^{\mb{Z}}\bigr)\Bigr)\\
  \notag&\qquad\qquad\qquad
  +q^{(r-1)s}\sum_{\mc{E}\in\{\mc{E}\}}c_{\mc{E}}\,
  \mr{Tr}_{\mc{E}}(\mr{Frob}_x^{-s/\deg(x)})\\
  \notag
  =&\mr{Tr}\Bigl(f\times\mr{Frob}_x^{-s/\deg(x)}:
   \mc{H}^*_{\mr{c}}\bigl(\Chtpr{r,\overline{p}\leq
  p}{N}/a^{\mb{Z}},\overline{\mb{Q}}_\ell\bigr)\Bigr)\\
  \notag&\qquad\qquad\qquad
  +q^{(r-1)s}\sum_{\mc{E}\in\{\mc{E}\}}c_{\mc{E}}\,
  \mr{Tr}_{\mc{E}}(\mr{Frob}_x^{-s/\deg(x)})\\
  \notag
  =&\mr{Tr}\bigl(f\times\mr{Frob}_x^{-s/\deg(x)}:
   \mc{H}_{N,\mr{ess},\ell}\bigr)\\
  \notag&\qquad\qquad\qquad
  +q^{(r-1)s}\sum_{A\in\{\mc{E}\}\cup\{\sigma\}}c_{A}\,
  \mr{Tr}_{A}(\mr{Frob}_x^{-s/\deg(x)})
 \end{align}
 where the first and second equality hold by Lemma
 \ref{calccorresprime} (ii) and (i) respectively,
 the third by (\ref{complandpappl}), and the last by repeating the
 corresponding argument in the $\ell$-adic situation.
 Now, for $\mc{G}\in\mc{I}(S_N)$, an endomorphism $f$ of $\mc{G}$, and an
 integer $n$, we put
 \begin{align*}
  \overline{\mr{Tr}}(f\times\mr{Frob}_x^n:\mc{G}):=&
  \frac{1}{r^2!}\sum_{k=1}^{r^2!}\mr{Tr}
  \bigl(f\times\mr{Frob}_x^n:
  (\mr{Frob}_X^k\times\mr{id}_X)^+(\mc{G})\bigr)\\
  =&
  \frac{1}{r^2!}\sum_{k=1}^{r^2!}\mr{Tr}_{\mc{G}}
  \bigl(f\times\mr{Frob}^n_{(\mr{Frob}^k_X\times\mr{id}_X)(x)}\bigr),
 \end{align*}
 and similarly for objects in $\mc{W}_{\ell}(S_N)$. Note that when
 $\mc{G}\in\{\mc{E}\}$, there exist $\mc{G}'_\iota$
 and $\mc{G}''_\iota$ of rank $<r$ and pure of weight $0$ in
 $\mc{I}(X-N)$, and constants $c_\iota$, $\lambda_\iota$ such that
 \begin{equation*}
  \overline{\mr{Tr}}(\mr{Frob}_x^n:\mc{G})=
   \sum_{\iota}c_{\iota}\lambda_{\iota}^n\,
   \mr{Tr}_{q'^+\mc{G}'_{\iota}\otimes q''^+\mc{G}''_{\iota}}
   (\mr{Frob}_x^n).
 \end{equation*}
 Put $U_f:=\bigcap_{n=0}^{r^2!}
 (\mr{Frob}^n_X\times\mr{id}_X)^{-1}(U'_f)$.
 Since $\mc{H}_{N,\mr{ess}}$ and
 $\mc{H}_{N,\mr{ess},\ell}$ are invariant under the pull-back by
 $\mr{Frob}_X\times\mr{id}_X$, the computation (\ref{compfirstuseind})
 implies that for $x\in U_f$,
 \begin{align*}
  \mr{Tr}\bigl(f\times\mr{Frob}&_x^{-s/\deg(x)}:
  \mc{H}_{N,\mr{ess}}\bigr)
  =
  \overline{\mr{Tr}}\bigl(f\times\mr{Frob}_x^{-s/\deg(x)}:
  \mc{H}_{N,\mr{ess}}\bigr)
  \\
  =&
  \overline{\mr{Tr}}\bigl(f\times\mr{Frob}_x^{-s/\deg(x)}:
  \mc{H}_{N,\mr{ess},\ell}\bigr)\\
  \notag&\qquad\qquad\qquad
  +q^{(r-1)s}\sum_{A\in\{\mc{E}\}\cup\{\sigma\}}c_A\,
  \overline{\mr{Tr}}_{A}(\mr{Frob}_x^{-s/\deg(x)})\\
  =&
  \mr{Tr}\bigl(f\times\mr{Frob}_x^{-s/\deg(x)}:
  \mc{H}_{N,\mr{ess},\ell}\bigr)\\
  \notag&\qquad\qquad\qquad
  +q^{(r-1)s}
  \sum_{\iota}c_\iota\lambda_\iota^s\,
  \mr{Tr}_{q'^+\mc{E}'_\iota\otimes q''^+\mc{E}''_\iota}
  (\mr{Frob}_x^{-s/\deg(x)})
 \end{align*}
 where $\mc{E}'_\iota$ and $\mc{E}''_\iota$ are irreducible isocrystals
 on $X-N$ of rank $<r$ and pure of weight $0$, and we used Langlands
 correspondence for rank $<r$ for the last equality.
 By \cite[VI.25]{Laf}, if we further shrink $U_f$, we finally obtain
 \begin{align}
  \label{equcaldiffcor}
  \tag{$\heartsuit$}
  q^{-(r-1)s}\,\mr{Tr}
  \bigl(f\times\mr{Frob}_x^{n}:\mc{H}_{N,\mr{ess}}\bigr)-
  \sum_{\pi\in\{\pi\}^r_N}&\mr{Tr}_\pi(f)\cdot
  z_{\bullet}^{-s'}(\pi_{\infty_x})\cdot
  z_{\bullet}^{u'}(\pi_{0_x})\\
  \notag
  &=\sum_{\iota}c_\iota\lambda_\iota^s\,
  \mr{Tr}_{q'^+\mc{E}'_\iota\otimes q''^+\mc{E}''_\iota}
  (\mr{Frob}_x^{-s/\deg(x)}).
 \end{align}
 Since $\mc{H}_{N,\mr{ess}}$ is pure of weight $2r-2$ and we know that
 $|z_i(\pi_{\infty})|$ and $|z_j(\pi_0)|$ are $1$, we have
 $|\lambda_\iota|=1$.
 We need to show that the right hand side of (\ref{equcaldiffcor}) is
 $0$. The argument is essentially the same as {\it ibid}.,
 but for the reader, we recall it. Assume otherwise. Then there would
 exist irreducible isocrystals $\mc{E}'$, $\mc{E}''$ of rank $<r$ and
 pure of weight $0$ such that the series
 \begin{equation*}
  \sum_{\iota}c_\iota\frac{d}{dZ}\log L_{U_f}
   \bigl((q'^+\mc{E}'_{\iota}\otimes q''^+\mc{E}''_\iota)
   \otimes(q'^+\mc{E}'^\vee\otimes q''^+\mc{E}''^\vee),
   \lambda_{\iota}Z\bigr)
 \end{equation*}
 has a pole on $|Z|=q^{-2}$ by Corollary \ref{conseqweighidenti} since
 $|\lambda_\iota|=1$. On the other hand, the series
 \begin{equation*}
   \frac{d}{dZ}\log L_{U_f}\bigl(\mc{H}
   \otimes(q'^+\mc{E}'^\vee\otimes q''^+\mc{E}''^\vee),
   q^{1-r}Z\bigr)
 \end{equation*}
 does not have poles at $|Z|=q^{-2}$ since $\mc{H}$ is essential. Now,
 let $\pi'$ and $\pi''$ be the automorphic cuspidal representations
 corresponding to $\mc{E}'$ and $\mc{E}''$.
 For a locally closed subscheme $Y$ of $(X-N)\times(X-N)$, let us denote
 by $\mf{C}_{s,Y}$ the subset of
 $(x,s',u')\in|Y|\times\mb{N}\times\mb{N}$ such that
 $s\cdot\deg(\infty_x)^{-1}=s'\in\mb{N}$ and
 $s\cdot\deg(0_x)^{-1}=u'\in\mb{N}$, and put
 \begin{align*}
  \mr{Ser}_Y(Z):=
  \sum_{s\geq 1}Z^{s-1}\sum_{(x,s',u')\in\mf{C}_{s,Y}}
  \Bigl[
  \bigl(\deg(\infty_x)\cdot
  &z_{\bullet}^{-s'}(\pi_{\infty_x})\cdot
  z_{\bullet}^{s'}(\pi'_{\infty_x})\bigr)\\
  &\times
  \bigl(\deg(0_x)\cdot
  z_{\bullet}^{u'}(\pi_{0_x})\cdot
  z_{\bullet}^{u'}(\pi''_{0_x})\bigr)
  \Bigr].
 \end{align*}
 The series
 \begin{equation*}
  \frac{d}{dZ}\log L_{X-N}(\pi\times\pi'^\vee,Z),\qquad
   \frac{d}{dZ}\log L_{X-N}(\pi^\vee\times\pi''^\vee,Z)
 \end{equation*}
 do not have poles on $|Z|\leq q^{-1}$ by [{\it ibid.}, B.10]. Thus, the
 ``product series'' $\mr{Ser}_{S_N}$ does not have poles at $|Z|\leq
 q^{-2}$. We claim that the series $\mr{Ser}_{U_f}$ does not have
 poles at $|Z|\leq q^{-2}$ either. Indeed, putting
 $W:=S_N\setminus U_f$, we have
 $\mr{Ser}_W=\mr{Ser}_{S_N}-\mr{Ser}_{U_f}$. Since
 $|z_i(\pi)|=|z_j(\pi')|=|z_k(\pi'')|=1$, we have
 \begin{equation*}
  |\mr{Ser}_W(Z)|\leq\sum_{s}
   |Z|^{s-1}\sum_{(x,s',u')\in\mf{C}_W}
   \deg(\infty_x)\cdot\deg(0_x).
 \end{equation*}
 Since $W$ is of dimension $1$, the latter series converges on
 $|Z|<q^{-1}$, and thus $\mr{Ser}_W$ converges absolutely on the same
 area, which implies the claim.
 Combining these, if we put
 \begin{equation*}
  \sum_sZ^{s-1}\sum_{\mf{C}_{s,U_f}}
   \deg(\infty_x)\cdot\deg(0_x)\cdot
   z^{s'}(\pi'_{\infty'_x})\cdot
   z^{u'}(\pi''_{0_x})
 \end{equation*}
 at the head of the both sides of (\ref{equcaldiffcor}), the
 left side does not have poles at $|Z|=q^{-2}$ whereas the right side
 does at $|Z|=q^{-2}$, which is a contradiction.
\end{proof}

\begin{lem}[{\cite[VI.26]{Laf}}]
 \label{compuoffrobeigvalcon}
 As an object of $\mc{H}^r_N/a^{\mb{Z}}\text{-}\mb{Z}\mc{I}(S_N)$, we
 can write $\mc{H}_{N,\mr{ess}}$ as
 \begin{equation*}
  \bigoplus_{\pi\in\{\pi\}^r_N}\pi\boxtimes \mc{H}_{\pi}(1-r)
 \end{equation*}
 and there exists an open dense subscheme $U_\pi\subset S_N$ for
 any $\pi\in\{\pi\}^r_N$ such that $\mc{H}_\pi$ is pure of weight $0$,
 and for any closed point $x\in U_\pi$ and
 $s=\deg(\infty_x)s'=\deg(0_x)u'$, we have
 \begin{equation*}
  \mr{Tr}_{\mc{H}_\pi}\bigl(\mr{Frob}_x^{-s/\deg(x)}\bigr)=
      z_{\bullet}^{-s'}(\pi_{\infty_x})\cdot
   z_{\bullet}^{u'}(\pi_{0_x}).
 \end{equation*}
\end{lem}
\begin{proof}
 This can be proven similarly to the $\ell$-adic case. By Lemma
 \ref{exteritensprodcat},
 there exists a finite set of irreducible representations $\{\pi'\}$ of
 $\mc{H}^r_N/a^{\mb{Z}}$ and semi-simple objects $\mc{H}_{\pi'}$ in
 $\mb{Z}\mc{I}(S_N)$ for each $\pi'\in\{\pi'\}$ such that
 $\mc{H}_{N,\mr{ess}}=\bigoplus_{\pi'\in\{\pi'\}^r_N}\pi'\boxtimes
 \mc{H}_{\pi'}(1-r)$. For $\pi\in\{\pi\}^r_N\cup\{\pi'\}$, we
 can choose $f_\pi\in\mc{H}^r_N/a^{\mb{Z}}$ such that
 $\mr{Tr}_\pi(f_\pi)=1$ and $\mr{Tr}_{\pi'}(f_\pi)=0$ for any
 $\pi\neq\pi'\in\{\pi\}^r_N\cup\{\pi'\}$ (cf.\ proof of [{\it ibid.},
 VI.26]), and apply Lemma \ref{computoftrauselp} by taking $f=f_\pi$.
 Then the lemma holds with $U_\pi:=U_{f_\pi}$.
\end{proof}

\begin{thm}[{\cite[VI.27]{Laf}}]
 \label{concludethm}
 For any $\pi\in\{\pi\}^r_N$, we have
 \begin{equation*}
  \mc{H}_\pi=q'^+\mc{E}_\pi\otimes q''^+\mc{E}_\pi^\vee
 \end{equation*}
 as objects in $\mc{I}(S_N)$, where $\mc{E}_\pi$ is an
 isocrystal of rank $r$ on $X-N$ pure of weight $0$ corresponding to
 $\pi$ in the sense of Langlands.
\end{thm}
\begin{proof}
 Take a closed point $x\in U_\pi$, which lies over
 $(\infty,0)\in |X-N|\times|X-N|$. Let $X^0:=X\times0\hookrightarrow
 X\times X$ be the closed immersion, and let $(X-N)^0\subset X^0$ be the
 pull-back of $S_N$ by the closed immersion. Let
 $\mc{E}^0$ be the semi-simplification in $\mc{I}((X-N)^0)$
 of the pull-back of $\mc{H}_\pi$,
 which is pure of weight $0$. Let $\mc{H}^0_\pi$ be the pull-back on
 $S_N\otimes k(0)$.
 Let $\chi_i$ be the character ({\it i.e.}\ rank $1$ isocrystal on the
 point $0$) corresponding to the Hecke eigenvalue $z_i(\pi_0)$.
 Then the two semi-simple objects in $\mc{I}(S_N\otimes k(0))$
 \begin{equation*}
  q'^+\mc{E}^0\otimes q''^+\mc{E}^{0\vee},\qquad
   \mc{H}^0_\pi\otimes\bigoplus^{r}_{i=1}\chi_i\otimes
   \bigoplus^{r}_{i=1}\chi^\vee_i
 \end{equation*}
 have the same Frobenius eigenvalues at each closed point of
 $U_\pi\otimes k(0)$ by the explicit description of
 the Frobenius trace in Lemma \ref{compuoffrobeigvalcon}.
 Thus by the \v{C}ebotarev density \ref{Cebdensity}, these two objects
 coincide.

 Thus, there exist $\mc{E}'$ and $\mc{E}''$ on $X-N\subset X$
 which are pure of weight $0$ such that $q'^+\mc{E}'\otimes
 q''^+\mc{E}''$ and $\mc{H}_\pi$ have at least one constituent in
 common. We may assume that $\mc{E}'$ and $\mc{E}''$ are irreducible,
 and since $\mc{H}_\pi$ is stable under the action of
 $(\mr{Frob}_X\times\mr{id})^+$, we may assume that $q'^+\mc{E}'\otimes
 q''^+\mc{E}''$ is a subobject of $\mc{H}_\pi$.

 Let us show that $\mc{E}'$ is of rank $\geq r$. The argument is similar
 to the last part of the proof of Lemma \ref{computoftrauselp}.
 If $\mc{E}'$ were of rank $<r$, there would
 exist an automorphic cuspidal representation $\pi'$ corresponding to
 $\mc{E}'$. Consider the series
 \begin{equation*}
  \frac{d}{dZ}\log L_{X-N}(\pi\times\pi'^\vee,Z),\qquad
   \frac{d}{dZ}\log L_{X-N}(\pi^\vee\times\mc{E}''^\vee,Z).
 \end{equation*}
 The first one converges absolutely on $|Z|<q^{-1+\epsilon}$ for some
 $\epsilon>0$ by \cite[Thm B.10]{Laf}.
 Since $|z_i(\pi)|=1$ and $\mc{E}''$ is of weight $0$, the second one
 converges absolutely on $|Z|<q^{-1}$. Thus the ``product series''
 converges absolutely on $|Z|<q^{-2+\epsilon}$. On the other hand,
 consider the series
 \begin{equation}
  \label{wantseriprodlogd}
   \tag{$\star$}
   \frac{d}{dZ}\log L_{U_\pi}\bigl(\mc{H}_{\pi}\otimes
   (q'^+\mc{E}'^\vee\otimes q''^+\mc{E}''^\vee)\bigr).
 \end{equation}
 Let $C:=S_N\setminus U_\pi$. The difference with the product
 series is nothing but
 \begin{align}
  \label{differencesereval}
  \tag{$\star\star$}
  \sum_{s\geq1}Z^{s-1}\sum_{(x,s',u')\in\mf{C}_s}
  \Bigl[
  \bigl(\deg(\infty_x)\cdot
  &z_{\bullet}^{-s'}(\pi_{\infty_x})\cdot
  z_{\bullet}^{s'}(\pi'_{\infty_x})\bigr)\\
  \notag
  &\times
  \bigl(\deg(0_x)\cdot
   z_{\bullet}^{u'}(\pi_{0_x})\cdot
  z_{\bullet}^{u'}(\mc{E}''_{0_x})\bigr)
  \Bigr]
 \end{align}
 where $\mf{C}_s\subset|C|\times\mb{N}\times\mb{N}$ such that
 $s\cdot\deg(\infty_x)^{-1}=s'\in\mb{N}$ and
 $s\cdot\deg(0_x)^{-1}=u'\in\mb{N}$.
 Since $C$ is of dimension $1$ and the complex absolute values of
 $z_i(\pi_x)$, $z_i(\pi'_x)$, $z_i(\mc{E}'')$ are $1$, we get that the
 series (\ref{differencesereval}) converges in $|Z|<q^{-1}$. Thus,
 considering the radius of convergence of the product series,
 (\ref{wantseriprodlogd}) should converge on $|Z|<q^{-2+\epsilon}$.
 However, since by Corollary \ref{conseqweighidenti}, it should have a
 pole at $|Z|=q^{-2}$, which is a contradiction.

 This shows that the rank of $\mc{E}'$ is $\geq r$. By symmetry, the
 rank of $\mc{E}''$ is $\geq r$ as well. Since the rank of $\mc{H}_\pi$
 is $r^2$, we get that the rank
 of $\mc{E}'$ and $\mc{E}''$ are $r$,
 and $\mc{H}_\pi\cong q'^+\mc{E}'\otimes q''^+\mc{E}''$. By
 the induction hypothesis \ref{starproof} \eqref{prodfoumindhyp}, there
 exist cuspidal automorphic representations $\pi'$, $\pi''$ of
 $\mr{GL}_r(\mb{A})$ corresponding to $\mc{E}'$ and $\mc{E}''$.
 Now, since $\dim(X-N)=1$ and $|z_i(\pi)|=|z_j(\pi')|=|z_k(\pi'')|=1$,
 the rational functions
 \begin{equation*}
  \frac{d}{dZ}\log L_{X-N}(\pi\times\pi'^\vee,Z),\qquad
   \frac{d}{dZ}\log L_{X-N}(\pi^\vee\times\pi''^\vee,Z)
 \end{equation*}
 do not have poles at $|Z|<q^{-1}$. This implies that they have a pole
 at $Z=q^{-1-s}$ and $q^{-1+s}$ respectively for some $s\in\mb{C}$ such
 that $\mr{Re}(s)=0$. Otherwise, the series
 \begin{equation*}
  \frac{d}{dZ}\log L_{U_\pi}\bigl(\mc{H}_{\pi}
   \otimes(q'^+\mc{E}'^\vee\otimes q''^+\mc{E}''^\vee)\bigr)
   =
   \frac{d}{dZ}\log L_{U_\pi}\bigl(\mc{H}_{\pi}
   \otimes\mc{H}_\pi^\vee\bigr)
 \end{equation*}
 would not have a pole at $q^{-2}$. This shows that
 $\mc{E}'\cong\mc{E}''^\vee$, and $\mc{E}_\pi=\mc{E}'(s)$ as requested,
 and $\mc{H}_\pi\cong q'^+\mc{E}_\pi\otimes q''^+\mc{E}_\pi^\vee$.
\end{proof}

\subsubsection{Conclusion of the proof}
\mbox{}\\
By Theorem \ref{concludethm} and
\ref{thmdeligneprinciple}, \ref{starproof} \eqref{prodfoumindhyp} holds
for $n=r$. Let us check \ref{starproof} \eqref{rneghyp}.
Theorem \ref{concludethm} combined with Proposition \ref{fundpropnegH}
(ii) shows that the ``average''
$(r^2!)^{-1}\sum_{n=1}^{r^2!}(\mr{Frob}^n_X\times\mr{id}_X)^+
\mc{H}_\mr{c}^{2r-2}(\dots)^{\mr{ss}}$ is $(r+1)$-negligible.
Since the averaging involve only positive coefficients,
$\mc{H}_\mr{c}^{2r-2}(\dots)$ is $(r+1)$-negligible. Combining with
Proposition \ref{fundpropnegH} (i),
\ref{starproof} \eqref{rneghyp} holds for $n=r$, and $(\mr{S})_{r}$ is
shown. Thus, we conclude the proof of Theorem \ref{main} by induction.
\qed

\subsection{A few applications}
To conclude this paper, let us collect some applications of the
Langlands correspondence.

\begin{thm}[{\cite[1.2.10 (vi)]{De}}]
 Let $X$ be a smooth {\em curve} over a finite field $k$ of
 characteristic $p$. Let $\ell$ be a prime number different from $p$.
 Then for any irreducible smooth $\overline{\mb{Q}}_\ell$-sheaf whose
 determinant is of finite order, there exists a ``{\it petit camarade
 cristallin}''.
\end{thm}
\begin{proof}
 Use Lafforgue's result and Theorem \ref{main}.
\end{proof}

\subsubsection{}
\label{Bertiniconj}
In the $p$-adic cohomology theory, Bertini type results do not seem to
be known as pointed out in \cite{KedWeil}. We consider the situation and
notation in \ref{algcloscoefftheory} {\em exclusively in this
paragraph}. In particular, $k$ does not need to be finite in the
following conjecture.

\begin{conj*}[Bertini type conjecture]
 Let $X$ be a smooth scheme over $k$, and $\mc{E}$ be an irreducible
 $L$-isocrystal. Then there exists a dense Zariski open subset $U$ of
 $X$ such that the following holds: for any $x\in|U|$, there
 exists an immersion (not necessarily closed) from a smooth curve
 $i\colon C\hookrightarrow U$ passing through $x$ such that $i^+\mc{E}$
 remains irreducible.
\end{conj*}

\begin{thm}
 \label{anycomplmixedcomplex}
 Assume we are in the situation of \ref{situationfix}. Let $X$ be a
 scheme of finite type over $k$. Then any complex in
 $D^{\mr{b}}_{\mr{hol}}(X/\overline{\mb{Q}}_{p,F})$ is mixed if $X$ is
 of dimension $1$. If Conjecture \ref{Bertiniconj} is true, we do not
 need to assume $X$ to be of dimension $1$.
\end{thm}
\begin{proof}
 See \cite[6.3]{ALang}.
\end{proof}

\begin{cor}
 \label{chebotarevdensity}
 The \v{C}ebotarev density theorem for smooth curves holds for
 overconvergent $F$-isocrystals. If Conjecture \ref{Bertiniconj} is
 true, it holds for any smooth variety.
\end{cor}
\begin{proof}
 Apply Proposition \ref{Cebdensity} and Theorem
 \ref{anycomplmixedcomplex}.
\end{proof}

\begin{thm}
 \label{convpcc}
 Let $X$ be a smooth scheme of finite type over a finite field $k$. Let
 $\mc{E}$ be an irreducible $\overline{\mb{Q}}_p$-isocrystal with
 Frobenius structure on $X$ such that the determinant is of finite
 order. {\em Assume that Conjecture \ref{Bertiniconj} holds}.

 (i) There exists a number field $E/\mb{Q}$ such that, for any $x\in
 |X|$, all the coefficients of the Frobenius eigenpolynomial of $\mc{E}$
 at $x$ is in $E$.

 (ii) For any prime $\ell\neq p$, there exists an
 $\overline{\mb{Q}}_{\ell}$-adic smooth sheaf $\ms{F}$ corresponding to
 $\mc{E}$ such that the sets of Frobenius eigenvalues coincide for any
 closed point of $X$.
\end{thm}
\begin{proof}
 Let us show (i). We denote by $\mc{I}_r(X)$ the set of
 $\overline{\mb{Q}}_p$-isocrystals of rank $r$ on $X$ up to isomorphism
 and semi-simplification. We use the notation of \cite[\S2]{EK}.
 We prove by induction on the dimension of $X$.
 Let $\overline{X}$ be a normal compactification of $X$, and
 $\overline{X}\backslash X$ is a Cartier divisor. Then there exists a
 map $\mc{I}_r(X)\rightarrow\mc{V}_r(X)$ using the Langlands
 correspondence. This is injective by the \v{C}ebotarev density. We show
 the following: Let $\mc{E}\in\mc{I}_r(X)$. Then there exists an dense
 open subscheme $U\subset X$ and Cartier divisor $D$
 of $\overline{X}$ contained in $\overline{X}\setminus U$ such that
 $\mc{E}_U\in\mc{V}_r(U,D)$ where $\mc{E}_U$ is the restriction of
 $\mc{E}$ on $U$. Once this is shown, Conjecture \ref{Bertiniconj}
 implies that $\mc{E}_U$ is irreducible, and we get (i) of the theorem
 for $\mc{E}_U$ by \cite[8.2]{EK}, and by induction hypothesis, we
 conclude.

 When $\overline{X}$ is proper smooth, $\overline{X}\setminus X$
 is a simple normal crossing divisor, and $\mc{E}$ is log-extendable to
 $\overline{X}$, then we may take $D=0$. Indeed, for a smooth curve,
 let $\mc{E}$ be an $\overline{\mb{Q}}_p$-isocrystal whose determinant
 is finite and $\ms{F}$ be its $\ell$-adic companion. Then the
 ramification of $\mc{E}$ and $\ms{F}$ at the boundary are the same by
 Theorem \ref{main} \eqref{epsilonlocalL}. In general, take a
 semi-stable reduction $\overline{Y}\rightarrow\overline{X}$ of $\mc{E}$
 (cf.\ \cite{KeSS}). We take $U\subset X$ so that $p\colon
 V:=U\times_{\overline{X}}\overline{Y}\rightarrow U$ is finite
 \'{e}tale. There exists a Cartier divisor $D$ such that the
 ramification of $p_*\overline{\mb{Q}}_\ell$ is in $\mc{V}_{d}(U,D)$
 where $d=\deg(V/U)$. Then we can check easily that
 $\mc{E}\in\mc{V}_r(X,D)$.

 For (ii), copy the proof of Drinfeld \cite[\S2.3]{Dr}.
\end{proof}

\begin{rem}
 Using the main results of this paper, K. S. Kedlaya recently proved
 Theorem \ref{anycomplmixedcomplex} as well as Theorem \ref{convpcc}
 without assuming Conjecture \ref{Bertiniconj} in \cite{Kenote}. His
 argument reduces to the curve case using an ingenious induction.
 By proving Conjecture \ref{Bertiniconj} when $k$ is finite, the author
 together with H. Esnault gave another proof of this Kedlaya's result in
 \cite{AH}.
 The existence of crystalline companions on smooth schemes of finite
 type over $k$ is still an open problem.
\end{rem}

\appendix

\section{Appendix}

\subsection{Beilinson-Drinfeld's gluing of derived category}
\label{BDgluingcat}
In this subsection, we recall the construction and some results of
Beilinson and Drinfeld \cite[7.4]{BD} shortly, which is used in the main
text.

\subsubsection{}
Let $\mc{M}\rightarrow\Delta^+$ be a cofibered category such that its
fiber $\mc{M}_i$ over $[i]\in\Delta^+$ is an abelian
category, and for $\phi\colon[i]\rightarrow[j]$, the push-forward
$\phi_*$ is exact. We denote by $\mc{M}_{\mr{tot}}$ the abelian category
of cartesian sections; the category of collections
$\bigl\{M_n,\alpha_\phi\bigr\}$ such that $M_n\in\mc{M}_n$ and for
$\phi\colon[i]\rightarrow[j]$, $\alpha_\phi\colon\phi_*
M_i\xrightarrow{\sim} M_j$ satisfying the cocycle condition. Now, we
want to construct a suitable triangulated category associated to
$\mc{M}$ whose heart is $\mc{M}_{\mr{tot}}$. For
this, we consider the category $\mr{sec}_+(\mc{M})$. The objects consist
of collections $\bigl\{M_n,\alpha_\phi\bigr\}$ where $M_n\in\mc{M}_n$,
and for $\phi\colon[i]\rightarrow[j]$, $\alpha_\phi\colon\phi_*
M_i\rightarrow M_j$, satisfying the condition
$\alpha_{\phi\circ\psi}=\alpha_{\phi}\circ\phi_*(\alpha_\psi)$ for
composable morphisms $\phi$ and $\psi$ in $\Delta^+$,
and $\alpha_{\mr{id}}=\mr{id}$. We put
$\mr{sec}_-:=(\mr{sec}_+(\mc{M}^\circ))^\circ$. A profound observation
of Beilinson and Drinfeld is that there are functors
\begin{equation*}
 c_+\colon C(\mr{sec}_-(\mc{M}))\rightarrow C(\mr{sec}_+(\mc{M})),\qquad
  c_-\colon C(\mr{sec}_+(\mc{M}))\rightarrow C(\mr{sec}_-(\mc{M}))
\end{equation*}
such that $(c_+,c_-)$ is an adjoint pair, and the adjunction
homomorphisms $c_+c_-\rightarrow\mr{id}$ and $\mr{id}\rightarrow c_-c_+$
are quasi-isomorphisms (cf.\ \cite[7.4.4]{BD}). With these functors, we
are able to identify $D(\mr{sec}_+(\mc{M}))$ and
$D(\mr{sec}_-(\mc{M}))$. Now, let $C_{\mr{tot\pm}}\subset
C(\mr{sec}_{\pm}(\mc{M}))$ be the full subcategory consisting of
complexes $M$ such that $\H^i(M)\in\mc{M}_{\mr{tot}}$. We
denote by $K_{\mr{tot}\pm}(\mc{M})$ and $D_{\mr{tot}\pm}(\mc{M})$ the
corresponding homotopy and derived categories.
By means of $c_{\pm}$, we are able to identify
$D_{\mr{tot+}}(\mc{M})$ and $D_{\mr{tot-}}(\mc{M})$, and denote them by
$D_{\mr{tot}}(\mc{M})$.
The functors $\H^i\colon D_{\mr{tot}\pm}\rightarrow\mc{M}_{\mr{tot}}$
induce $D_{\mr{tot}}(\mc{M})\rightarrow\mc{M}_{\mr{tot}}$.

\subsubsection{}
Now, an important aspect of the theory is the existence of a spectral
sequence connecting $\bigl\{\mc{M}_n\bigr\}$ and $\mc{M}$.
For $N\in D^-(\mr{sec}_-(\mc{M}))$ and $M\in D^+(\mr{sec}_+(\mc{M}))$,
we have the following spectral sequence by \cite[7.4.8]{BD}:
\begin{align}
 \label{BDssingen}
 E_1^{p,q}=\mr{Ext}^q_{\mc{M}_p}(N_p,M_p)\Rightarrow
 &\mr{Hom}_{D(\mc{M})}\bigl(c_+(N),M[p+q]\bigr)\\
 \notag
 \cong
 &\mr{Hom}_{D(\mc{M})}\bigl(N,c_-(M)[p+q]\bigr).
\end{align}

\begin{rem*}
 Since the proof of {\it ibid.}\ is rather sketchy, it might
 be hard to follow their argument in some cases. Let us add a short
 explanation. When $\mr{sec}_+(\mc{M})$ have enough injectives, then we
 can take the right derived functors of the
 functor $\mr{Hom}(N,-)\colon K^+\mr{sec}_+(\mc{M})\rightarrow DF$,
 where $DF$ denotes the derived category of filtered modules,
 in a usual way and we get the spectral sequence as written in {\it
 ibid.}.
 However, there might be a situation that the functor does
 not admit a right derived functor. Even in this case, we can define the
 derived functor $\mb{R}\mr{Hom}(N,-)\colon
 D^+(\mr{sec}_+(\mc{M}))\rightarrow\mr{Ind}(DF)$ as in [SGA 4, XVII,
 1.2]. We have a functor
 \begin{equation*}
  \widetilde{\H}^i\colon\mr{Ind}(DF)\xrightarrow{\H^i}\mr{Ind}(F\mr{Ab})
   \xrightarrow{\indlim} F\mr{Ab}.
 \end{equation*}
 Using this, we have
 $\widetilde{\H}^i\mr{gr}_F^n\mb{R}\mr{Hom}(N,M)\cong
 \mr{Hom}_{D(\mc{M}_n)}(N_n,M_n[i-n])$. This follows from the fact that
 we have
 $\mr{gr}_F^n\mr{Hom}(N,M)\cong\mr{Hom}_{K(\mc{M}_n)}(N_n,M_n)[-n]$ by
 construction of the functor $\mr{Hom}(N,-)$, and
 the functor $M\mapsto M_n$ has an exact right adjoint as in
 \ref{functorrhodef} which implies the existence of $g$ and $h$ in
 {\it ibid.}. Since the category of spectral sequences of abelian
 groups admits inductive limits, we get the desired result.
\end{rem*}

\subsection{Some properties of algebraic stacks}
The results of this subsection is used implicitly in Lafforgue's proof
of Langlands correspondence. Even though we believe that the results are
well-known to experts, since we are not able to find references, we
decided to write down the detail.

\begin{lem}
 \label{neutralgerbsmpres}
 Let $f\colon\st{X}\rightarrow\st{Y}$ be a gerb-like morphism
 \cite[5.1.3]{Behr} such that the structural group is flat. Then there
 exists a presentation $Y\rightarrow\st{Y}$ such that
 $\st{X}\times_{\st{Y}}Y\rightarrow Y$ is a neutral gerb.
\end{lem}
\begin{proof}
 First, we note that $f$ is smooth surjective. Indeed, since the
 verification is fppf-local, we may assume that $f$ is neutral, and thus
 $\st{Y}=:Y$ is a scheme and $\st{X}=B(G/Y)$. Since $G$ is assumed
 flat, $f$ is smooth by \cite[5.1.2]{Behr}.

 Let $P\colon Y\rightarrow\st{X}$ be a presentation, and consider
 the smooth morphism $Q:=f\circ P\colon Y\rightarrow\st{Y}$, which is a
 presentation of $\st{Y}$ since $f$ is smooth surjective.
 We have the morphism $(P,\mr{id})\colon
 Y\rightarrow\st{X}\times_{\st{Y}}Y$. This defines a section of the
 second projection $\st{X}\times_{\st{Y}}Y\rightarrow Y$.
\end{proof}

\begin{lem}
 \label{genersepa}
 Let $f\colon\st{X}\rightarrow\st{Y}$ be a representable morphism of
 algebraic stacks over an integral scheme $S$. For the generic point
 $\eta\in S$, if $f_\eta$ is separated, then there exists an open
 subscheme $U\subset S$ such that
 $f_U\colon\st{X}\times_SU\rightarrow\st{Y}\times_SU$ is separated.
\end{lem}
\begin{proof}
 Let $\algsp{Y}\rightarrow\st{Y}$ be a presentation, and
 $\algsp{X}\rightarrow\st{X}$ be the induced presentation.
 Let $f'\colon\algsp{X}\rightarrow\algsp{Y}$ be the induced
 morphism. The morphism $f$ being separated is equivalent to $f'$ being
 separated, and thus by [EGA VI, 8.10.5], the lemma follows.
\end{proof}

\begin{lem}
 Let $f\colon\st{X}\rightarrow\st{Y}$ be a representable morphism of
 locally noetherian algebraic stacks. Then
 there exists an open dense substack $\st{U}$ of $\st{X}$ such that
 $f|_{\st{U}}$ is separated.
\end{lem}
\begin{proof}
 By [EGA I, 5.5.1 (vi)], we may assume that $\st{X}$ and $\st{Y}$ are
 reduced.
 By shrinking $\st{X}$ and $\st{Y}$, we may assume that
 $\st{X}\rightarrow\algsp{X}$ and $\st{Y}\rightarrow\algsp{Y}$ are gerbs
 over algebraic spaces by \cite[11.5]{LM}. By shrinking $\st{X}$ and
 $\st{Y}$, we may assume that $\algsp{X}$ is a separated scheme and
 $\algsp{X}\rightarrow\algsp{Y}$ is separated. Consider the following
 diagram:
 \begin{equation*}
  \xymatrix{
   \st{X}\ar[r]^{\alpha}\ar@/_1.5ex/[rd]&
   \st{X}'\ar[r]\ar[d]\ar@{}[rd]|\square&
   \algsp{X}\ar[d]\\
  &\st{Y}\ar[r]&\algsp{Y}.
   }
 \end{equation*}
 Note that $\alpha$ is representable by \cite[3.12 (c)]{LM}.
 Let $\eta$ be a generic point of $\algsp{X}$.
 By Lemma \ref{genersepa}, it suffices to show that $\alpha_\eta$ is
 separated. So the statement is reduced to the following
 special case of the lemma:
 \begin{cl}
  Let $\st{X}$ and $\st{X}'$ be gerbs over $\mr{Spec}(K)$, and let
  $\alpha\colon\st{X}\rightarrow\st{X}'$ be a representable
  morphism. Then this is separated.
 \end{cl}
 \begin{proof}
  Since there exists a scheme $X\rightarrow\mr{Spec}(K)$ such that
  $\st{X}$ and $\st{X}'$ are neutral over $X$, by taking a closed point
  of $X$, $\st{X}$ and $\st{X}'$ are neutral over a
  finite extension of $K$.
  Since the claim is stable under finite
  extension, we may assume that $\st{X}$ and $\st{X}'$ are neutral gerbs
  over $\mr{Spec}(K)$. Thus, by taking the
  automorphism groups $G$ and $G'$ of $\st{X}(K)$ and $\st{X}'(K)$, we
  have a homomorphism of $K$-group spaces $\rho\colon G\rightarrow G'$
  over $\algsp{X}$, which induces $\alpha$. We have the cartesian
  diagram
  \begin{equation*}
   \xymatrix{
    G\ar[r]^<>(.5){\rho}\ar[d]\ar@{}[dr]|\square&G'\ar[d]\\
   BG\ar[r]_<>(.5){\Delta}&BG\times_{BG'}BG.
    }
  \end{equation*}
  Since the diagonal morphism $\Delta$ is quasi-compact by
  \cite[7.7]{LM}, $\rho$ is quasi-compact. By [SGA 3,
  $\mbox{VI}_{\mr{A}}$, Cor 6.7], $\rho$ decomposes as
  $G\twoheadrightarrow G/N\hookrightarrow G'$. This induces a morphism
  $BG\rightarrow B(G/N)\rightarrow BG'$. Since $\alpha$ is assumed
  representable, the morphism $BG\rightarrow B(G/N)$ is representable as
  well. This can only happen when $N$ is trivial. Thus, $\rho$ is a
  closed immersion by the same corollary of [SGA 3]. This shows that
  $\Delta$ is a closed immersion, and thus $\alpha$ is separated.
  \renewcommand{\qedsymbol}{$\square\blacksquare$}
 \end{proof}
 \renewcommand{\qedsymbol}{}
\end{proof}
\vspace{-7ex}

\begin{lem}
 \label{factorrepdimsu}
 Let $f\colon\st{X}\rightarrow\st{Y}$ be a morphism locally of finite
 type between {\em reduced} algebraic stacks such that
 \begin{equation*}
  \dim\bigl(\st{A}ut_{\st{Y}}\,\st{X}\bigr)=
   \dim(\st{Y})-\dim(\st{X}).
 \end{equation*}
 Then there exists a dense open substack $\st{V}\subset\st{Y}$ such that
 $f|_{f^{-1}(\st{V})}$ can be factorized as
 $\st{X}\xrightarrow{p}\st{X}'\xrightarrow{g}\st{Z}
 \xrightarrow{h}\st{Y}$ where $p$ is a gerb-like morphism with the
 structure group space $\st{A}ut_{\st{Y}}\,\st{X}$, $g$ is a
 representable universal homeomorphism, and $h$ is a representable
 finite \'{e}tale morphism.
\end{lem}
\begin{proof}
 Locally on $\st{X}$, $f$ factors as
 $\st{X}\xrightarrow{p}\st{X}'\xrightarrow{\alpha}\st{Y}$ such that
 $p$ is gerb-like and $\alpha$ is a representable morphism
 \cite[5.1.13, 5.1.14]{Behr}. By the assumption on the dimension,
 $\alpha$ is a representable
 quasi-finite morphism. By shrinking $\st{X}$, and by using Lemma
 \ref{genersepa}, we may assume that $\alpha$ is separated. By using
 Zariski's main theorem (cf.\ \cite[16.5]{LM}), by shrinking if $\st{Y}$
 necessary, we may assume that $\alpha$ is a finite
 morphism.

 Let $f\colon X\rightarrow Y$ be a finite morphism between integral
 schemes such that $X$ is normal.
 The finite extension $K(X)/K(Y)$ of fields can be factorized
 canonically as $K(X)/M/K(Y)$ such that $K(X)/M$ is purely inseparable
 and $M/K(Y)$ is separable. Let $Z$ be the normalization of $Y$ in
 $\mr{Spec}(M)$. The morphism $f$ factors as the composition of finite
 morphisms $X\rightarrow Z\rightarrow Y$. This construction is compatible
 with smooth base change $Y'\rightarrow Y$ by \cite[16.2]{LM}. Thus,
 given a finite morphism $\st{X}\rightarrow\st{Y}$ of
 reduced algebraic stacks, by shrinking $\st{Y}$ if necessarily, we have
 a factorization
 $\st{X}\rightarrow\st{Z}\rightarrow\st{Y}$ such that the first morphism
 is generically purely inseparable and the second is generically
 finite \'{e}tale by \cite[14.2.4]{LM}.

 Apply this factorization to $\alpha$, and we get a factorization
 $\st{X}'\xrightarrow{g}\st{Z}\xrightarrow{h}\st{Y}$ satisfying the
 condition above. Take a presentation $Z\rightarrow\st{Z}$, and let
 $X'\rightarrow\st{X}'$ be the pull-back. Then by construction,
 $\widetilde{g}\colon X'\rightarrow Z$ is generically purely
 inseparable. By [EGA IV, 1.8.7], by shrinking $Z$, we may
 assume that $\widetilde{g}$ is radicial surjective for any fiber of
 $Z$. Thus $\widetilde{g}$ is radicial surjective as well, and since
 moreover $\widetilde{g}$ is finite, it is a universal homeomorphism by
 [EGA IV, 2.4.5]. By replacing $\st{Z}$ by the image of $Z$ and
 $\st{X}'$ by the pull-back of newly constructed $\st{Z}$, $g$ can be
 made universally homeomorphic. By removing the ramification locus of
 $h$ from $\st{Y}$, we may assume that $h$ if finite \'{e}tale.
\end{proof}

\subsection{Lefschetz fixed point theorem}
\label{LFTforsch}
In this subsection, we prove a Lefschetz fixed point theorem for
arithmetic $\ms{D}$-modules. In the case of realizable schemes, the
theorem has already been proven in \cite{AC}.
This can be generalized to the case of separated scheme of finite type
in an obvious manner, but we present here for the convenience of the
reader, and for the future reference.

\subsubsection{}
We consider a similar situation as \ref{situationfix}:
Let $k$ be a finite field with $q=p^s$ elements, and
we fix an algebraic closure $\overline{\mb{Q}}_p$ of $K$. We denote by
$\overline{k}$ the residue field of $\overline{\mb{Q}}_p$ which is
algebraically closed as well. We put
$\mf{T}_k:=(k,R:=W(k),K:=\mr{Frac}(R),L,s,\sigma:=\mr{id})$ to be an
arithmetic base tuple (cf.\ \ref{5tuplesover},
\ref{algcloscoefftheory}).
Let $X$ be a scheme over $k$, and let $i_x\colon
x\hookrightarrow X$ be a closed point of $X$.
Choose a geometric point $\overline{x}\in X(\overline{k})$ lying above
$x$. This defines an embedding
$K_x:=\mr{Frac}(W(k(x)))\hookrightarrow\overline{\mb{Q}}_p$.
Let $L_x$ be the field generated by $K_x$ and $L$ in
$\overline{\mb{Q}}_p$. Put $\mf{T}_{x,k}:=(k,R,K,L_x,s,\mr{id})$,
$\mf{T}_{k(x)}:=(k(x),R_x,K_x,L_x,[k(x):k]\cdot s,\mr{id})$,
and we have the following functor
\begin{equation*}
 \iota_{\overline{x}}\colon
  D^{\mr{b}}_{\mr{hol}}(X/\mf{T}_k)
  \xrightarrow{i_x^+}
  D^{\mr{b}}_{\mr{hol}}(k(x)/\mf{T}_k)
  \xrightarrow{\otimes_LL_x}
  D^{\mr{b}}_{\mr{hol}}(k(x)/\mf{T}_{x,k})\cong
  D^{\mr{b}}_{\mr{hol}}(k(x)/\mf{T}_{k(x)}),
\end{equation*}
where the equivalence follows by Corollary \ref{behaviorofbasech} using
the embedding. For $\ms{E}\in D^{\mr{b}}_{\mr{hol}}(X/\mf{T}_k)$, the
automorphism on $\iota_{\overline{x}}(\ms{E})$ is denoted by $F_x$.
We note that the eigenpolynomial of $F_x$, which is {\it a priori} in
$L_x[t]$, only depends on $x$ not on $\overline{x}$ and belongs in fact
to $L[t]$. Indeed, the independence
of $\overline{x}$ follows from the construction of the equivalence in
Corollary \ref{behaviorofbasech}. Let us check that the eigenpolynomial,
which we denote by $\chi(\iota_{\overline{x}}\ms{E},t)$ for the moment,
is in $L[t]$. We may replace $L_x$ by the Galois closure,
and may assume that $L_x$ is a Galois extension. Take an automorphism
$\gamma$ of $L_x$ over $L$. This induces an endofunctor $\gamma^*$ on
$D^{\mr{b}}_{\mr{hol}}(k(x)/\mf{T}_{k(x)})$, and
$\chi(\gamma^*\iota\ms{E},t)=\gamma(\chi(\iota\ms{E},t))$. However, we have a
canonical isomorphism
$\gamma^*\iota\ms{E}\cong\iota\ms{E}$
since we are taking $\otimes_LL_x$ in the definition of
$\iota_{\overline{x}}$, which implies that
$\gamma(\chi(\iota\ms{E},t))=\chi(\iota\ms{E},t)$.

Let $X$ be a separated scheme of finite type over $k$ and $\ms{E}\in
D^{\mr{b}}_{\mr{hol}}(X/\mf{T}_k)$. We put
$\mb{R}\Gamma_{\mr{c}}(X,\ms{E}):=f_!(\ms{E})$ where $f$ is the structural
morphism of $X$, and $H^\nu_{\mr{c}}(X,\ms{E}):=\H^\nu
\mb{R}\Gamma_{\mr{c}}(X,\ms{E})$ as usual. The object
$\mb{R}\Gamma_{\mr{c}}(X,\ms{E})$ is equipped with an automorphism, and we
denote this automorphism by $\mr{Frob}_X$.
For an extension $L'$ of $K$ and
$\ms{F}\in D^{\mr{b}}_{\mr{fin}}(\mr{Vec}_{L'})$ equipped with automorphism
$\varphi$, we put
$\mr{Tr}_{L'}(\varphi;\ms{F}):=\sum_{\nu\in\mb{Z}}(-1)^\nu
\mr{Tr}_{L'}(\varphi;\H^\nu(\ms{F}))$, and similarly for the determinant
except that we take an alternating product instead of sum.
By using the same argument as \cite[4.3.9]{AC}, we have the following
theorem:

\begin{thm}[Essetially due to \cite{EL}]
 \label{LFTforschELcoeff}
 Let $X$ be a separated scheme of finite type over $k$, and $\ms{E}\in
 D^{\mr{b}}_{\mr{hol}}(X/L_F)$. Let $k_n$ be the extension of $k$
 of degree $n$. Then we have the following identity in
 $\overline{\mb{Q}}_p$:
 \begin{equation*}
  \mr{Tr}_L\bigl(\mr{Frob}_X^n;\mb{R}\Gamma_{\mr{c}}(X,\ms{E})\bigr)=
   \sum_{x\in X(k_n)}
   \mr{Tr}_{L_x}(F^n_x;\iota_{\overline{x}}(\ms{E})).
 \end{equation*}
\end{thm}
\begin{proof}
 First of all, let us show the theorem in the case $L=K$.
 We argue by induction on the dimension of $X$.
 Since both sides of the formula is multiplicative with respect to exact
 triangles, we may assume that $X$ is affine by using the localization
 exact sequence \ref{localtriangs}.
 By the Noether normalization lemma, we may find a finite dominant
 morphism $f\colon X\rightarrow\mb{A}^d$.
 Thus, we may assume that $X=\mb{A}^d$. By induction hypothesis, we may
 shrink $X$, and assume that $X$ is affine and $\ms{E}$ is
 smooth on $X\subset\mb{A}^d$. In this case, we know that
 $H^\nu_{\mr{c}}(X,\ms{E})$ is isomorphic to the rigid cohomology by
 \ref{relbetourrig}, thus the formula is a result of Etesse and Le Stum
 \cite{EL}.

 Now, we show the general case.
 We may assume that the extension $L/K$ is finite.
 By {\it d\'{e}vissage}, we may assume that $X(k_n)=\emptyset$,
 and we only need to show that the left side of the equality vanishes.
 Replacing $L$ by its finite extension, we may assume that an
 $n$-th root of $\alpha:=\mr{Tr}_L(\mr{Frob}_X^n;
 \mb{R}\Gamma_{\mr{c}}(X,\ms{E}))$ is contained in $L$.
 We denote by $\beta$ an $n$-th root.
 Our goal is to show $\alpha=0$. We assume the contrary, so
 $\beta\neq0$.
 Let $\ms{E}=(\ms{F},\Phi_{\ms{F}})$ where $\ms{F}$ is the underlying
 object in $\mr{Hol}(X/L_{\emptyset})$ and $\Phi_{\ms{F}}$ is the
 Frobenius structure. We define $\ms{E}^{(1/\beta)}$ to be an object in
 $\mr{Hol}(X/L_F)$ such that the underlying object is the same
 as $\ms{F}$, and the Frobenius structure is the composition
 \begin{equation*}
  \ms{F}\xrightarrow[\sim]{\cdot(1/\beta)}\ms{F}
   \xrightarrow[\sim]{\Phi_{\ms{F}}}F^*\ms{F}.
 \end{equation*}
 Then $\mr{Tr}_L(\mr{Frob}_X^n;
 \mb{R}\Gamma_{\mr{c}}(X,\ms{E}^{(1/\beta)}))=
 (1/\beta)^n\times\mr{Tr}_L(\mr{Frob}_X^n;
 \mb{R}\Gamma_{\mr{c}}(X,\ms{E}))
 =1$. For any $L$-vector space $V$ with automorphism $\phi$,
 we have $\mr{Tr}_{L/K}\bigl(\mr{Tr}_L(\phi; V)\bigr)=
 \mr{Tr}_K(\phi; V)$, where $\mr{Tr}_{L/K}$
 denotes the field trace. Combining these, we have
 \begin{align*}
  0=\mr{Tr}_K(\mr{Frob}_X; \mb{R}\Gamma_{\mr{c}}(X,\ms{F}^{(1/\beta)}))
   &=
   \mr{Tr}_{L/K}\bigl(\mr{Tr}_L(\mr{Frob}_X; \mb{R}\Gamma_{\mr{c}}
   (X,\ms{E}^{(1/\beta)}))\bigr)\\
   &=\mr{Tr}_{L/K}(1)=[L:K],
 \end{align*}
 where the first equality holds by the case $L=K$.
 This is a contradiction and implies that $\alpha=0$.
\end{proof}

\subsubsection{}
\label{cohintGroth}
Now, let $\ms{E}$ be an object in $D^{\mr{b}}_{\mr{hol}}(X/L_F)$.
We define series in $L\dd{Z}$
\begin{equation*}
 L_x(\ms{E},Z):=\det\bigl(1-Z^{\deg(x)}F_x;
  \iota_{\overline{x}}(\ms{E})
  \bigr)^{-1},\qquad
   L_X(\ms{E},Z):=\prod_{x\in|X|}L_x(\ms{E},Z).
\end{equation*}
Since the first one does not depend on the choice of $\overline{x}$,
these are well-defined. The first one (resp.\ second one) is called the
{\em local $L$-function} (resp.\ {\em global $L$-function}).
By a standard argument (cf.\ for example, [SGA $4\frac{1}{2}$,
Rapport, \S3]), we have a cohomological interpretation of $L$-function
as a consequence of the Lefschetz fixed point theorem.

\begin{cor*}
 Let $X$ be a separated scheme of finite type over $k$, and $\ms{E}\in
 D^{\mr{b}}_{\mr{hol}}(X/L_F)$. Then we have an identity of formal power
 series:
 \begin{equation*}
  L_X(\ms{E},Z)=\prod_{\nu\in\mb{Z}}
   \det\bigl(1-Z\cdot\mr{Frob}_X;
   H^\nu_{\mr{c}}(X,\ms{E})
   \bigr)^{(-1)^{\nu+1}}.
 \end{equation*}
\end{cor*}

\subsection{\v{C}ebotarev density (after N. Tsuzuki)}
The \v{C}ebotarev density theorem for curves and mixed isocrystals is
proven in \cite{ALang}. We need the \v{C}ebotarev density for surfaces
and mixed isocrystals, which we show in this appendix. We could have
included in the main text, but since the author learned the proof from
N. Tsuzuki before start writing this paper, we decided to keep it
separate. We consider the situation of \ref{LFTforsch}.

\begin{prop}
 \label{Cebdensity}
 Let $X$ be a smooth variety over a finite field $k$. Let $\mc{E}$ and
 $\mc{E}'$ be $\iota$-mixed overconvergent $F$-isocrystals
 in $\mr{Isoc}^\dag(X/L_F)$ such that the sets of
 Frobenius eigenvalues are the same for any closed point of $X$. Then
 $\mc{E}^{\mr{ss}}=\mc{E}'^{\mr{ss}}$ where the semi-simplification is
 taken in $\mr{Isoc}^\dag(X/L_F)$.
\end{prop}
\begin{proof}
 Since we have weight filtration on $\mc{E}$ and $\mc{E}'$ by
 \cite[4.3.4]{AC}, we may assume that $\mc{E}$ and $\mc{E}'$ are
 $\iota$-pure. Let $\mc{F}$ be an irreducible overconvergent
 $F$-isocrystal. Since Frobenius eigenvalues of $\mc{E}$ and $\mc{E}'$
 are the same, we have
 \begin{equation*}
  L(\mc{E}\otimes\mc{F}^\vee,Z)=L(\mc{E}'\otimes\mc{F}^\vee,Z).
 \end{equation*}
 If $\mu$ is the multiplicity of $\mc{F}$ in $\mc{E}^{\mr{ss}}$,
 Proposition \ref{conseqweighidenti} implies that
 $L(\mc{E}\otimes\mc{F},Z)$ has a pole of order $\mu$ at
 $Z=q^{-\dim(X)}$. The equality tells us that
 $L(\mc{E}'\otimes\mc{F}^\vee,Z)$ has a pole of the same order at the
 same point, and using the proposition again, we get that
 $\mc{E}'^{\mr{ss}}$ contains $\mc{F}$ with multiplicity $\mu$.
 Note that Proposition \ref{conseqweighidenti} is stated for
 $\overline{\mb{Q}}_p$-coefficients, but the proof goes through also for
 $L$-coefficients with obvious changes.
\end{proof}

\subsection{Gabber-Fujiwara's $\ell$-independence}
We generalize Gabber-Fujiwara's $\ell$-independence results (cf.\
\cite{FG}) to admissible stacks. For a category $\mc{C}$, we denote by
$[\mc{C}]$ the set of isomorphism classes of $\mc{C}$.

\begin{thm}[Trace formula]
 \label{traceformulasta}
 Let $\st{X}$ be a c-admissible stack over finite field $\mb{F}_q$.
 Let $\ms{M}$ be a complex in
 $D^{\mr{b}}_{\mr{hol}}(\st{X}/\overline{\mb{Q}}_{p,F})$. For
 $x\in[\st{X}(\mb{F}_q)]$, we denote by
 $i_x\colon\mr{Spec}(\mb{F}_q)\rightarrow\st{X}$ the corresponding
 morphism. Then, for any $n>0$, we have
 \begin{equation*}
  \sum_{i\in\mb{Z}}(-1)^i\cdot\mr{Tr}\bigl(F^n:H^i_{\mr{c}}
   (\st{X},\ms{M})\bigr)=\sum_{x\in[\st{X}(\mb{F}_q)]}
   \frac{1}{\#\mr{Aut}(x)}\cdot\mr{Tr}\bigl(F^n:
   i_x^+(\ms{M})\bigr).
 \end{equation*}
\end{thm}

\begin{rem*}
 Note that both sides of the equality are finite sums since we are
 dealing with c-admissible stacks, contrary to the case of more general
 algebraic stacks. This prevents us from struggling with the
 convergence issues as in \cite{Behr}, which makes it much easier to
 formulate and prove.
\end{rem*}

\begin{proof}
 Since we can prove similarly to \cite[A.14]{Laf} or
 \cite[6.4.10]{Behr}, we only sketch. Let us denote the right hand side
 of the equality by $L(\st{X},\ms{M})$. For a morphism
 $f\colon\st{X}\rightarrow\st{Y}$ of c-admissible stacks, it suffices to
 show the equality $L(\st{X},\ms{M})=L(\st{Y},f_!(\ms{M}))$ since the
 theorem is the particular case where $\st{Y}=\mr{Spec}(\mb{F}_q)$. When
 $f$ is a morphism between schemes, then this equality is
 already known by Theorem \ref{LFTforschELcoeff}.
 Now, by localization triangle, the verification is local
 with respect to $\st{X}$. By using \cite[11.5]{LM} and some standard
 {\it d\'{e}vissage} argument, it suffices to treat the case where $f$
 is gerb-like. By definition of $L(-,-)$ combining with
 \cite[6.4.2]{Behr}, it is reduced to showing the theorem in the case
 $\st{X}=BG$ with a finite flat group scheme $G$ over
 $\mr{Spec}(\mb{F}_q)$ and $\ms{M}$ is $\overline{\mb{Q}}_p$. Since the
 morphism $BG_{\mr{red}}\rightarrow BG$ is a representable universal
 homeomorphism, we have $H^i(BG,\overline{\mb{Q}}_p)\xrightarrow{\sim}
 H^i(BG_{\mr{red}},\overline{\mb{Q}}_p)$, and we may assume $G$ to be
 smooth. By considering the universal torsor
 $\mr{Spec}(\mb{F}_q)\rightarrow BG$, which is finite since $G$ is, we
 get $H^i(BG,\overline{\mb{Q}}_p)=0$ for $i\neq0$. The calculation of
 $H^0$ is left to the reader.
\end{proof}

\begin{dfn}
 Let $\st{X}$ be an algebraic stack over $\mb{F}_q$, and $\ms{E}$ (resp.\
 $\mc{F}$) be an object in
 $D^{\mr{b}}_{\mr{hol}}(\st{X}/\overline{\mb{Q}}_{p,F})$ (resp.\
 $D^{\mr{b}}_{\mr{c}}(X/\overline{\mb{Q}}_\ell)$). For
 $x\in[\st{X}(\mb{F}_{q^d})]$, we denote by
 $i_x\colon\mr{Spec}(\mb{F}_{q^d})\rightarrow\st{X}$ and
 $\rho\colon\mr{Spec}(\mb{F}_{q^d})\rightarrow\mr{Spec}(\mb{F}_q)$ the
 canonical morphisms. We say that $\ms{E}$ and $\mc{F}$ are {\em
 compatible} if for any point $x\in[\st{X}(\mb{F}_{q^d})]$, the Frobenius
 trace of $\rho_+\circ i_x^+(\ms{E})$ and $\rho_*\circ i_x^*(\mc{F})$ are
 equal.
\end{dfn}

\begin{lem}
 \label{charofcompsys}
 The couple $(\ms{E},\mc{F})$ are compatible if and only if for any
 $X\in\st{X}_{\mr{sm}}$, the pull-back $\ms{E}_X$ and $\mc{F}_X$ are
 compatible.
\end{lem}
\begin{proof}
 Use \cite[6.3]{LM}.
\end{proof}

\begin{thm}
 \label{Gabberfujiwarateh}
 (i) Let $f\colon\st{X}\rightarrow\st{Y}$ be a morphism between
 c-admissible stacks. Then $f_+$, $f_!$, $f^+$, $f^!$, $\mb{D}$, and
 $\otimes$ preserve compatible systems.

 (ii) When $j\colon\st{U}\hookrightarrow\st{X}$ is an immersion of
 c-admissible stacks, $j_{!+}$ preserves compatible systems.
\end{thm}
\begin{proof}
 By Lemma \ref{charofcompsys}, the theorem for $f^+$, $\mb{D}$, $j_{!+}$
 follows from \cite[4.3.11]{AC}. We only need to show the theorem for
 $f_!$. For this, use the trace formula \ref{traceformulasta}.
\end{proof}

\end{document}